\documentclass[11pt]{amsart}

\usepackage[marginratio=1:1,height=660pt,width=480pt,tmargin=70pt]{geometry}

\usepackage{amssymb,latexsym,amsmath,amsthm,amscd}
\usepackage{mathrsfs}   
\usepackage{tikz}  
\usepackage{tikz-cd}
\usepackage{fancyhdr}
\usepackage{ifthen} 
\usepackage{xstring}
  
\usepackage{dynkin-diagrams}

\usepackage{etoolbox}
 
\DeclareFontFamily{U}{MnSymbolC}{}
\DeclareSymbolFont{MnSyC}{U}{MnSymbolC}{m}{n}
\DeclareFontShape{U}{MnSymbolC}{m}{n}{
    <-6>  MnSymbolC5
   <6-7>  MnSymbolC6
   <7-8>  MnSymbolC7
   <8-9>  MnSymbolC8
   <9-10> MnSymbolC9
  <10-12> MnSymbolC10
  <12->   MnSymbolC12}{}
\DeclareMathSymbol{\im}{\mathbin}{MnSyC}{'270}

\appto\appendix{\addtocontents{toc}{\protect\setcounter{tocdepth}{1}}}
\appto\listoffigures{\addtocontents{lof}{\protect\setcounter{tocdepth}{1}}}
\appto\listoftables{\addtocontents{lot}{\protect\setcounter{tocdepth}{1}}}
\theoremstyle{plain}
\newtheorem{theorem}{Theorem}[section]
\newtheorem{corollary}[theorem]{Corollary}

\newtheorem{lemma}[theorem]{Lemma}
\newtheorem{proposition}[theorem]{Proposition}
\newtheorem{definition}[theorem]{Definition}

\theoremstyle{remark}

\newtheorem{remark}[theorem]{Remark}

\newtheorem{example}[theorem]{Example}
\newtheorem*{example*}{Example}
\newtheorem{notation}[theorem]{Notation}

\numberwithin{equation}{section}

\usepackage{enumitem,array,rotating}  
\usepackage{color}
\usepackage{xcolor}
\definecolor{carmine}{rgb}{0.59, 0.0, 0.09}
\definecolor{mediumpersianblue}{rgb}{0.0, 0.4, 0.65}
\definecolor{persianplum}{rgb}{0.44, 0.11, 0.11}
\usepackage[colorlinks=true,
            linkcolor=persianplum, 
            urlcolor=olive,
            citecolor=mediumpersianblue, 
            backref=page]{hyperref}
\usepackage{amsmath}
\usepackage{amsfonts}
\usepackage{amssymb}
 
\newcommand{\g}{\mathfrak{g}}
\newcommand{\p}{\mathfrak{p}}
\newcommand{\q}{\mathfrak{q}}

\newcommand{\mfr}{\mathfrak{r}}

\newcommand{\mfk}{\mathfrak{k}}
\newcommand{\mfl}{\mathfrak{l}}

\newcommand{\Aut}{\mathrm{Aut}}
\newcommand{\pr}{\mathrm{pr}}

\newcommand{\ssA}{\mathsf{B}} 
\newcommand{\ssB}{\mathsf{A}}
 
\newcommand{\sA}{\mathsf{A}}
\newcommand{\sB}{\mathsf{B}}
\newcommand{\sM}{\mathsf{M}}
\newcommand{\sN}{\mathsf{N}}
\newcommand{\sQ}{\mathsf{Q}}
\newcommand{\sh}{\mathsf{h}}
\newcommand{\sq}{\mathsf{q}}
\newcommand{\sP}{\mathsf{P}}

\newcommand{\sT}{\mathsf{T}}
\newcommand{\sC}{\mathsf{C}}
\newcommand{\sE}{\mathsf{E}}

\newcommand{\sx}{\mathsf{x}}
\newcommand{\sbb}{\mathsf{b}}
\newcommand{\scc}{\mathsf{c}}
\newcommand{\sw}{\mathsf{w}}

\newcommand{\cC}{\mathcal{C}}

\newcommand{\cG}{\mathcal{G}}
\newcommand{\cH}{\mathcal{H}}
\newcommand{\cV}{\mathcal{V}}
\newcommand{\cE}{\mathcal{E}}
\newcommand{\cD}{\mathcal{D}}

\newcommand{\cP}{\mathcal{P}}
\newcommand{\cF}{\mathcal{F}}

\newcommand{\cI}{\mathcal{I}}

\newcommand{\balpha}{\boldsymbol\alpha}
\newcommand{\bbeta}{\boldsymbol\beta}
\newcommand{\beeta}{\boldsymbol\eta}
\newcommand{\bzeta}{\boldsymbol\zeta}
\newcommand{\bnu}{\boldsymbol\nu}

\newcommand{\fso}{\mathfrak{so}}
\newcommand{\fp}{\mathfrak{p}}
\newcommand{\fg}{\mathfrak{g}}

\newcommand{\fe}{\mathfrak{e}}
\newcommand{\fk}{\mathfrak{k}}

\newcommand{\fv}{\mathfrak{v}}

\newcommand{\fy}{\mathfrak{y}}
\newcommand{\fx}{\mathfrak{x}}
\newcommand{\fgl}{\mathfrak{gl}}
\newcommand{\ri}{\mathrm{i}}

\newcommand{\rO}{\mathrm{O}}

\newcommand{\rG}{\mathrm{G}}

\newcommand{\R}{\mathbb{R}}
\newcommand{\RR}{\mathbb{R}}

\newcommand{\PP}{\mathbb{P}}
\newcommand{\HH}{\mathbb{H}}
\newcommand{\SSS}{\mathbb{S}}

\newcommand{\ba}{\mathbf{a}}
\newcommand{\bb}{\mathbf{b}}
\newcommand{\bc}{\mathbf{c}}

\newcommand{\bh}{\mathbf{h}}

\newcommand{\bw}{\mathbf{w}}

\newcommand{\bA}{\mathbf{A}}
\newcommand{\bB}{\mathbf{B}}
\newcommand{\bC}{\mathbf{C}}

\newcommand{\bF}{\mathbf{F}}

\newcommand{\bI}{\mathbf{I}}

\newcommand{\bN}{\mathbf{N}}

\newcommand{\bR}{\mathbf{R}}

\newcommand{\bT}{\mathbf{T}}

\newcommand{\bW}{\mathbf{W}}

\newcommand{\bq}{\mathbf{q}}
\newcommand{\tg}{\tilde{g}}

\newcommand{\tcC}{\tilde\cC}
\newcommand{\tcE}{\tilde\cE}
\newcommand{\tcV}{\tilde\cV}
\newcommand{\tcH}{\tilde\cH}
\newcommand{\tcG}{\tilde\cG}
\newcommand{\tpi}{\tilde\pi}
\newcommand{\tomega}{\tilde\omega}
\newcommand{\ttheta}{\tilde\theta}
\newcommand{\tgamma}{\tilde\gamma}

\newcommand{\tOmega}{\tilde\Omega}
\newcommand{\tTheta}{\tilde\Theta}
\newcommand{\tphi}{\tilde\phi}
\newcommand{\txi}{\tilde\xi}
\newcommand{\tpsi}{\tilde\psi}
\newcommand{\ttau}{\tilde\tau}
\newcommand{\tnu}{\tilde\nu}
\newcommand{\trho}{\tilde\rho}
\newcommand{\tM}{\tilde M}
\newcommand{\tS}{\tilde S}

\newcommand{\tI}{\tilde I}
\newcommand{\tscD}{\cD}

\newcommand{\tE}{\tilde E}
\newcommand{\tV}{\tilde V}
\newcommand{\tbh}{\tilde{\mathbf{h}}}
\newcommand{\tlambda}{\tilde \lambda}

\newcommand{\w}{{\,{\wedge}\;}}
\newcommand{\uo}{\underline 1}
\newcommand{\ud}{\underline 2}
\newcommand{\ut}{\underline 3}
\newcommand{\uz}{\underline 0}

\newcommand{\ub}{\underline b}
\newcommand{\uc}{\underline c}
\newcommand{\bara}{\bar a}
\newcommand{\barb}{\bar b}
\newcommand{\barc}{\bar c}
\newcommand{\bard}{\bar d}

\newcommand{\exd}{\mathrm{d}}
\newcommand{\ts}{\textstyle}
\newcommand{\ve}{\varepsilon}

\newcommand{\tr}{\operatorname{tr}}

\newcommand{\Ker}{\mathrm{ker}}
\newcommand{\biw}{\bigwedge\nolimits}

\newcommand{\half}{\textstyle{\frac 12}}

\makeatletter
\renewcommand*{\p@section}{\S\,}
\renewcommand*{\p@subsection}{\S\,}
\renewcommand*{\p@subsubsection}{\S\,}
 
\makeatother
\usepackage{float}

\usepackage[small,nohug]{mydiagrams}

\begin{document}

\author{Omid Makhmali}\author{Katja Sagerschnig}

\address{\newline    Omid Makhmali (corresponding author)\\\newline
  Department of Mathematics and Natural Sciences, Cardinal Stefan Wyszy\'nski University, ul. Dewajtis 5,  Warszawa, 01-815, Poland \\\newline
\textit{Email address: }{\href{mailto:o.makhmali@uksw.edu.pl}{\texttt{o.makhmali@uksw.edu.pl}}}\\\newline
Departamento de Geometr\'ia y Topolog\'ia and IMAG, Universidad de Granada, Granada 18071, Spain \\\newline
    \textit{Email address: }{\href{mailto:omakhmali@ugr.es}{\texttt{omakhmali@ugr.es}}}\\\newline
Department of Mathematics and Statistics, UiT The Arctic University of Norway, Troms\o\  90-37,Norway \\\newline\newline 
    Katja Sagerschnig\\\newline
Center for Theoretical Physics, PAS, Al. Lotnik\'ow 32\slash46, 02-668 Warszawa, Poland \\\newline
\textit{Email address: }{\href{mailto:katja@cft.edu.pl}{\texttt{katja@cft.edu.pl}}}
 }

\title[]
          {Parabolic quasi-contact cone structures\\ with an infinitesimal symmetry} 
\date{\today}

\begin{abstract} 
  We interpret the property of having an infinitesimal symmetry as a variational property in certain geometric structures. This is achieved by establishing a one-to-one correspondence between a class of \emph{cone structures} with an infinitesimal symmetry and geometric structures arising from certain systems of ODEs  that are variational. Such cone structures include pseudo-Riemannian conformal structures and distributions of growth vector (2,3,5) and (3,6). The correspondence is obtained via symmetry reduction and \emph{quasi-contactification}. Subsequently, for each class of such cone structures we provide invariant conditions that imply more specific properties, such as having a null infinitesimal symmetry, being foliated by null submanifolds,  or having reduced holonomy to the appropriate contact parabolic subgroup. As an application of our results we give an alternative proof of the variationality of chains in CR geometry.     
\end{abstract}


\maketitle
  
\vspace{-.5 cm}

\setcounter{tocdepth}{2}
\tableofcontents

\section{Introduction}
\label{sec:introduction}

The main purpose of this  article is to give a  characterization of certain geometric structures equipped with a choice of infinitesimal symmetry and, conversely, present a construction that gives rise to  all such geometric structures plus a distinguished choice of  infinitesimal symmetry.

A well-known  example of such local characterizations is the one-to-one correspondence between  integrable CR structures with an infinitesimal symmetry and K\"ahler metrics on the (local) leaf space of the infinitesimal symmetry \cite{Webster-CR2, Webster-CR}.
 The CR structure is flat if and only if the K\"ahler metric is Bochner-flat. Bochner-flat K\"ahler metrics \cite{Bochner,Bryant-Bochner} are  a particular case of the so-called \emph{special symplectic connections} studied in \cite{CS-special}, which are induced on the (local) leaf space of an infinitesimal symmetry for  flat contact parabolic structures.
 
More recently, in a series of articles \cite{CS-cont0, CS-cont1,CS-cont2}, the authors  carried out a similar characterization for \emph{contact parabolic structures} with an infinitesimal symmetry and showed that they are in one-to-one  correspondence with what they refer to as \emph{parabolic conformally symplectic} structures (PCS-structures) on the (local) leaf space of the infinitesimal symmetry. Additionally, in \cite{CS-cont3}, the authors used their characterization to explain and generalize the  construction in \cite{EG-BGG} of a family of differential complexes  on the complex projective space.

 As an example that will appear later in this article, consider a Riemannian manifold $M$ with metric $g.$ The Lorentzian conformal  structure $[\tilde g],$ where $\tilde g=g-(\exd t)^2,$ on $M\times\RR$ has a conformal Killing field given by $\frac{\partial}{\partial t}.$ Moreover,  $[\tilde g]$ is conformally flat if and only if $g$ has constant sectional curvature. However, this is not the only way to obtain (flat)  conformal structures with a conformal Killing field. One of the objectives of this paper is to give a local characterization of   all pseudo-Riemannian conformal  structures with a choice of conformal Killing field as a geometric structure on the leaf space of the conformal Killing field. Subsequently, we   also characterize all such conformal structures with additional properties such as having a null conformal Killing field, being foliated by null hypersurfaces, having reduced conformal holonomy, or being flat.

\subsection{Main ingredients}
\label{sec:main-observations}

In this  section we would like to explain the main ingredients and steps in our characterization of parabolic quasi-contact cone structures with an infinitesimal symmetry by taking the case of a pseudo-Riemannian conformal structures on an $(n+1)$-dimensional manifold $\tM$ with a conformal Killing field $\xi$.

The first step is to identify the structure on the \emph{sky bundle} of the conformal structure, i.e. the $2n$-dimensional bundle of projectivized null cones, which is denoted by $\tnu\colon\tcC\to\tM.$ The sky bundle has two  foliations one of which is the lift of null geodesics and the other one is the sky at each point, i.e. the fibers $\tcC_x=\tnu^{-1}(x)$ where $x\in\tM.$ The tangent directions to these foliations give distributions of rank 1 and $n-1$ at each point, denoted by $\tcE$ and $\tcV,$ respectively. The Lie bracket of these distributions is a maximally non-integrable corank 1 distribution $\tcH\subset T\tcC.$ Note that if $(x;[y])=:p\in\tcC\subset\PP T\tM$ then $\tnu_*\tcH_{p}\subset T_x\tM$ is the affine tangent space to the null cone at $x$ along the line $\mathrm{span}\{y\}\subset T_xM.$
 
The second step  is to show that  $\xi$ lifts to an infinitesimal symmetry of  the sky bundle which is  transverse to $\tcH$ almost everywhere.  Thus,  at each point of the local leaf space of the trajectories of  $\xi,$ denoted by $\cC,$ the tangent space can be identifies with $\tcH.$ Moreover, the distributions $\cV:=\pi_*\tcV$ and $\cE:=\pi_*\tcE$ are well-defined on $\cC$. Consequently, one obtains that $\nu\colon\cC\to M$ defines a \emph{path geometry} where $M$ is the $n$-dimensional local leaf space of $\cV.$ Classically, a path geometry on a manifold is a family of paths on a manifold with the property that along each direction in the tangent space of each point there passes a unique path. However, we use a generalized notion of path geometry referred to as \emph{generalized path geometry} (e.g. see \cite[Section 4.4.4]{CS-Parabolic}) in which   $\cC$ can be identified locally as an open subset of $\PP TM$  and is foliated by curves.

The third step is to show that such induced path geometries are \emph{variational}, i.e. the paths are geodesics of a \emph{generalized (pseudo-)Finsler structure.} The main observation involves the existence of a closed 2-form of maximal rank on $\cC.$ The existence of such 2-form defines what we refer to as a  \emph{conformally quasi-symplectic structure}  on $\cC$ whose existence is a ``residue'' of  the exterior derivative of the 1-form on $\tcC$ that annihilates $\tcH.$ On the other hand, the projective second fundamental form of skies $\tcC_x\subset\PP T_x\tM$  descends to $\cC$ and defines a conformal class of bundle metrics $[\bh]\subset\mathrm{Sym}^2\cV^*.$ This motives the definition of an \emph{orthopath structure}, i.e. a path geometry augmented with a non-degenerate symmetric bilinear form  $[\bh]\subset\mathrm{Sym}^2\cV^*.$ Consequently, an orthopath geometry is called  variational if it has a \emph{compatible} conformally quasi-symplectic structure. It turns out that only variational orthopath geometries satisfying certain invariant condition, which is the vanishing of a cubic tensor, can arise from conformal structures with an infinitesimal symmetry.

Lastly, we provide an inverse to this construction, referred to quasi-contactification, and therefore, establish a one-to-one correspondence between pseudo-conformal  structure with an infinitesimal symmetry and variational orthopath geometries for which a fundamental cubic invariant vanishes. We interpret the fundamental cubic invariant as the trace-free part of the \emph{Cartan torsion} for the divergence equivalence class of generalized (pseudo-)Finsler structures which is canonically  associated to a variational orthopath geometry.

Figure \ref{fig:conformal} is a visual attempt to demonstrate the main steps discussed above. Note that the leaf space of $\xi$ can be identified with any hypersurface that is transverse to the trajectories of $\xi.$ As a result, the induced variational orthopath geometry is  defined on all  transverse hypersurfaces.
\begin{center}
\begin{figure}[h]
  \includegraphics[width=.7\linewidth]{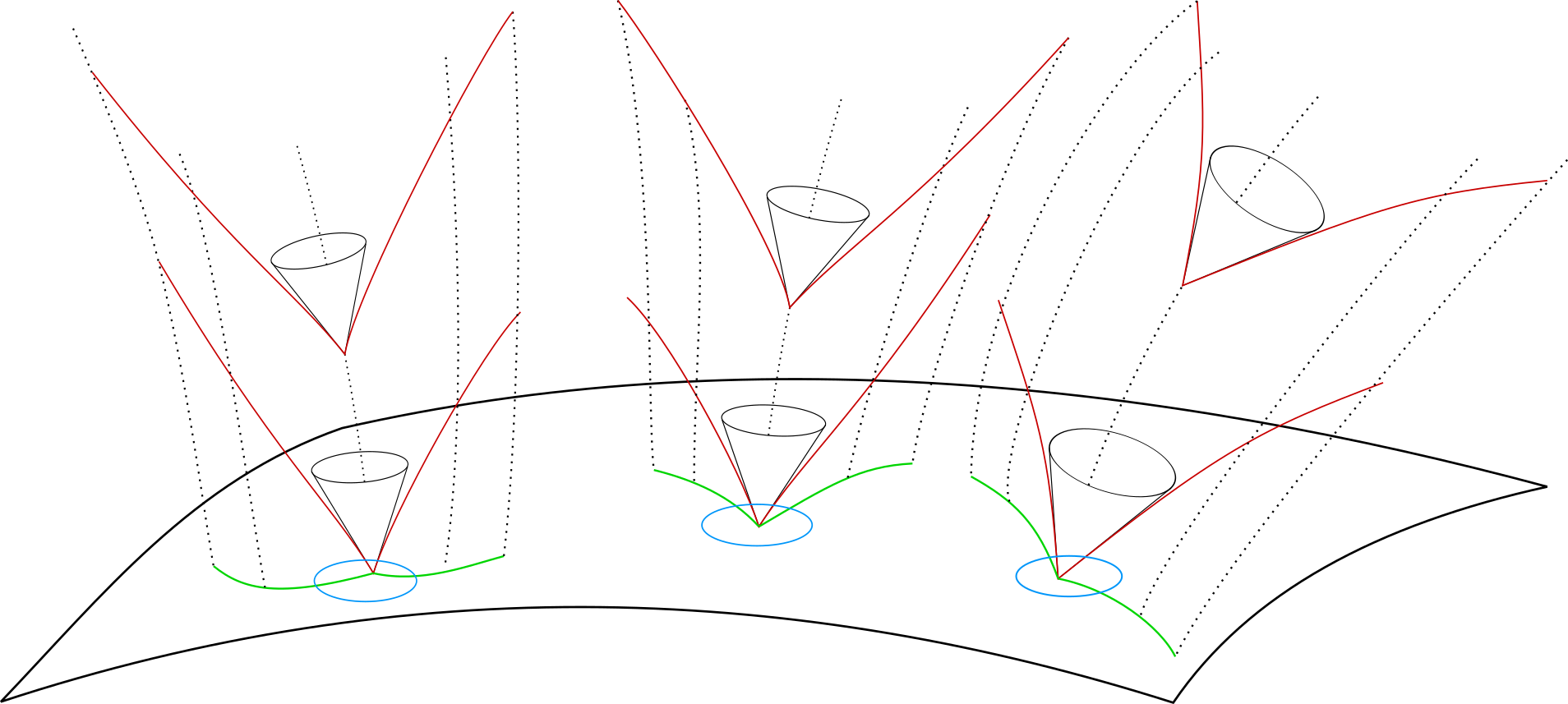}
  \caption{A schematic sketch that shows how null geodesics of a conformal Lorentzian structure with a time-like conformal Killing field $\xi$ induce distinguished paths on any transverse hypersurface. The dotted curves are the trajectories of $\xi.$ The red curves are null geodesics whose projection along the dotted curves give the distinguished  green curves on the hypersurface. The blue closed curves on the hypersurface  represent the unit  vectors of an induced Finsler structure on the hypersurface in the tangent space of that point. The green curves are the geodesics of this Finsler structure. Here only the future null cones and future null geodesics  are depicted. }
  \label{fig:conformal}
\end{figure}
\end{center}

The class of geometric structures for which the main steps discussed above can be carried out 
are referred to as \emph{parabolic quasi-contact cone structures}. As will be discussed, they  share characteristic features of pseudo-Riemannian conformal structures, which allows us to give a universal treatment of the symmetry reduction and quasi-contactification construction outlined above in the case of conformal structures.

Recall that a (smooth) \emph{cone structure} on a manifold $M$ is (locally)  a sub-bundle $\cC\subset\PP TM$ whose fibers, $\cC_x\subset \PP T_xM,$ are  connected and smooth projective submanifolds of the same dimension   e.g. see  \cite{Hwang-Survey, HN-cone}. A cone structure is called \emph{quasi-contact} if the bundle of projectivized cones, $\cC\to M,$ has a natural  \emph{quasi-contact structure}, i.e. a maximally non-integrable hyperplane distribution on an even-dimensional manifold, also known as an \emph{even-contact structure}. The class of \emph{parabolic quasi-contact cone structures} is defined, in Definition \ref{def-parabolicquasicontact}, as those quasi-contact cone structures that correspond to a regular and normal \emph{parabolic geometry}, i.e.   Cartan geometries modeled on homogeneous spaces of the form $G/P$, where $G$ is a semisimple Lie group and $P\subset G$ a parabolic subgroup. The  Cartan geometric description will be essential for the curvature analysis leading to the main results  in \ref{sec:curvatuer-analysis}.

It turns out, in Proposition \ref{prop-qcontactgrading}, that the class of parabolic quasi-contact cone structures is  comprised of a 
generalization of pseudo-Riemannian conformal structures,
called  \emph{causal structures}, where the null cone at each tangent space does not have to  be quadratic but can be instead any cone with Gauss map of maximal rank, see Definition \ref{def-causal},
and  certain types of  non-integrable distributions. The types of distributions are  generic rank 2 distributions in dimension 5, commonly referred to as $(2,3,5)$ distributions, and generic rank 3 distributions in dimension 6, commonly referred to as $(3,6)$ distributions. In both  cases where we start with a distribution $\mathcal{D}$, the sky bundle is replaced by the bundle $\tilde{\cC}:=\mathbb{P}(\mathcal{D})\to\tilde{M}$ whose fiber at each point $x\in \tilde{M}$ is  $\mathbb{P}\mathcal{D}_x$. The role of null geodesics is replaced by so-called \emph{abnormal curves} of the distributions.  Any infinitesimal symmetry of the distribution $\mathcal{D}$ can be naturally lifted to $\tilde{\cC}$. Locally, in a neighborhood of a point where the symmetry is transverse to the quasi-contact distribution on $\tilde{\cC}$, one has a submersion onto a leaf space $\cC$. Starting with a $(2,3,5)$ distribution, this leaf space has the induced structure of a variational scalar 4th order ODE up to contact transformations, and starting with a $(3,6)$ distribution, the induced structure on $\cC$ generalizes that of a variational  pair of 3rd order ODEs up to point transformations. 
Conversely, the inverse construction to the symmetry reduction   shows that any variational  scalar 4th order ODE and variational pair of 3rd order ODEs can be quasi-contactified to a $(2,3,5)$ distribution and $(3,6)$ distribution, respectively.

Although the \emph{ODE geometries} above are easy to describe, in order to capture their essential underlying features   and avoid a purely case-by-case study, we consider \emph{parabolic almost conformally quasi-symplectic structures,} which we refer to as \emph{PACQ structures}. 
 These structures are defined, in Definition \ref{def-PACQ},  in terms of their tangent bundle filtration and their structure group, i.e. the  group of filtration-preserving automorphism. We define their tangent bundle filtration to be that of  parabolic quasi-contact cone structures   truncated in the last step and their structure group to be that of  parabolic quasi-contact cone structures. The name \emph{almost conformally quasi-symplectic structure}, as shown in  Proposition \ref{prop-lineb}, has to do with the fact that  such manifolds are odd-dimensional and  naturally carry  the conformal class of a 2-form of  maximal rank.

\subsection{Outline of the article and main results} 

In \ref{sec:quasi-contact-cone} we introduce the class of parabolic quasi-contact cone structures and their characteristic properties.
The section starts with a brief review of general theory of filtered manifolds, Cartan geometries and  parabolic geometries in \ref{sec:background}. It then provides  definitions of $(2,3,5)$ distributions, $(3,6)$ distributions and causal structures in \ref{sec-defquasi}.
In order to view $(2,3,5)$ distributions, $(3,6)$ distributions,  indefinite conformal structures and more general causal structures in a uniform way,  we lift the geometric structure from the original manifold $\tilde{M}$ to the total space $\tilde{\cC}\subset\mathbb{P}\tilde{M}$ of the  bundle of projectivized cones.  This is done  in \ref{sec-descriptionofquasicontactcone} using the canonical parabolic geometries associated to these structures and the correspondence space construction  \cite{Cap-corr}.
More precisely, Proposition \ref{prop-qcontactcone} shows that $(2,3,5)$ distributions, $(3,6)$ distributions and causal structures are locally equivalent to regular and  normal parabolic geometries $(\tilde{\mathcal{G}}\to\tilde{\mathcal{C}},\tilde{\psi})$ of type $(\g,Q)$,
where $(\g,Q)$ is one of the types $(G_2,P_{12})$, $(B_3,P_{23})$, $(A_3,P_{123})$, $(B_n,P_{12})$  for $n\geq3$, and $(D_n,P_{12})$ for $n\geq 4$. 
Since for these types of parabolic geometries a certain Lie algebra cohomology condition is satisfied, any geometry from this class is completely determined by a bracket-generating distribution $T^{-1}\tilde{\cC}\subset T\tilde{\cC}$. We then define parabolic quasi-contact cone structures in terms of these types of bracket generating distributions, Definition \ref{def-parabolicquasicontact},
 and subsequently characterize the symbol algebras  $\q_{-}=\q_{-k}\oplus\cdots\oplus\q_{-1}$ of these distributions in Proposition \ref{prop-qcontactgrading}.
In particular, this identifies parabolic quasi-contact cone structures with a subclass of the non-rigid parabolic Monge geometries considered in \cite{ANN-Monge}. 

In  \ref{sec:conf-quasi-sympl} we introduce a class of geometric structures that we call \emph{parabolic almost conformally  quasi-symplectic} (PACQ)  structures, following the terminology in \cite{CS-cont1}. The structures are introduced as certain types of filtered $G$-structures in \ref{sec:PACQ-structure-def}. A PACQ structure on an odd dimensional manifold $\cC$ is comprised of a bracket-generating distribution $T^{-1}\cC\subset T\cC$
 together with an  additional reduction of structure group  (see Definition \ref{def-qsymplectic}).
The symbol algebras $\mfk_{-}$ of the distributions $T^{-1}\cC$ are quotients of the symbol algebras  $\q_{-}$ of parabolic quasi-contact structures introduced in \ref{sec:quasi-contact-cone}, namely, one has $\mfk_{-}\cong \q_{-}/\q_{-k}$.
In more concrete terms, the class of PACQ structures includes: 
\begin{enumerate}
\item  geometries of scalar fourth order ODEs up to contact transformations, 
\item pairs of third order ODEs up to point transformations, 
\item \emph{orthopath geometries}, which are defined as   path geometries (equivalently, systems of second order ODEs up to point transformations) together with an additional geometric data that can be  viewed as a conformal class $[\bh]$ of symmetric bundle metrics on the vertical bundle $\mathcal{V}$ of the projection $\mathbb{P}TM\to M$.
\end{enumerate}

Our first theorem,  Theorem \ref{thm-CartanConnection} stated  in \ref{sec:norm-cond}, establishes an equivalence of categories between
 \[\left\{
 \begin{array}{c}
  \mbox{parabolic almost conformally}\\ \mbox{  quasi-symplectic  structures}
  \end{array}
  \right\}\leftrightarrow \left\{ \begin{array}{c}
  \mbox{regular and  normal Cartan geometries}\\
  \mbox{of a particular type $(\mfk,L)$}\end{array}
  \right\},\]
 where the  Lie algebra and Lie group pair $(\mfk,L)$ is given in Lemma \ref{lem-kL}.
 The proof of the theorem uses results on constructions of canonical Cartan connections for filtered $G$-structures (\cite{Morimoto-Filtered}, \cite{Cap-Cartan}) and follows in most parts the construction in  \cite{CDT-ODE} for systems of ODEs, which covers the $(G_2,P_{12})$ and $(B_3,P_{23})$ cases. We derive our  normalization condition (in particular, the inner product used to define it) from the normalization condition for parabolic geometries.

In \ref{sec:spec-conf-quasi} we recall the definition of an almost conformally  quasi-symplectic structure, which is given by  a line bundle $\ell\subset \biw^2T^*\cC$ such that each element $\rho_x\in\ell_x$ has maximal rank, or equivalently, one-dimensional kernel. We show that a PACQ structure may be equivalently described as a bracket generating distribution together with a suitably   compatible almost conformally quasi-symplectic structure.
Next we investigate the condition that the bundle $\ell\subset \biw^2 T^*\cC$ has, locally, closed sections.  PACQ structures with this property will be called \emph{parabolic conformally quasi-symplectic} (PCQ) structures.
Using  general results on the inverse problem of the calculus of variations  for ODEs (\cite{AT-Book, Fels-ODE}), we then show
 that for PACQ structures with underlying geometries of ODEs, the existence of closed $2$-forms belonging to $\ell$ means that the corresponding ODE geometries are variational, i.e. they arise from the Euler-Lagrange equations of a Lagrangian, see Theorem \ref{prop-variational}. 
 
In \ref{sec:quasi-cont-reudct} we discuss the relationship between the two types of geometric structures studied in \ref{sec:quasi-contact-cone} and \ref{sec:conf-quasi-sympl}  via symmetry reduction and quasi-contactification.
We show the following.
\newtheorem*{thm1}{\bf Theorem \ref{thm-quasicont}}
\begin{thm1}
  \textit{
Given a  parabolic quasi-contact cone structure   with a transverse infinitesimal symmetry $\xi\in\mathfrak{X}(\tilde{\mathcal{C}})$, the symmetry determines  a local leaf space $\pi\colon\tilde{\mathcal{C}}\to \mathcal{C}$ and a canonical PCQ structure  on $\cC$ such that the filtration on $T\tilde{\cC}$ projects to the filtration on $T\cC$ and  the quasi-contact structure on $\tilde{\cC}$ induces the canonical conformally quasi-symplectic structure $\ell\subset\biw^2T^*\mathcal{C}$.
 Conversely, any PCQ structure  can be locally realized as the structure induced on the local leaf space of a parabolic quasi-contact cone structure endowed with a   transverse infinitesimal symmetry in this way.}
\end{thm1}

Together with Theorem \ref{prop-variational} this yields the following corollary.

\newtheorem*{cor1}{\bf Corollary \ref{corr-variational}}
\begin{cor1}
  \textit{
ODE geometries arising from parabolic quasi-contact cone structures via symmetry reduction  are variational. Conversely, any variational scalar fourth order ODE, variational pair of third order ODEs, and variational orthopath geometry arises as a symmetry reduction of a quasi-contact cone structure determined by a  $(2,3,5)$ distribution, a $(3,6)$ distribution, and a causal structure, respectively.}
\end{cor1}

Interesting instances of parabolic quasi-contact cone structures with infinitesimal symmetries correspond to certain Cartan holonomy reductions. The general framework and two important types of such holonomy reductions are briefly discussed in \ref{subsec-holonomy}.
In \ref{sec:relate-betw-cartan-geometries} we relate the canonical Cartan geometry of a PCQ structure with that of its quasi-contactification.

In \ref{sec:curvatuer-analysis} we take full advantage of the Cartan geometric nature of  PCQ structures and analyze their curvature and that of their quasi-contactification. As mentioned before, the Cartan curvature of  PCQ structures of type $(\fk,L),$ where $\fk$ is as \eqref{k_g2} and \eqref{k_so34}, have been well-studied although our normalization condition for the Cartan curvature arises from an inner product that is different from previous studies and obtained from the corresponding parabolic quasi-contact cone structure. We give explicit invariant conditions for such structures to have a conformally quasi-symplectic structure and, consequently, be variational in Propositions \ref{prop:variationality-Fels-4th-ODE} and \ref{prop:variationality-pair-3rd-order-ODEs}. A key result in this section is the solution of the equivalence problem for variational orthopath geometries  which is obtained via Cartan's reduction procedure. The theorem is   paraphrased below. 
\newtheorem*{thm2}{\bf Theorem \ref{thm:var-orthopath-equiv-prob}}
\begin{thm2}
  \textit{   Variational orthopath geometries on  $\nu\colon\cC\to M,$ where $\cC$ has   dimension $2n-1\geq 5,$ and $[\bh]\subset \mathrm{Sym^2}(\cV^*)$ has signature $(p,q),$  $p+q=n-1,$ are regular and normal Cartan geometries $(\tau\colon\cG\to\cC,\psi)$ of type $(\mfk,L)$ where  $\mfk=\fp^{op}\slash\fp_{-2},$   $L=B\times O(p,q),$ $B\subset\mathrm{GL}(2,\RR)$ is the Borel subgroup and $\fp^{op}\subset\fso(p+2,q+2)$ is the opposite contact parabolic subalgebra, as given in \eqref{k_sopq}. Their fundamental (relative) invariants are given by a trace-free weighted cubic $\bA,$ a symmetric trace-free weighted bilinear form $\bT,$ a skew-symmetric weighted bilinear form $\bN,$ and a weighted scalar function $\bq$ on $M.$
 } 
\end{thm2}
The proof is an implementation of certain appropriate reductions of path geometries  augmented with a conformal structure of bundle metrics, $[\bh].$ In fact,    every orthopath geometry and its principal bundle $\cG\to \cC$ is obtained via such a reduction as stated in Proposition \ref{prop:path-geom-red-vari-orth-geom}. The invariant $\bT$ is the pull-back of the torsion of the corresponding path geometry to $\cG$  with its indices lowered using the bilinear form $\bh.$ Subsequently, in \ref{sec:div-equiv-pseudo-Finsler-as-var-orthopath} we define the notion of divergence equivalence in (pseudo-)Finsler geometry and show that as geometric structures they are in one-to-one correspondence with  variational orthopath structures.

The main results of this section are Corollaries \ref{cor:235-from-4th-ODE-various-curv-conditions}, \ref{cor:36-from-pair-3rd-ODE-various-curv-conditions}  and \ref{cor:causal-from-orthopath-various-curv-cond}   whose proof is based on   Proposition \ref{prop-relharmonics}. We prove Proposition \ref{prop-relharmonics} in three parts given  in \ref{sec:235-4th-ODE-quasi-cont}, \ref{sec:36-quasi-cont-from-pair-3rd-ODE} and \ref{sec:general-causal-quasi-cont} in which we explicitly relate the Cartan connection of a PCQ structure and its quasi-contactification. Subsequently, we use the explicit relation between Cartan connections to describe several distinguished classes of parabolic quasi-contact cone structures in the form of Corollaries \ref{cor:235-from-4th-ODE-various-curv-conditions}, \ref{cor:36-from-pair-3rd-ODE-various-curv-conditions}  and \ref{cor:causal-from-orthopath-various-curv-cond}. In particular, the following is a short version of Corollary \ref{cor:causal-from-orthopath-various-curv-cond} tailored to the case of pseudo-conformal structures.
\newtheorem*{cor2}{\bf Corollary \ref{cor:causal-from-orthopath-various-curv-cond}}
\begin{cor2}
\textit{There is a one-to-one correspondence between pseudo-Riemannian conformal  structures of signature $(p+1,q+1)$ with a conformal Killing field and  variational orthopath geometries of signature $(p,q)$  satisfying the invariant condition $\bA=0$ on the leaf space of the integral curves of the infinitesimal symmetry. Moreover, one has the following.
  \begin{enumerate}
  \item   The conformal Killing field is null if and only if  $\bA=0$ and $\bq=0$.  
      \item   The conformal Killing field is null and its orthogonal  hyperplane distribution is integrable if and only if  $\bA=0,\bN=0,\bq=0$ hold. In this case, the orthopath geometry defines a variational projective structure   on $M$.
  \item A variational orthopath geometry descends to a Cartan geometry $(\tau_1\colon\cG\to M,\psi)$ of type $(\fk,\fx),$  if and only if  the invariant conditions  $\bA=0,\bN=0,\bq=0,$ and $\sT_{ab;\uc}\ve^{bc}=0$ for all $1\leq a\leq n$ are satisfied. In this case Cartan holonomy of the corresponding pseudo-conformal structure is  a proper subgroup of the  contact parabolic subgroup $P\subset \mathrm{SO}(p+2,q+2).$
       \item Variational orthopath geometries whose quasi-contactification is the flat pseudo-conformal structures in dimension $n+1$ satisfy $\bT=0$ and $\bA=0$. They are in one-to-one correspondence with the $\lfloor\half({n+1})\rfloor$-parameter family of $(2n-2)$-dimensional torsion-free  $\mathrm{SL}(2,\RR)\times \mathrm{CO}(p,q)$-structures. 
  \end{enumerate}}
\end{cor2}
By our discussion in  \ref{sec:div-equiv-pseudo-Finsler-as-var-orthopath}, it follows that the one-to-one correspondence involving variational orthopath geometries   extends the  observation made in \cite{CJS-Finsler} relating Randers metrics to stationary conformal structures of Lorentzian signature where the leaf space $\cC$ is identified with a transverse hypersurface. Analogously, we have extended the result in \cite{DZ-DivEquiv} to the case of (2,3,5)-distributions with an infinitesimal symmetry (see Remark \ref{rmk:235-Monge-Lagrangian}.) 

Examples of PCQ structures with various properties are given in \ref{sec:local-form-invar}, \ref{sec:36-local-form-invar} and \ref{sec:general-causal-local-form-invar} and local generalities of certain classes are found. In particular, in \ref{sec:local-form-invar} we describe all (2,3,5)-distributions with a \emph{null} infinitesimal symmetry as the quasi-contactification of the Euler-Lagrange equations  of a class  of second order Lagrangian  expressed parametrically as \eqref{eq:4th-ODE-Lagrangian-null-symm}. Similarly, in \ref{sec:36-local-form-invar} we show how all (3,6)-distributions with \emph{null} infinitesimal symmetry arise as the quasi-contactification of variational pairs of third order ODEs. Moreover, in    \ref{sec:general-causal-local-form-invar} we show that conformal structures obtained from quasi-contactifying the orthopath geometry of a Riemannian metric $g$ on a manifold $M$ coincides with the  standard product metric of Lorentzian signature on the direct product $M\times \RR,$  as mentioned at the beginning of this section.  More generally, in \ref{sec:general-causal-local-form-invar}, starting from a (pseudo-)Finsler structure with \emph{Cartan torsion} $\bI,$ we find the cubic form $\bA$  for its associated variational orthopath geometry and the cubic form $\bF$ for the causal structure obtained from its quasi-contactification.  A relation between these cubic forms via the affine, centro-affine and projective geometry of hypersurfaces will be given in Remark \ref{rmk:affine-centroaffine-Fubini}. As a byproduct of expressing the invariants of an orthopath structure in terms of the invariants of a (pseudo-)Finsler structure, as done in \ref{sec:general-causal-local-form-invar}, we derive  several curvature conditions for  (pseudo-)Finsler structures that are invariant under divergence equivalence relation.

As an application of our results, we study the variationality of chains in partially integrable almost CR structures in \ref{sec:vari-chains-cr-orthopath}. We show the following theorem which extends the result obtained in \cite{CMMM-CR}.
\newtheorem*{thm3}{\bf Theorem \ref{thm-chains}}
\begin{thm3}
\textit{  The canonical orthopath geometry defined by the chains of a non-degenerate partially integrable almost  CR structure is variational if and only if the CR structure is integrable. Such variational orthopath structures  satisfy the invariant conditions $\bA=0$ and $\bq=0$.}
\end{thm3}
We relate the quasi-contactification of the  orthopath geometry of chains to the well-known Fefferman conformal structure  \cite{Fefferman-CR}. In the spirit of \cite{Graham-Sparling}, it would be interesting to give a characterization of variational orthopath geometries defined by the chains of CR structures.  Finally, in Appendix we provide the full Cartan curvature for PCQ structures.

\subsection{Conventions} 

In this article our consideration will be over smooth real manifolds, although most results remain valid in the complex setting.   We denote a Cartan geometry of type $(\fg,P)$  as $(\cG\to M,\psi)$ where $\cG\to M$ is a principal $P$-bundle and $\psi$ is the corresponding Cartan connection. The two main classes of geometric structures considered in this paper are parabolic quasi-contact cone structures and parabolic conformally quasi-symplectic structures. Since we show a one-to-one correspondence between these two classes,   as Cartan geometries the former is denoted by $(\tcG\to\tcC,\tpsi)$ and the latter by $(\cG\to\cC,\psi).$ 
Parabolic quasi-contact cone structures are equipped with a pair of   distributions given by a line field $\tcE\subset T\tcC$ and an integrable distribution $\tcV\subset T\tcC$ that is transverse to $\tcE$. Consequently, we define (local) leaf spaces $\tnu\colon\tcC\to\tM=\tcC\slash\cI_{\tcV}$  and $\tlambda\colon\tcC\to\tS=\tcC\slash\cI_{\tcE}$ where $\cI_{\tcV}$ and $\cI_{\tcE}$ denote the foliations of $\tcC$ by  the integral manifolds of $\tcV$ and $\tcE,$ respectively. Similarly, any parabolic conformally quasi-symplectic  structure is equipped with a pair of transverse  distributions which are the line field $\cE$ and an integrable distribution $\cV$. We define (local) leaf spaces $\nu\colon\cC\to M=\cC\slash\cI_{\cV}$  and $\lambda\colon\cC\to S=\cC\slash\cI_{\cE}$ where $\cI_{\cV}$ and $\cI_{\cE}$ denote the foliations of $\cC$ by the integral manifolds of $\cV$ and $\cE,$ respectively.  When considering the leaf space of a foliation, we always restrict to sufficiently small open sets where the quotient manifold is smooth and Hausdorff. Since we consider parabolic quasi-contact cone structures with an infinitesimal symmetry, we have the natural projection $\pi\colon\tcC\to\cC$ where $\cC$ is the leaf space of the integral curves of the infinitesimal symmetry.  The following diagram shows the main fibrations appearing in this article.
\begin{equation}
  \label{eq:AllFib-FINAL}
  \begin{gathered}
  \begin{diagram}
    &&         \tcC                   &&    \\
     \tM & \ldTo(2,1)^\tnu    & \dTo(0,2)_\pi  &\rdTo(2,1)^\tlambda & \tS        \\
    &&         \cC                   &&    \\
    M & \ldTo(2,1)^\nu    &   &\rdTo(2,1)^\lambda & S        
            \end{diagram}\\ \\ 
             \textsc{Diagram }\text{1 : Local fibrations for a parabolic quasi-contact cone structure}\\\text{ with an infinitesimal symmetry and its corresponding PCQ structure}
  \end{gathered}         
\end{equation}
Given any PCQ structure $(\cG\to\cC,\psi)$, the semi-basic 1-forms in $\psi$ are always denoted as $\omega^i$ and $\theta^i.$ Given a coframe $(\omega^i,\theta^i)$ the corresponding frame is denoted as $(\tfrac{\partial}{\partial\omega^i},\tfrac{\partial}{\partial\theta^i}).$ The  \emph{coframe derivatives} of a function $f\colon\cG\to\RR$ are defined as
\[f_{;i}=\tfrac{\partial}{\partial\omega^i}\im\exd f,\quad f_{;\underline i}=\tfrac{\partial}{\partial\theta^i}\im\exd f,\quad f_{; \underline i j } = \tfrac{\partial}{\partial\omega^j}\im\exd f_{;\underline i},\quad f_{; \underline{ij} } = \tfrac{\partial}{\partial\theta^j}\im\exd f_{;\underline i},\quad f_{; {ij} } = \tfrac{\partial}{\partial\omega^j}\im\exd f_{;i} \]
where $0\leq i\leq n.$  This notation will be used in \ref{sec:curvatuer-analysis} to give invariant conditions that define certain distinguished classes of PCQ structures.  The structure equations and Cartan curvature for all  PCQ structures considered in the article are given in the Appendix.

When dealing with differential ideals, the algebraic ideal generated by 1-forms $\alpha^1,\cdots,\alpha^k$ will be denoted by $\{\alpha^1,\cdots,\alpha^k\}.$ The span of $v_1,\cdots,v_k$ where $v_i$'s are  tangent vectors or differential 1-forms is denoted by $\langle v_1,\cdots,v_k\rangle.$ Derived systems of a distribution $\cD\subset T\cC$ is the distribution whose sheaf of sections if given by $\Gamma(\cD)+[\Gamma(\cD),\Gamma(\cD)]$ and, by abuse of notation, is denoted as $[\cD,\cD].$ Similarly, given two distributions $\cD_1$ and $\cD_2,$ we denote by $[\cD_1,\cD_2]$ the distribution whose sheaf of sections is   $\Gamma(D_1)+\Gamma(D_2)+[\Gamma(\cD_1),\Gamma(\cD_1)]+[\Gamma(\cD_1),\Gamma(\cD_2)]+[\Gamma(\cD_2),\Gamma(\cD_2)].$ 

We denote by $\Omega^k(\cC)$ and $\Omega^k(\cC,\fg)$ the sheaf of sections of  $\biw^k T^*\cC$ and $\biw^k T^*\cC\otimes\fg,$ respectively. The sheaf of vector fields on $\cC$ is denoted by $\mathfrak X(\cC).$  To denote symmetric product of tensors we use the symbol $\circ$ e.g. for two 1-forms $\alpha$ and $\beta$ we have $\alpha\circ\beta=\half(\alpha\otimes\beta+\beta\otimes\alpha)$ and $\alpha^k$ denotes the $k$th symmetric power of $\alpha.$

To address variationality and give examples we will need to work with various jet spaces. The symbols used as coordinates of jet spaces in \ref{sec:curvatuer-analysis} is different from what is used in \ref{sec:spec-conf-quasi}  so that the parametric expression of invariants do not appear cumbersome.  Furthermore, path geometries appearing in this article are always defined in the generalized sense and therefore we avoid using the term generalized path geometry.

\section{Parabolic quasi-contact cone structures}
\label{sec:quasi-contact-cone}
We start this section with a short review of basic notions and results on filtered $G$-structures and Cartan geometries. Then we introduce  parabolic quasi-contact cone structures. We characterize them as a class of parabolic geometries and as a class of bracket generating distributions with symbol algebras of a certain type. These descriptions reveal common features of parabolic quasi-contact cone structures that will be used in later parts of the paper. 
\subsection{Background on filtered $G$-structures and Cartan  geometries}
\label{sec:background}

A  Lie algebra $\g$ is called \emph{filtered} if it is endowed with a filtration of nested subspaces
$\g^{k}\subset\cdots\subset\g^{-p}=\g$ compatible with the Lie bracket in the sense that $[\g^i,\g^j]\subset\g^{i+j}$, where $\g^l:=\g$ for $l<-p$ and $\g^l:=\{0\}$ for $l>k$. It is called \emph{graded} if it is endowed with a vector space decomposition $\g=\g_{-p}\oplus\cdots\oplus\g_{k}$ compatible with the Lie bracket. The \emph{associated graded} Lie algebra of a filtered Lie algebra $\g$ is $\mathrm{gr}(\g):=\bigoplus_{i=-p}^{k}\g^i/\g^{i+1}$
endowed with the induced Lie bracket. A graded Lie algebra determines a filtration by defining $\g^i=\g_i\oplus\dots\oplus\g_k$.

A \emph{filtered manifold} is a smooth manifold $M$ together with a filtration by sub-bundles
$$T^{-1}M\subset\cdots\subset T^{-p+1}M\subset T^{-p}M=TM$$
 that is compatible with the Lie bracket in the sense that $[\Gamma(T^{-i}M),\Gamma(T^{-j}M)]\subset\Gamma(T^{-i-j}M)$ for all $i,j>0$, where we set $T^{-i}M=TM$ for $i\geq p$.
 Given a distribution $T^{-1}M\subset TM$, its \emph{weak derived flag} is defined iteratively by taking Lie brackets with sections of $T^{-1}M$.  A distribution is called \emph{bracket generating} if  iterated Lie brackets of $T^{-1}M$ span the entire tangent space at each point.  Assuming that all elements in the flag have constant rank, one obtains a filtered manifold.

  Given a filtered manifold, the Lie bracket of vector fields induces a tensorial map on the associated graded vector bundle $\mathrm{gr}(TM)=T^{-p}M/T^{-p+1}M\oplus\dots\oplus T^{-1}M$, called the \emph{Levi bracket} $\mathfrak{L}$. The graded nilpotent Lie algebra $(\mathrm{gr}(T_xM),\mathfrak{L}_x)$  is called the \emph{symbol algebra} of a filtered manifold at a point $x\in M$. 
  If the symbol algebra of a filtered manifold is at each point isomorphic to a fixed graded Lie algebra $\mathfrak{m}$, then it has a natural graded frame bundle with structure group $\Aut_{gr}(\mathfrak{m})$ whose fibers consist of all graded Lie algebra homomorphisms $\mathfrak{m}\to\mathrm{gr}(T_xM)$.

Consider a subgroup $G_0\subset\Aut_{gr}(\mathfrak{m})$. A \emph{filtered $G_0$-structure of type $\mathfrak{m}$} is given by
  \begin{enumerate}
  \item a filtered manifold $(M,T^{-1}M\subset\cdots\subset T^{-k}M)$ whose symbol algebras form a locally trivial bundle modeled on the graded nilpotent Lie algebra $\mathfrak{m}$ and
\item a reduction of structure group of the graded frame bundle of the filtered manifold to a principal $G_0$-bundle $\tau_0\colon\mathcal{G}_0\to M$.
  \end{enumerate}

 Certain classes of filtered $G_0$-structures admit equivalent descriptions as Cartan geometries. Let $G/P$ be a homogeneous space with corresponding Lie algebras $\p\subset\g$.  A \emph{Cartan geometry} $(\mathcal{G}\to M,\psi)$ of type $(\g,P)$  is given by
a (right) principal $P$-bundle $\tau\colon\mathcal{G}\to M$ together with
 a \emph{Cartan connection} $\psi\in\Omega^1(\mathcal{G},\g)$, i.e a $\g$-valued $1$-form  on $\mathcal{G}$ with the following properties:
 \begin{itemize}
 \item $\psi$ is $P$-equivariant, i.e. $r_g^*\psi=\mathrm{Ad}_{g^{-1}}\circ\psi$ for any $g\in P$,
 \item  $\psi$ maps fundamental vector fields to their generators, i.e. $\psi(\zeta_X)=X$ for any $X\in\p$,
 \item  $\psi$ defines an isomorphism $\psi\colon T_u\mathcal{G}\to \g$ for any $u\in\mathcal{G}.$
\end{itemize}
The \emph{curvature} of a Cartan connection $\psi$ is the $2$-form $\Psi\in\Omega^2(\mathcal{G},\g)$ defined as $$\Psi(X,Y)=\mathrm{d}\psi(X,Y)+[\psi(X),\psi(Y)]\quad\mbox{for}\quad X,Y\in\Gamma(T\mathcal{G}).$$ Since it is $P$-equivariant and horizontal, it can be equivalently encoded in the curvature function $\kappa\colon\mathcal{G}\to\biw^2(\g/\p)^*\otimes\g$.

Let $\g=\g_{-p}\oplus\cdots\oplus\g_{k}$ be a graded Lie algebra
and  $\g_{-}$ 
be its negative part. We  assume that $\g_{-1}$ generates $\g_{-}$ and that there is no ideal in $\g$ that is contained in $\g^0$. Let
$$\cdots\xrightarrow{\partial}\biw^{s-1}(\g_{-})^*\otimes\g\xrightarrow{\partial}\biw^s(\g_{-})^*\otimes\g \xrightarrow{\partial}\biw^{s+1}(\g_{-})^*\otimes\g\xrightarrow{\partial}\cdots$$
be the complex that defines Lie algebra cohomology $H^s(\g_{-},\g)$ of $\g_{-}$ with values in $\g$. Explicitly, one has
\begin{equation}
\label{formuladel}
\begin{aligned}
\partial\Psi (X_0,\dots, X_{s})&=\sum_i(-1)^i[X_i,\Psi(X_0,\dots,\hat{X}_i,\dots,X_s)]
\\& 
+\sum_{i<j}(-1)^{i+j}\Psi([X_i,X_j],\dots,\hat{X}_i,\dots,\hat{X}_j,\dots,X_s)
\end{aligned}
\end{equation}
for $X_0,\cdots,X_s\in\g_{-}$, where  hats denote omission. 
Since the differential $\partial$ preserves homogeneity, the grading on $\g$ induces a grading on the cohomology spaces $H^s(\g_{-},\g)=\bigoplus_{l}H^s(\g_{-},\g)_{l}$.

Let  $P\subset \mathrm{Aut}_f(\g)$ be a subgroup of  automorphisms preserving the filtration of $\g$ that has Lie algebra $\p=\g^0=\g_0\oplus\cdots\oplus\g_k$ and assume $P$ is of split exponential type as in  \cite[Definition 4.11]{Cap-Cartan}.
 Suppose further  that for the pair $(\g,P)$ the following holds:
\begin{enumerate}
\item The Lie algebra $\g$ is the full prolongation of  $\g_{-}\oplus\g_0$. Equivalently, this means that  the cohomological condition  $H^1(\mathfrak{g}_{-}, \g)_l=\{0\}$ holds for all $l>0$.
\item  There exists a $P$-invariant subspace $$\mathcal{N}\subset\biw^2(\g/\p)^*\otimes\g$$
 such that for each $l>0$ the induced subspace $\mathrm{gr}_l(\mathcal{N})\subset \mathrm{gr}_{l}(\biw^2(\g/\p)^*\otimes\g)\cong(\biw^2\mathfrak{g}_{-}^*\otimes\g)_{l}$ is a complement to the image of $\partial:(\mathfrak{g}_{-}^*\otimes\g)_l\to (\biw^2\mathfrak{g}_{-}^*\otimes\g)_l.$ 
\end{enumerate}
A choice of subspace $\mathcal{N}$ is called a \emph{normalization condition}.
A Cartan connection of type $(\g,P)$ is  called \emph{normal} if the curvature function takes values in the distinguished subspace $\mathcal{N}$. It is called \emph{regular} if  $\kappa(\g^i,\g^j)\subset\g^{i+j+1}$ for all $i,j$. 
\begin{theorem}(\cite[Theorem 4.12]{Cap-Cartan})
\label{eqcat}
 Suppose that $H^1(\mathfrak{g}_{-}, \g)_l=\{0\}$ for all $l>0$. Then there exists an equivalence of categories between filtered $G_0$-structures of type $\mathfrak{g}_{-}$ and  regular and normal Cartan geometries of type $(\g,P)$.  
\end{theorem}

There is a well-known class of geometric structures, called parabolic geometries, where a natural normalization condition is known to exist.

Let $\g$ be a real semisimple Lie algebra. A \emph{$\vert k\vert$-grading} on $\g$ is a grading of the form
\begin{align}\label{grading}\g=\underbrace{\g_{-k}\oplus\cdots\oplus\g_{-1}}_{\g_{-}}\oplus\; \g_{0}\oplus\underbrace{\g_1\oplus\cdots\oplus\g_k}_{\g_+},  \quad [\g_i,\g_j]\subset\g_{i+j},
\end{align}
where  $\g_{-}$ is generated by $\g_{-1}$ and  no simple ideal of $\g$ is contained in $\g_0$. Then the Killing form defines isomorphisms $\g_{-i}^*\cong\g_i$. The negative part $\g_{-}$ is a nilpotent Lie algebra and the non-negative part $\p=\g^0=\g_0\oplus\g_{+}$ is a parabolic subalgebra of $\g$.  
There is a well-known correspondence between $\vert k\vert$-gradings and subsets $\Sigma\subset\Delta^0$ of simple (restricted) roots, which can be depicted by means of Satake diagrams with crosses. More precisely, a parabolic subalgebra $\p$ is represented by crosses on the nodes of the diagram corresponding to simple roots $\alpha_i$ such that the root space $\g_{\alpha_i}$ is not contained in $\p$.
For details we refer the reader to \cite{CS-Parabolic}.
Let $G$ be a Lie group with Lie algebra $\g$ and $P\subset G$ a closed subgroup  with Lie algebra $\p=\g^0$ consisting of elements that via the adjoint action preserve the filtration on $\g$. Then a Cartan geometry of type $(\g,P)$ is called a \emph{parabolic geometry}.  

To introduce a normalization condition for parabolic geometries, recall that the Killing form induces a $P$-module identification $(\g/\p)^*\cong\g_+$ and thus an isomorphism between $\biw^2(\g/\p)^*\otimes\g$ and $\biw^2\g_{+}\otimes\g$. The latter space is a chain space in the complex
$$\cdots\xrightarrow{\tilde{\partial}^*}\biw^{k+1} \g_+\otimes\g\xrightarrow{\tilde{\partial}^*}\biw^k \g_+\otimes\g \xrightarrow{\tilde{\partial}^*}\biw^{k-1}\g_+\otimes\g \xrightarrow{\tilde{\partial}^*}\cdots$$   using which one defines the Lie algebra homology $H_l(\mathfrak{g}_{+},\g)$ of the nilpotent Lie algebra $\g_{+}$ with values in $\g$. Explicitly, the 
$P$-equivariant boundary operator
   is given by 
   \begin{equation}
   \label{KostantCodif}
\begin{aligned}
\tilde{\partial}^*(Z_1\wedge\cdots\wedge & Z_k \otimes A)=\sum_{i=1}^{k}(-1)^i Z_1\wedge\cdots\wedge \hat{Z}_i\wedge\cdots\wedge Z_k\otimes[Z_i,A]\\&+\sum_{i<j}(-1)^{i+j}[Z_i,Z_j]\wedge Z_1\wedge\cdots\wedge \hat{Z}_i\wedge\cdots\wedge \hat{Z}_j\wedge\cdots\wedge Z_k\otimes A
\end{aligned}
\end{equation}
for $Z_1,\cdots,Z_k\in\g_+$, $A\in \g$,  where  hats denote omission. 
The operator $\tilde{\partial}^*$ is traditionally called the \emph{Kostant codifferential}. 
The kernel
$$\mathcal{N}=\ker(\tilde{\partial}^*)\subset\biw^2(\g/\p)^*\otimes\g$$
provides a normalization condition for parabolic geometries of type $(\g,P)$.

 To relate  $\tilde{\partial}^*$ and the Lie algebra cohomology operator $\partial$, note that as $\g_0$-modules $(\g_{-})^*\cong\g_{+}$ and thus  the Kostant codifferential defines a map
$\tilde{\partial}^*:\biw^k (\g_{-})^*\otimes\g \to\biw^{k-1}(\g_{-})^*\otimes\g .$
As in the proof of  \cite[Proposition 3.3.1]{CS-Parabolic}, one considers a positive definite inner product $\left\langle X, Y\right\rangle=-B(X,\theta(Y))$ on $\g$, where the Cartan involution $\theta:\g\to\g$ has the property that 
$\theta(\g_i)=\g_{-i}$.
It induced an inner product on the cochain spaces $\biw^k(\g_{-})^*\otimes\g$ such that the operators $\partial$ and $\tilde{\partial}^*$ are adjoint, i.e.  
$$\left\langle  \partial\phi,\psi\right\rangle=\left\langle\phi,\tilde{\partial}^*\psi\right\rangle.$$
As a result, one obtains a $\g_0$-invariant Hodge decomposition of the form
\begin{equation}\label{eqHodge}\biw^k(\g_{-})^*\otimes\g\cong\mathrm{im}(\partial)\oplus\ker(\Box)\oplus\mathrm{im}(\tilde{\partial}^*),
\end{equation} where $\Box=\partial\circ\tilde{\partial}^*+\tilde{\partial}^*\circ\partial$, the first two summands $\mathrm{im}(\partial)\oplus\ker(\Box)$ add up to $\ker(\partial)$, and the second two summands $\ker(\Box)\oplus\mathrm{im}(\tilde{\partial}^*)$ add up to $\ker(\tilde{\partial}^*)$. 
The projection  of the curvature function $\kappa$ onto 
the quotient 
\begin{equation}
\label{harmon}
\ker(\tilde{\partial}^*)/\mathrm{im}(\tilde{\partial}^*)\cong 
H^2(\g_{-},\g)\cong\ker(\Box)
\end{equation} is called the \emph{harmonic curvature} and will be denoted by $\kappa_H$. It is a complete obstruction to the local equivalence of a parabolic geometry to its homogeneous model.

%\subsection{$(2,3,5)$ distributions, $(3,6)$ distributions and causal structures}
%\label{sec-defquasi}

\subsection{Definition and descriptions of parabolic quasi-contact cone structures}
\label{sec-descriptionofquasicontactcone}\label{sec-defquasi}

We will now introduce the class of parabolic geometries that will be studied in this article. We start by recalling the definitions of two types of non-integrable distributions on manifolds of dimension $5$ and $6$, and a generalization of pseudo-Riemannian conformal geometry.

\begin{definition}
 A (2,3,5) distribution $\cD\subset T\tilde{M}$ is a rank $2$ distribution on a  smooth $5$-manifold such that $[\cD,\cD]$ is a sub-bundle of rank $3$ and $[\cD,[\cD,\cD]]=T\tilde{M}$.  
\end{definition} 
Let $\mathfrak{g}_2^*$ be the split real form of the  $14$-dimensional exceptional simple Lie algebra and $P_1\subset\mathrm{G}^*_2$ the stabilizer of a highest weight line in the $7$-dimensional fundamental representation. Then $P_1$ is a maximal parabolic subgroup corresponding to the short root $\{\alpha_1\}$.
There is a well-known equivalence of categories between $(2,3,5)$ distributions  and regular and normal parabolic geometries $(\tilde{\mathcal{G}}\to \tilde{M},\tilde{\psi})$ of type $(\mathfrak{g}_2^*,P_1)$, going back to the classical work of Cartan \cite{Cartan-235}. 

\begin{definition}
 A (3,6) distribution $\cD\subset T\tilde{M}$ is a rank $3$ distribution on a  smooth $6$-manifold such that $[\cD,\cD]= T\tilde{M}$.
\end{definition}
 Similarly, there is a well-known equivalence of categories between $(3,6)$ distributions and regular and normal parabolic geometries of type $(\mathfrak{so}(3,4),P_3)$. Here $P_3\subset \mathrm{SO}(3,4)$  is  defined as the stabilizer of a null $3$-plane in $\mathbb{R}^{3,4}$. It  is a maximal parabolic subgroup corresponding to the short root $\{\alpha_3\}$.

To introduce the next class of geometric structures, recall that a pseudo-Riemannian conformal  structure of indefinite signature  $(\tilde{M},[\tilde{g}])$ can be equivalently described in terms of its bundle of projectivized null cones, denoted by  $\tilde{\mathcal{C}}\subset \mathbb{P}T\tilde{M}$, which is also referred to as its \emph{sky bundle}.  
We will consider a generalization of pseudo-Riemannian conformal  structures by relaxing the condition that the fibers of the sky bundle are non-degenerate quadrics. In other words, the cone in each tangent space  can be defined as the null set of a (pseudo-)Finsler norm rather than an inner product. 
\begin{definition}\label{def-causal}
 A causal structure  $(\tilde{M},\tilde{\mathcal{C}})$ of signature $(p+1,q+1)$ is given by a submanifold of the projectivized tangent bundle, $\tilde{\nu}\colon\tilde{\mathcal{C}}\subset\mathbb{P}T\tilde{M}\to \tilde{M}$, with the properties that the projection $\tilde{\nu}$ restricts to a submersion to the base $\tilde{M}$ and the fibers   $\tilde{\mathcal{C}}_x:=\tilde{\nu}^{-1}(x)\subset\mathbb{P}T_x\tilde{M}$ are connected projective hypersurfaces with second fundamental form of signature $(p,q)$.
\end{definition}
 Causal structures have an equivalent description in terms of regular and normal parabolic geometries. 
For $n=p+q+2\geq 5$ they correspond to regular and normal parabolic geometries of type $(\mathfrak{so}(p+2,q+2),P_{12})$, where $P_{12}\subset\mathrm{SO}(p+2,q+2)$ is the stabilizer of a null line contained in a null $2$-plane in $\mathbb{R}^{p+2,q+2}$. For  $n=p+q+2=4$ they correspond to a subclass of regular and normal parabolic geometries of type $(\mathfrak{so}(p+2,q+2),P_{123})$ that can be characterized by the vanishing of one of the harmonic curvature components, where $P_{123}\subset\mathrm{SO}(p+2,q+2)$ is a minimal parabolic subgroup \cite{Omid-Sigma}. 

\begin{remark}\label{rmk:general-defin-causal-str}
One can give a more general definition of causal structures on $\tM^{n+1}$  as a $2n$-dimensional manifold $\tcC$ with an immersion $\iota\colon\tcC\to \PP T\tM,$ such that $\hat\tnu\circ\iota\colon\tcC\to\tM$ is a submersion with connection fibers where $\tnu\colon\PP T\tM\to \tM$ is the projection, and fibers $\tcC_x:=(\tnu\circ\iota)^{-1}(x),$ for each $x\in\tM,$  are immersed via $\iota_x\colon\tcC_x\to \PP T_x\tM$ as hypersurfaces in $\PP T_x\tM$  with second fundamental form of signature $(p,q).$  This definition gives a generalization of causal structures in the same spirit that the definition of Finsler structures has been broadened to   generalized Finsler structures in \cite{Bryant-Finsler}. See Remark \ref{rmk:generalized-causal-geometry-Finsler} for further discussion.
\end{remark}

Note that any $(2,3,5)$ distribution, $(3,6)$ distribution and causal structure on a manifold $\tilde{M}$ defines a submanifold $\tilde{\cC}\subset\mathbb{P}T\tilde{M}$ in the projectivized tangent bundle such that the projection restricts to a submersion over the base $\tilde{M}$.  For causal structures this is part of the definition and for the two types of distributions one defines $\tilde{\cC}=\mathbb{P}\cD$. In other words, the structures can be viewed as special types of so-called $\emph{cone structures}$, \cite{Hwang-Survey} and \cite{HN-cone} for a definition and results on cone structures in the context of certain (holomorphic) parabolic geometries.

As a first step towards the geometric correspondence presented in this article, we will now lift the geometric structures on $\tilde{M}$ to certain geometric structures on the total space of the bundle $\tilde{\mathcal{C}}\to \tilde{M}$ of projectivized cones, which we call \emph{parabolic quasi-contact cone structures}. This is done using their equivalent description as parabolic geometries.

Note that the regular and normal parabolic geometries  $(\tilde{\mathcal{G}}\to \tilde{M},\tilde{\psi})$  associated to $(2,3,5)$ distributions, $(3,6)$ distributions, and pseudo-Riemannian conformal structures are of type  $(\mathfrak{g},R)$, where
 $R$ is a maximal parabolic subgroup  corresponding to a single simple root $\Sigma^R=\{\alpha\}$. We denote  by $\mathfrak{r}\subset\g$ the Lie algebra of $R$.
\begin{notation}
In the following, we will be dealing with a number of different $\vert k\vert$-gradings and corresponding parabolic subalgebras of a given semisimple Lie algebra $\g$. 
We will denote by $\mfr_{-k}\oplus\cdots\oplus\mfr_k$ the grading and by $\mfr^{k}\subset\cdots\subset\mfr^{-k}$ the filtration corresponding to the parabolic subalgebra $\mfr=\mfr^0$ and use analogous notation for other parabolic subalgebras.
\end{notation} 
  Let  $Q\subset R$ be the subgroup stabilizing  the point $[v]\in\mathbb{P}(\mathfrak{r}^{-1}/\mathfrak{r})$ corresponding to the root space for $-\alpha$ and let $R\cdot[v]\subset \mathbb{P}(\mathfrak{r}^{-1}/\mathfrak{r})$ be its $R$-orbit.  Then $Q$ is a parabolic subgroup corresponding to the subset $\Sigma^Q$ given by $\alpha$ and all simple roots connected to $\alpha$ in the Dynkin diagram for $\g$, see  \cite[Lemma 2.14]{HN-cone}. Consequently, we have the following:

 \begin{lemma} 
Let $(\tilde{\mathcal{G}}\to \tilde{M},\tilde{\psi})$ be a regular and normal Cartan geometry of type $(\mathfrak{g},R)$ associated with a $(2,3,5)$ distribution, $(3,6)$ distribution, or  pseudo-Riemannian conformal structure. Let $\tilde{\cC}\subset \mathbb{P}T\tilde{M}$ be the associated submanifold of the projectivized tangent bundle and $Q\subset R$ be the parabolic subgroups introduced above. Then we have a natural identification
$$\tilde{\cC}\cong\tilde{\mathcal{G}}/Q.$$
 
\end{lemma}
\begin{proof}
The orbit space $\tilde{\mathcal{G}}/Q$ can be identified with $\tilde{\mathcal{G}}\times_{R}R/Q$. On the other hand we have an isomorphism $R/Q\cong R\cdot[v]\subset \mathbb{P}(\mathfrak{r}^{-1}/\mathfrak{r})$. Using this, one obtains an identification of $\tilde{\mathcal{G}}/Q$ with $\mathbb{P}\mathcal{D}$ in the distribution cases and with the sky bundle in the case of pseudo-Riemannian conformal structures.
\end{proof}

  Moreover, 
$(\tilde{\mathcal{G}}\to\tilde{\cC}=\tilde{\mathcal{G}}/Q,\tilde{\psi})$ is a Cartan geometry of type $(\g,Q)$, which is referred to as the \emph{correspondence space} of $(\tilde{\mathcal{G}}\to \tilde{M}=\tilde{\mathcal{G}}/R,\tilde{\psi})$ with respect to the inclusion $Q\subset R$. If $(\tilde{\mathcal{G}}\to \tilde{M},\tilde{\psi})$ is normal, then the correspondence space is  normal as well \cite{Cap-corr}. However, regularity is in general not preserved by the correspondence space construction.
Using these general facts we now show the following.
\begin{proposition}
\label{prop-qcontactcone}
Let  $(\tilde{\mathcal{G}}\to \tilde{M},\tilde{\psi})$ be a  regular and normal parabolic geometry associated with a $(2,3,5)$ distribution, $(3,6)$ distribution or causal structure. Then  $(\tilde{\mathcal{G}}\to\tilde{\mathcal{C}},\tilde{\psi})$ is a  regular and normal parabolic geometry of type $(\g,Q)$, where $(\g,Q)$ is one of the following pairs.
\begin{enumerate} 
\item 
$(\g_2^*,P_{12})$, where $P_{12}$ is the Borel subalgebra of $\mathrm{G}_2^*$,
\item 
 $(\mathfrak{so}(3,4),P_{23})$, where $P_{23}\subset\mathrm{SO}(3,4)$ is the
 stabilizer of a filtration consisting of a null $2$-plane inside a null $3$-plane in $\mathbb{R}^{3,4}$,
\item 
  $(\mathfrak{so}(p+2,q+2), P_{12})$ for $n=p+q+2\geq 5$, 
and  $(\mathfrak{so}(p+2,q+2),P_{123})$ for $n=p+q+2=4$,
 where the parabolic subgroups are the ones corresponding to causal structures as defined in Definition \ref{def-causal}.
\end{enumerate}

Conversely, any regular and normal parabolic geometry of the types (1) and (2) is locally equivalent to the correspondence space of a regular and normal parabolic geometry determined by a $(2,3,5)$ and $(3,6)$ distribution with respect to the inclusions $P_{12}\subset P_1$ and $P_{23}\subset P_3,$ respectively.

 A regular and normal parabolic geometry of type (3) is locally equivalent to a causal structure for $n\geq 5$;  for $n=4$ this holds if and only if one of the harmonic curvature components of homogeneity one vanishes. It is locally  the correspondence space of a normal pseudo-conformal  structure with respect to $P_{12}\subset P_{1}$ and  $P_{123}\subset P_1$ when $n\geq 5$ and $n=4,$ respectively,  if and only if in addition the remaining harmonic curvature component $\bF$ \eqref{eq:harmonic-inv-causal-represented} of homogeneity one vanishes.
\end{proposition}
\begin{proof}
Let $(\tilde{\mathcal{G}}\to \tilde{M},\tilde{\psi})$ be the regular and normal Cartan geometry of type $(\mathfrak{g},R)$. In case this Cartan geometry is associated to a $(2,3,5)$ and $(3,6)$ distribution, $R$ is the parabolic subgroup corresponding to $\{\alpha_1\}$ and $\{\alpha_3\}$, respectively. Therefore, $Q$ is the parabolic subgroup corresponding to $\{\alpha_1,\alpha_2\}$ and $\{\alpha_2,\alpha_3\}$, and we obtain parabolic geometries of types $(\g_2^*,P_{12})$ and $(\mathfrak{so}(3,4),P_{23})$ on $\tilde{\mathcal{C}}$, respectively. Since the correspondence spaces are automatically normal, it remains to verify that they are also regular.
 This can be easily seen if one uses the facts that any regular and normal Cartan connection of type $(\mathfrak{g}_2^*,P_1)$ and $(\mathfrak{so}(3,4),P_3)$ is torsion-free, i.e. the curvature function takes values in $\mathfrak{r}$, and that  $\mathfrak{r}\subset \q^{-1}$.

 Conversely, suppose we are given a regular and normal parabolic geometry of type (1) or (2). Using Kostant's theorem one algorithmically determines a highest weight vector in the harmonic curvature representations  $H^2(\q_{-},\q)$. In both cases there is exactly one irreducible component in positive homogeneity, which is  of homogeneity $4$ for geometries of type $(\g_2^*,P_{12})$
 and  of homogeneity $3$ for geometries of type $(\mathfrak{so}(3,4),P_{23})$. The structure of the harmonic curvature spaces further shows that the harmonic curvature function satisfies $\tilde{\kappa}_{H}(X,Y)=0$ if one of the entries $X$, $Y$ lies in $\mathfrak{r}/\q$. Then using \cite[Theorems 2.7 and 3.3]{Cap-corr}, one obtains the  local equivalence of regular and normal parabolic geometry of type $(\g,Q)$ with the correspondence space of a normal parabolic geometry of type  $(\g,R)$. It is a matter of checking that this geometry is regular as well.

 The statements about case (3) are shown in \cite{Omid-Sigma}.
\end{proof}

Now consider a $\vert k\vert$-graded Lie algebra
 \begin{align}\label{gradingq}
\g=\underbrace{\q_{-k}\oplus\dots\oplus\q_{-1}}_{\q_{-}}\oplus\underbrace{\q_0\oplus\q_1\oplus\dots\oplus\q_k}_{\q}
\end{align}  with its non-negative part being 
one of the  parabolic subalgebras $\q$  from Proposition \ref{prop-qcontactcone}. It is known that for most gradings of simple Lie algebras and, in particular, for the graded Lie algebras \eqref{gradingq} the grading component $\q_0$ is isomorphic to the Lie algebra of  derivations of $\q_{-}$ that preserve the grading, see  \cite{Yamaguchi-Simple} or \cite[Proposition 4.3.1]{CS-Parabolic}. Defining $Q=\Aut_f(\g)$ to be the subgroup of all automorphisms of $\mathfrak{g}$ that preserve the filtration corresponding to $\q$, it then follows that  regular and normal parabolic geometries of type $(\g,Q)$  are equivalent to filtered manifolds
\begin{equation}T^{-1}\tilde{\mathcal{C}}\subset T^{-2}\tilde{\mathcal{C}}\subset\dots \subset T^{-k+1}\tilde{\mathcal{C}}\subset T^{-k}\tilde{\mathcal{C}}=T\tilde{\mathcal{C}}
\end{equation}
 whose symbol algebras form a locally trivial bundle modeled on  $\q_{-}$, see \cite[Proposition 4.3.1]{CS-Parabolic}. Such filtered manifolds are generated by certain types of bracket generating distributions.  
 This leads us to define parabolic quasi-contact cone structures as the bracket generating distributions corresponding to regular and normal parabolic geometries of type $(\g,Q)$.
 
 \begin{definition}\label{def-parabolicquasicontact}
A parabolic quasi-contact cone structure of type $\q_{-}$ on an even-dimensional manifold $\tilde{\mathcal{C}}$ is a bracket-generating distribution $T^{-1}\tilde{\mathcal{C}}\subset T\tilde{\mathcal{C}}$ whose symbol algebras form a locally trivial bundle modeled on the  negative part $\q_{-}$  in a grading \eqref{gradingq}. 
\end{definition}
We then have the following immediate corollary of Proposition \ref{prop-qcontactcone}.

\begin{corollary}
\label{prop-filt}
 Any $(2,3,5)$ distribution, $(3,6)$ distribution and conformal structure determines a canonical parabolic quasi-contact cone structure on the correspondence space $\tilde{\mathcal{C}}$.
 
  Conversely, for gradings corresponding to $(\g_2^*,\{\alpha_1,\alpha_2\})$ and
 $(\mathfrak{so}(3,4),\{\alpha_2,\alpha_3\})$ any parabolic quasi-contact cone structure of type $\q_{-}$ is locally equivalent to one determined by a $(2,3,5)$ distribution and $(3,6)$ distribution, respectively. For gradings corresponding to  $(\mathfrak{so}(p+2,q+2), \{\alpha_1,\alpha_2\})$ when $n=p+q+2\geq 5$ 
and  $(\mathfrak{so}(p+2,q+2),\{\alpha_1,\alpha_2,\alpha_3\})$ when $n=p+q+2=4$,  a parabolic quasi-contact cone structure of type $\q_{-}$ is locally equivalent to one determined by a conformal structure of signature $(p+1,q+1)$ if and only if its harmonic curvature of homogeneity one vanishes.
 \end{corollary}
 The harmonic  curvature condition in the above corollary is equivalent to vanishing of $\bF$ \eqref{eq:harmonic-inv-causal-represented} and integrability of the distribution $\tilde{\mathcal{V}}$ as in  \eqref{dec-T^-1}, which then corresponds to the vertical bundle of the sky-bundle $\tilde{\cC}\to\tilde{M}$.

 \begin{remark}
\label{rem-correspondence}
For  geometries arising as correspondence spaces the various sub-bundles in the above filtration can be geometrically described as follows. First note that $(2,3,5)$  and $(3,6)$ distributions determine canonical conformal structures of split signature  with respect to which the distributions $\cD=T^{-1}\tilde{M}$ are totally null \cite{Nurowski-G2,Bryant-36}. 
  Now let $\tilde{\nu}\colon\tilde{\mathcal{C}}\to \tilde{M}$ be the correspondence space so that a point $p\in\tilde{\mathcal{C}}$ corresponds to a null line $l\subset T^{-1}_{\pi(p)}\tilde{M}$. 
Then one has 
\begin{equation}
\label{eq-T-1}
T^{-1}_p\tilde{\mathcal{C}}=\{\tilde{X}_p\in T_p\tilde{\cC}:\tilde{\nu}_*\tilde{X}_p\in l\}.
\end{equation} The distribution $T^{-1}\tilde{\mathcal{C}}$ decomposes into two sub-bundles, the vertical bundle $\tilde{\mathcal{V}}\subset T^{-1}\tilde{\mathcal{C}}$ and a line bundle $\tilde{\mathcal{E}}\subset T^{-1}\tilde{\mathcal{C}}$ that can be understood as a generalized null geodesic flow. More precisely, in the case of $(2,3,5)$  distributions, the integral curves of $\tcE$ are the abnormal curves of the distribution \cite{Zelenko}. Moreover,  the fiber of the corank one distribution $T^{-k+1}\tilde{\cC}$ is given by 
\begin{equation}
\label{eq-quascond}
T^{-k+1}_p\tilde{\cC}=\{\tilde{X}_p\in T_p\tilde{\cC}:\tilde{\nu}_*\tilde{X}_p\in l^{\perp}\},
\end{equation}
 where the orthogonal is taken with respect to the conformal structure on $\tilde{M}$. We will see that $T^{-k+1}\tilde{\mathcal{C}}$ defines a quasi-contact structure,  see Definition \ref{def-qcontact}.  The remaining filtration components are preimages under the tangent map of $\tilde{\nu}$ of the filtration components on $T\tilde{M}$.

Arguing  along these lines should lead to a direct proof of Proposition \ref{prop-filt}, i.e. one that does not use Cartan connections.
Using the equivalent description as parabolic geometries not only clarifies matters,  but also    will be  essential for the curvature analysis in \ref{sec:curvatuer-analysis}.
\end{remark}

To give a uniform description of parabolic quasi-contact cone structures as filtered manifolds, it remains to characterize the graded Lie algebras \eqref{gradingq}. At first sight they seem to be fairly diverse, in particular the gradings have different lengths.
 However, note that on the level of the homogeneous models we have  double fibrations 
\begin{center}
\begin{tikzcd}[row sep=0em,column sep=1em,
cells={nodes={anchor=center}}]
& G/Q \arrow{dr} \arrow{dl} & \\
G/R &
& G/P 
\end{tikzcd}
\end{center}
where $R$, $P$ and $Q=R\cap P$ are the following parabolic subgroups.
\begin{enumerate}
  \item $G_2$: $R=P_1$, $Q=P_{12}$, $P=P_2$\ \begin{tikzcd}[row sep=0em,column sep=1em,
cells={nodes={anchor=center}}]
& \dynkin[reverse arrows, root
radius=.06cm] G{xx}\arrow{dr} \arrow{dl} & \\
\dynkin[reverse arrows, root
radius=.06cm] G{xo} &
& \dynkin[reverse arrows, root
radius=.06cm] G{ox}  \\
\end{tikzcd}
\item $B_3$: $R=P_3$, $Q=P_{23}$, $P=P_2$\ \begin{tikzcd}[row sep=0em,column sep=1em,
cells={nodes={anchor=center}}]
& \dynkin[root
radius=.06cm] B{oxx} \arrow{dr} \arrow{dl} & \\
\dynkin[root
radius=.06cm] B{oox}  &
& \dynkin[root
radius=.06cm] B{oxo}  \\
\end{tikzcd}
\item $B_n$, $n\geq 2$: $R=P_1$, $Q=P_{12}$, $P=P_2$\
\begin{tikzcd}[row sep=0em,column sep=1em,
cells={nodes={anchor=center}}]
& \dynkin[root
radius=.06cm] B{xxooo} \arrow{dr} \arrow{dl} & \\
\dynkin[root
radius=.06cm] B{xoooo}  &
& \dynkin[root
radius=.06cm] B{oxooo} \\
\end{tikzcd}
\newline
$A_3=D_3$: $R=P_2$, $Q=P_{123}$, $P=P_{13}$\
\begin{tikzcd}
[row sep=0em,column sep=1em,
cells={nodes={anchor=center}}]
& \dynkin[root
radius=.06cm] A{xxx} \arrow{dr} \arrow{dl} & \\
\dynkin[root
radius=.06cm] A{oxo}  &
& \dynkin[root
radius=.06cm] A{xox} \\
\end{tikzcd}
\newline
$D_n$, $n\geq 4$: $R=P_1$, $Q=P_{12}$, $P=P_2$\
\begin{tikzcd}
[row sep=0em,column sep=1em,
cells={nodes={anchor=center}}]
& \dynkin[root
radius=.06cm] D{xxoooo} \arrow{dr} \arrow{dl} & \\
\dynkin[root
radius=.06cm] D{xooooo}  &
& \dynkin[root
radius=.06cm] D{oxoooo} \\
\end{tikzcd}
\end{enumerate}    
\textsc{Diagram 2:} {\text Crossed Dynkin diagrams depicting the parabolic subalgebras corresponding to $(2,3,5)$ distributions, $(3,6)$ distributions and pseudo-Riemannian conformal structures (on the left), their quasi-contact correspondence spaces and causal structures (in the middle), and the corresponding parabolic contact structures (on the right).}

\vspace{.5cm}

  The base space $G/P$ in  the above fibration $G/Q\to G/P$  is a  homogeneous contact manifold, and the one-dimensional fiber over $eP$ is $P/Q$.  
 In particular, we have the  algebraic set-up of two nested parabolic subalgebras $\mathfrak{q}\subset\mathfrak{p}\subset\g$, $\mathrm{dim}(\p/\q)=1$, where 
 $\p$ is the parabolic corresponding to a parabolic contact structure.
This means that $\p$ is the non-negative part in a \emph{contact grading}, i.e. a  $\vert 2 \vert$-grading 
\begin{align}\label{contactgrading}\g=\p_{-2}\oplus\p_{-1}\oplus\p_{0}\oplus\p_1\oplus\p_2,  \quad \mbox{dim}(\p_{\pm 2})=1,
\end{align}
  where $[,]:\biw^2\p_{-1}\to\p_{-2}$ is non-degenerate as a bilinear form.   It is well-known that contact gradings exist on all simple complex Lie algebras and most real forms and are unique up to conjugation. The full list of complex contact gradings can be found e.g. in \cite{Yamaguchi-Simple}.

  \begin{proposition}\label{prop-qcontactgrading}
Let
  \begin{equation}\label{quasicontactgrading}
\g=\q_{-k}\oplus\dots\oplus
\q_0\oplus
\dots\oplus\q_k
\end{equation}
be a  $\vert k\vert$-graded simple Lie algebra corresponding to a subset $\Sigma^Q$ of simple restricted roots, where $(\g,\Sigma^Q)$ is $(\g_2^*,\{\alpha_1,\alpha_2\})$,
 $(\mathfrak{so}(3,4),\{\alpha_2,\alpha_3\})$, 
  $(\mathfrak{so}(p+2,q+2), \{\alpha_1,\alpha_2\})$ for $n=p+q+2\geq 5$, 
or  $(\mathfrak{so}(p+2,q+2),\{\alpha_1,\alpha_2,\alpha_3\})$ for $n=p+q+2=4$, as in Proposition \ref{prop-qcontactcone}.
Then the graded Lie algebra $\g$ has the following properties:
\begin{enumerate}
\item $\mathrm{dim}(\q_{- k})=1,$ $\mathrm{dim}(\q_{ -k+1}\oplus\cdots\oplus\q_{-1})$ is odd, and  
\begin{equation}\label{qsymplecticbracket}\pr_{-k}\circ[,]:\biw^2(\q_{-1}\oplus\cdots\oplus\q_{-k+1})\to\q_{-k}
\end{equation}
when viewed as a $2$-form  has one-dimensional kernel $\tilde{\mathfrak{e}}\subset\q_{-1}$. 
\item Let $\tilde{\mathfrak{v}}$ be the $\q_0$-invariant complement to $\tilde{\mathfrak{e}}$ in $\q_{-1}$ so that $$\q_{-1}=\tilde{\mathfrak{e}}\oplus\tilde{\mathfrak{v}}.$$
Then the only other non-trivial components of the Lie bracket on $\q_{-}$ are  $[,]:\tilde{\mathfrak{e}}\otimes\tilde{\mathfrak{v}}\to\q_{-2}$ and $[,]:\tilde{\mathfrak{e}}\otimes\mathfrak{q}_{i}\to\q_{i-1}$ for $i=-2,\dots -k+2$, which are isomorphisms. In particular, $\tilde{\mathfrak{v}}$ is abelian.
\end{enumerate}
Conversely, any $\vert k \vert$-graded simple Lie algebra of depth $k\geq 3$ satisfying these conditions is of one of the above types $(\g,\Sigma^Q)$.
  \end{proposition}
  
 \begin{proof} Let $\Sigma^Q$ be the subset of simple roots corresponding to the grading \eqref{quasicontactgrading}. Let $\Sigma^P=\Sigma^Q\setminus\{\alpha\}$ be the subset of simple roots associated with the corresponding contact grading 
 \eqref{contactgrading}. Then one has  $\q_0\subset\p_0$ and   the following $\q_0$-invariant refinement of the grading \eqref{quasicontactgrading}:
\begin{equation*} 
\g=\underbrace{\q_{-k}}_{\mathfrak{p}_{-2}}\oplus\underbrace{\q_{-k+1}\oplus\dots\oplus(\q_{-1}\cap\p_{-1})}_{\mathfrak{p}_{-1}}\oplus\underbrace{(\q_{-1}\cap\p_{0})\oplus\q_0\oplus(\p_{0}\cap\q_{1})}_{\p_0}\oplus\underbrace{(\p_{1}\cap\q_{1})\oplus\dots\oplus\q_{k+1}}_{\p_1}\oplus\underbrace{\q_k}_{\p_2}.
\end{equation*}
Let us now define the subspaces \begin{equation}
\label{eq-tildeeandf}
\tilde{\mathfrak{e}}:=(\q_{-1}\cap\p_0)\quad\mbox{and}\quad  \tilde{\mathfrak{v}}:=(\q_{-1}\cap\p_{-1}).
\end{equation} Then $\tilde{\mathfrak{e}}$ is one-dimensional since $\mathrm{dim}(\p/\q)=1$, and  $\tilde{\mathfrak{v}}$ is abelian since $[\p_{-1},\p_{-1}]\subset\p_{-2}=\q_{-k}$ and $[\q_{-1},\q_{-1}]\subset\q_{-2}$.  Moreover, (1)  follows from the properties of a contact grading and the fact that $[\p_{-1},\p_{0}]\subset\p_{-1}$. Since $[\p_{-1},\p_{-1}]\subset\p_{-2}$, the only non-trivial brackets on $\q_{-}$ other than \eqref{qsymplecticbracket} have to involve $\tilde{\mathfrak{e}}$, and one verifies in each case that they define isomorphisms as in (2).

Conversely, starting with a $\vert k \vert$-graded simple Lie algebra of depth $k\geq 3$ satisfying properties (1) and (2), one can construct a contact grading as follows. Recall that the Killing form induces an isomorphism $(\q_{-i})^*\cong\q_i$. Define $\tilde{\mathfrak{e}}^o$ and $\tilde{\mathfrak{v}}^o$ to be the subspaces in $\q_{1}$ annihilating $\tilde{\mathfrak{e}}$ and $\tilde{\mathfrak{v}}$, respectively. Now set the $-2$ grading component to be $\q_{-k}$, the $-1$ grading component to be $\q_{-k+1}\oplus\cdots\oplus\tilde{\mathfrak{v}}$, the $0$ grading component to be  $\tilde{\mathfrak{e}}\oplus\q_0\oplus\tilde{\mathfrak{v}}^o$, the $+1$ grading component  to be $\tilde{\mathfrak{e}}^o\oplus\cdots\oplus\q_{k-1}$, and the $+2$ grading component to be $\q_k$. Using invariance of the Killing form and the fact that it restricts to a non-degenerate pairing on $\q_{-i}\times\q_i$,  one verifies that the decomposition is indeed compatible with the Lie bracket, and by (1) it defines a contact grading. Finally, inspecting the list of all contact gradings, one concludes that the only $\vert k\vert$-graded Lie algebras that are related to contact gradings in this way are the Lie algebras associated with the three families of parabolic quasi-contact cone structures. 
 \end{proof}

Henceforth, we  refer to the gradings in Proposition \ref{prop-qcontactgrading}  as  \emph{quasi-contact gradings}.

\begin{remark}
In \cite{ANN-Monge}  a class of parabolic geometries called parabolic geometries of Monge type was introduced. These geometries are distinguished by the property that their flat models can be realized by systems of underdetermined ODEs. 
They generalize the classical description of $(2,3,5)$ distributions in terms of the so-called Monge equations \cite{Cartan-235}.
The corresponding $\vert k\vert$-gradings are characterized by the property that $\g_{-1}$ contains a non-zero abelian codimension one subalgebra $\tilde{\mathfrak{v}}$ with $\mathrm{dim}(\tilde{\mathfrak{v}})=\mathrm{dim}(\g_{-2})$.  The authors give a  classification of all $\vert k \vert$-graded Lie algebras of Monge type. Moreover, in Theorem B of \cite{ANN-Monge}  the authors list all Monge type geometries  that are non-rigid, i.e. have non-zero Lie algebra cohomology $H^2(\mathfrak{g}_{-},\mathfrak{g})$
in positive homogeneity.
Conditions (2) implies that quasi-contact gradings  are in particular of Monge type. More precisely, by inspection, one finds that they are exactly the non-rigid gradings of Monge type that   satisfy condition (1) of Proposition \ref{prop-qcontactgrading}. 
\end{remark}

\begin{remark}
Let $\Sigma^Q$ be the subset of simple roots associated with the $\vert k \vert$-graded Lie algebra $\g$ satisfying properties (1) and (2) as in Proposition \ref{prop-qcontactgrading}. Then it is easy to see that $\tilde{\mathfrak{e}}=\g_{-\alpha}$ for a simple root $\alpha$ (see the proof of  Proposition 2.3, [iii], in \cite{ANN-Monge}). 
 The contact grading corresponds to the subset $\Sigma^P=\Sigma^Q\setminus \{\alpha\}$. In terms of Dynkin diagram notation, the contact gradings of Lie algebras that admit a corresponding quasi-contact  grading have a single uncrossed node that is not connected to any other uncrossed node. To pass from the contact grading to the quasi-contact  grading 
 we cross that node. 
 \end{remark} 
 
 \begin{remark}
The class of conformal geometries, $(2,3,5)$ and $(3,6)$ distributions also appears in \cite{Biquard-G2} as the class of non-rigid parabolic geometries that admit a generalized geodesic flow, which is used to construct quaternionic K\"{a}hler metrics.
\end{remark}

The  properties of quasi-contact gradings reflect the properties of parabolic quasi-contact cone structures.
Recall the definition of a quasi-contact structure (also called even contact structure):
\begin{definition}
  \label{def-qcontact}
On a  $2n$-dimensional manifold $M$ we define the following.
  \begin{enumerate}
\item
A $1$-form $\alpha$ on $M$  is called a \emph{quasi-contact form} if the  restriction of  $\mathrm{d}\alpha$ to $\ker(\alpha)$ has maximal rank, equivalently,   $\alpha\wedge(\mathrm{d}\alpha)^{n-1}\neq 0$. 
\item
A \emph{quasi-contact structure} $(M,\mathcal{H})$ is given by a 
corank one distribution $\mathcal{H}\subset TM$  such that for each point $p\in M$ there exists  an open neighborhood $U$ where  $\mathcal{H}$ can be written as the kernel  of a quasi-contact form. 
\end{enumerate}
\end{definition}
The maximal rank condition means that the restriction of $\mathrm{d}\alpha$ to $\ker(\alpha)$ has a one-dimensional kernel at each point which defines the characteristic line field of the quasi-contact structure, denoted by $\tilde{\mathcal{E}}$. The integral curves of $\tilde{\mathcal{E}}$ are referred to as the characteristic curves of the quasi-contact structure.

For a parabolic quasi-contact cone structure the  corank one distribution $T^{-k+1}\tilde{\mathcal{C}}\subset T\tilde{\mathcal{C}}$  is 
  a quasi-contact structure.
The $\q_0$ invariant decomposition $\q_{-1}=\tilde{\mathfrak{e}}\oplus\tilde{\mathfrak{v}}$ corresponds to the decomposition 
 \begin{equation}\label{dec-T^-1}
 T^{-1}\tilde{\mathcal{C}}=\tilde{\mathcal{E}}\oplus \tilde{\mathcal{V}}
 \end{equation}  into the characteristic line field and a complementary distribution $\tilde{\mathcal{V}}$.
 From the structure of the torsion of  the parabolic geometries from Proposition \ref{prop-qcontactcone}  it follows that $\tilde{\mathcal{V}}$ is integrable in almost all cases. The only  exception are parabolic quasi-contact cone structures of type $D_3=A_3$, where integrability of $\tilde{\mathcal{V}}$ is ensured by the requirement that one of the harmonic curvature components with homogeneity one of the geometry vanish.

 \begin{remark}
 The local leaf space of a parabolic quasi-contact cone structure by its characteristic  curves has an induced contact structure. However, in general one  does not obtain a parabolic geometry on the leaf space. In fact, other than the flat models, the only geometries for which the leaf space is  equipped with a parabolic geometry are  causal structures with vanishing harmonic curvature $\bW$ in \eqref{eq:harmonic-inv-causal-represented}, for which the induced structures are known as \emph{Lie contact structures}.
 \end{remark}

\section{Parabolic almost conformally quasi-symplectic structures}
\label{sec:conf-quasi-sympl}
In this section we will introduce and study the class of geometric structures that are naturally induced on the leaf space of the integral curves of an infinitesimal symmetry for  parabolic quasi-contact cone structures. We will first consider a larger category of structures that we call \emph{parabolic almost conformally quasi-symplectic} (PACQ) structures. We find a natural normalization condition that allows one to associate a canonical Cartan connection to any PACQ structure. Then we identify the sub-class of \emph{parabolic conformally quasi-symplectic} (PCQ) structures, which are distinguished by the property that they admit local closed $2$-forms of a particular type. For PACQ structures corresponding to geometries of ODEs we show that PCQ structures correspond to \emph{variational} ODEs.

\subsection{Definition and descriptions of PACQ structures}
\label{sec:PACQ-structure-def}
We start by introducing parabolic almost conformally quasi-symplectic structures; these are filtered $G$-structures whose symbol algebras are quotients of the symbol algebras of parabolic quasi-contact cone structures.

\begin{lemma} \label{lem-s-}
Let $\g$ be  
 equipped with a quasi-contact grading $\g=\q_{-k}\oplus\dots\oplus\q_0\oplus\dots\oplus\q_{k}$ as in \eqref{quasicontactgrading}. 
Define
\begin{align}\label{def-k-}
\mfk_{-}=\mfk_{-k+1}\oplus\cdots\oplus\mfk_{-1}:=(\q_{-k}\oplus\dots\oplus\q_{-1})/\mathfrak{q}_{-k}
\end{align}
 to be the quotient of the graded nilpotent Lie algebra $\q_{-}$ by the last grading component $\q_{-k}$, which is an ideal in $\q_{-}$. Then $\mfk_{-}$ is also a graded nilpotent Lie algebra.
 \end{lemma}

The graded Lie algebra $\mfk_{-}$ can be identified with the graded subspace
$\mfk_{-}\cong\q_{-k+1}\oplus\cdots\oplus\q_{-1} \subset \q_{-}$
equipped with the Lie bracket 
 $[X,Y]_{\mfk_-}:=[X,Y]-\mathrm{pr}_{\q_{-k}}[X,Y],$
 where $\mathrm{pr}_{\q_{-k}}:\mathfrak{g}\to\q_{-k}$ denotes the projection onto the $\q_{-k}$-component.

 Let $Q_0:=\mathrm{Aut}_{gr}(\q_{-})$ be the group of  Lie algebra automorphisms of $\q_{-}$ preserving the grading. Since $\q_{-}$ is generated by $\q_{-1}$ as a Lie algebra, every element in $Q_0$ is completely determined by its restriction to $\q_{-1}$. In particular, we have an inclusion  
 \begin{equation} \label{inclQ0}
 Q_0\subset\mathrm{Aut}_{gr}(\mfk_{-}).
 \end{equation}

\begin{definition}
\label{def-PACQ}
 Let $\mfk_{-}$ be defined as in Lemma \ref{lem-s-} and $ Q_0\subset\mathrm{Aut}_{gr}(\mfk_{-})$ as in \eqref{inclQ0}.
A \emph{parabolic almost conformally quasi-symplectic} (PACQ) structure of type $(\mfk_{-},Q_0)$ is a filtered $Q_0$-structure of type $\mfk_{-}=\q_{-}/\q_{-k}$. 
More precisely, it is given by 
\begin{itemize}
\item a bracket generating distribution $T^{-1}\mathcal{C}$ with weak derived flag
$$T^{-1}\mathcal{C}\subset T^{-2}\mathcal{C}\subset\cdots\subset T^{-k+1}\mathcal{C}=T\mathcal{C}$$
whose symbol algebras form a locally trivial bundle modeled on the graded Lie algebra $\mfk_{-}$ and
\item  a reduction of structure group of the graded frame bundle of the filtered manifold with respect to the inclusion $Q_0\subset \mathrm{Aut}_{gr}(\mfk_{-})$. 
\end{itemize}
\end{definition}
Note that by definition of $\mfk_{-}$ and $Q_0$ and Proposition \ref{prop-qcontactgrading}, we have a $Q_0$-invariant decomposition
$$\mfk_{-1}=\mathfrak{e}\oplus\mathfrak{v}.$$ 
Consequently, for any PACQ structure the distribution 
$$T^{-1}\mathcal{C}=\mathcal{E}\oplus \mathcal{V}$$ splits into a line bundle $\mathcal{E}$ and a distribution $\mathcal{V}$. 
Moreover, it follows from the bracket relations on $\mfk_{-}$ that
$$[\mathcal{E},\mathcal{V}]
=T^{-2}\mathcal{C},\quad\mbox{and}\quad [\mathcal{E},T^{-i}\mathcal{C}]
=T^{-i-1}\mathcal{C}.$$

In what follows, we describe  classes of PACQ structures associated with (systems of) ODEs.
 Note that in some cases these filtered $G$-structures are more general than the respective ODE geometries, which always contain integrable sub-distributions coming from jet space fibrations.
 
\subsubsection{Geometry of 4th order scalar ODEs  up to contact transformations}\label{sec-4thorder}
Starting with $\g=\g_2^*$ and the grading corresponding to the Borel subalgebra $\q=\p_{12}$,  the Lie algebra $\mfk_{-}=\q_{-}/\q_{-5}$ has a grading of the form $\mfk_{-4}\oplus\cdots\oplus\mfk_{-1}$, where $\mathrm{dim}(\mfk_{-1})=2$ and $\mathrm{dim}(\mfk_{-i})=1$ for $i=2,3,4$.  The Lie group $Q_0\cong(\mathbb{R}^*)^2$ is the diagonal subgroup in $\mathrm{GL}(\mfk_{-1})$ preserving a decomposition  $\mfk_{-1}=\mathfrak{e}\oplus\mathfrak{v}.$

Filtered $Q_0$-structures of this type arise from  $4$th order ODEs considered modulo contact transformations.
Geometrically, a $4$th order ODE $$y_4=f(x,y,y_1,y_2,y_3)$$ defines a submanifold
$\mathcal{C}\subset J^4(\mathbb{R},\mathbb{R})$  locally
diffeomorphic to $J^3(\mathbb{R},\mathbb{R})$. The canonical rank $2$ distribution on $J^3(\mathbb{R},\mathbb{R})$ corresponds to a rank $2$ distribution $T^{-1}\mathcal{C}$, which splits into
the vertical bundle $\mathcal{V}$ of the projection $\mathcal{C}\to J^2(\mathbb{R},\mathbb{R})$, spanned by $\partial_{y_3}$, and the line bundle $\mathcal{E}$ corresponding to the foliation by prolongations of solutions of the ODE, spanned by the total derivative $\frac{\exd}{\exd x}=\partial_x+y_1\partial_y+y_2\partial_{y
_1}+y_3\partial_{y_2}+f\partial_{y_3}$. In fact, locally all PACQ structures of this algebraic type 
arise from $4$th order ODEs.

\begin{center}
\begin{figure}[h]
  \begin{tikzpicture}[scale=0.85,baseline=-5pt]
 
    \draw ( 0  , 1.732) -- (0,0); 
   
    \filldraw ( 0  ,  1.732) circle (0.05);
    \draw (0, 0) -- (1.5,  0.866); 
    \filldraw ( 1.5,  0.866) circle (0.05);
    \draw (0, 0) -- (0.5,  0.866); 
    \filldraw ( 0.5,  0.866) circle (0.05);
    \draw (0, 0) -- (-0.5,  0.866); 
    \filldraw ( -0.5,  0.866) circle (0.05);
    \draw (0, 0) -- (-1.5,  0.866);
    \filldraw ( -1.5,  0.866) circle (0.05);
    \draw (0, 0) -- (0.866,  0); 
    \filldraw (0.866,  0) circle (0.05);
    \draw (0, 0) -- (1.5,  -0.866); 
   
    \filldraw ( 1.5,  -0.866) circle (0.05);
    \draw ( 0  , -1.732) -- (0,0);
    \filldraw ( 0  ,  -1.732) circle (0.05);
    \draw (0, 0) -- (-1.5,  -0.866); 
    \filldraw ( -1.5,  -0.866) circle (0.05);
    \draw (0, 0) -- (-0.5,  -0.866); 
    \filldraw ( -0.5,  -0.866) circle (0.05);
    \draw (0, 0) -- (0.5,  -0.866); 
    \filldraw ( 0.5,  -0.866) circle (0.05);
     \node at  ( 3.8  , -0.866){} ;
    \node at  ( 2.4  , -1.29) {$\left\langle \partial_{y_3}\right\rangle=\mathcal{V}$};
    \node at  ( 0.7  , -1.29) {$\left\langle \partial_{y_2}\right\rangle$};
     \node at  ( -0.7  , -1.29) {$\left\langle \partial_{y_1}\right\rangle$};
      \node at  ( -1.85  , -1.29) {$\left\langle \partial_{y}\right\rangle$};
    \node at  ( 1.8  , -0.85) {-$1$};
    \node at  ( 0.8  , -0.85) {-$2$};
    \node at  ( -0.8  , -0.85) {-$3$};
    \node at  ( -1.8 , -0.85) {-$4$};
    \draw (0, 0) -- (-0.866,  0); 
     \node at  ( -2.9  , 0.00) {$\mathcal{E}=\left\langle \frac{\exd}{\exd x} \right\rangle $};
     \node at  ( -1.2  , 0) {-$1$};
      \node at  ( 4  , 0){} ;
    \filldraw (-0.866,  0) circle (0.05);   
    \filldraw (0,0) circle (0.1); 
    
 \end{tikzpicture}
  \caption{The figure shows the root diagram for $G_2$. The roots whose corresponding root spaces form the grading components $\q_{i}$   are  labeled by $i$, for $i=-1,\dots,-4$. Since $\mfk_{-}=\q_{-}/\q_{-5}$, we can use the diagram to compare the filtration on $\mfk_{-}$  with the filtration on $T\cC$ determined by a $4$th order ODE.}     
\end{figure}
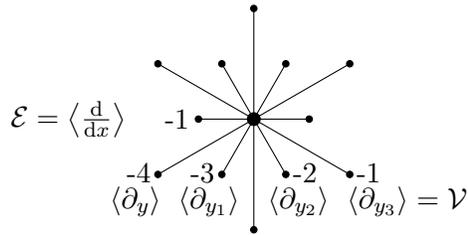
\end{center}

\subsubsection{Geometry of pairs of $3$rd order ODEs   up to point transformations}\label{sec-pairsthirdorder}
Starting with $\g=\mathfrak{so}(3,4)$ and the grading corresponding to  $\q=\p_{23}$, we obtain a graded Lie algebra $\mfk_{-}=\mfk_{-3}\oplus\cdots\oplus\mfk_{-1}$, where $\mathrm{dim}(\mfk_{-1})=3$ and $\mathrm{dim}(\mfk_{-i})=2$ for $i=2,3$.  The Lie group $Q_0\cong\mathbb{R}^*\times\mathrm{GL}(2,\mathbb{R})$ is the subgroup in $\mathrm{GL}(\mfk_{-1})$ preserving the decomposition $\mfk_{-1}=\mathfrak{e}\oplus\mathfrak{v}$. 

A sub-class of PACQ structures of this type arises from  pairs of third order ODEs
$$y^i_3=f^i\left(x,y^1,y^2,y^1_1,y^2_1,y^1_2,y^2_2\right), \quad i\in\{1,2\}.$$ Such a pair of ODEs defines a submanifold $\mathcal{C}\subset J^3(\mathbb{R},\mathbb{R}^2)$ locally diffeomorphic to $J^2(\mathbb{R},\mathbb{R}^2)$. The canonical rank three distribution on $J^2(\mathbb{R},\mathbb{R}^2)$ corresponds via this diffeomorphism to $T^{-1}\mathcal{C}$, which decomposes as $T^{-1}\mathcal{C}=\mathcal{E}\oplus\mathcal{V}$ into the vertical bundle $\mathcal{V}$ of the projection $\mathcal{C}\to J^1(\mathbb{R},\mathbb{R}^2)$, spanned by $\partial_{y^1_2}$ and $\partial_{y^2_2}$, and the line distribution $\mathcal{E}$, which corresponds to the foliation  by prolongations of solutions of the ODE. 
However, not all PACQ structures of this algebraic type locally arise this way, which will be discussed in \ref{sec:pairs-third-order}.

\subsubsection{Orthopath geometry} Before describing the last class of PACQ structures, we  need to recall the definition of a path geometry.
A classical path geometry on a $n$-dimensional manifold $M$ is given by a family of unparametrized curves (immersed one-dimensional submanifolds) on $M$ such that for each point $x\in M$ and each  direction $l_x\subset T_xM$ one has a unique curve from the family  through $x$ in direction $l_x$. Such a structure determines a line bundle $\mathcal{E}$ on the  projectivized tangent bundle $\mathbb{P}TM$. 
More generally, one defines a path geometry (sometimes also called generalized path geometry) as follows: 
\begin{definition}
\label{def-pathgeometry}
An  $n$-dimensional path geometry is given by a triple $(\cC,\cE, \cV)$, where $\cC$ is a manifold  of dimension $2n-1$, $\mathcal{E}\subset T\cC$ is a line field  and $\mathcal{V}\subset T\cC$ is a distribution of rank $n-1$ such that $\cE$ and $\cV$ have trivial intersection,  $[\cV,\cV]\subset \cE\oplus\cV=:T^{-1}\cC$ 
and the Levi bracket $\mathfrak{L}:\mathcal{E}\otimes\mathcal{V}\cong T\cC/T^{-1}\cC$, given by mapping $X\in\Gamma(\mathcal{E})$ and $Y\in \Gamma(\mathcal{V})$ to the projection of  $[X,Y]$ onto  $T\cC/T^{-1}\cC$, is an isomorphism.
\end{definition}
It turns out that for $n= 2$ and $n\geq 4$ the distribution $\mathcal{V}$ is automatically integrable. For $n=3$ integrability of $\mathcal{V}$ is equivalent to  vanishing of the lowest homogeneous component of the harmonic curvature of the path geometry, see \cite{CS-Parabolic}. In this case, in a neighborhood of each point, one can form a local leaf space $M$ for the foliation determined by $\mathcal{V}$. Projecting leaves of $\mathcal{E}$ endows  $M$ with a family of unparametrized curves. In fact, it can be shown that any generalized path geometry $(\cC,\cE, \cV)$ is locally equivalent 
 to   the one defined on an open subset of $\mathbb{P}TM$ determined by the family of projected curves.
 Finally we point out that  any path geometry, in the generalized sense, can be locally realized as a system of second order ODEs up to point transformations, such that the paths correspond to graphs of solutions of the ODE system. 
 We refer the reader to \cite[Section 4.4.3]{CS-Parabolic}  for more details.
 
Now, starting with the Lie algebra $\g=\mathfrak{so}(p+2,q+2)$ and the grading corresponding to the parabolic subalgebra $\p_{12}$,
we obtain a graded Lie algebra $\mfk=\mfk_{-2}\oplus\mfk_{-1}$ that is isomorphic to the symbol algebra of a  path geometry.  
The reduction of structure group to 
$Q_0\cong \mathbb{R}^*\times\mathrm{CO}(p,q) $ amounts to augmenting the path geometry with an additional geometric structure that can be viewed as a conformal class of bundle metrics on $\mathcal{V}$. In particular, a PACQ structure of this type is equivalent to an orthopath geometry as defined below.

\begin{definition}\label{def-orthopath}
An orthopath geometry of signature $(p,q)$ on a manifold $\mathcal{C}$ of dimension $\geq 5$ is a  path geometry together with a conformal class $[\bh]$ of non-degenerate bundle metrics $\bh\in\Gamma(\mathrm{Sym}^2\mathcal{V}^*)$ of signature $(p,q)$.
\end{definition}

\subsection{Canonical Cartan connection for PACQ structures}
\label{sec:norm-cond}
The aim of this section is to associate a canonical Cartan geometry to any parabolic almost conformally quasi-symplectic structure.
We  start by describing the Lie algebraic data needed to express the type of the Cartan geometries associated to PACQ  structures.
 
Let $\g$ be one of the Lie algebras $\g_2^*$ or $\mathfrak{so}(p+2,q+2)$  from Proposition \ref{prop-qcontactcone},  $\g=\q_{-k}\oplus\cdots\oplus\q_0\oplus\cdots\oplus\q_k$  a grading as in \eqref{gradingq}  corresponding to a subset $\Sigma^Q$ and  $\g=\p_{-2}\oplus\p_{-1}\oplus\p_0\oplus\p_1\oplus\p_2$  the contact grading as in \eqref{contactgrading} corresponding to $\Sigma^P=\Sigma^Q\setminus\{\alpha\}$. Let $$\p^{op}=\p_{-2}\oplus\p_{-1}\oplus\p_0$$ denote the \emph{opposite parabolic subalgebra}.
We  define the Lie algebra
\begin{equation}
\label{Liealgs}
\mfk=\p^{op}/\p_{-2}=\p^{op}/\q_{-k}
\end{equation}
and endow it with the following grading
\begin{equation}
\label{Liealgsgrading}
\mfk=\mfk_{-k+1}\oplus\cdots\oplus\mfk_0\oplus\mfk_1:=\left(\q_{-k}\oplus\q_{-k+1}\oplus\cdots\oplus\q_0\oplus (\p_0\cap\q_+)\right)/\q_{-k}.
\end{equation}
Let $\mfl=\mfk^0=\mfk_0\oplus\mfk_1$ be its non-negative part. The negative part of $\mfk$ is naturally identified with $\mfk_{-}$ introduced in \eqref{def-k-}. Note that we can identify $\p_0$ with a subalgebra in $\mfk$ and similarly we have an  identification 
$$\mfl\cong\p_0\cap\q=\q_0\oplus (\p_0\cap\q_+)\subset\p_0.$$
We  further consider the decomposition   
 \begin{equation}\label{aplusp0}
 \mfk=\mathfrak{a}\oplus\p_0 = \mathfrak{a}\oplus \mathfrak{e}\oplus\mfl,
 \end{equation}
 where 
 $\mathfrak{e}\cong (\p_0\cap\q_{-})$ and
  $\mathfrak{a}\subset\mfk$ is an abelian ideal which is isomorphic to $\p_{-1}$ as a $\p_0$-module. 
  \begin{lemma}
  \label{lem-kL}
  For the individual Lie algebras from Proposition \ref{prop-qcontactcone} equipped with gradings \eqref{gradingq}  the Lie algebras $\mfk$ and $\mfl$ are given as follows:
\begin{enumerate}
\item For $\g=\g_2^*$ and  $\Sigma^Q=\{\alpha_1,\alpha_2\}$ we have 
$\p_0=\mathfrak{gl}(2,\mathbb{R})$ and  $\mathfrak{a}\cong \mathrm{Sym}^3\mathbb{R}^2$ as a $\p_0$-module, hence
\begin{equation}
\label{k_g2}
\mfk\cong \mathfrak{gl}(2,\mathbb{R})\ltimes \mathrm{Sym}^3\mathbb{R}^2.
\end{equation} The Lie algebra $\mfl=\mathfrak{b}\subset\p_0\cong \mathfrak{gl}(2,\mathbb{R})$ is the Borel subalgebra stabilizing a line $l\subset\mathbb{R}^2$.

 \item For $\g=\mathfrak{so}(3,4)$ and  $\Sigma^Q=\{\alpha_2,\alpha_3\}$ we have 
  $\p_0= \mathfrak{sl}(2,\mathbb{R})\oplus\mathfrak{gl}(2,\mathbb{R})$ and  $\mathfrak{a}\cong \mathbb{R}^2\otimes \mathrm{Sym}^2\mathbb{R}^2$, hence
  \begin{equation}\label{k_so34}
  \mfk\cong \mathfrak{sl}(2,\mathbb{R})\oplus\mathfrak{gl}(2,\mathbb{R})\ltimes\mathbb{R}^2\otimes \mathrm{Sym}^2\mathbb{R}^2.
  \end{equation} We have $\mfl\cong \mathfrak{sl}(2,\mathbb{R})\oplus \mathfrak{b}$, where $\mathfrak{b}\subset\mathfrak{gl}(2,\mathbb{R})$ is the stabilizer of a line $l\subset\mathbb{R}^2$. 
 
\item For $\g=\mathfrak{so}(p+2,q+2)$ and the subsets $\Sigma^Q=\{\alpha_1,\alpha_2\}$ if $p+q>2$ and $\Sigma^Q=\{\alpha_1,\alpha_2,\alpha_3\}$ if $p+q=2$  we have
 $\p_0\cong\mathfrak{sl}(2,\mathbb{R})\oplus\mathfrak{co}(p,q)$ and  $\mathfrak{a}\cong\mathbb{R}^2\otimes\mathbb{R}^{p+q}$, hence
 \begin{equation}
 \label{k_sopq}
 \mfk\cong \mathfrak{sl}(2,\mathbb{R})\oplus\mathfrak{co}(p,q)\ltimes \mathbb{R}^2\otimes\mathbb{R}^{p+q}.
 \end{equation}
  We have $\mfl\cong \mathfrak{b}\oplus\mathfrak{co}(p,q)$, where $\mathfrak{b}\subset\mathfrak{sl}(2,\mathbb{R})$ is the stabilizer of a line $l\subset\mathbb{R}^2$.
\end{enumerate}
Let $P_0$ be a Lie group with Lie algebra $\p_0$ and $Q$ a parabolic subgroup with Lie algebra $\q$, then
$$L=P_0\cap Q$$
is a Lie group with Lie algebra $\mfl$.
In particular,  the following  Lie groups 
\begin{enumerate}
\item[(a)] $L=B\subset P_0=\mathrm{GL}(2,\mathbb{R})$,
\item[(b)]  $L=\mathrm{SL}(2,\mathbb{R})\times B\subset P_0=\mathrm{SL}(2,\mathbb{R})\times \mathrm{GL}(2,\mathbb{R})$, 
\item[(c)] $L=B\times \mathrm{CO}(p,q)\subset P_0=\mathrm{SL}(2,\mathbb{R})\times \mathrm{CO}(p,q)$,
\end{enumerate}
where  $B$ denotes the Borel subgroup defined as the stabilizer of a line $l\subset\mathbb{R}^2$, are Lie groups with Lie algebra $\mfl$.
\end{lemma}
\begin{remark}
Let $\mathfrak{a}\subset\mfk$ be the abelian ideal as in  decomposition  \eqref{aplusp0} of $\mfk$ and $P_0$  one of the groups  $\mathrm{GL}(2,\mathbb{R})$, $\mathrm{SL}(2,\mathbb{R})\times \mathrm{GL}(2,\mathbb{R})$, $\mathrm{SL}(2,\mathbb{R})\times \mathrm{CO}(p,q)$  with Lie algebra $\p_0$ as in Lemma \ref{lem-kL}, (a)-(c). Define the Lie group $K$ with Lie algebra $\mfk$ to be the semidirect product $$K=P_0\ltimes \mathfrak{a}.$$ Then the homogeneous space $K/L$ fibers over   $K/P_0$ and 
 the $1$-dimensional fiber over the identity coset is $P_0/L\cong \mathbb{P}^1$. The homogeneous spaces $K/P_0$ are  homogeneous models for certain \emph{parabolic almost conformally symplectic (PACS) structures}  as introduced in \cite{CS-cont1}, namely for so-called $\mathrm{GL}(2,\mathbb{R})$-structures in dimension $4$, Segr\'e structures in dimension $6$, and PACS structures of Grassmannian type, respectively.

\end{remark}

The next  step towards the construction of a canonical Cartan connection for a PCAQ structure is to find a suitable normalization condition, which we derive from the canonical normalization conditions for the corresponding parabolic geometry. 
Recall that on the semisimple Lie algebra $\g$ we have a  positive definite inner product $\left\langle X,Y\right\rangle=-B(X,\theta(Y))$  for which  $\partial:\Lambda^{s-1}(\q_{-})^*\otimes\g\to\Lambda^s(\q_{-})^*\otimes\g$ as in \eqref{formuladel} and the map $\tilde{\partial}^*:\Lambda^{s}(\q_{-})^*\otimes\g\to\Lambda^{s-1}(\q_{-})^*\otimes\g$ induced by the Kostant codifferential \eqref{KostantCodif} are adjoint, wherein $\q_-=\g_-$(see the proof of  \cite[Proposition 3.3.1]{CS-Parabolic}). 
 The grading components of $\g$ are mutually orthogonal with respect to this inner product. Now we restrict the inner product to  $\q_{-k+1}\oplus\cdots\oplus\q_0\oplus (\p_0\cap\q_+)$, which we identify   with $\mfk$ as a vector space. The isomorphism induces an inner product on  $\mfk$. 
 The grading on $\mfk$ is then orthogonal with respect to the induced inner product $\left\langle , \right\rangle$.
 Using the invariance of the Killing form $B$ and the fact that the  Cartan involution $\theta$ is a Lie algebra automorphism satisfying $\theta^2=\mathrm{id}$, one verifies that $\left\langle [A, X],Y\right\rangle=-\left\langle X,[\theta(A), Y]\right\rangle$ holds for any $A\in\p_0$.

Having introduced an inner product on $\mfk$, we now proceed with the construction of a codifferential, which is derived analogously to the one  for (systems of) ODEs as given in \cite{CDT-ODE}. 
 Consider the Lie algebra cohomology differential  
\begin{equation}
\label{dels}
\partial_{\mfk}:\biw^s\mfk^*\otimes\mfk\to\biw^{s+1}\mfk^*\otimes\mfk
\end{equation} in the standard complex that is used to define Lie algebra cohomology of $\mfk$ with values in $\mfk$. Let $$\partial_{\mfk}^*:\biw^{s+1}\mfk^*\otimes\mfk\to\biw^{s}\mfk^*\otimes\mfk$$  be the adjoint of $\partial_{\mfk}$ with respect to the induced inner product on the spaces $\biw^{s}\mfk^*\otimes\mfk$.

Since $\partial_{\mfk}$ is $\mfk$-equivariant and the inner product satisfies $\left\langle A\cdot X,Y\right\rangle=-\left\langle\phi,\theta(A)\cdot\psi\right\rangle$ for any $A\in\p_0$, we have
\begin{equation}
\begin{aligned}
\label{equivariance}
\left\langle A\cdot \partial_{\mfk}^*(\phi),\psi\right\rangle&=-\left\langle \partial_{\mfk}^*(\phi),\theta(A)\cdot\psi\right\rangle=-\left\langle\phi,\partial_{\mfk}(\theta(A)\cdot\psi)\right\rangle\\&=-\left\langle\phi,\theta(A)\cdot\partial_{\mfk}(\psi)\right\rangle=\left\langle A\cdot\phi,\partial_{\mfk}(\psi)\right\rangle=\left\langle \partial_{\mfk}^*(A\cdot\phi),\psi\right\rangle .
\end{aligned}
\end{equation}
It follows that 
$\partial_{\mfk}^*:\biw^{s+1}\mfk^*\otimes\mfk\to\biw^{s}\mfk^*\otimes\mfk$
is $\p_0$-equivariant.

Now define the space
$$\biw^{s}(\mfk/\mfl)^*\otimes\mfk=\left\{\Phi\in \biw^s\mfk^*\otimes\mfk=L(\biw^s\mfk,\mfk):\,A\im\Phi=0\ \forall A\in\mfl\right\}.$$
One verifies that
$\partial_{\mfk}^*$  
 restricts to a map
\begin{equation}\label{Codifferential}\partial^*:\biw^{s+1}(\mfk/\mfl)^*\otimes\mfk\to\biw^{s}(\mfk/\mfl)^*\otimes\mfk
\end{equation}
which is $L$-equivariant for the natural $L$-action.

Moreover,  $\partial^*$ can be shown to define a codifferential in the sense of  \cite[Definition 3.9]{Cap-Cartan}, which then immediately implies that  $\ker(\partial^*)$ is a normalization condition by  \cite[Proposition 3.10]{Cap-Cartan}.
In particular, this requires 
considering the associated graded vector space of $\biw^{s}(\mfk/\mfl)^*\otimes\mfk$, which can be identified with $\biw^s(\mfk_{-})^*\otimes\mfk$.  
The inner product $\left\langle , \right\rangle$ induces an inner product on $\biw^s(\mfk_{-})^*\otimes\mfk$, and one verifies that the induced map  
$\partial_{\mfk_{-}}^*:\biw^s(\mfk_{-})^*\otimes\mfk\to\biw^{s-1}(\mfk_{-})^*\otimes\mfk$ and the Lie algebra cohomology differential $\partial_{\mfk_{-}}:\biw^{s-1}(\mfk_{-})^*\otimes\mfk\to\biw^{s}(\mfk_{-})^*\otimes\mfk$ are adjoint with respect to this induced inner product. The details and  verifications of the remaining properties of a codifferential are left to the reader since they are  analogous to the ones in the proofs of  \cite[Lemma 3.2 and Proposition 3.3]{CDT-ODE}.

Summarizing, we have the following lemma.
\begin{lemma}\label{prop-codif}
The  adjoint  of $\partial_{\mfk}$ restricts to  a $L$-equivariant map
\begin{equation}\partial^*:\biw^{s+1}(\mfk/\mfl)^*\otimes\mfk\to\biw^{s}(\mfk/\mfl)^*\otimes\mfk
\end{equation}
which defines a codifferential.  In particular, its kernel
$$\ker(\partial^*)\subset \biw^2(\mfk/\mfl)^*\otimes\mfk$$
is a normalization condition for Cartan geometries of type $(\mfk,L)$.
\end{lemma}

As a result, normality can  be defined as follows.
\begin{definition}
A Cartan geometry of type $(\mfk,L)$ is called \emph{normal} if its curvature function takes values in the $L$-module $\mathcal{N}=\ker\partial^*$ from Proposition \ref{prop-codif}.
\end{definition}
We can now formulate the main theorem of this section. Let $(\mfk_{-},Q_0)$ be defined as in \eqref{def-k-} and \eqref{inclQ0} and $(\mfk, L)$ 
as in Lemma  \ref{lem-kL}.

\begin{theorem}
\label{thm-CartanConnection}
There is an equivalence of categories between parabolic almost conformally quasi-symplectic structures of type $(\mfk_{-},Q_0)$ and regular and normal Cartan geometries of type $(\mfk,L)$. 
\end{theorem}

In order to prove the theorem we first need to show that $\mfk$ is the full prolongation of its non-positive part $\mfk_{-}\oplus\mfk_0$. This is the content of Proposition \ref{prop-prol} below.

Recall the  decomposition 
\begin{equation}\label{apluspluseplusl}
\mfk=\mathfrak{a}\oplus\p_0=\mathfrak{a}\oplus\mathfrak{e}\oplus\mfl,\quad\mbox{where}\quad\mfl=\mfk_0\oplus \mfk_1,
\end{equation}
from \eqref{aplusp0}.
Let $X$ and $Y$ be a generator of $\mathfrak{e}$ and $\mfk_1$, respectively. Choose $H\in\mfk_0$ such that $X,Y,H$ form a $\mathfrak{sl}(2,\mathbb{R})$ triple.
Denote by $\omega^X$  the functional that maps $X$ to one and vanishes on  $\mathfrak{a}$.
Then we can uniquely decompose  any element  $\Phi\in\biw^s(\mfk/\mfl)^*\otimes\mfk$ as
\begin{equation}\label{decomp}\Phi=\omega^X\wedge\phi_1+\phi_2=\begin{pmatrix} \phi_1\\ \phi_2\end{pmatrix},\end{equation}
where $\phi_1:=X\im\Phi$ and $\phi_2:=\Phi-\omega^X\wedge\phi_1$ both vanish upon insertion of $X$ and thus correspond to elements $\phi_1\in \biw^{s-1}\mathfrak{a}^*\otimes\mfk$ and $\phi_2\in \biw^s\mathfrak{a}^*\otimes\mfk$, respectively.

Let $\partial_{\mathfrak{a}}:\biw^{s}\mathfrak{a}^*\otimes\mfk\to\biw^{s+1}\mathfrak{a}^*\otimes\mfk\,$ be the differential in the complex defining Lie algebra cohomology of the abelian Lie algebra $\mathfrak{a}$ with values in $\mfk$. Explicitly, it is given by
\begin{equation}\label{partiala}
\partial_{\mathfrak{a}}\phi(V_0,\cdots,V_s)=\sum_{i}(-1)^i[V_i,\phi(V_0,\dots,\hat{V}_i,\dots,V_s)].
\end{equation} 
Using the formulas for $\partial_{\mfk_{-}}:\biw^{s}(\mfk_{-})^*\otimes\mfk\to\biw^{s+1}(\mfk_{-})^*\otimes\mfk$ and $\partial_{\mathfrak{a}}$, one verifies  that
\begin{align}\label{delformula}
\partial_{\mfk_{-}}\begin{pmatrix}\phi_1\\ \phi_2\end{pmatrix}=\begin{pmatrix}-\partial_{\mathfrak{a}}\phi_1+ X\cdot\phi_2\\ \partial_{\mathfrak{a}}\phi_2\end{pmatrix},
\end{align}
where 
$$(X\cdot\phi_2)(V_1,\cdots,V_s)=[X,\phi_2(V_1,\cdots,V_s)]+\sum_{j=1}^{s}(-1)^j\phi_2([X,V_j],V_1,\cdots,\hat{V}_j,\cdots,V_s).$$

\begin{proposition}\label{prop-prol}
Let $\mfk$ be a graded Lie algebra  as  in \eqref{Liealgs} and \eqref{Liealgsgrading}. Then  $H^1(\mfk_{-},\mfk)$ is contained in non-positive homogeneous degrees, or equivalently, $\mfk$ is the full prolongation of its non-positive part $\mfk_{-}\oplus\mfk_0$. 
\end{proposition}
\begin{proof}
We show that the cohomology $H^1(\mfk_{-},\mfk)$ is contained in non-positive homogeneous degrees. This means that for any $\Phi\in (\mfk_{-})^*\otimes\mfk$ of positive homogeneity and contained in the kernel of $\partial_{\mfk_{-}}$, we need to show that there exists $A\in\mfk$ such that $\partial_{\mfk_{-}} A =\Phi.$

We decompose $\Phi$ as in \eqref{decomp} with $\phi_1\in\mfk$ and $\phi_2\in\mathfrak{a}^*\otimes\mfk$. Then by \eqref{delformula}, we have $\partial_{\mfk_{-}}\Phi=0$ if and only if $\partial_{\mathfrak{a}}\phi_2=0$ and $\partial_{a}\phi_1=X\cdot \phi_2$.  
We further decompose  $\phi_2$ according to its values in $\mfk=\mathfrak{a}\oplus\p_0$. The restriction $$\partial_{\mathfrak{a}}\vert_{\mathfrak{a}^*\otimes\mathfrak{p}_0}:\mathfrak{a}^*\otimes\mathfrak{p}_0\to\biw^2\mathfrak{a}^*\otimes\mathfrak{a}$$ of $\partial_{\mathfrak{a}}$ as defined in \eqref{partiala} is  the Spencer differential corresponding to $\p_{0}\subset\mathfrak{a}^*\otimes\mathfrak{a}$. This has been shown to be injective in   \cite[Theorem 4.2]{CS-cont1}. In particular, $\partial_{\mathfrak{a}}\phi_2=0$ implies that $\phi_2\in\mathfrak{a}^*\otimes\mathfrak{a}$.

Now  we decompose $\mathfrak{a}^*\otimes\mathfrak{a}$ as a $\p_0$-representation into $\p_{0}\subset\mathfrak{a}^*\otimes\mathfrak{a}$ and a direct sum of complementary  irreducible components.
Recall that we have $\mathfrak{sl}(2,\mathbb{R})\subset\p_0$  spanned by $X\in\mathfrak{e}$, which is of negative homogeneity, $Y\in\mfk_1$, which is of positive homogeneity, and $H\in\mfk_0$, which is of homogeneity zero. 
Since $X\in\p_0$, the action of $X$ respects the decomposition of $\mathfrak{a}^*\otimes\mathfrak{a}$ introduced above.
 Using this, we conclude from
 \begin{equation}
 \label{eq-prol1}
 X\cdot\phi_2=\partial_{\mathfrak{a}}(\phi_1)=-\mathrm{ad}(\phi_1)\vert_{\mathfrak{a}}
 \end{equation}
  that $\phi_2$ has to be the sum of an element in $\p_0$ and  an element in the complement of $\p_{0}\subset\mathfrak{a}^*\otimes\mathfrak{a}$ that is annihilated by the action of $X$. But then the assumption that $\phi_2$ is of positive homogeneity implies that  $\phi_2=\lambda\, \mathrm{ad}(Y)\vert_{\mathfrak{a}}\subset\mathfrak{a}^*\otimes\mathfrak{a}$. It follows that 
 $X\cdot\phi_2=\lambda\, \mathrm{ad}(H)\vert_{\mathfrak{a}},$
  which combined with \eqref{eq-prol1} gives $\phi_1=-\lambda H$. 
Thus, we obtain
$$\Phi=\begin{pmatrix} \phi_1\\ \phi_2\end{pmatrix}=\begin{pmatrix} X\cdot(-\lambda Y)\\ \partial_{\mathfrak{a}}(-\lambda Y)\end{pmatrix}=\partial_{\mfk_{-}}(-\lambda Y).$$
\end{proof}

Now we can prove Theorem \ref{thm-CartanConnection}.
\begin{proof}[Proof of Theorem \ref{thm-CartanConnection}]
Since $H^1(\mfk_{-}, \mfk)_l=\{0\}$ for all $l>0$ by Proposition \ref{prop-prol},  the result follows from Theorem \ref{eqcat}.
\end{proof}

\subsection{Conformally quasi-symplectic condition and variational ODEs}
\label{sec:spec-conf-quasi}
In the following, we  introduce  PCQ structures as a  subclass of PACQ structures. To do so, we first recall the following terminology.
\begin{definition}
\label{def-qsymplectic}
Let $\cC$ be a smooth manifold of odd dimension.
\begin{itemize}
\item A  $2$-form $\rho$ on $\cC$ is called  \emph{quasi-symplectic} if it is closed and has 
maximal rank. 
\item An \emph{almost  conformally quasi-symplectic structure} $(\ell, \cC)$  is given by a line sub-bundle $\ell\subset\biw^2T^*\cC$ such that for each $x\in \cC$ any non-zero element of $\ell_x$ has
maximal rank. 
\item The structure $(\ell, \cC)$ is called \emph{conformally quasi-symplectic} if in a neighborhood of  each point $x\in \cC$ there exists a nowhere vanishing smooth closed section $\rho\in\Gamma(\ell\vert_{U})$.
\end{itemize}
\end{definition}
Since the dimension of $\cC$ is odd, the condition that $\rho$ has maximal rank means that, when viewed as a map $T\cC\to T^*\cC$, it has one-dimensional kernel. We refer to  $\mathrm{ker}(\rho)$ as the \emph{characteristic} of $\rho$ or $\ell$ interchangeably.

The following proposition explains the name parabolic almost conformally quasi-symplectic (PACQ) structures for the   structures introduced in Definition \ref{def-PACQ}.
\begin{proposition}
\label{prop-lineb}
Every  PACQ
 structure  determines a canonical line bundle  $\ell\subset\biw^2T^*\cC$ with the property that for each $x\in \cC$ any non-zero element of $\ell_x$ has
maximal rank.
In other words, it has a canonical almost conformally quasi-symplectic structure.
\end{proposition}
\begin{proof}
By Theorem \ref{thm-CartanConnection}, a parabolic almost quasi-symplectic structure on a smooth manifold $\mathcal{C}$ determines a canonical Cartan geometry $(\mathcal{G}\to \mathcal{C},\psi)$ of type $(\mfk,L)$ and thus an identification
\begin{equation}\label{eq-ident2form}
\biw^2 T^*\mathcal{C}\cong\mathcal{G}\times_{L}\biw^2(\mfk/\mfl)^*.
\end{equation}
Considering the  contact grading \eqref{contactgrading} on $\g$, the Lie bracket defines a map $[,]:\biw^2\mathfrak{p}_{-1}\to\mathfrak{\p}_{-2}$. Via any identification $\p_{-2}\cong \mathbb{R}$, this map defines a non-degenerate $2$-form on $\p_{-1}$. In particular, it determines a  one-dimensional $\p_0$-invariant 
subspace in $\biw^2\mathfrak{p}_{-1}^*$ whose non-zero elements are non-degenerate $2$-forms.

Using the decomposition \eqref{aplusp0}, where $\mathfrak{a}\cong\p_{-1}$, we can identify $\mfk/\mfl\cong  \mathfrak{p}_{-1}\oplus \p_0/\mfl$ as  $L$-modules. Consequently, the map $[,]$  determines a one-dimensional $L$-invariant subspace in $\biw^2(\mfk/\mfl)^*$  comprised of 2-forms with one-dimensional kernel corresponding to $\p_0/\mfl$. 
Via the identification \eqref{eq-ident2form}, this subspace gives rise to an almost conformally quasi-symplectic structure on $\mathcal{C}$.
\end{proof}
The canonical almost conformally quasi-symplectic structure determined by a PACQ structure is \emph{compatible} with its underlying  filtration in the following way.

\begin{definition}
\label{def_compatible}
Let 
$$T^{-1}\mathcal{C}=\mathcal{E}\oplus\mathcal{V}\subset\cdots\subset T^{-k+1}\mathcal{C}=T\mathcal{C}$$
 be the filtration plus splitting (pseudo-product structure) corresponding to a PACQ structure.
 We say that  an almost conformally quasi-symplectic structure $\ell\subset\biw^2T^*\mathcal{C}$ is \emph{compatible}  if the following conditions hold:
\begin{enumerate}
\item The line bundle $\mathcal{E}$ is the characteristic of $\ell$.
\item  For any $\rho\in\Gamma(\ell)$ one has 
$$\rho\vert_{T^{-i}\mathcal{C}\times T^{-j}\mathcal{C}}\equiv 0\quad \mbox{if}\quad i+j<k .$$ In particular, any $T^{-i}\mathcal{C}$ with $i\leq (k-1)/2$ is isotropic.
\item For  $E\in\Gamma(\mathcal{E})$,  $X\in\Gamma(T^{-i}\cC)$ and  $Y\in\Gamma(T^{-j}\cC)$, where $i+j<k$, one has
\begin{equation}\label{eq-Ecomp}
\rho([E,X],Y)=-\rho(X,[E,Y]),
\end{equation}
or equivalently, 
$$\mathcal{L}_{E}\rho\vert_{T^{-i}\mathcal{C}\times T^{-j}\cC}\equiv 0\quad \mbox{if}\quad i+j<k .$$
\end{enumerate}
\end{definition}

The properties of the Lie bracket imply compatibility of the canonical almost conformally quasi-symplectic structure from Proposition \ref{prop-lineb}:
\begin{proposition}
\label{prop_filtPACQ}
The canonical almost conformally quasi-symplectic structure $\ell\subset\biw^2T^*\mathcal{C}$ determined by a PACQ structure 
is compatible with the underlying filtration and splitting corresponding to the PACQ structure.
\end{proposition}
\begin{proof}
The statements follow from the fact that $\rho$ is induced from the Lie bracket component $\biw^2\p_{-1}\to\p_{-2}$ of the  contact grading \eqref{contactgrading}, or equivalently  $\biw^2(\q_{-1}\oplus\cdots\oplus\q_{-k+1})\to\q_{-k}$ of the corresponding quasi-contact grading \eqref{qsymplecticbracket}. To check condition (1) in Definition \ref{def_compatible} note that $\mathcal{E}$ corresponds to $\mathfrak{e}\cong\p_0/\mfl$. For condition (2) in Definition \ref{def_compatible} one uses that the Lie bracket is compatible with the grading and thus for $X\in\q^{-i},$ $Y\in\q^{-j}$ one has $[X,Y]\in\q^{-i-j}$, which has no $\q_{-k}$ component if $i+j<k$. Finally, one computes
$$[[E,X],Y]=-[[X,Y],E]-[[Y,E],X]=-[X,[E,Y]]$$
for $X\in\q_{-i}$, $Y\in\q_{-j}$, $E\in\tilde{\mathfrak{e}}\subset\q_{-1}$ with $i+j+1=k$.
Taking into account condition (2) this implies condition (3) in Definition \ref{def_compatible}.
\end{proof}

For  any choice of $\rho\in\Gamma(\ell)$   and $E\in\Gamma(\mathcal{E})$ one can define a tensor field  $\bh\in\Gamma(\mathcal{V}^*\otimes\mathcal{V}^*)$ via
$\mathbf{h}(X,Y)=\rho(\mathcal{L}_E^{k-2}X,Y)$. It follows from properties (2) and (3) in Definition \ref{def_compatible} that $\bh$ is non-degenerate, and symmetric for $k$ odd and skew-symmetric for $k$ even. 
In particular, we have the following.

\begin{lemma}\label{corr-orthopath}
A path geometry $T^{-1}\cC=\mathcal{E}\oplus\mathcal{V}\subset T^{-2}\cC=T\cC$  together with a compatible almost conformally quasi-symplectic structure $\ell\subset\biw^2T^*\mathcal{C}$ determines  the structure of an orthopath geometry on $\cC$. 
\end{lemma}
\begin{proof}
For a section $\rho\in\Gamma(\ell)$   and $E\in\Gamma(\mathcal{E})$ one obtains a  bundle metric on $\mathcal{V}$  defined as
\begin{equation}\label{eq-hfromrho}
\bh(X,Y)=\rho([E,X],Y)
\end{equation}  
for any $X,Y\in\Gamma(\mathcal{V})$.  Using the properties of the filtration and compatibility of $\rho$, one verifies that $\bh$ is non-degenerate and symmetric and its conformal class $[\bh]$ is uniquely determined by $\ell$.
\end{proof}

We now introduce the class of PCQ structures.
\begin{definition}
A parabolic conformally quasi-symplectic (PCQ) structure is a PACQ structure whose canonical almost conformally quasi-symplectic structure admits local closed sections.
\end{definition}
It turns out that for PCQ structures one can derive a characterization of the associated canonical conformally quasi-symplectic structure $\ell$  based on the properties discussed before. 
\begin{proposition}
\label{lem-canonicalconfqsymp} Consider a PCQ structure and let $\ell$ be
 a compatible conformally quasi-symplectic structure on $\cC$. In case the PCQ structure defines an orthopath geometry, assume that  the bilinear form associated to any section $\rho\in\Gamma(\ell)$ as in \eqref{eq-hfromrho} is contained in the conformal class $[\bh]$ belonging to the orthopath geometry. Then $\ell$ is the canonical conformally quasi-symplectic structure associated to the PCQ structure.
\end{proposition}
\begin{proof}
The result follows from   \cite[Lemma 4.3]{Fels-ODE} for the case of 4th order ODEs, and Lemmas \ref{lemm:36-variationality-3rd-pair-ODEs} and \ref{lemm:causal-variationality-orthopath} for the other two cases of PCQ structures.
\end{proof}

\begin{remark}
Note that if $\rho\in\Gamma(\ell)$ is closed, then  $$\mathcal{L}_E\rho=\mathrm{d}(E\im\rho)+X\im \mathrm{d}\rho=0$$ for $E\in\Gamma(\mathcal{E})=\Gamma(\mathrm{ker}(\rho))$. In particular, property (3) in Definition \ref{def_compatible} is a consequence of property (1) for a closed $2$-form.
\end{remark}

To derive the following results,
recall that a $2$-form $\rho$ of rank $p$ on a manifold of dimension $n=2p+1$ can be locally written as
$\smash{\rho=\sum_{i=1}^{p}\varphi^i\wedge\varphi^{i+p}}$
for pointwise linearly independent $1$-forms $\varphi^1,\dots,\varphi^{2p}$. The condition  $\rho\wedge\alpha=0$ for a $1$-form $\alpha$  implies $\alpha=0$ if $n\geq 5$, and  the condition $\rho\wedge\beta=0$ for a $2$-form $\beta$ implies $\beta=0$   if $n>5$.

\begin{proposition}
Let $\ell\subset\biw^2T^*\cC$ be an almost conformally  quasi-symplectic structure on a manifold of dimension $\mathrm{dim}(\cC)\geq 5$.  Then the following holds:
\begin{itemize}
\item
Local closed sections of $\ell$ are uniquely determined up to constants.
\item Conditions (1) and (2) below are equivalent.
\begin{enumerate}
\item For any local section $\rho$ of $\ell$ there exists a one-form $\phi$ such that $\mathrm{d}\rho=\phi\wedge\rho$.
\item  For any compatible connection $\nabla$  with torsion $T$ the trace-free part of $\rho_{c[d}T_{ab]}{}^c$ vanishes.
\end{enumerate}
\item
If $\mathrm{dim}(\cC)>5$, then the structure is conformally quasi symplectic if and only if the equivalent conditions (1) and (2) are satisfied.
\end{itemize}

\end{proposition}
\begin{proof}
For the uniqueness statement consider a rescaling $\hat{\rho}=f \rho$ by a smooth function. If both $\rho$ and $\hat{\rho}$ are closed, then one has $\mathrm{d}f\wedge\rho=0$, which implies $\mathrm{d}f=0$.

For a connection $\nabla$ with torsion $T$ we have
\begin{equation}
\begin{aligned}
\mathrm{d}\rho(X,Y,Z)&=(\nabla_X\rho)(Y,Z)+(\nabla_Y\rho)(Z,X)+(\nabla_Z\rho)(X,Y)\\
&+\rho(T(X,Y),Z)+\rho(T(Y,Z),X)+\rho(T(Z,X),Y)
\end{aligned}
\end{equation}
for any $X,Y,Z\in\Gamma(T\cC)$. Since by assumption $\nabla$ preserves $\ell$, this shows that conditions (1) and (2) are equivalent.

To prove the last claim, first note that if $\rho$ is closed and $\hat{\rho}=f\,\rho$, then $\mathrm{d}\hat{\rho}=\mathrm{d}f\wedge \rho=(\mathrm{d}f/f)\wedge\hat{\rho}$. Conversely,  the condition $\mathrm{d}\rho=\phi\wedge\rho$ implies $\mathrm{d}\phi\wedge\rho=0$. Since  $\rho$ has maximal rank and $\mathrm{dim}(\cC)>5$, this implies $\mathrm{d}\phi=0$. It follows that one can  find a function $f$ on some open subset such that $-\phi=\mathrm{d}(\mathrm{log}(f))=\mathrm{d}f/f$. Consequently, a straightforward calculation shows that the rescaled 2-form $f\,\rho$ is closed.

\end{proof}

\begin{remark}
If $\mathrm{dim}(\cC)=5$ the equivalent conditions (1) and (2) are not sufficient for the structure to be conformally quasi-symplectic.
 Sufficient conditions for $5$-dimensional PACQ structures in terms of the associated regular and normal Cartan connection will be discussed in \ref{sec:inverse-probl-vari} and \ref{sec:conf-quasi-sympl-1}.
\end{remark}

It turns out that those PACQ structures that have underlying geometries of ODEs, the conformally quasi-symplectic condition means that the corresponding ODEs are variational.  This follows from 
general results on the inverse problem of the calculus of variations for ordinary differential equations \cite{AT-Book}, as discussed below.

In the following we denote by $(x,y^i_0,\cdots,y^i_n)$, $i=1,\cdots,k$, coordinates on jet space $J^n(\mathbb{R},\mathbb{R}^k)$.  The $1$-forms constituting the contact system on jet space  will be denoted by 
$$\theta^i_0=\mathrm{d}y^i_0-y^i_1\mathrm{d}x,\quad \theta^i_1=\mathrm{d}y^i_1-y^i_2\mathrm{d}x,\quad \dots \,\quad \theta^i_{n-1}=\mathrm{d}y^i_{n-1}-y^i_{n}\mathrm{d}x.$$
Given a differential equation,  we will, by abuse of notation,  use the same notation for the pull-backs of the forms to the equation manifold.

An $n$th order variational problem is the problem of finding a function $u:(a,b)\to\mathbb{R}^k$ that extremizes the functional
$$I[u]=\int_{a}^bL\left(x,u^i(x),\cdots,(u^i(x))^{(n)}\right)\mathrm{d}x,$$
where $L=L(x,y^i_0,\cdots,y^i_n)$ is a function of $k(n+1)+1$ variables, referred to as the \emph{Lagrangian}. Any (smooth) extremal is a solution to the corresponding \emph{Euler-Lagrange equations}
$$E_i(L)=\tfrac{\partial L}{\partial {y^i_0}}-\tfrac{\mathrm{d}}{\mathrm{d}x} \tfrac{\partial L}{\partial {y^i_1}}
+(\tfrac{\mathrm{d}}{\mathrm{d}x})^2
 \tfrac{\partial L}{\partial {y^i_2}}-...+(-1)^n(\tfrac{\mathrm{d}}{\mathrm{d}x})^n
 \tfrac{\partial L}{\partial {y^i_n}}=0,$$
 where
 \begin{equation}
 \tfrac{\mathrm{d}}{\mathrm{d}x}=\tfrac{\partial}{\partial x}+y^i_1\tfrac{\partial}{\partial y^i_0}+\cdots+y^i_{n+1}\tfrac{\partial}{\partial y^i_{n}} 
 \end{equation}
is the total derivative.

The inverse problem of the calculus of variations is concerned with the problem of determining whether a given system of differential equations is variational.
A system of ordinary differential equations 
\begin{align}\label{equationsPACQ}
y^i_n=F^i\left(x,y^j_0,y^j_1,\cdots,y^j_{n-1}\right),\quad i=1,\cdots,m
\end{align} is said to be \emph{variational} if there exists a non-singular matrix $B=[B_{ij}]$ of smooth functions (the \emph{variational multiplier}) and a Lagrangian $L$ such that 
\begin{equation}\label{varmult}E_j(L)= B_{ij}\left(y^i_{n}-F^i \right).
\end{equation}

\begin{theorem}
\label{prop-variational}
A (system of) ordinary differential equations 
\eqref{equationsPACQ}
  determined by a parabolic conformally quasi-symplectic structure is variational. More precisely,
  for  (1) a scalar fourth order ODE (2) a pair of third order ODEs (3) a system of n second order ODEs determined by a PCQ structure there exists a variational multiplier $B$ and
  \begin{enumerate}
  \item a non-degenerate second order Lagrangian 
  $$L=L(x,y_0,y_1,y_2),\quad \tfrac{\partial^2 L}{\partial {y_2}\partial {y_2}}\neq 0,$$
  \item a degenerate second order Lagrangian 
$$L=L_1(x,y^1_0,y^2_0,y^1_1,y^2_1) y^1_2+L_2(x,y^1_0,y^2_0,y^1_1,y^2_1)
y^2_2+L_0(x,y^1_0,y^2_0,y^1_1,y^2_1),\quad\tfrac{\partial L_1}{\partial {y^2_1}}-\tfrac{\partial L_2}{\partial {y^1_1}}\neq 0,$$
  \item a non-degenerate first order Lagrangian 
  $$L=L\left(x,y^1_0,\cdots,y^n_0,y^1_1,\cdots,y^n_1\right), \quad\mathrm{det}\left(\frac{\partial^2 L}{\partial {y^i_1}\partial {y^j_1}}\right)\neq 0,\quad i,j\in\{1,\cdots,n\},$$
  \end{enumerate}
such that \eqref{varmult} holds. 

Conversely,  a  Lagrangian of the forms above determines a unique PCQ structure on the equation manifold of the corresponding Euler-Lagrange equations.
\end{theorem}
 
\begin{proof}
We first show that a Lagrangian of any of the types above determines a PCQ structure on the equation manifold $\cC$ of the corresponding Euler-Lagrange equations. 
Given a degenerate second order Lagrangian as in (2), the Euler-Lagrange equations are of the form
$$E_j(L)=\sum_{i=1}^2(\tfrac{\partial L_i}{\partial_{y^j_1}}-\tfrac{\partial L_j}{\partial_{y^i_1}})y^i_3+G_j(x,y^1_0,y^2_0,y^1_1,y^2_1,y^1_2,y^2_2)=0, \quad j\in\{1,2\}.$$
Since by assumption $\tfrac{\partial L_1}{\partial_{y^2_1}}-\tfrac{\partial L_2}{\partial_{y^1_1}}\neq 0,$ these are pairs of $3$rd order ODEs. 
Similarly, given a non-degenerate second order Lagrangian as in (1), the Euler Lagrange equation is a scalar 4th order ODE,
and given a non-degenerate first order Lagrangian as in (3), the Euler-Lagrange equations are a second order system of ODEs.

In cases (1) and (2) the Euler-Lagrange equations already uniquely determine a corresponding PACQ structure by the discussion in \ref{sec-4thorder} and \ref{sec-pairsthirdorder}. To see that the PACQ structure is indeed PCQ and to obtain the bundle metric $\bh$ defining an orthopath geometry in case (3), observe that the  Lagrangian also determines a closed $2$-form $\rho$ up to scale as follows. For a Lagrangian of order one and two, define  $\eta$,  as 
$$\eta=-\tfrac{\partial L}{\partial y^i_1}\theta^i_0,\quad\mbox{and}\quad\eta= (-\tfrac{\partial L}{\partial y^i_1}+\tfrac{\mathrm{d}}{\mathrm{dx}}\tfrac{\partial L}{\partial y^i_2})\theta^i_0-\tfrac{\partial L}{\partial y^i_2}\theta^i_1,$$ respectively.  Denote  the pull-back of $\eta$ to the equation manifold by the same name. Then   $\rho=\mathrm{d}_V\eta$ is the corresponding closed $2$-form.  It can also be written as the pull-back of  $\mathrm{d}\Theta$ to the equation manifold, where $\Theta=\eta-L\mathrm{d}x$ is the Poincar\'{e}-Cartan form. Computing the $2$-form $\rho=\mathrm{d}_V\eta$ shows that it is compatible with the canonical filtration on the equation manifold $\cC$ in the sense of Proposition \ref{prop_filtPACQ}. 
 In particular, starting with a first order Lagrangian as in case (3) of the proposition above, the $2$-form $\rho$  satisfies 
$\rho(\tfrac{\partial}{\partial {y^j_1}},\tfrac{\partial}{\partial {y^i_1}})=0$ and $\rho(\tfrac{\partial}{\partial {y^j}},\tfrac{\partial}{\partial {y^i_1}})=\frac{\partial^2 L}{\partial {y^i_1}\partial {y^j_1}}$. By assumption one has $\mathrm{det}(\frac{\partial^2 L}{\partial {y^i_1}\partial {y^j_1}})\neq 0$. The bundle metric $\bh$ is then determined by $\rho$ as in Lemma \ref{corr-orthopath}. In all three cases, by Proposition \ref{lem-canonicalconfqsymp}, the algebraic form of the $2$-form $\rho$ shows that it spans the  canonical conformally quasi-symplectic structure associated with the PACQ structure. Consequently, the structure is PCQ.

To show that the ODEs corresponding to PCQ structures are variational, we use  that by Propositions \ref{prop-lineb} and \ref{prop_filtPACQ} any PCQ structure is locally equipped with a closed $2$-form $\rho\in\Gamma(\ell)$ of a particular algebraic type  compatible with the underlying filtration. 
For systems of ODEs of even order    \cite[Theorem 2.6]{AT-Book} shows that the existence of a closed $2$-form with these properties implies that the equations are variational with respect to non-degenerate Lagrangians as in (1) and (3). 

The case of pairs of third order ODEs is not explicitly treated in \cite{AT-Book}, but it can be dealt with in an analogous way. 
Let $\rho$ be a closed section of the canonical line bundle $\ell\subset\biw^2T^*\mathcal{C}$.
 The algebraic properties from Proposition \ref{prop_filtPACQ} imply that when written out in terms of pull-backs of the contact forms to $\cC$, i.e. $\theta^i_0=\mathrm{d}y^i-y^i_1\mathrm{d}x$, $\theta^i_1=\mathrm{d}y^i_1-y^i_2\mathrm{d}x$, $\theta^i_2=\mathrm{d}y^i_2-F^i\mathrm{d}x$, the $2$-form $\rho$ is of the form 
 \begin{equation}
 \label{eq-formrho}
 \rho=a_{ij}\theta^i_2\wedge\theta^j_0+b_{ij}\theta^i_1\wedge\theta^j_1+c_{ij}\theta^i_1\wedge\theta^j_0+d_{ij}\theta^i_0\wedge\theta^j_0,
 \end{equation} where $a_{ij}=-a_{ji}\neq 0$. Consider the  inclusion $\iota:\mathcal{C}\to J^\infty(\RR,\RR^2)$ into infinite order jet space by prolongation of the equation.
Let  $\mathrm{d}=\mathrm{d}_V+\mathrm{d}_H$ be the decomposition of the exterior derivative  into vertical and horizontal part on  $J^\infty(\RR,\RR^2)$, and denote by the same symbols their pull-back to $\mathcal{C}$. 
Using the variational bicomplex, one shows as in the proofs of  \cite[Theorem 2.6 and Lemma 2.7]{AT-Book} that there exists a $1$-form $\eta$ on $\mathcal{C}$ such that one has  
\begin{equation}\label{eq-rho}
\rho=\mathrm{d}_V\eta,
\end{equation}
where $ \eta=B_i\theta^i_0+C_i\theta^i_1$ is contained in the span of the contact forms $\theta^i_0$ and $\theta^i_1$. Moreover, there exists  a horizontal $1$-form $\lambda=L \mathrm{d}x$ such that
\begin{equation}\label{eq-dVla=dHeta}
\mathrm{d}_V\lambda=\mathrm{d}_H\eta,
\end{equation}
 where $L$ is a Lagrangian for the given system of equations. This implies that $C_i=-\tfrac{\partial L}{\partial y^i_2}$ and $B_i=-\tfrac{\partial L}{\partial y^i_1}-\tfrac{\mathrm{d}}{\mathrm{d}x}C_i$. Consequently, $\tfrac{\partial^2 L}{\partial {y^i_2}\partial {y^j_2}}$ coincides with the coefficient of $\theta^i_2\wedge\theta^j_1$ in $\rho$, but the assumed form \eqref{eq-formrho} of $\rho$ shows that this is zero. Assuming vanishing of the  Hessian   of $L$ with respect to the variables $y^i_2$, one computes that the coefficient of $\theta^i_2\wedge\theta^j_0$ coincides with $\tfrac{\partial L_i}{\partial {y^j_1}}-\tfrac{\partial L_j}{\partial {y^i_1}}$, and this  is by assumption non-zero. Thus, $L$ is of the stated form.

\end{proof}
\begin{remark}
Note that for conformally quasi-symplectic  4th order ODEs $y_4=f(x,y,y_1,y_2,y_3)$   one has $f_{y_3 y_3 y_3}=0$ by \cite{Fels-ODE}, see also \ref{sec:local-form-invar}.
Conformally quasi-symplectic pairs of third order ODEs are linear in second order derivatives by Proposition \ref{prop:variationality-pair-3rd-order-ODEs} and the explicit expression of $\mathbf{b}_1$ given in \ref{sec:local-form-invar}. It follows that the equations satisfy a polynomial dependency condition as required in   \cite[Theorem 2.6]{AT-Book}. Note also that the Lagrangian associated with a PCQ structure in Theorem \ref{prop-variational} is not unique. It is  known that \emph{divergence equivalent} Lagrangians determine the same Euler-Lagrange equations.  For a detailed discussion in the case  of orthopath geometries see  \ref{sec:div-equiv-pseudo-Finsler-as-var-orthopath}.
\end{remark}

\section{Symmetry reduction and quasi-contactification}
\label{sec:quasi-cont-reudct}
In this section we present the geometric constructions that relate  parabolic quasi-contact cone structures with an infinitesimal symmetry and parabolic conformally quasi-symplectic  structures.
\subsection{Transverse symmetries,  quotients, and quasi-contactification}
\label{sec:reduct-an-transv}
In what follows we consider PCQ structures that admit a transverse infinitesimal symmetry which means the following.
\begin{definition}\phantom{}
\begin{enumerate}
\item
An infinitesimal symmetry  of a  quasi-contact  structure $(\tilde{\mathcal{C}},\mathcal{H})$ is a vector field  $\xi\in\mathfrak{X}(\tilde{\mathcal{C}})$ such that $\mathcal{L}_{\xi}\eta\subset \Gamma(\mathcal{H})$ for any $\eta\in \Gamma(\mathcal{H})$. 
\item
An infinitesimal symmetry  of a parabolic quasi-contact cone structure is a vector field $\xi\in\mathfrak{X}(\tilde{\mathcal{C}})$ such that  $\mathcal{L}_{\xi}\eta\subset \Gamma(T^{-1}\tilde{\mathcal{C}})$ for any $\eta\in \Gamma(T^{-1}\tilde{\mathcal{C}})$. Such a vector field preserves the filtration generated by $T^{-1}\tilde{\mathcal{C}}$ and in particular  the quasi-contact distribution $\mathcal{H}= T^{-k+1}\tilde{\mathcal{C}}$.
\item An infinitesimal symmetry $\xi\in\mathfrak{X}(\tilde{\mathcal{C}})$ is called \emph{transverse} if  $\xi_x\notin\mathcal{H}_x$ for any $x\in\tilde{\mathcal{C}}$, where $\mathcal{H}= T^{-k+1}\tilde{\mathcal{C}}$.
\end{enumerate}
\end{definition}

\begin{remark} Note that any infinitesimal symmetry of a parabolic quasi-contact cone structure is transverse almost everywhere.   To see this, suppose $\xi$ is an infinitesimal symmetry that is contained in $\mathcal{H})=\ker(\alpha)$ on an open neighborhood $U$ of a point $x\in M$. Then $\mathrm{d}\alpha(\xi,\eta)=-\alpha([\xi,\eta])=0$ for any $\eta\in\Gamma(\mathcal{H})$ by the symmetry property. It follows, that $\xi$ is contained in the kernel of $\mathrm{d}\alpha$ restricted to $\mathcal{H}$, i.e. in $\tilde{\mathcal{E}}$. Moreover, for a parabolic quasi-contact cone structure the map $\mathfrak{L}:\tilde{\mathcal{E}}\otimes\tilde{\mathcal{V}}\to \mathrm{gr}_{-2}(\tilde{\mathcal{C}})=T^{-2}\tilde{\mathcal{C}}/T^{-1}\tilde{\mathcal{C}}$ induced by the Lie bracket is an isomorphism by Proposition \ref{prop-qcontactgrading}. But an infinitesimal symmetry, via the Lie bracket, preserves the distribution $T^{-1}\tilde{\mathcal{C}}=\tilde{\mathcal{E}}\oplus\tilde{\mathcal{V}}$. This implies that $\xi$ has to be contained in  the kernel of $\mathfrak{L}$ as defined above and thus to vanish on $U$. 
Since an infinitesimal symmetry that vanishes on an open subset vanishes on the entire connected component, the claim follows.
\end{remark}

\begin{remark}
Parabolic quasi-contact cone structures that arise as correspondence spaces generally do not admit (everywhere) transverse infinitesimal symmetries, with the exception of  
the class of  Lorentzian conformal structures with a  time-like conformal Killing field. 
 This can be seen from the description of the quasi-contact distribution $T^{-k+1}\tilde{\mathcal{C}}$ on $\tilde{\cC}$ given in Remark \ref{rem-correspondence}, which shows that in general the lift of a downstairs infinitesimal symmetry will satisfy $\xi_x\in T^{-k+1}_x\tilde{\mathcal{C}}$ at some points $x$ in each fiber, namely those points that correspond to null lines orthogonal to the symmetry.
Only in the case of  Lorentzian conformal structures the lift of any time-like  conformal Killing field  to the correspondence space is a nowhere vanishing, i.e. transverse infinitesimal symmetry of the corresponding quasi-contact cone structure. 
\end{remark}

Since a transverse infinitesimal symmetry is nowhere vanishing, it defines a one-dimensional foliation on $\tilde{\cC}$. 
After restricting to an open subset of $\tilde{\mathcal{C}}$, it determines a local leaf space $\pi\colon\tilde{\mathcal{C}}\to \mathcal{C}$.

\begin{proposition}
\label{prop-qsymplectic}
A transverse infinitesimal symmetry $\xi$ of a quasi-contact structure $(\tilde{\mathcal{C}},\mathcal{H})$ determines a unique quasi-contact form $\alpha$ on $\tilde{\mathcal{C}}$ such that $\alpha(\xi)=1$
 and a unique quasi-symplectic form $\rho$ on the local leaf space $\mathcal{C}$  of the integral curves of $\xi$ such that $\mathrm{d}\alpha=\pi^*\rho$.
\end{proposition}
\begin{proof}
A transverse symmetry $\xi$ determines a unique quasi-contact form $\alpha$ by the requirement that $\mathrm{ker}(\alpha)=\mathcal{H}$ and $\alpha(\xi)=1$. From  $\alpha(\mathcal{L}_\xi \eta)=\alpha([\xi,\eta])=-\mathrm{d}\alpha(\xi,\eta)=0$ for any $\eta\in\Gamma(\ker(\alpha))$ one derives  that $\xi\im\mathrm{d}\alpha=0$ and $\mathcal{L}_\xi\mathrm{d}\alpha=0$, which implies that there exists a $2$-form $\rho$ on $\cC$ such that $\mathrm{d}\alpha=\pi^*\rho$ \cite[Corollary 2.3]{EDS-book}. Since $\mathrm{d}$ commutes with pull-backs and $\pi$ is a surjective submersion, $\rho$ is uniquely determined and closed. Moreover, since $\mathrm{d}\alpha$ restricted to the quasi-contact distribution has maximal rank, this implies that $\rho$ is  quasi-symplectic. 
\end{proof}

\begin{theorem}\label{thm-quasicont}
Given a  parabolic quasi-contact cone structure   with a transverse infinitesimal symmetry $\xi\in\mathfrak{X}(\tilde{\mathcal{C}})$, the symmetry determines  a local leaf space $\pi\colon\tilde{\mathcal{C}}\to \mathcal{C}$ and a canonical PCQ structure  on $\cC$ such that the filtration on $T\tilde{\cC}$ projects to the filtration on $T\cC$ and  the quasi-contact structure on $\tilde{\cC}$ induces the canonical conformally quasi-symplectic structure $\ell\subset\biw^2T^*\mathcal{C}$ corresponding to the PCQ structure.

Conversely, any PCQ structure on a manifold $\mathcal{C}$ can be locally realized as the quotient of a parabolic quasi-contact cone structure by a transverse infinitesimal symmetry in this way.
\end{theorem}

\begin{proof}
After restricting to an open subset if necessary, we have a surjective submersion $\pi\colon\tilde{\mathcal{C}}\to \mathcal{C}$ such that the kernel of its tangent map is spanned by $\xi_x$ for any $x\in\tilde{\mathcal{C}}$. Since $\xi$ is transverse and preserves all the sub-bundles in the filtration $$T^{-1}\tilde{\cC}\subset T^{-2}\tilde{\cC}\subset \cdots\subset T^{-k+1}\tilde{\cC}\subset T^{-k}\tilde{\cC}=T\tilde{\cC},$$ there is an induced filtration 
$$T^{-1}\cC\subset T^{-2}\cC\subset \cdots\subset T^{-k+1}\cC=T\cC$$
by sub-bundles of $T\cC$ such that $\pi_{*}:T^i_x\tilde{\cC}\to T^i_{\pi(x)}\cC$ is an isomorphism for any $i=-1,\cdots,-k+1$ and $x\in \tilde{\cC}$.
Using the fact that for any $X\in\Gamma(T^i \cC), Y\in\Gamma(T^j \cC)$ and horizontal lifts $\tilde{X}\in\Gamma(T^i\tilde{\cC}), \tilde{Y}\in\Gamma(T^j\tilde{\cC})$ the Lie bracket $[\tilde{X},\tilde{Y}]$ is $\pi$-related to $[X,Y]$, one verifies that  $T^{-1}\cC$ is bracket generating with symbol algebra $\mfk_{-}=\mathfrak{q}_{-}/\mathfrak{q}_k,$ where $\mathfrak{q}_{-}$ is the symbol algebra of $T^{-1}\tilde{\cC}$.

Moreover, as shown in the proof of Proposition \ref{prop-qsymplectic}, an infinitesimal symmetry $\xi$ determines a quasi-symplectic form $\rho$, which spans a compatible conformally quasi-symplectic structure $\ell$ on $\cC$. This conformally quasi-symplectic structure determines a reduction of structure group to $Q_0\subset\mathrm{Aut}_{gr}(\mfk_{-})$ and thus a PCQ structure on $\cC$. More directly, the sub-bundles $\tilde{\mathcal{E}}$ and $\tilde{\mathcal{V}}$ of $T^{-1}\tilde{\cC}$ are preserved by $\xi$ and descend to sub-bundles $\mathcal{E}$ and $\mathcal{V}$ of $T^{-1}\mathcal{C}$. In the case where we start with a causal structure, the conformal class $[\bh]$ on $\mathcal{V}$ is determined by $\rho$ as in Lemma \ref{corr-orthopath}. 

For the converse, suppose we are given a PCQ structure on $\cC$. Let  $\beta\in\Omega^1(\cC)$  be a $1$-form  such that $\rho=\mathrm{d}\beta$ is a nowhere vanishing section of $\ell$ (possibly after restriction to an open subset). Defining $\tilde{\cC} =\cC\times \mathbb{R}$, let $t$ be the standard coordinate on the second factor and $\pi\colon\tilde{\cC}\to \cC$  the projection onto the first factor.
Now we define  $\alpha=\mathrm{d}t+\pi^*\beta,$ which satisfies $\mathrm{d}\alpha=\pi^*\rho$, and thus its kernel $T^{-k+1}\tilde{\cC}:=\ker(\alpha)$ defines a quasi-contact structure on $\tilde{\cC}$. We lift the remaining filtration components from $\cC$ to $\tilde{\cC}$ by setting $T^i_x\tilde{\cC}=\{X\in T^{-k+1}_x\tilde{\cC} :\, \pi_* X\in T^i_{\pi(x)}\cC\}$
for $i=-1,\dots,-k+2$. This defines a filtration of $T\tilde{\cC}$, which determines a corresponding quasi-contact cone structure. Moreover,  by construction $\tfrac{\partial}{\partial t}$ is a transverse infinitesimal symmetry of the structure. Forming the leaf space of $\tfrac{\partial}{\partial t}$, we recover  $\cC$ with its original PCQ structure. 
\end{proof}

As a consequence of Theorem \ref{prop-variational} and Theorem \ref{thm-quasicont} we have the following.
\begin{corollary}
\label{corr-variational}
ODE geometries arising from parabolic quasi-contact cone structures via symmetry reduction  are variational.
Conversely, any variational scalar 4th order ODE, variational pair of third order ODEs, and variational orthopath geometry arises as a symmetry reduction of a quasi-contact cone structure determined by a  $(2,3,5)$ distribution,  $(3,6)$ distribution, and causal structure, respectively.
\end{corollary}

A realization of a PCQ structure as a local quotient of a parabolic quasi-contact cone structure, as in Theorem \ref{thm-quasicont},  is referred to as a \emph{parabolic quasi-contactification}. Parabolic quasi-contactifications are locally unique in the following sense.

\begin{proposition}
Consider two PCQ structures of the same type, realized as quotients of parabolic quasi-contact cone structures by transverse infinitesimal symmetries. Then locally any morphism of the PCQ structures lifts to a morphism of the quasi-contact cone structures which is compatible with the infinitesimal symmetries up to a constant. 
\end{proposition}
\begin{proof}
 A morphism  between two conformally quasi-symplectic structures lifts to a morphism of the quasi-contactifications that respects the infinitesimal symmetries up to a constant. This can be shown exactly as in the classical case of  contactifications of conformally symplectic structures (see  the proof of   \cite[Proposition 3.1]{CS-cont0}). The proposition follows by noting that  such a lift automatically respects the lifted filtrations, i.e. it is an automorphism of the parabolic quasi-contact cone structures. 

\end{proof}

\subsection{Cartan holonomy reductions  associated with infinitesimal symmetries} 
\label{subsec-holonomy}
Interesting classes of parabolic quasi-contact cone structures with  an infinitesimal symmetry correspond to certain types of Cartan holonomy reductions. The aim of this section is to provide some background on how such holonomy reductions are related with symmetries and then to discuss two examples that are of interest in this context.

Let $(\tilde{\mathcal{G}}\to \tilde{\cC},\tilde{\psi})$ be a regular and normal parabolic geometry of  type $(\g,Q)$ for some parabolic subgroup $Q\subset G$ (not necessarily of quasi-contact type). The  associated vector bundle $\tilde{\mathcal{A}}=\tilde{\mathcal{G}}\times_Q\g$ is called the \emph{adjoint tractor bundle}. It is equipped with  the connection $\nabla^{\tilde{\mathcal{A}}}$ induced by the Cartan connection $\tilde{\psi}$, the so-called \emph{adjoint tractor connection}. We denote by $\Pi\colon\tilde{\mathcal{A}}\to \tilde{\cC}$ the natural projection corresponding to  $\g\to\g/\q$. 
There is a bijective correspondence between $Q$-invariant vector fields $\xi\in\mathfrak{X}(\tilde{\mathcal{G}})$  and sections of the  adjoint tractor bundle $\tilde{\mathcal{A}}$. The correspondence is given by mapping $\xi$ to the $Q$-equivariant function $\tilde{\psi}(\xi):\tilde{\mathcal{G}}\to\g$, which defines the section $s\in\Gamma(\tilde{\mathcal{A}})$.  Infinitesimal symmetries of the parabolic geometry correspond to  $Q$-invariant vector fields $\xi$ satisfying $\mathcal{L}_{\xi}\tilde{\psi}=0$, and via the above correspondence to sections $s\in\Gamma(\tilde{\mathcal{A}})$ satisfying $\nabla^{\mathrm{inf}}s:=\nabla^{\tilde{\mathcal{A}}}s+{\Pi(s)}\im\tilde{\Psi}=0$; here the curvature $\tilde{\Psi}$ of $\tilde{\psi}$ is viewed as a $2$-form on $\tilde{\cC}$ with values in $\tilde{\mathcal{A}}$.

By \cite[Corollary 3.5]{Cap-deform}, it follows that in torsion-free normal parabolic geometries, satisfying  a certain cohomological condition,
any  section $s\in\Gamma(\tilde{\mathcal{A}})$ that satisfies $\nabla^{\tilde{\mathcal{A}}}s=0$ automatically satisfies $\Pi(s)\im\tilde{\Psi}=0$ and thus $\nabla^{\mathrm{inf}}s=0$.  In particular, the result  applies  to conformal geometries, $(2,3,5)$ and $(3,6)$ distributions, and also their (quasi-contact) correspondence spaces. 
Infinitesimal automorphisms corresponding to parallel sections  for the normal tractor connection $\nabla^{\tilde{\mathcal{A}}}$ are called \emph{normal}. 
They give rise to holonomy reductions of the Cartan geometry $(\tilde{\mathcal{G}}\to \tilde{\cC},\tilde{\psi})$ of type $\mathcal{O}$, where $\mathcal{O}\subset\g$ is some $G$-orbit in $\g$, called the $G$-type of $s\in\Gamma(\tilde{\mathcal{A}})$, see \cite{CHG-holonomy} for details.
Two examples of such $G$-types are discussed below.

 \begin{example}\label{ex-Feff}
Consider a normal conformal parabolic geometry of signature $(2p+1,2q+1)$. Let $\g=\mathfrak{so}(2p+2,2q+2)\cong \biw^2\mathbb{R}^{2(p+q)+4}$ be the corresponding orthogonal Lie algebra  of   skew-symmetric endomorphisms with respect to a symmetric bilinear form of signature $(2p+2,2q+2)$ on $\mathbb{R}^{2(p+q)+4}$.  
Let $\mathbb{J}\in\g\cong\biw^2\mathbb{R}^{2(p+q)+4}$ be a skew-symmetric endomorphism satisfying $\mathbb{J}\circ\mathbb{J}=-\mathrm{id}$. Consider the corresponding $G$-orbit  $\mathcal{O}=G\cdot\mathbb{J}\subset\g$. For a conformal structure admitting a parallel adjoint tractor field $s\in\Gamma(\tilde{\mathcal{A}})$ of this $G$-type, the conformal holonomy  is contained in $\mathrm{U}(p+1,q+1)\subset \mathrm{SO}(2p+2,2q+2)$; in fact, it turns out that it further reduces to $\mathrm{SU}(p+1,q+1)$. 

The local leaf space  determined by the  (null) normal conformal  Killing field corresponding to a parallel adjoint tractor $s$ of the $G$-type above has an induced CR structure of hypersurface type. Conversely, the Fefferman construction associates to any non-degenerate CR structure of hypersurface type a canonical conformal structure on a circle bundle over the CR manifold that admits a holonomy reductions of this type. Fefferman conformal structures can be characterized either in terms of their conformal holonomy or in terms of the corresponding (null) normal conformal Killing field  \cite{Graham-Sparling}.  
Note that, in particular, Fefferman conformal structures have an associated canonical variational orthopath geometry.
This orthopath geometry, that we refer to as the orthopath geometry of chains, will be studied in   \ref{sec:vari-chains-cr-orthopath}.

\end{example}

\begin{example} Let $\g=\p_{-2}\oplus\p_{-1}\oplus\p_0\oplus\p_1\oplus\p_2$ be a contact grading of $\g$ as in \eqref{contactgrading}. Let $\mathcal{O}=G\cdot V\subset\g$ be the orbit of a root vector $V\in\p_{-2}$ under the adjoint representation.
  Then the stabilizer $H$ of $V$ is contained in the parabolic subgroup $P^{op}$ stabilizing the line $\p_{-2}\subset\mathbb{P}(\g)$. In particular, if a parabolic quasi-contact cone structure admits a parallel adjoint tractor field $s\in\Gamma(\tilde{\mathcal{A}})$ of this $G$-type $\mathcal{O}$, then the holonomy of the Cartan connection is a proper subgroup of the contact parabolic subgroup. In \ref{sec:curvatuer-analysis} we will construct, via  quasi-contactification, parabolic quasi-contact cone structures that admit holonomy reductions of this type (Corollaries \ref{cor:235-from-4th-ODE-various-curv-conditions}, \ref{cor:36-from-pair-3rd-ODE-various-curv-conditions}, and \ref{cor:causal-from-orthopath-various-curv-cond}). 
 \end{example}

\subsection{Relation between Cartan geometries and  curvatures}
\label{sec:relate-betw-cartan-geometries}

We have seen that PCQ structures as well as parabolic quasi-contact cone structures have associated canonical Cartan connections. Our next goal is to describe the relationship between these Cartan connections via parabolic quasi-contactifications.

 Let $\pi\colon\tilde{\cC}\to\cC$ be a local quotient of a quasi-contact cone structure by an infinitesimal symmetry $\xi$. Let $\alpha$ be the quasi-contact form on $\tilde{\cC}$ such that $\alpha(\xi)=1$ and $\rho\in\Gamma(\ell)$  be the closed section of  $\ell\subset\biw^2T^*\cC$ such that $\mathrm{d}\alpha=\pi^*\rho$. Denote by $(\tau\colon\mathcal{G}\to \cC,\psi)$  the regular and normal Cartan geometry of type $(\mathfrak{k},L)$ associated with the PCQ structure on $\cC$ and let $\Psi$ denote its Cartan curvature.

To construct the Cartan geometry associated to the quasi-contact cone structure, we use the principal bundle $\tau\colon\mathcal{G}\to \cC$ and projection  $\pi\colon\tilde{\cC}\to \cC$ to define the pull-back bundle $\hat{\tau}\colon\hat{\mathcal{G}}\to \tilde{\cC}$ 
  where
   \begin{align*} 
\hat{\mathcal{G}}:=\pi^*\mathcal{G}=\{(u,\tilde{x})\in\mathcal{G}\times\tilde{\cC}: \tau(u)=\pi(\tilde{x}) \}\quad\mbox{and}\quad\hat{\tau}\left( (u,\tilde{x})\right)=\tilde{x},
  \end{align*}
which is a $L$-principal bundle with the  $L$-action given by $r^g(u,\tilde{x})=(r^g(u),\tilde{x})$ for any $g\in L$. Projection onto the first factor defines $\hat{\pi}\colon\hat{\mathcal{G}}\to\mathcal{G}$.

Recall that $\mathfrak{k}$ was defined as a Lie algebra quotient $\p^{op}/\q_{-k}\cong\p^{op}/\p_{-2}$ in \eqref{Liealgs}. The quotient map 
restricts to a linear isomorphism between $\p_{-1}\oplus\p_0\subset\p^{op}$ and $\mfk$, whose  inverse will be denoted by \begin{align}
i:\mathfrak{k}\cong\p_{-1}\oplus\p_0=\q_{-k+1}\oplus\cdots\oplus\q_0\oplus(\p_0\cap\q_{+})\subset\g.
 \end{align}
Hence $i\circ\hat{\pi}^*\psi$ defines a $1$-form on $\hat{\mathcal{G}}$ with values in  $i(\mathfrak{k})\subset\g$.   
   One can  introduce an additional  $1$-form $\hat{\psi}_{-k}\in\Omega^1(\hat{\mathcal{G}},\mathfrak{q}_{-k})$ with values in the kernel of the quotient map as follows.
 For any $u\in\mathcal{G}$ with $\tau(u)=x$ there exists a unique $f_u:\mathfrak{q}_{-k}\to\mathbb{R}$ such that 
  $$\rho_x(\tau_* X_u,\tau_* Y_u)=f_u\left(\mathrm{pr}_{-k}([i\circ\psi(u)(X_u),i\circ\psi(u)(Y_u)])\right),
  $$
for any $X_u,Y_u\in T_u\mathcal{G}$,  where $\mathrm{pr}_{-k}$ denotes the projection from $\p^{op}$ onto $\q_{-k}\cong \p_{-2}$. This determines a smooth and $L$-equivariant function $f:\mathcal{G}\to L(\q_{-k},\mathbb{R})$, where $L(\q_{-k},\mathbb{R})$ denotes the space of linear maps from $\q_{-k}$ to $\mathbb{R}$. 
We now define  $$\hat{\psi}_{-k}(u,\tilde{x})=-f_u^{-1}\circ \hat{\tau}^*\alpha (u,\tilde{x})\quad \mbox{and}\quad\hat{\psi}=\hat{\psi}_{-k}+i\circ\hat{\pi}^*\psi\in\Omega^1(\hat{\mathcal{G}},\p^{op}).$$ 
 The 1-form $\hat{\psi}$ defined above is $L$-equivariant, $\hat{\psi}(u):T\hat{\mathcal{G}}\to\p^{op}$ defines an isomorphism for each $u\in\hat{\mathcal{G}}$,  and $\hat{\psi}$ reproduces generators of fundamental vector fields since $\psi$ does and $\hat{\psi}_{-k}(\zeta_X)=0$ for any fundamental vector field $\zeta_X$.  
Consequently, $\hat{\psi}$ defines a Cartan connection of type $(\p^{op},L)$ on $\hat{\cG}\to\tcC$.

Let $\hat{X},\hat{Y}$ be two vector fields on $\hat{\mathcal{G}}$ in the kernel of $\hat{\psi}_{-k}$.
One computes that
\begin{equation}
  \label{eqdnewform}
  \begin{aligned}
  \mathrm{d}\hat{\psi}_{-k}(u,\tilde{x})(\hat{X},\hat{Y})&=-\hat{\psi}_{-k}(u,\tilde{x})([\hat{X},\hat{Y}])
 =f_u^{-1}\left( \hat{\tau}^*\alpha (u,\tilde{x})([\hat{X},\hat{Y}])\right)\\&=-f_u^{-1}\left( \hat{\tau}^*\pi^*\rho (u,\tilde{x})(\hat{X},\hat{Y})\right)
 \\&=-\mathrm{pr}_{-k}([i(\hat{\pi}^*\psi(u,\tilde{x}) (\hat{X})),i(\hat{\pi}^*\psi(u,\tilde{x}) (\hat{Y}))]).
  \end{aligned}
  \end{equation}
Using this,  we get for the curvatures
 \begin{equation} 
 \label{curv1}
 \begin{aligned}
 \hat{\Psi}(\hat{X},\hat{Y})=&\mathrm{d}\hat{\psi}(\hat{X},\hat{Y})+[\hat{\psi}(\hat{X}),\hat{\psi}(\hat{Y})]\\=& i \left(\hat{\pi}^*\mathrm{d}\psi( \hat{X}, \hat{Y})\right)+\mathrm{d}\hat{\psi}_{-k}(\hat{X},\hat{Y})+[i(\hat{\pi}^*\psi( \hat{X})),i(\hat{\pi}^*\psi( \hat{Y}))]\\
 =&i\left(\hat{\pi}^*\Psi( \hat{X}, \hat{Y})\right).
  \end{aligned}
  \end{equation}

 Next consider the extended  $Q$-principal bundle $\tilde{\mathcal{G}}=\hat{\mathcal{G}}\times_{L}Q$ with respect to $j:L\to Q$ and denote by $\iota:\hat{\mathcal{G}}\to\tilde{\mathcal{G}}$  the natural inclusion.
 There is a unique equivariant extension of $\hat{\psi}$ to a Cartan connection on $\tilde{\mathcal{G}}$ with values in $\g$ which, by abuse notation, we shall denote by $\hat{\psi}\in\Omega^1(\tilde{\mathcal{G}},\g)$ as well.
 By regularity of $\psi$ and \eqref{eqdnewform}, $\hat{\psi}$ is a regular Cartan connection  on $\tilde{\cC}$. However, in general it is not normal.  Let $\tilde{\psi}=\hat{\psi}+\phi\in\Omega^1(\tilde{\mathcal{G}},\g)$ be the regular and normal Cartan connection inducing the filtration corresponding to the parabolic quasi-contact cone structure, which we know exists by the general theory. Then $\phi\in\Omega^1(\tilde{\mathcal{G}},\g)$ is a  horizontal and  $Q$-equivariant  1-form of homogeneity $\geq l>0$ (see \cite[ Proposition 3.1.10]{CS-Parabolic}).  The curvatures are related by
\begin{equation}
\label{curv2}
\begin{aligned}
\tilde{\Psi}(\tilde{X},\tilde{Y})=&\mathrm{d}\hat{\psi}(\tilde{X},\tilde{Y})+\mathrm{d}\phi(\tilde{X},\tilde{Y})+[\hat{\psi}(\tilde{X})+\phi(\tilde{X}),\hat{\psi}(\tilde{Y})+\phi(\tilde{Y})]\\
=& \hat{\Psi}(\tilde{X},\tilde{Y})+\mathrm{d}^{\hat{\psi}}\phi(\tilde{X},\tilde{Y})+[\phi(\tilde{X}),\phi(\tilde{Y})]
\end{aligned}
\end{equation}
for $\tilde{X}, \tilde{Y}\in\mathfrak{X}(\tilde{\mathcal{G}}).$
Here $\mathrm{d}^{\hat{\psi}}\phi$ is of homogeneity $\geq l$ and the last term is of homogeneity $\geq 2l\geq l+1$. 
 Moreover, let $h:\tilde{\mathcal{G}}\to(\g/\q)^*\otimes\g$ be the function corresponding to $\phi$, and let $g:\tilde{\mathcal{G}}\to\biw^2(\g/\q)^*\otimes\g$ be the function corresponding to  $\mathrm{d}^{\hat{\psi}}\phi$.  
 Then 
 \begin{equation}
 \label{curv3}
\mathrm{gr}_l\circ g=\partial_{\q_{-}}\circ\mathrm{gr}_l\circ h, 
 \end{equation}
 where $\mathrm{gr}_l$ denotes the projection of an element of homogeneity $\geq l$ to the component  $(\biw^i\q_{-}^*\otimes\g)_l=(\biw^i(\g/\q)^*\otimes\g)^{l}/(\biw^i(\g/\q)^*\otimes\g)^{l+1}$ of homogeneity $l$ in the associated graded space with respect to the filtration on $\biw^i(\g/\q)^*\otimes\g$, and $\partial_{\q_{-}}\colon(\q_{-})^*\otimes\g\to\biw^2(\q_{-})^*\otimes\g$ denotes the operator \eqref{formuladel}. See \cite[Proposition 4.3 and Theorem 4.4]{Cap-Cartan} for details.
 
 Viewing $\mathfrak{k}$ as a subspace in $\g$, we have an orthogonal $L$-invariant decomposition $\g=\mfk\oplus\mfk^{\perp}$ with respect to the inner product introduced in \ref{sec:norm-cond}.   
 Then $(\biw^2\mfk_{-}^*\otimes\mfk)_l$ can be identified with a subspace in $ (\biw^2\q_{-}^*\otimes\g)_l$ comprised of those maps that have values in $\mfk\subset\g$ and that vanish upon insertion of any element belonging to $\q_{-k}$. Let
 $\ker(\Box)\cong\ker(\tilde{\partial}^*)/\mathrm{im}(\tilde{\partial}^*)=H^2(\q_{-},\g)$ be the harmonic curvature space \eqref{harmon}. Using Kostant's theorem,  e.g. see  \cite[Theorem 3.3.5]{CS-Parabolic}, it is straightforward to determine the structure of this space. In particular, for quasi-contact cone structures the lowest homogeneous component of $\ker(\Box)$, denoted by
 $\ker(\Box)_l$, is contained in $(\biw^2\mfk_{-}^*\otimes\mfk)_l\subset (\biw^2\q_{-}^*\otimes\g)_l$.
\begin{proposition}
Let  $\mathrm{gr}_l(\kappa)$ be the component of lowest homogeneity of the  curvature function of the regular and normal Cartan connection $\psi\in \Omega^1(\mathcal{G},\mfk)$ and let $\mathrm{gr}_l(\tilde{\kappa})$ be the component of lowest homogeneity of the  curvature function of the regular and normal Cartan connection $\tilde{\psi}\in \Omega^1(\tilde{\mathcal{G}},\g)$. Denote by $\iota:\hat{\mathcal{G}}\to\tilde{\mathcal{G}}$ the inclusion and by $\hat{\pi}:\hat{\mathcal{G}}\to\mathcal{G}$ the projection introduced above.
Then $\iota^*\mathrm{gr}_l(\tilde{\kappa})$  coincides with the pull-back $\hat{\pi}^*(\mathrm{pr}_{\ker(\Box)_l}\mathrm{gr}_l(\kappa))$ 
of the part of $\mathrm{gr}_l(\kappa)$ that has values in $\ker(\Box)_l\subset (\biw^i\mathfrak{k}_{-}^*\otimes\mathfrak{k})_l$. 
\end{proposition}
\begin{proof}
 The lowest homogeneous component $\mathrm{gr}_l(\tilde{\kappa})$ of a regular and normal parabolic geometry has values in $\ker(\Box)_l\subset(\biw^i\q_{-}^*\otimes\g)_l$, see  \cite[Theorem 3.1.12]{CS-Parabolic}. Moreover, $\ker(\Box)_l$ can be identified with a subspace of $(\biw^i\mathfrak{k}_{-}^*\otimes\mathfrak{k})_l$ by the discussion above.
It follows from  \eqref{curv1}, \eqref{curv2} and \eqref{curv3}  that both $\kappa$ and $\tilde{\kappa}$ are of homogeneity $\geq l$ for the same $l$ and that the difference between $\iota^*\mathrm{gr}_l(\tilde{\kappa})$ and  $\hat{\pi}^*\mathrm{gr}_l(\kappa)$ is contained in $\mathrm{im}(\partial_{\q_{-}})$. Using the fact that $\mathrm{im}(\partial_{\q_{-}})$ is contained in the orthogonal complement to $\ker(\Box)_l\subset(\biw^i\q_{-}^*\otimes\g)_l$ by the Hodge-decomposition \eqref{eqHodge}, the result follows. 

\end{proof}
A more detailed relationship of the Cartan connections and  curvatures of the PCQ structures and the associated parabolic quasi-contact cone structures will be obtained in \ref{sec:curvatuer-analysis} by explicit computations for the individual geometries. In particular, we will prove the following proposition in three parts in \ref{sec:235-4th-ODE-quasi-cont}, \ref{sec:36-quasi-cont-from-pair-3rd-ODE} and \ref{sec:general-causal-quasi-cont}.
\begin{proposition}
\label{prop-relharmonics}
Let $(\tau:\mathcal{G}\to \cC,\psi)$ be the regular and normal Cartan geometry associated with a parabolic conformally quasi-symplectic structure and $(\tilde{\tau}:\tilde{\mathcal{G}}\to \tilde{\cC},\tilde{\psi})$ be the regular and normal Cartan geometry associated with the induced parabolic quasi-contact cone structure as in Theorem \ref{thm-quasicont}. Let $\xi$ denote the infinitesimal symmetry on $\tilde{\cC}$, $\alpha$ the quasi-contact form such that $\alpha(\xi)=1$, and $\hat{\xi}$ the lift of the symmetry  to $\hat{\mathcal{G}}=\pi^*\mathcal{G}$. Denote by $\pi\colon\tilde{\cC}\to\cC$, $\hat{\pi}\colon\hat{\mathcal{G}}\to\mathcal{G}$ and $\hat{\tau}:\hat{\mathcal{G}}\to\tilde{\cC}$ the projections and by $\iota:\hat{\mathcal{G}}\to\tilde{\mathcal{G}}$ the natural inclusion. Then one has 
\begin{equation}
  \label{eq:Relation-Cartan-Connections}
  \mathrm{pr}_{\mfk}\circ\iota^*\tilde{\psi}=\hat{\pi}^*\psi+\mathrm{pr}_{\mfk}\circ\iota^*\tilde{\psi}(\hat{\xi})\,\hat{\tau}^*\alpha,
  \end{equation}
  where $\mathrm{pr}_{\mfk}$ denotes the orthogonal projection onto $\mfk\cong\p_{-1}\oplus\p_{0}\subset\g$.
\end{proposition}

\section{Curvature analysis of PCQ structures and their quasi-contactification} 
\label{sec:curvatuer-analysis}

In this section we analyze the curvature of parabolic conformally quasi-symplectic  structures. We express the Cartan connection for the quasi-contact cone structures that arise from quasi-contactifying PCQ structures and relate their fundamental invariants. Using the constructed Cartan connection, we identify several distinguished  curvature conditions which can be interpreted as  certain integrability conditions. We give explicit examples and provide parametric expressions for invariants in some local coordinate system. 

\subsection{(2,3,5)-distributions from variational  scalar 4th order ODEs}
\label{sec:2-3-5}
In this section we analyze  the one-to-one correspondence between (2,3,5)-distributions with an infinitesimal symmetry and variational 4th order scalar ODEs. 
\subsubsection{Scalar fourth order ODEs}
\label{sec:fourth-order-scalar}

It was shown in  Theorem \ref{thm-CartanConnection} that a class of PACQ structures is equivalent to regular and normal  Cartan geometries   $(\tau\colon\cG\to\cC,\psi)$ of type $(\mfk,L),$ where $\mfk=\fp^{op}\slash\fp_{-2}\cong\fgl(2,\RR)\ltimes\mathrm{Sym}^3\RR^2,$ in which $\fp^{op}\subset\fg_2^*$ is the opposite contact parabolic subalgebra with the grading $\fp^{op}=\fp_{-2}\oplus\fp_{-1}\oplus\fp_0$ and $L=B\subset\mathrm{GL}(2,\RR)$ is the Borel subgroup (see \eqref{Liealgs} and \eqref{k_g2}).  The Cartan connection for such geometries  can be expressed as
\begin{equation}
   \label{eq:4-ODE-Conn2}
    \def\arraystretch{1.1}
   \psi=\textstyle{
\begin{pmatrix}
  \textstyle{}\phi_0+\mu& \xi_0& 0& 0& 0 \\
\omega^0 & \textstyle{}\phi_1-\phi_0+\mu& 0 & 0&  0\\
\omega^1& \theta^1 & \textstyle{2}\phi_0-\phi_1+\mu & \xi_0 & 0  \\
\omega^2 & -2\omega^1 & 2\omega^0 & \mu & -\xi_0 \\
\omega^3 & \omega^2 & 0 & -2\omega^0 & \phi_1-2\phi_0+\mu\\
\end{pmatrix}}
 \end{equation}
 where $\mu=-\textstyle{\frac 15}\phi_1$. The tangent bundle of $\cC$ has a 4-step filtration
 \begin{equation}   \label{eq:filtration-C-4th-order-ODE}
       T^{-1}\cC\subset\cdots\subset T^{-4}\cC=T\cC
     \end{equation}
     where $T^{-1}\cC=\cV\oplus \cE$     and 
 \begin{equation}
   \begin{gathered}
     \label{eq:filtraction-4th-order-ODE-each-step}
\cE=\tau_*\left(\ker\{\omega^3,\omega^2,\omega^1,\theta^1\}\right)=\tau_*\langle\tfrac{\partial}{\partial\omega^0}\rangle,\quad \cV=\tau_*\left(\ker\{\omega^3,\omega^2,\omega^1,\omega^0\}\right)=\tau_*\langle\tfrac{\partial}{\partial\theta^1}\rangle, \\
     T^{-2}\cC=[T^{-2}\cC,\cV]=[\cE,\cV]=\tau_*\left(\ker\{\omega^3,\omega^2\}\right),\\ T^{-3}\cC=[\cE,T^{-2}\cC]=[T^{-2}\cC,T^{-2}\cC]=\tau_*\left(\ker\{\omega^3\}\right),\quad T\cC=T^{-4}\cC=[\cE,T^{-3}\cC].
      \end{gathered}
 \end{equation}
Assuming regularity and normality, the structure equations for such geometries are as in \eqref{eq:curv-matrix-4-ODE-Conn2} and \eqref{eq:curv-2-forms-4-ODE}. Taking a section $s\colon\cC\to\cG,$ the fundamental invariants, represented as sections of line bundles $(\cV)^{k_1}\otimes(\cE)^{k_2},$  are given as  
\begin{equation}
  \label{eq:fund-inv-4-ODE}
  \begin{gathered}
    \ts \bc_0=s^*\scc_0(\frac{\partial}{\partial s^*\omega^0})^{-2}\otimes(\frac{\partial}{\partial s^*\theta^1})^{-2}, \,\qquad \bc_1=s^*\scc_1(\frac{\partial}{\partial s^*\omega^0})^{-1}\otimes(\frac{\partial}{\partial s^*\theta^1})^{-2},\\
 \ts\bw_0=  s^*\sw_0(\frac{\partial}{\partial s^*\omega^0})^{-4},\quad \bw_1= s^*\sw_1(\frac{\partial}{\partial s^*\omega^0})^{-3},
  \end{gathered}
\end{equation}
where $\frac{\partial}{\partial s^*\omega^0}$ and $\frac{\partial}{\partial s^*\theta^1}$ are vector fields dual to  $s^*\omega^0$ and $s^*\theta^1$ in the coframe $(s^*\omega^0,\cdots,s^*\omega^3,s^*\theta^1)$ for  $\cC.$
As was mentioned in \ref{sec-4thorder}, The splitting of $T^{-1}\cC$ together with the bracket relations
 \eqref{eq:filtraction-4th-order-ODE-each-step} imply that such geometries are locally realizable as   scalar 4th order ODEs under contact equivalence (see \cite{Yamaguchi-Geometrization,Krynski-ODE} for more detail). In \ref{sec:local-form-invar} we give a parametric expression for the fundamental invariants starting from a 4th order ODE.

Lastly, we describe the  canonical conformally almost quasi-symplectic structure on $\cC$. By inspection one can show that there is a unique  2-form $\rho\in \Omega^2(\cG)$ of maximal rank that is semi-basic with respect to the fibration $\tau\colon\cG\to\cC,$   given by
\begin{equation}
  \label{eq:rho-4th-order-ODE-almost-q-symp}
  \rho=\theta^1\w\omega^3-3\omega^1\w\omega^2.
\end{equation}
Consequently, taking any section  $s\colon\cC\to\cG,$  the  conformal class of $\rho$ is well-defined on $\cC$ and defines the canonical almost conformally quasi-symplectic structure $\ell\subset\biw^2(T^*\cC).$ 
Note that the line bundle $\cE,$ defined in  \eqref{eq:filtraction-4th-order-ODE-each-step}, is the characteristic direction of $\ell$ and $T^{-1}\cC=\cE\oplus\cV$ is isotropic with respect to $\ell.$

\subsubsection{Variationality }
\label{sec:inverse-probl-vari}
By Proposition \ref{prop-variational}, we know that a 4th order ODE is variational if and only if $\cC$ is a conformally quasi-symplectic structure, i.e.   $\ell$ 
has a closed representative. Using structure equations \eqref{eq:curv-matrix-4-ODE-Conn2} and \eqref{eq:curv-2-forms-4-ODE},  it follows that
\[\exd\rho=\phi_1\w\rho+ 2\sw_1\omega^0\w\omega^2\w\omega^3+\textstyle{\frac{5}{7}\sw_{1;\uo}}\omega^1\w\omega^2\w\omega^3.\]
Since one has $\ell=[s^*\rho],$ for some section $s\colon\cC\to\cG,$ the condition
\[\sw_1=0\]
implies $\exd\rho=\phi_1\w\rho.$ As a result, $\sw_1=0$ is a necessary condition for $[\rho]$ to have a closed representative. Furthermore, assuming $\sw_1=0,$ if $\rho_0\in[\rho]$ is a closed representative then one has $\exd\rho_0=0,$ which implies that  $\phi_1$ is a closed form. Using the structure equations and assuming $\sw_1=0,$ it follows that   $\exd\phi_1=0$ is equivalent to  
\[ \scc_1=0.\]
Thus, we have  shown the following.
\begin{proposition}[\cite{Fels-ODE}]\label{prop:variationality-Fels-4th-ODE} 
  A scalar 4th order ODE is variational if and only if the vanishing conditions $\bc_1=0,\bw_1=0$ are satisfied.
\end{proposition}

\begin{remark}\label{rmk:235-variationality-subFinsler}
  Geometrically one can interpret variational 4th order ODEs as the Euler-Lagrange equations of a  variational problem, in the sense of Griffiths \cite{Hsu-VarCal,Griffiths-EDS}, for some  geometric structure.  
  In \cite{Ivey-ODE} Griffiths' formalism is used to relate sub-Finsler structures on contact 3-folds to variational 4th order ODEs. It turns out that with some more work one can establish a one-to-one correspondence between variational 4th order ODEs and the geometry of sub-Finsler structures on contact 3-folds under \emph{divergence equivalence}. See  \ref{sec:div-equiv-pseudo-Finsler-as-var-orthopath} for a discussion on the notion of divergence equivalence. However, in this article we will not elaborate on this aspect of variational 4th order ODEs. 
\end{remark}

\subsubsection{(2,3,5)-distributions}
\label{sec:2-3-5-1}

Given a regular and  normal Cartan geometry $(\tpi\colon\tcG\to\tcC,\tpsi)$ of type $(\fg^*_2,P_{12}),$ the connection form $\tpsi$ can be written as
\begin{equation}
   \label{eq:G2-CarNur-Conn2}
    \def\arraystretch{1.1}
   \tpsi=\textstyle{
\begin{pmatrix}
  \textstyle{}\tphi_0& \txi_0& \txi_1& 2\txi_2& \txi_3& \txi_4& 0\\
\tomega^0 & \textstyle{}\tphi_1-\tphi_0& \tgamma_1 & -\txi_1&  \txi_2 & 0 & -\txi_4\\
\tomega^1& \ttheta^1 & \textstyle{2}\tphi_0-\tphi_1 & \txi_0 & 0  &-\txi_2 & -\txi_3\\
\tomega^2 & -2\tomega^1 & 2\tomega^0 & 0 & -\txi_0 & \txi_1 & -2\txi_2\\
\tomega^3 & \tomega^2 & 0 & -2\tomega^0 & \tphi_1-2\tphi_0 & -\tgamma_1 & -\txi_1\\
\tomega^4 & 0 & -\tomega^2 & 2\tomega^1 & -\ttheta^1 & \tphi_0-\tphi_1 & -\txi_0\\
0 & -\tomega^4 & -\tomega^3 & -\tomega^2& -\tomega^1 & -\tomega^0  &-\tphi_0
\end{pmatrix}}
 \end{equation}
which is $\mathfrak g^*_2$-valued. Here  $\fg^*_2$ is defined  using the following inner product  and 3-form
\[\langle v,v\rangle=2(v_1v_7+v_2v_6+v_3v_5)+v_4^2,\qquad \Phi(v,v,v)=2v_{237}+v_{345}+v_{246}-v_{147}+v_{156},\]
where $v_{ijk}:=v_i\w v_j\w v_k.$
Using the connection form $\tpsi,$ the corresponding 5-step filtration 
 \begin{equation}   \label{eq:filtration-tC-G2-P12}
       T^{-1}\tcC\subset\cdots\subset T^{-5}\tcC=T\tcC,
     \end{equation}
     where $T^{-1}\tcC=\tcV\oplus \tcE,$     is expressed as
\begin{equation}
   \label{eq:filtration-G2P12-5grading}
   \begin{gathered}  
\tcE=\ttau_*\left(\Ker\{\tomega^4,\tomega^3,\tomega^2,\tomega^1,\ttheta^1\}\right)=\ttau_*\langle\tfrac{\partial}{\partial\tomega^0}\rangle,\quad \tcV=\ttau_*\left(\Ker\{\tomega^4,\tomega^3,\tomega^2,\tomega^1,\tomega^0\}\right)=\ttau_*\langle\tfrac{\partial}{\partial\ttheta^1}\rangle, \\
T^{-2}\tcC=[\tcE,\tcV]=\ttau_*\left(\Ker\{\tomega^4, \tomega^3,\tomega^2\}\right),\quad      T^{-3}\tcC=[\tcE,T^{-2}\tcC]=\ttau_*\left(\Ker\{\tomega^4, \tomega^3\}\right),\\
T^{-4}\tcC=[\tcE,T^{-3}\tcC]=\ttau_*\left(\Ker\{\tomega^4\}\right),\quad T^{-5}\tcC=[\tcV,T^{-4}\tcC]=[T^{-2}\tcC,T^{-3}\tcC]=T\tcC.
      \end{gathered}
    \end{equation}
To describe the harmonic curvature of such geometries define
\[
\begin{aligned}
\tTheta^1&=\exd\tilde\theta^1+\tilde\xi_3\w\tilde\omega^4-3\tilde\xi_0\w\tilde\omega^1+(3\tilde\phi_0-2\tilde\phi_1)\w\tilde\theta^1.
\end{aligned}
\]
It follows that
\begin{equation}
  \label{eq:a0-str-eqn-theta1}
  \begin{aligned}
  \tTheta^1&\equiv a_0\tilde\omega^0\w\tilde\omega^3\quad\mod\quad\{\tilde\theta^1,\tilde\omega^1,\tilde\omega^2,\tilde\omega^4\}
\end{aligned}
\end{equation}
and, taking a section $s\colon\tcC\to\tcG,$  
\begin{equation}
  \label{eq:a0-harmonic-inv-G2P12}
  \ba_0:= s^*a_0(\ts\frac{\partial}{\partial s^*\tomega^0})^{-4}\in \Gamma( \tcE)^{-4}
\end{equation}
  represents the harmonic curvature of a $(\fg^*_2,P_{12})$ geometry. 
By Proposition \ref{prop-qcontactcone}, such geometries always arise as the correspondence space of regular and normal Cartan geometries associated to (2,3,5)-distributions.  The corresponding (2,3,5)-distribution is defined on  the local leaf space
\begin{equation}
  \label{eq:tM-235-mfld}
  \tnu\colon\tcC\to \tM=\tcC\slash\cI_{\tcV},
  \end{equation}
  where $\cI_{\tcV}$ denotes the foliation of $\tcC$ by the integral curves of  $\tcV.$ Since $[\tcV,T^{-2}\tcC]\subset T^{-2}\tcC,$ the rank 2 distribution on $\tM$ is given by  $\tscD=\tnu_*(T^{-2}\tcC)\subset T\tM.$ Conversely, starting from a (2,3,5)-distribution, one has $\cC\cong\PP\tscD,$ where $\PP\tscD$ is the $\PP^1$-bundle over $\tM$ whose fiber at $p\in\tM$ is the projective line $\PP\tscD_p$.

  The harmonic invariant of a (2,3,5)-distribution, referred to as the \emph{Cartan quartic,} is 
\begin{equation}
  \label{eq:Cartan-quartic-235}
  \bC=a_0(\tomega^0)^4+4a_1\tomega^1(\tomega^0)^3+6a_2(\tomega^1)^2(\tomega^0)^2+4a_3(\tomega^1)^3\tomega^0+a_4(\tomega^1)^4\in \Gamma(\text{Sym}^4(\tscD^*)),
\end{equation}
where, we have suppressed denoting the pull-back by a section $s\to\tM\to\tcG,$ and
\begin{equation}
  \label{eq:Cartan-quartic-coeffs}
  a_1=\textstyle{\frac 14\frac{\partial}{\partial\ttheta^1} a_{0},\quad a_2=\frac{1}{12}\frac{\partial^2}{(\partial\ttheta^1)^2}a_{0},\quad a_3=\frac{1}{24}\frac{\partial^2}{(\partial\ttheta^1)^3}a_{0},\quad a_4=\frac{1}{24}\frac{\partial^2}{(\partial\ttheta^1)^4}a_{0}}.
  \end{equation}
Equivalently, the formula above  infinitesimally expresses the fact that the coefficients of the Cartan quartic can be obtained via the action of the subdiagonal element of  $\mathrm{GL}(2,\RR)\subset P_1$ on $a_0$ where $P_1\subset \mathrm{G}_2^*$ is the first parabolic subgroup.

Using the embedding $\fg_2^*\subset\mathfrak{so}(3,4),$ it is known  \cite{Nurowski-G2} that the Cartan connection \eqref{eq:G2-CarNur-Conn2} defines a normal conformal connection  for the conformal structure $[s^*\tg]\subset\mathrm{Sym}^2 T^*\tM$ for some section $s\colon \tM\to\tcG$ where 
\begin{equation}
  \label{eq:235-Nur-conf-str}
\tg=2\tilde\omega^4\circ\tilde\omega^0+2\tilde\omega^3\circ\tilde\omega^1+\tilde\omega^2\circ\tilde\omega^2\in\Gamma(\mathrm{Sym}^2T^*\tcG).
\end{equation}
The conformal holonomy of such conformal structures is reduced to $\mathrm{G}^*_2$.

\subsubsection{Quasi-contactification and  various curvature conditions}
\label{sec:235-4th-ODE-quasi-cont}
Now we relate the Cartan connections of a (2,3,5)-distribution with an infinitesimal symmetry and a variational 4th order ODE. Subsequently, we relate their Cartan curvatures and study a variety a geometric properties arising from    additional curvature conditions.

\begin{proof}[Proof of Porposition \ref{prop-relharmonics}; first part]
  Let $\rho_0\in\Gamma(\ell)$ be a quasi-symplectic representative. Then, as was mentioned in \ref{thm-quasicont}, locally, there is a primitive 1-form $\beta_0\in \Omega^1(\cC)$ for $\rho_0,$ i.e. $\exd\beta_0=\rho_0.$ Define $\omega^4=\exd t+\pi^*\beta_0\in \Omega^1(\tcC),$ where $\pi\colon\tcC:=\RR\times\cC\to\cC$ is the  projection map and $t$ is a coordinate on $\RR.$ Since the scaling action of the structure group on $\rho_0$ induces a scaling action on $\beta_0,$ the scaling action can be extended to $\omega^4$ so that its lift, $\tomega^4\in \Omega^1(\pi^*\cG),$ transforms equivariantly along the fibers of the  pull-back bundle $\pi^*\cG\to\tcC.$ It is straightforward to check that the coframing $\hat\psi:=(\tomega^4,\pi^*\psi)$ defines the Cartan connection for a Cartan geometry  $(\hat\tau_0\colon\pi^*\cG\to\tcC,\hat\psi)$ of type $(\fp^{op},B),$ where $\fp^{op}\subset\fg_2^*$ is the opposite contact parabolic subalgebra and $B\subset\mathrm{GL}(2,\RR)$ is the Borel subgroup. Furthermore, the vector field $\frac{\partial}{\partial\omega^4}=\tfrac{\partial}{\partial t}$ is an infinitesimal symmetry for this Cartan geometry.

  Using the filtration \eqref{eq:filtration-C-4th-order-ODE} and the construction of the Cartan connection $\hat\psi,$ one obtains that $T\tcC$ has a 5-step filtration as in \eqref{eq:filtration-tC-G2-P12}. Hence one can associate to such filtraction  a regular and normal Cartan geometry $(\ttau\colon\tcG\to\tcC,\tpsi)$ of type $(\fg^*_2,P_{12}).$ To relate  $\tpsi$ and $\hat\psi,$ one first finds an equivariant extension of $\hat\psi$ to $\tcG\to\tcC.$ Subsequently,  using  structure equations \eqref{eq:curv-2-forms-4-ODE-variational}, one imposes the appropriate normality conditions on   the extended $\hat\psi$ in order  to obtain  $\tpsi.$ It is a matter of computation, using the matrix forms \eqref{eq:G2-CarNur-Conn2} and \eqref{eq:4-ODE-Conn2}, to show that
 $\pr_\fk\circ\iota^*\tpsi$ and $\hat\pi^*\psi$ are related as follows 
  \begin{equation}
  \label{eq:Q-contactification-235}
  \begin{gathered}
  \tilde\omega^a=\omega^a,\quad \tilde\theta^1=\theta^1,\quad \tilde\omega^0=\omega^0+\textstyle{\frac{1}{2} \scc_{0}}\tilde\omega^4,\quad 
 \tilde\phi_1 = \phi_1,\\
  \tilde\phi_0 = \phi_0-\textstyle{ \frac{1}{4}} \scc_{0;0}\tilde\omega^4,\quad \tilde\xi_0 = \xi_0+ \textstyle{(\frac{1}{24}\sw_{0;\uo\uo}-\frac{1}{4}\scc_{0;00})}\tilde\omega^4,
\end{gathered}
\end{equation}
 wherein  by abuse of notation the pull-back $\pi^*$ is dropped.
The entries of $\iota^*\tpsi$  that are not included in $\fk$ can be expressed similarly but we will not provide them here since the expressions are rather long and not essential for what follows.
\end{proof}
To state the main corollary of this section we recall that the Lie algebra $\fk=\mathrm{Sym}^3(\RR^2)\rtimes\fgl_2(\RR)$  has the grading
 $\fk=\fk_{-4}\oplus\cdots\oplus\fk_{1}$ where $\fk_{-1}=\fv\oplus\fe.$
Define the Lie algebras $\fx,\fy\subset\fk$ as 
\begin{equation}
  \label{eq:LieAlg-fx-fy-4th-ODE}  
  \fx=\fv\oplus\fk_{0}\oplus\fk_{1},\quad \fy=\fk_{-2}\oplus\fv\oplus\fk_0\oplus\fk_1.
\end{equation}
The Lie groups $X$ and $Y$ with Lie algebras $\fx$ and $\fy$ are
\begin{equation}
  \label{eq:235-corr-X-Y}
  \begin{aligned}
    X&=B\ltimes \mathrm{Sym}^3(\RR^2)^{-1},\quad
    Y&=B\ltimes \mathrm{Sym}^3(\RR^2)^{-2},
  \end{aligned}
\end{equation}
where $B\subset\mathrm{GL}(2,\mathbb{R})$ is the Borel subgroup stabilizing the filtration $\mathrm{Sym}^3(\RR^2)^{-1}\subset \mathrm{Sym}^3(\RR^2)^{-2}\subset \mathrm{Sym}^3(\RR^2)^{-3}\subset \mathrm{Sym}^3(\RR^2)^{-4}=\mathrm{Sym}^3(\RR^2)$.
\begin{corollary}\label{cor:235-from-4th-ODE-various-curv-conditions}
Let $(\ttau\colon\tcG\to\tcC,\tpsi)$ be a regular and normal Cartan geometry  of type  $(\fg^*_2,P_{12})$ with an infinitesimal symmetry and $(\tau\colon\cG\to\cC,\psi)$ be the  Cartan geometry for its corresponding  variational scalar 4th order ODE  induced on the leaf space of the integral curves of the infinitesimal symmetry via the quotient map $\pi\colon\tcC\to\cC.$ Then one has
\begin{equation}
  \label{eq:a_0-w_0-relation}
  \ba_0=\pi^*(\bw_0),
  \end{equation}
  where $\ba_0$ is given in \eqref{eq:a0-harmonic-inv-G2P12} and $\bw_0$ is the second generalized Wilczy\'nski invariant of the ODE as in \eqref{eq:fund-inv-4-ODE}. Furthermore,  for the naturally induced (2,3,5)-distribution on the  leaf space $\tM,$ as defined in \eqref{eq:tM-235-mfld}, the following holds.
  \begin{enumerate}
  \item  The infinitesimal symmetry of the (2,3,5)-distribution is null with respect to its canonical conformal structure if and only if $\scc_0=0$ for the corresponding variational 4th order ODE. In this case the Cartan quartic \eqref{eq:Cartan-quartic-235} of the (2,3,5)-distribution  has a repeated root of multiplicity at least 2 and $\tM$ is foliated by 3-dimensional submanifolds with null conormal bundle.
  \item A variational 4th order  ODE corresponds to a Cartan geometry $(\tau_1\colon\cG\to M,\psi)$ of type $(\fk,X),$ where $M\cong J^2(\RR,\RR),$ if and only if  relation \eqref{eq:Relation-Cartan-Connections} becomes
    \begin{equation}
      \label{eq:tpsi-psi-0-235-J2}      
      \pr_\fk\circ\iota^*\tpsi-\hat\pi^*\psi=0,
          \end{equation}
i.e.  the invariant conditions $\scc_0=\sw_{0;\uo\uo}=0$ hold. In this case the Cartan quartic of the corresponding (2,3,5)-distribution has a repeated root of multiplicity at least 3.
\item  A variational 4th order  ODE defines a Cartan geometry $(\tau_2\colon\cG\to M_1,\psi)$ of type $(\fk,Y),$ where $M_1\cong J^1(\RR,\RR),$ if and only if the invariant conditions $\scc_0=\sw_{0;\uo}=0$ hold. In this case   the Cartan holonomy of the corresponding (2,3,5)-distribution is  a proper subgroup of the  contact parabolic subgroup $P\subset\rG^*_2$   and the Cartan quartic  has a repeated root of multiplicity at least 4.
\item Variational 4th order ODEs whose quasi-contactification is the flat (2,3,5)-distribution are in one-to-one correspondence with the 1-parameter family of 4-dimensional torsion-free  $\mathrm{GL}(2,\RR)$-structures  with symmetric Ricci curvature. 
\end{enumerate}
\end{corollary}
\begin{proof}
  The first claim  that  the fundamental invariants are related via the projection map $\pi$ can be obtained by replacing expressions \eqref{eq:Q-contactification-235} in $\tpsi$ given by \eqref{eq:G2-CarNur-Conn2} and computing $\ba_0$ defined in \eqref{eq:a0-harmonic-inv-G2P12} and \eqref{eq:a0-str-eqn-theta1}.
  
Using the connection forms \eqref{eq:Q-contactification-235} for $\iota^*\tpsi$ in Proposition \ref{prop-relharmonics}  and relation \eqref{eq:Cartan-quartic-coeffs}, it follows that  the coefficients of the Cartan quartic \eqref{eq:Cartan-quartic-235} are
\begin{equation}
  \label{eq:coeffs-Cartan-Quartic}
  a_0=\sw_0,\quad a_1=\textstyle{\frac 14 \sw_{0;\uo},\quad a_2=\frac{1}{12}\sw_{0;\uo\uo},\quad a_3=\frac{1}{24}\sw_{0;\uo\uo\uo},\quad a_4=\frac{1}{24}\sw_{0;\uo\uo\uo\uo}}.
  \end{equation}
Moreover, for variational 4th order ODEs one has
\[\sw_{0;\uo\uo\uo\uo\uo}=0.\]
To show \emph{(1)},  firstly note that   Proposition \ref{prop-relharmonics} and expressions \eqref{eq:Q-contactification-235}  imply that the canonical bilinear form $\iota^* \tg$ in \eqref{eq:235-Nur-conf-str} is given by
\begin{equation}
  \label{eq:h_0-235-quasi-cont}
\iota^*\tg= 2\omega^4\circ\omega^0+2\omega^3\circ\omega^1+\omega^2\circ\omega^2+\scc_0\omega^4\circ\omega^4\in\Gamma(\mathrm{Sym}^2T^*\pi^*\cG).
\end{equation}
As discussed in the proof of Proposition \ref{prop-relharmonics}, the vector field $\tfrac{\partial}{\partial\omega^4}$ is the infinitesimal symmetry of a $(\fg^*_2,P_{12})$ Cartan geometry arising from quasi-contactifying a variational 4th order ODE. As a result, from \eqref{eq:h_0-235-quasi-cont} it follows that $\scc_0=0$ is equivalent to the nullity of  $\tfrac{\partial}{\partial\omega^4}$ with respect to $\tg.$

From structure equations \eqref{eq:curv-2-forms-4-ODE-variational} and  expressions \eqref{eq:Q-contactification-235} it is clear that    $\scc_0=0$ is equivalent to the integrability of the rank 3 Pfaffian system $\iota^*I,$ where
\begin{equation}
  \label{eq:I-4th-ODE-3integrability-c0}
  I=\{\tomega^3,\tomega^0\}.
  \end{equation}
Thus,   $\tM$ is foliated by 3-dimensional submanifold whose conormal bundle, given by $\langle s^*\tomega^3,s^*\tomega^0\rangle$ for a section $s\colon\tM\to\tcG,$ is   null with respect to the  conformal structure induced by $\tg$ on $T^*\tM.$

Lastly, it is a straightforward computation to show that for variational 4th order ODEs  $\scc_0=0$ implies $ \sw_{0;\uo\uo\uo}=0.$
Thus, using \eqref{eq:coeffs-Cartan-Quartic}, the Cartan quartic \eqref{eq:Cartan-quartic-235} has a repeated root of multiplicity at least  2.

Part \emph{(2)} follows from requiring $\pr_{\fk}\circ\iota^*\tpsi=\hat\pi^*\psi$ using  \eqref{eq:Q-contactification-235} and the fact that  $\scc_0=\sw_{0;\uo\uo}=0$ are necessary and sufficient conditions for the Cartan curvature  \eqref{eq:curv-matrix-4-ODE-Conn2} to vanish upon insertion of  $\tfrac{\partial}{\partial\theta^0}.$ 
 By \eqref{eq:coeffs-Cartan-Quartic}, in this case the Cartan quartic has a root of multiplicity at least 3. Lastly, checking the identity $\exd^2=0$ for structure equations, it follows that in this case the only non-zero entries  in the Cartan curvature  \eqref{eq:curv-matrix-4-ODE-Conn2}  are given by
 \begin{equation}
   \label{eq:Cartan-Curv-J2-4-ODE}
      \begin{gathered}
     \Theta^1=\sw_0\omega^0\w\omega^3+\sx_{12}\omega^1\w\omega^3+\sx_{19}\omega^2\w\omega^3,\quad \Phi_0=\sx_{12}\omega^0\w\omega^3,\\ \Xi_1=\sx_{12}\omega^0\w\omega^2+\sx_{19}\omega^0\w\omega^3+\sx_{22}\omega^2\w\omega^3.
      \end{gathered}
 \end{equation}
As a result, the curvature 2-form is horizontal with respect to the fibration $\tau_1\colon\cG\to M$ defined by the quotient map $\nu_{1}\colon \cC\to \cC\slash \cI_{\cV}$ where $\cI_\cV$ is the foliation of $\cC$ by the integral curves of $\cV=\tau_*\langle\tfrac{\partial}{\partial\theta^1}\rangle$ as given in \eqref{eq:filtraction-4th-order-ODE-each-step}. Defining $\tau_1=\nu\circ\tau,$ it follows that $(\tau_1\colon\cG\to M,\psi)$ is a Cartan geometry of type $(\fk,X),$ where $X$ is defined in \eqref{eq:235-corr-X-Y}. Lastly, because $[\cV,T^{-k}\cC]\subset T^{-k}\cC$ for $k\geq2,$ the 4-step filtration \eqref{eq:filtration-C-4th-order-ODE} of $T\cC$ descends to a 3-step filtration $T^{-1}M\subset T^{-2}M\subset T^{-3}M=TM$  where $T^{-k}M:=\tau_{1*}T^{-k+1}\cC$ for $1\geq k\geq 3,$ with the property that $\cV_1=\tau_{1*}\langle\tfrac{\partial}{\partial\omega^1}\rangle$ is a well-defined  direction in $T^{-1}M$ and 
 \begin{equation}
   \begin{gathered}
     \label{eq:3step-filtaction-of-J2}
     T^{-2}M=[T^{-1}M,\cV_1]=[\cV_1,T^{-2}M],\quad      T^{-3}M=[T^{-2}M,T^{-2}M].
      \end{gathered}
    \end{equation}
It is a classical fact that any manifold with  3-step filtration \eqref{eq:3step-filtaction-of-J2} is locally isomorphic to $J^2(\RR,\RR)$.

Part \emph{(3)} is shown by first noting that by part \emph{(2)} the conditions  $\scc_0=\sw_{0;\uo\uo}=0$ are necessary to descent the geometry to $M$ as a result of which the non-zero curvature 2-forms are  as in \eqref{eq:Cartan-Curv-J2-4-ODE}. Furthermore, it can be checked that $\sx_{12}=\tfrac{1}{4}\sw_{0;\uo}.$ Thus, $\sw_{0;\uo}=0$ is necessary to obtain a Cartan geometry of type $(\fk,Y)$ on the local leaf space $\nu_{2}\colon M\to M_1=M\slash \cI_{\cV_1}$ where $\cI_{\cV_1}$ is the foliation of $M$ by the integral curves of $\cV_1=\tau_{1*}\langle\tfrac{\partial}{\partial\omega^1}\rangle$ and $Y$ is given in \eqref{eq:235-corr-X-Y}. From part \emph{(2)} it follows that the only obstruction for $\iota^*\tpsi$ to be  $\fp^{op}$-valued is the vanishing of $\txi_3$ given by
\begin{equation}
  \label{eq:txi3-235-4th-ODE-quasi-cont}
  \txi_3=-\tfrac{1}{4}\sw_{0;\uo}\omega^0-\tfrac{1}{12}\sw_{0;\uo 2}\omega^3.
\end{equation}
Thus, the conditions  $\scc_0=\sw_{0;\uo}=0$ imply that $\iota^*\tpsi$ is $\fp^{op}$-valued and, therefore, the Cartan holonomy is reduced to the contact parabolic subgroup $P\subset\mathrm{G}_2^*.$     Checking the identity $\exd^2=0$ for the structure equations after setting $\sw_{0;\uo}=0$, it follows that the only non-zero entries  in the Cartan curvature  \eqref{eq:curv-matrix-4-ODE-Conn2}  are given by
\begin{equation}
  \label{eq:4-ODE-holonomy-reductio-str-eqns}
     \begin{gathered}
     \Theta^1=\sw_0\omega^0\w\omega^3+\sx_{19}\omega^2\w\omega^3,\quad \Xi_1=\sx_{19}\omega^0\w\omega^3.
      \end{gathered}
\end{equation}
    As a result, the curvature 2-form is semi-basic  with respect to the fibration $\tau_2\colon\cG\to M_1.$ It follows that $(\tau_2\colon\cG\to M_1,\psi)$ is a Cartan geometry of type $(\fk,Y).$ Furthermore,  $TM_1$ has a contact distribution given by $\Ker\langle\omega^3\rangle$ and therefore is locally isomorphic to $J^1(\RR,\RR)$.    By \eqref{eq:coeffs-Cartan-Quartic}, in this case the Cartan quartic has a root of multiplicity at least 4.

Part \emph{(4)} follows from \eqref{eq:a_0-w_0-relation} since $\ba_0=0$ implies flatness of the (2,3,5)-distributions and, therefore, flat (2,3,5)-distributions are in one-to-one correspondence with 4th order ODEs satisfying $\scc_1=\sw_0=\sw_1=0$. Furthermore, 4th order ODEs satisfying  $\sw_1=\sw_0=0$ define integrable $\mathrm{GL}(2,\RR)$-structure \cite{Bryant-exotic,CDT-ODE} on the solution space of the ODE, denoted by $S,$ which is defined as the local leaf space $\lambda\colon\cC\to\cC/I_\cE$ where $I_\cE$ is the foliation of $\cC$ by the integrable curves of $\cE,$ i.e. the solution curves of the 4th order ODE.  The local generality of such $\mathrm{GL}(2,\RR)$-structures which additionally satisfy $\scc_1=0$ depends on  one parameter \cite{CS-special}. It is a matter of straightforward computation to show that  among 4-dimensional torsion-free $\mathrm{GL}(2,\RR)$-structures with connection $\nabla,$ the Ricci curvature of $\nabla$ is symmetric if and only if $\scc_1=0$ (e.g. see \cite[Example 1]{N-ODE}). Lastly, we point out that this class of $\mathrm{GL}(2,\RR)$-structures are among parabolic conformally symplectic structures as defined in \cite{CS-cont1}.
\end{proof}
\begin{remark}
 The integrability of the Pfaffian system \eqref{eq:I-4th-ODE-3integrability-c0} in part (1) of Corollary \ref{cor:235-from-4th-ODE-various-curv-conditions}    defines a  point equivalence class in the contact equivalence class of the variational 4th order ODE. Furthermore, in regard  to Remark \ref{rmk:235-variationality-subFinsler}, the geometries induced  on $J^2(\RR,\RR)$ and $J^1(\RR,\RR)$ in parts (2) and (3) of Corollary \ref{cor:235-from-4th-ODE-various-curv-conditions}  seem to be an interesting topic of study in relation to  sub-projective structures on Engel distributions and sub-Finsler geometry on contact 3-folds. Lastly, the 1-parameter family of  $\mathrm{GL}(2,\RR)$-structures in part (4) of Corollary \ref{cor:235-from-4th-ODE-various-curv-conditions} corresponds to the 2-parameter family of $H_3$-structures obtained in \cite{Bryant-exotic}. The difference in the number of parameters is due to the additional homothety acting via the structure group $\mathrm{GL}(2,\RR)$ which is absent in $H_3$-structures as their structure group is $\mathrm{SL}(2,\RR)$.
\end{remark}

\subsubsection{Parametric expressions, examples and local generality}
\label{sec:local-form-invar}
Unlike the notation used in \ref{sec:spec-conf-quasi} and Theorem \ref{prop-variational} for jet coordinates, to simplify the parametric expressions in this section we use $(x,y,p,q,r,u)$ to  denote  coordinates on $J^4(\RR,\RR)$ with contact system
\[\rho^3=\exd y-p\exd x,\quad \rho^2=\exd p-q\exd x,\quad \rho^1=\exd q-r\exd x,\quad \rho^0=\exd x,\quad \sigma^1=\exd r-u\exd x.\]
Let $\cC\subset J^4(\RR,\RR)$ be the hypersurface defined by a 4th order ODE
\begin{equation}
  \label{eq:4th-order-ode}
  y''''=f(x,y,y',y'',y''').
\end{equation}
The pull-back of the contact system on $J^4(\RR,\RR)$ to $\cC$ is given by
\begin{equation}
  \label{eq:contact-system-on-C-parametric}
  \rho^3=\exd y-p\exd x,\quad \rho^2=\exd p-q\exd x,\quad \rho^1=\exd q-r\exd x,\quad \rho^0=\exd x,\quad \sigma^1=\exd r-f(x,y,p,q,r)\exd x.
  \end{equation}
  As was mentioned in \ref{sec:fourth-order-scalar}, the geometry of contact equivalence classes of scalar 4th order ODEs corresponds to Cartan geometries $(\tau\colon\cG\to\cC,\psi)$ of type $(\fk,L).$ As a result, after imposing appropriate coframe adaptation on the 1-forms \eqref{eq:contact-system-on-C-parametric}, one obtains the following adapted coframe
  \begin{equation}
    \label{eq:4th-order-ODE-adapted-coframe-coordinates}
    \begin{aligned}
    \omega^3&=\rho^3\\
    \omega^2&=\rho^2\\
    \omega^1&=-\tfrac{3}{4}\rho^1+\tfrac{1}{8}f_r\rho^2+(-\tfrac{3}{20}\tfrac{\exd}{\exd x}f_r+\tfrac{9}{40}f_q+\tfrac{11}{160}f_r^2)\rho^3\\
    \omega^0&=\tfrac{1}{3}\rho^0+\tfrac{1}{18}f_{rr}\rho^2+(-\tfrac{1}{15}\tfrac{\exd}{\exd x}f_{rr}+\tfrac{1}{30}f_{qr}-\tfrac{1}{180}f_rf_{rr})\rho^3\\
    \theta^1&=-\tfrac{9}{4}\sigma^1
    +\tfrac{9}{8}f_r\rho^1+(-\tfrac{27}{40}\tfrac{\exd}{\exd x}f_r+\tfrac{63}{40}f_q+\tfrac{57}{160}f_r^2)\rho^2\\
    &+(\tfrac{9}{80}\tfrac{\exd^2}{\exd x^2}f_r-\tfrac{9}{20}\tfrac{\exd}{\exd x}f_q-\tfrac{9}{32}f_r\tfrac{\exd}{\exd x}f_r+\tfrac{23}{320}f_r^3+\tfrac{27}{80}f_rf_q+\tfrac{8}{9}f_p)\rho^3\\      
    \end{aligned}
  \end{equation}
in which  $\frac{\exd}{\exd x}=\frac{\partial}{\partial x}+p\frac{\partial}{\partial y}+q\frac{\partial}{\partial p}+r\frac{\partial}{\partial q}+f\frac{\partial}{\partial r}$ is the total derivative. With respect to the coframe above, the fundamental invariants of a 4th order ODE as given in \eqref{eq:fund-inv-4-ODE} are found to be
\[
  \begin{aligned}
    \scc_1=&f_{rrr}\\
    \scc_0=&\ts\frac{\exd}{\exd x}f_{rrr}+\frac 32f_rf_{rrr}+\frac 29 f_{rr}^2+\frac 43f_{qrr}\\
    \sw_1=&\ts\frac{\exd^2}{\exd x^2}f_r-\frac 32f_r\frac{\exd}{\exd x}f_r-2\frac{\exd}{\exd x}f_q+\frac 14f_r^2+f_rf_q+2f_p\\
    \sw_0=&\ts\frac{\exd^3}{\exd x^3}f_r-4\frac{\exd^2}{\exd x^2}f_q-3f_r\frac{\exd}{\exd x}f_r-\frac{3}{10}\frac{\exd}{\exd x}f_r^2+10\frac{\exd}{\exd x}f_p+5f_r\frac{\exd}{\exd x}f_q+\frac{27}{5}f_q\frac{\exd}{\exd x}f_r+\frac{39}{10}f_r^2\frac{\exd}{\exd x}f_r\\
    &-\ts\frac{13}{5}f_r^2f_q-5f_rf_p+\frac{39}{80}f_r^4-\frac{9}{5}f_q^2-20f_y.
  \end{aligned}
\]
 One can also obtain a parametric expression for quantities such as $\sw_{0;\uo}$ or higher order derivatives of $\sw_{0;\uo}$ used in Corollary \ref{cor:235-from-4th-ODE-various-curv-conditions}, which will not be presented here.

As an example of variational 4th order ODEs, consider the 2nd order Lagrangian $L=L(x,y,p,q)$   such that $L_{qq}\neq 0.$ The Euler-Lagrange equations for such Lagrangians are given by
\[\textstyle{L_y-\frac{\exd}{\exd x}L_p+\frac{\exd^2}{\exd x^2}L_q=0},\]
which gives the 4th order ODE $y''''=f(x,y,y',y'',y''')$ where
   \begin{equation}
     \label{eq:EL-Lag-4th-ODE-general}
   \begin{aligned}     
     f(x,y,p,q,r)=&\textstyle{-\frac{1}{L_{qq}}}\left( L_{qqq}r^2+L_{qpp}q^2+L_{qyy}p^2+2(L_{qqp}q+L_{qqy}p+L_{qqx})r\right.\\
       &\left.+(2L_{qpy}p+2L_{qpx}+L_{qy}-L_{pp})q +(2L_{qyx}-L_{py})p+L_{qxx}-L_{px}+L_{y}\right)
   \end{aligned}
   \end{equation}
   Such 4th order ODEs are by definition variational and therefore satisfy $\scc_1,\sw_1=0.$ Moreover,  one obtains
   \begin{equation}
     \label{eq:cc0-4th-oder-ODE-null}     
     \scc_0=\tfrac{1}{L_{qq}^2}(4 L_{qqq}^2-3L_{qqqq}L_{qq}).
        \end{equation}
   By Corollary \ref{cor:235-from-4th-ODE-various-curv-conditions}, if $\scc_0\neq 0,$ then the quasi-contactification of such ODEs give rise to (2,3,5)-distributions with a non-null infinitesimal symmetry. The equation $\scc_0=0$ can be solved explicitly and has the following two solutions
 \begin{equation}
   \label{eq:4th-ODE-Lagrangian-null-symm}
   \begin{aligned}
     L(x,y,p,q)&=\textstyle{\frac{h_1(x,y,p)}{h_2(x,y,p)+q}+qh_3(x,y,p)+h_4(x,y,p)},\\
         L(x,y,p,q)&=h_1(x,y,p)q^2+qh_3(x,y,p)+h_4(x,y,p)
   \end{aligned}
    \end{equation}
    for arbitrary functions $h_1(x,y,p),\cdots,h_4(x,y,p)$ where $h_{1}(x,y,p)\neq 0.$ As a result, substituting \eqref{eq:4th-ODE-Lagrangian-null-symm} in   \eqref{eq:EL-Lag-4th-ODE-general} gives a closed form for all variational 4th order ODEs whose quasi-contactification gives all (2,3,5)-distributions with a null infinitesimal symmetry.
    \begin{remark} 
Note that the 4th order ODE $y''''=4(y''')^2/(3y'')$ appearing from $\scc_0=0$ in \eqref{eq:cc0-4th-oder-ODE-null} for fixed values of $x,y,p$ is the submaximal 4th order ODE  with 6-dimensional symmetry algebra \cite{KT-ODE}.  This ODE is variational with $\sw_0=0$ and $\scc_0\neq 0.$ In fact,   the expression for  $\scc_0$  can be interpreted as the equi-affine invariant of a curve in a plane  for fixed values of $x,y,p$  (see  \cite[Tables 2 and 5]{Olver-book}.) In relation to Remark \ref{rmk:235-variationality-subFinsler} one can justify such interpretation as the indicatrix at each point of a sub-Finsler structure on a contact 3-fold is a non-degenerate curve and the divergence equivalence relation amounts to the appearance of  equi-affine transformations of these curves on the contact 3-dimensional manifold. See \ref{sec:div-equiv-pseudo-Finsler-as-var-orthopath} and \ref{sec:general-causal-local-form-invar} for the appearance of affine invariants of hypersurfaces in (pseudo-)Finsler structures under divergence equivalence relation. 
    \end{remark}

We would like to  point out that the expression for $\scc_0$ in \eqref{eq:cc0-4th-oder-ODE-null} appears in the study of (2,3,5)-distribution in terms of their Monge normal form in the following way.
   Recall that every rank 2 distribution in dimension five can be presented by an underdetermined ODE $z'=F(x,y,y',y'',z),$ referred to as its Monge normal form; see \cite{Strazzullo-Monge} for more detail. Given a Monge normal form,  an adapted coframe for the (2,3,5)-distribution is given by
   \begin{equation}
     \label{eq:Adapted-coframe-Monge}
     \begin{aligned}
       \tomega^4&=-\tfrac{9}{4F_{qq}}(\trho^4-F_q\trho^2)\\
     \tomega^3&=\trho^3\\
     \tomega^2&=\trho^2\\
     \tomega^1&=-\tfrac{3}{4}\trho^1-\tfrac{3}{4F_{qq}}(\tfrac{\mathrm{D}}{\exd x}F_q-F_p-F_zF_p)\trho^0+a^1_2\trho^2+a^1_3\trho^3+a^1_4\trho^4\\
     \tomega^0&=\tfrac{1}{3}\trho^0+a^0_2\trho^2+a^0_3\trho^3+\tfrac{1}{90F_{qq}^3}(4F_{qqq}^2-3F_{qq}F_{qqqq})\trho^4
     \end{aligned}
   \end{equation}
   for some functions $a^1_2,a^1_3,a^1_4, a^0_2,a^0_3,$ on $\tM,$   where $\tfrac{\mathrm{D}}{\exd x}=\tfrac{\partial}{\partial x}+p\tfrac{\partial}{\partial y}+q\tfrac{\partial}{\partial p}+F\tfrac{\partial}{\partial z}$ and
   \[\trho^4=\exd z-F\exd x,\quad \trho^3=\exd y-p\exd x,\quad \trho^2=\exd p-q\exd x,\quad \trho^1=\exd q,\quad \trho^0=\exd x.\]
The  conformal structure $[\tg]$ in \eqref{eq:235-Nur-conf-str} can be found explicitly using which one obtains
\begin{equation}
  \label{eq:cc0-appearing-Monge-form}
  \tg(\tfrac{\partial}{\partial z},\tfrac{\partial}{\partial z})=-\tfrac{1}{40F^4_{qq}}(4F_{qqq}^2-3F_{qq}F_{qqqq}),
  \end{equation}
whose vanishing implies the nullity of $\tfrac{\partial}{\partial z}$ and, as an ODE, coincides with  $\scc_0=0$  from \eqref{eq:cc0-4th-oder-ODE-null}.
\begin{remark}\label{rmk:235-Monge-Lagrangian}
The function $F$ has no $z$ dependency if and only if $\frac{\partial}{\partial z}$ is an infinitesimal symmetry of the corresponding (2,3,5)-distribution. As a result of \eqref{eq:cc0-appearing-Monge-form}, Monge systems $z'=F(x,y,p,q)$ for which $\frac{\partial}{\partial z}$ is a null infinitesimal symmetry can be expressed as \eqref{eq:4th-ODE-Lagrangian-null-symm} where $F=L(x,y,p,q)$.   Following the approach in \cite{DZ-DivEquiv},  it is expected that  any infinitesimal symmetry can be taken to be $\frac{\partial}{\partial z}$ in some coordinate system as a result of which $F$ belongs to the divergence equivalence class of the Lagrangian for the corresponding variational 4th order ODE.
\end{remark}

As another example of variational 4th order ODEs  let
\begin{equation}\label{eq:4th-ODE-example-2-use-cor}
  f(x,y,p,q,r)=\ts\frac{q}{p}g_p(y,p)+g_y(y,p)
\end{equation}
for a function $g=g(y,p).$ For such ODEs one has $\scc_0=\sw_{0;\uo\uo}=0.$ 
Furthermore, one can compute  
\begin{equation}
  \label{eq:w-01-4th-ODE-type-N}
  \sw_{0;\uo}=\ts\frac{1}{p^2}(g_{p,p}p-g_p).
  \end{equation}
  Solving $\sw_{0;\uo}=0$ in \eqref{eq:w-01-4th-ODE-type-N}, it follows from part (3) of Corollary \ref{cor:235-from-4th-ODE-various-curv-conditions} that  the Cartan quartic of the  (2,3,5)-distribution arising from quasi-contactifying the ODE \eqref{eq:4th-ODE-example-2-use-cor}  has a repeated root of multiplicity 3, unless $g(y,p)=h_1(y)p^2+h_2(y)$ for two functions $h_1(y)$ and $h_2(y).$ The case $g(y,p)=h_1(y)p^2+h_2(y)$ corresponds to part (3) of Corollary \ref{cor:235-from-4th-ODE-various-curv-conditions} which means the Cartan quartic  has a repeated root of multiplicity at least 4 and the Cartan holonomy is reduced to a proper subgroup of the contact parabolic subgroup $P\subset\mathrm{G}^*_2.$

Now we incorporate Cartan-K\"ahler analysis   to find the local generality of various classes of 4th order ODEs and the corresponding (2,3,5)-distribution. It is known that (2,3,5)-distributions locally depend on 1 function of 5 variables which can be represented as an underdetermined second order ODE, $z'=f(x,y,y',y'',z)$ referred to as the Monge form \cite{Cartan-235}. Also as is clear from \eqref{eq:4th-order-ode}, any 4th order ODE is defined by a function of 5 variables. Using Cartan-K\"ahler analysis, one obtains that variational 4th order ODEs are locally given by 1 function of 4 variables, which can be interpreted as a second order Lagrangian $L(x,y,y',y'')$  whose Euler-Lagrange equation always defines a variational 4th order ODE. It is easy to see that this generality is not affected by taking divergence equivalence classes of Lagrangians.  Furthermore, the local generality of  variational 4th order ODEs discussed in Corollary \ref{cor:235-from-4th-ODE-various-curv-conditions}, 
i.e. satisfying the vanishing conditions $\{\scc_0=0\},$ $\{\scc_0=\sw_{0;\uo\uo}=0\},$ and $ \{\scc_0=\sw_{0;\uo}=0\},$    depend on 2 functions of 3 variables, 1 function of 3 variables and 2 functions of 2 variables, respectively.

\subsection{(3,6)-distributions from variational pairs of 3rd order ODEs}
\label{sec:prec-scal-odes}

In this section we analyze  the one-to-one correspondence between (3,6)-distributions with an infinitesimal symmetry and certain geometries that include  variational pairs of third  order ODEs. 

\subsubsection{Pairs of third order ODEs and their generalization}
\label{sec:pairs-third-order}

By Theorem \ref{thm-CartanConnection}, another class of PACQ structures is equivalent to  regular and normal Cartan geometries   $(\tau\colon\cG\to\cC,\psi)$ of type $(\fk,L)$ where $\fk=\fp^{op}/\fp_{-2}\cong\mathfrak{sl}(2,\RR)\oplus\fgl(2,\RR)\ltimes\RR^2\otimes\mathrm{Sym}^2\RR^2,$  $\fp^{op}\subset\mathfrak{so}(3,4)$ is the opposite contact parabolic subalgebra with the grading $\fp=\fp_{-2}\oplus\fp_{-1}\oplus\fp_{0},$  and $L=\mathrm{SL}(2,\RR)\times B\subset\mathrm{SL}(2,\RR)\times \mathrm{GL}(2,\RR)$  (see \eqref{Liealgs} and \eqref{k_so34}). The Cartan connection for such geometries  can be expressed as
\begin{equation}
  \label{eq:var-pair-3rd-order-ODE-Cartan-conn}
\psi=\begin{pmatrix}
  \textstyle{\phi_0+\mu} & \xi_0 & 0 \\
  \omega^0 & \textstyle{\phi_1-\phi_0+\mu} & 0  \\
  \omega^a & \theta^a& \phi^a_b+\mu\delta^a_b \\
\end{pmatrix}
\ \mathrm{where}\ 
[\phi^a_b]=
\begin{pmatrix}
  \textstyle{-\phi_2} & -\gamma_0 & 0 \\
\theta^0 & 0 & \gamma_0\\
0 & -\theta^0 & \textstyle{\phi_2}
\end{pmatrix}
\end{equation}
and $\mu=\textstyle{-\frac 15}\phi_1.$ The tangent bundle of $\cC$ has a 3-step filtration
\begin{equation}
  \label{eq:3-filtration-pair-3rd-ODE}
  T^{-1}\cC\subset T^{-2}\cC\subset T^{-3}\cC=T\cC,
\end{equation}
where $T^{-1}\cC=\cE\oplus \cV$ and
 \begin{equation}
   \begin{gathered}
     \label{eq:filtration-pair-3rd-order-ODE-each-step}
\cE=\tau_*\left(\Ker\{\omega^3,\theta^3,\omega^2,\theta^2,\omega^1,\theta^1\}\right)=\tau_*\langle\tfrac{\partial}{\partial\theta^0}\rangle,\\ \cV=\tau_*\left(\Ker\{\omega^3,\theta^3,\omega^2,\theta^2,\theta^0\}\right)=\tau_*\langle\tfrac{\partial}{\partial\omega^1},\tfrac{\partial}{\partial\theta^1}\rangle,\quad [\cV,\cV]=\cV\\
     T^{-2}\cC=[\cE,\cV]=\tau_*\left(\Ker\{\omega^3,\theta^3\}\right),\quad      T\cC=T^{-3}\cC=[\cE,T^{-2}\cC].
      \end{gathered}
 \end{equation}
Assuming regularity and normality, the structure equations for such geometries are as in  \eqref{eq:Pairs-3-ODE-Cartan-curv} and \eqref{eq:curv-2form-pair-3rd-ODE} wherein we have only expressed the first order structure equations as the full structure equations would take much more space and will not be used for our purposes. Using the structure equations and taking a section $s\colon\cC\to\cG$, the fundamental invariants, represented as sections of  line bundles $(\cV)^{k_1}\otimes(\cE)^{k_2},$  are given by 
\begin{equation}
  \label{eq:fund-invs-generalized-pair-3rd-order-ODEs}
\begin{aligned}
  \bb_1&=\left(\sbb_{13}(\omega^1)^3+3\sbb_{12}(\omega^1)^2\circ\theta^1 +3\sbb_{11}\omega^1\circ(\theta^1)^2+\sbb_{10}(\theta^1)^3\right) \otimes\ts(\frac{\partial}{\partial\theta^1}\w\frac{\partial}{\partial\omega^1})\otimes(\frac{\partial}{\partial\theta^0})^{-1}\\
  \bb_2&=(\sbb_{22}(\omega^1)^2+2\sbb_{21}\omega^1\circ\theta^1+\sbb_{20}(\theta^1)^2)\otimes\ts(\frac{\partial}{\partial\theta^1}\w\frac{\partial}{\partial\omega^1})\otimes(\frac{\partial}{\partial\theta^0})^{-2}\\
    \bb_3&=\sbb_{30}\ts(\frac{\partial}{\partial\theta^0})^{-3}\\
\bb_4&=\left(\sbb_{42}(\omega^1)^2+2\sbb_{41}\omega^1\circ\theta^1+\sbb_{40}(\theta^1)^2\right)\otimes\ts(\frac{\partial}{\partial\theta^1}\w\frac{\partial}{\partial\omega^1})\otimes(\frac{\partial}{\partial\theta^0})^{-3}\\   \bb_5&=\sbb_{50}\ts(\frac{\partial}{\partial\theta^1}\w\frac{\partial}{\partial\omega^1})^{-1}\otimes(\frac{\partial}{\partial\theta^0})^{-1}\\
  \bb_6&=\left(\sbb_{62}(\omega^1)^2+2\sbb_{61}\omega^1\circ\theta^1+\sbb_{60}(\theta^1)^2\right)\otimes \ts(\frac{\partial}{\partial\theta^0})^{-2}\\
\end{aligned} 
\end{equation}
wherein by abuse of notation we have dropped the pull-back $s^*$ in the expressions above. The torsion components are $\bb_1$ and $\bb_2$ of homogeneity 2 and $\bb_3,\bb_4,\bb_5$ of homogeneity 3. The curvature component is $\bb_6$ which is of homogeneity 4. Moreover, the invariant $\bb_2,\bb_3,\bb_4$ are referred to as the Wilczy\'nski invariants of such structures \cite{CDT-ODE,Medvedev-ODE}.

For such geometries the 2-form  $\rho\in\Omega^2(\cG)$ defined as
\begin{equation}
  \label{eq:rho-pair-3rd-order-ODE}
  \rho=\theta^1\w\omega^3+\theta^2\w\omega^2+\theta^3\w\omega^1
\end{equation}
is the canonical semi-basic 2-form  of maximal rank with respect to the fibration $\tau\colon\cG\to\cC$ which induces the canonical conformally quasi-symplectic structure $\ell\subset\biw^2T^*\cC.$ 
Moreover, the line bundle $\cE$ defined in  \eqref{eq:filtration-pair-3rd-order-ODE-each-step} is the characteristic direction of $\ell$ and $T^{-1}\cC=\cE\oplus\cV$ is isotropic with respect to $\ell.$

Lastly, we note that   one can relate such Cartan geometries to differential equations. As was mentioned in \ref{sec-pairsthirdorder}, the geometry of pairs of third order ODEs under point equivalence can be defined in terms of   the Lie bracket relations \eqref{eq:filtration-pair-3rd-order-ODE-each-step} on  $\cC\cong J^2(\RR,\RR^2)$ and a splitting of $T^{-1}\cC,$  together with an integrability condition arising from  the fibration $J^2(\RR,\RR^2)\to J^0(\RR,\RR^2)\cong \RR^3$ \cite{Yamaguchi-Geometrization,Krynski-ODE}. Using the coframing $(\omega^3,\theta^3,\cdots,\theta^0),$ this fibration is equivalent to the integrability of the Pfaffian system
\begin{equation}
  \label{eq:I-Backlund-integ-3-ODE}  
  I=\{\omega^3,\theta^3,\theta^0\}.
  \end{equation}
Using the first order structure equations \eqref{eq:Pairs-3-ODE-Cartan-curv} and \eqref{eq:curv-2form-pair-3rd-ODE}, this integrability condition is equivalent to $\bb_5=0.$ Note that the integrability condition arising from the fibration $J^1(\RR,\RR^2)\to J^0(\RR,\RR^2)$ is already encoded in the structure equations. Consequently, the condition $\bb_5=0$ is necessary and sufficient for a regular and normal Cartan geometry of type $(\fk,L),$ as above, to be locally realizable as the point equivalence class of a pair of third order ODEs as a result of which, locally, one has $\cC\cong J^2(\RR,\RR^2)$. In \ref{sec:36-local-form-invar} we will  give a parametric expression of   the fundamental invariants of a pair of third order ODEs in terms of $\bb_i$'s in  \eqref{eq:fund-invs-generalized-pair-3rd-order-ODEs}.
\begin{remark}
The geometry of systems of third order ODEs has been studied in \cite{Medvedev-ODE}. Note that unlike   \ref{sec:fourth-order-scalar} where we considered scalar 4th order ODEs under contact equivalence,  it follows from B\"acklund theorem that for pairs of third order ODEs   contact and point  equivalence relations coincide.
\end{remark}

\subsubsection{Variationality}
\label{sec:36-inverse-probl-vari} 
Now we find invariant conditions for such Cartan geometries to be conformally quasi-symplectic. First we determine when  $\ell,$ defined by the 2-form  \eqref{eq:rho-pair-3rd-order-ODE}, has a closed representative.
It follows from structure equations \eqref{eq:Pairs-3-ODE-Cartan-curv} and \eqref{eq:curv-2form-pair-3rd-ODE} that
\[\exd\rho=\phi_1\w\rho+ \varrho\]
for some 3-form $\varrho$ which we will not write explicitly as the expression is rather long.
 
It is a matter of straightforward computation to show that  $\varrho=0$ is equivalent to the vanishing conditions
\begin{equation}
  \label{eq:b123-0-variationality-condition-pair-3rd-order-ODE}
  \bb_1=\bb_2=\bb_3=0.
  \end{equation}
Furthermore,  conditions \eqref{eq:b123-0-variationality-condition-pair-3rd-order-ODE} reduce the structure equations to \eqref{eq:Pairs-3-ODE-Cartan-curv} and \eqref{eq:curv-2form-pair-3rd-ODE-variational}. In particular, one obtains
\[\exd\phi_1=0.\]
As a result,  restricting to pairs of third order ODEs, i.e.  $\bb_5=0,$ the calculations above combined with Theorem \ref{prop-variational} gives the following.
\begin{proposition}\label{prop:variationality-pair-3rd-order-ODEs}
A pair of third order ODEs is variational if and only if  $\bb_1,\bb_2,$ and $\bb_3$ vanish.
\end{proposition}
In \ref{sec:36-local-form-invar}  an explicit expression for $\bb_1,\bb_2,\bb_3$ in terms of a pair of third order ODEs is found. 
\begin{remark}\label{rmk:36-variationality-3rd-pair-ODEs}
Unlike Remark \ref{rmk:235-variationality-subFinsler}, we do not have an  interpretation of variational pairs of third order ODEs as the Euler-Lagrange equations of a variational problem for some geometric structure. Note that even in the case $\bb_5\neq 0,$ the conditions $\bb_1=\bb_2=\bb_3=0$ amount to the  ``variationality'' of Cartan geometries of type $(\fk,L)$ in the sense of Griffiths. 
\end{remark}
Now we can state the first lemma needed in the proof of Proposition~\ref{lem-canonicalconfqsymp}.
  \begin{lemma}\label{lemm:36-variationality-3rd-pair-ODEs}
Given a regular and normal Cartan geometry $(\tau\colon\cG\to \cC,\psi)$ of type $(\fk,L)$ as in \eqref{k_so34}, then any compatible conformally quasi-symplectic structure is induced by the 2-form given in \eqref{eq:rho-pair-3rd-order-ODE}.
\end{lemma}
\begin{proof}
  By \eqref{eq-formrho} a compatible almost conformally quasi-symplectic structure $\ell\subset\biw^2(T^*\cC)$ is induced by  $\rho\in\Omega^2(\cG)$ where
  \begin{equation}
    \label{eq:rho-pair-3rd-order-PACQ-arbitrary}
      \begin{aligned}
    \rho&=p_1\theta^1\w\omega^3+p_1\theta^3\w\omega^1+p_2\theta^1\w\theta^3+p_{3}\omega^3\w\omega^1+q\theta^2\w\omega^2\\
    &+r_1\theta^2\w\omega^3+r_2\theta^2\w\theta^3+r_3\theta^3\w\omega^2+r_{4}\omega^3\w\omega^2+s\theta^3\w\omega^3
    \end{aligned}
    \end{equation}
for functions $p_1,p_2,p_3,q,r_1,r_2,r_3,r_4,s$  on $\cG.$  In the computations below we work modulo the ideal $\{\omega^0,\phi_0,\phi_1,\phi_2,\xi_0,\gamma_0\}$ and define
  \[V_h=\theta^0\w\theta^1\w\theta^2\w\theta^3\w\omega^1\w\omega^2\w\omega^3.\]
  Using the first order structure equations   \eqref{eq:Pairs-3-ODE-Cartan-curv} and \eqref{eq:curv-2form-pair-3rd-ODE} one obtains
  \[\exd\rho\w\theta^1\w\omega^2\w\theta^3\w\omega^3\equiv (p_1-q)V_h\Rightarrow p_1=q.\]
  Furthermore, one obtains
  \[
    \begin{aligned}
      \exd\rho\w\omega^1\w\omega^2\w\theta^2\w\theta^3&\equiv -(r_1+q_{;\uz})V_h\\
      \exd\rho\w\theta^1\w\theta^2\w\omega^2\w\omega^3&\equiv (r_3+q_{;\uz})V_h\\
      \exd\rho\w\theta^1\w\omega^1\w\theta^3\w\omega^3&\equiv (-r_1-r_3+q_{;\uz})V_h,
    \end{aligned}
  \]
  which implies $r_1=r_3=q_{;\uz}=0$ where $q_{;\uz}=\frac{\partial}{\partial\theta^0}q.$
  In the same spirit, one computes
  \[
    \begin{aligned}
      \exd\rho\w\omega^1\w\omega^2\w\omega^3\w\theta^3&\equiv -p_2V_h,\qquad       \exd\rho\w\theta^1\w\theta^2\w\theta^3\w\omega^3&\equiv -p_3V_h
    \end{aligned}
  \]
  which yields $p_2=p_3=0$ and 
    \[
    \begin{aligned}
      \exd\rho\w\omega^1\w\omega^2\w\omega^3\w\theta^2&\equiv r_2V_h,\qquad       \exd\rho\w\theta^1\w\theta^2\w\theta^3\w\omega^2&\equiv r_4V_h
    \end{aligned}
  \]
  which gives $r_2=r_4=0.$
  Lastly, one obtains
      \[
    \begin{aligned}
      \exd\rho\w\theta^1\w\omega^1\w\theta^2\w\omega^3&\equiv (s-2q\sbb_{21})V_h,\qquad          \exd\rho\w\theta^1\w\omega^1\w\omega^2\w\theta^3&\equiv -(s+2q\sbb_{21})V_h
    \end{aligned}
  \]
  which gives  $s=\sbb_{21}=0.$ Note that if $q=0$ then $\rho=0$ in \eqref{eq:rho-pair-3rd-order-PACQ-arbitrary}. If $q\neq 0,$ then $\rho$ in \eqref{eq:rho-pair-3rd-order-PACQ-arbitrary} is a multiple of the canonical quasi-symplectic 2-form  in \eqref{eq:rho-pair-3rd-order-ODE}. 
  Furthermore, by Proposition \ref{prop:variationality-pair-3rd-order-ODEs}, the condition $\sbb_{21}=0$ obtained in the last step is implied by the existence of a closed 2-form in the canonical quasi-symplectic structure induced by $\rho$.

\end{proof}

\subsubsection{(3,6)-distributions}
\label{sec:3-6-distributions}

The Cartan connection of a regular and normal  Cartan geometry $(\ttau\colon\tcG\to\tcC,\tpsi)$ of type  $(\mathfrak{so}(3,4),P_{23})$ can be expressed as 
\begin{equation}
  \label{eq:tpsi-Cartan-conn-B3-P23}
  \tpsi=\begin{pmatrix}
  \tphi_0 & \txi_0 & \txi_b & \txi_4 & 0\\
  \tomega^0 &\tphi_1-\tphi_0 & \tgamma_b & 0 & -\txi_4\\
  \tomega^a & \ttheta^a& \tphi^a_b & -\tgamma^a & -\txi^a\\
  \tomega^4 &0 & -\ttheta_b & \tphi_0-\tphi_1 & -\txi_0\\
  0 & -\tomega^4 & -\tilde\omega_b & -\tomega^0 & -\tphi_0
\end{pmatrix}
\quad \mathrm{where}\quad 
[\tphi^a_b]=
\begin{pmatrix}
-\tphi_2 & -\tgamma_0 & 0 \\
\ttheta^0 & 0 & \tgamma_0\\
0 & -\ttheta^0 & \tphi_2
\end{pmatrix}.
\end{equation}
The tangent bundle of $\tcC$ has a 4-step filtration
\begin{equation}
  \label{eq:4-filtration-36-dist-corr-space}
  T^{-1}\tcC\subset T^{-2}\tcC\subset T^{-3}\tcC\subset T^{-4}\tcC=T\tcC
\end{equation}
where $T^{-1}\tcC=\tcE\oplus \tcV$ and
 \begin{equation}
   \begin{gathered}
     \label{eq:4step-filtration-36-dist-corr-space-each-step}
\tcE=\ttau_*\left(\Ker\{\tomega^3,\ttheta^3,\tomega^2,\ttheta^2,\tomega^1,\ttheta^1\}\right)=\ttau_*\langle\tfrac{\partial}{\partial\ttheta^0}\rangle,\\ \tcV=\ttau_*\left(\Ker\{\tomega^3,\ttheta^3,\tomega^2,\ttheta^2,\ttheta^0\}\right)=\ttau_*\langle\tfrac{\partial}{\partial\tomega^1}, \tfrac{\partial}{\partial\ttheta^1}\rangle, \\
     T^{-2}\tcC=[\tcE,\tcV]=\ttau_*\left(\Ker\{\tomega^3,\ttheta^3,\tomega^4\}\right),\\      T^{-3}\tcC=[\tcE,T^{-2}\tcC]=\ttau_*\left(\Ker\{\tomega^4\}\right),\quad T\tcC=T^{-4}\tcC=[\tcV,T^{-3}\tcC]=[T^{-2}\tcC,T^{-2}\tcC].
      \end{gathered}
    \end{equation}
    Furthermore, assuming regularity, the distribution $\tcV$ is integrable. 
Defining the 2-forms
\[
\begin{aligned}
  \tOmega^1&=\exd \tomega^1+\tgamma_3\w\tomega^4- \tgamma_0\w\tomega^2-\tomega^0\w\ttheta^1  -(\tphi_0+\tphi_2)\w\tomega^1,\\
  \tTheta^1&=\exd \ttheta^1-\txi_3\w\tomega^4 -\tgamma_0\w\ttheta^2-\txi_0\w\tomega^1-(\tphi_1+\tphi_2-\tphi_0)\w\ttheta^1,
\end{aligned}
\] 
the normality conditions imply
\[
  \begin{aligned}
    \tOmega^1\equiv& -B_0\ttheta^3\w\ttheta^0 - B_1\tomega^3\w\ttheta^0+X_1\tomega^3\w\ttheta^3,\quad && \mod\quad \{\tomega^4,\tomega^2,\tomega^1,\ttheta^2,\ttheta^1\},\\
    \tTheta^1\equiv&\phantom{-} B_1\ttheta^3\w\ttheta^0 + B_2\tomega^3\w\ttheta^0+X_2\tomega^3\w\ttheta^3 \quad  &&\mod\quad \{\tomega^4,\tomega^2,\tomega^1,\ttheta^2,\ttheta^1\}.
  \end{aligned}
\]
for some functions $B_0,B_1,B_2,X_1,X_2$ on $\cG.$  Taking  a section $s\colon\tcC\to\tcG,$ the harmonic invariant of such geometries, given by    $\bB\in\Gamma(\mathrm{Sym}^2(\cV^*)\otimes\cE^{-3}\otimes\biw^2\cV),$ is
\begin{equation}
  \label{eq:B-harm-inv-36-corr-space}
  \bB=s^*\left ( B_2 \tomega^1\circ \tomega^1 + 2B_1\ttheta^1\circ\tomega^1 + B_0\ttheta^1\circ\ttheta^1\right)\otimes\ts(\frac{\partial}{\partial s^*\ttheta^1}\w\frac{\partial}{\partial  s^*\tomega^1}) \otimes(\frac{\partial}{\partial s^*\ttheta^0})^{-3}.
  \end{equation}
  By Proposition \ref{prop-qcontactcone}, such geometries always arise as the correspondence space of (3,6)-distributions, i.e.  rank 3 distributions $\tscD\subset T\tM$ such that $[\tscD,\tscD]=T\tM.$ More precisely, let $\tM$ be   the leaf space 
  \begin{equation}
    \label{eq:tM-quot-map-36-corr-space}
    \tnu\colon\tcC\to\tM=\tcC\slash\cI_{\tcV}
      \end{equation}
where $\cI_{\tcV}$ is the foliation of $\tcC$ by integral surfaces of the distribution $\tcV.$ Since $[\tcV,T^{-2}\tcC]=T^{-2}\tcC,$ it follows that 
\begin{equation}
  \label{eq:rank-3-distr-D-on-M-36}
\tscD:=\tnu_{*}T^{-2}\tcC
\end{equation}
is a distribution whose growth vector can be found from \eqref{eq:3-filtration-pair-3rd-ODE} and \eqref{eq:filtration-pair-3rd-order-ODE-each-step} to be (3,6). Conversely, starting from a (3,6)-distribution, one has $\tcC\cong \PP\tscD,$ where $\PP\tscD$ is the $\PP^2$-bundle  over $\tM$ whose fibers at each point $p\in\tM$ is the projective plane $\PP\tscD_p.$

Using the embedding $\mathfrak{so}(3,4)\subset\mathfrak{so}(4,4),$ it is known  \cite{Bryant-36} that the Cartan connection \eqref{eq:tpsi-Cartan-conn-B3-P23} defines a normal conformal connection for the conformal structure $[s^*\tg]\subset\mathrm{Sym^2}T^*\tM,$ for any section $s\colon\tM\to\tcG,$  where  
\begin{equation}
  \label{eq:h-36-Bryant-metric}
  \tg=\tilde\omega^4\circ\tilde\theta^0-\tilde\omega^3\circ\tilde\theta^2+\tilde\omega^2\circ\tilde\theta^3\in \Gamma(\mathrm{Sym}^2T^*\tcG).
\end{equation}
The conformal holonomy of such conformal structures is reduced to $\mathrm{SO}(3,4).$ 
\subsubsection{Quasi-contactification and various curvature conditions}
\label{sec:36-quasi-cont-from-pair-3rd-ODE}
As in \ref{sec:235-4th-ODE-quasi-cont}, we explain the one-to-one correspondence between (3,6)-geometries with an infinitesimal symmetry and Cartan geometries of type $(\fk,L),$ considered in this section, and relate their Cartan connections. This will be used to relate their fundamental invariants and give geometric interpretations for several curvature conditions. 

We first give the proof of Proposition \ref{prop-relharmonics} for these classes of Cartan geometries.
\begin{proof}[Proof of Proposition \ref{prop-relharmonics}; second part]
  Since this part is almost identical to the first part in \ref{sec:235-4th-ODE-quasi-cont}, we only highlight the differences. 
  The coframing $\hat\psi=(\tomega^4,\pi^*\psi)$ defines a Cartan connection for a Cartan geometry $(\hat\tau\colon\pi^*\cG\to\tcC,\hat\psi)$   of type $(\fp^{op},\mathrm{SL}(2,\RR)\times B)$ where $\fp^{op}\subset\mathfrak{so}(3,4)$ is the opposite contact parabolic subalgebra and $B\subset \mathrm{GL}(2,\RR)$ is the Borel subgroup. The 3-step filtration \eqref{eq:3-filtration-pair-3rd-ODE} and the fact that  $\tomega^4$ is quasi-contact imply that $T\tcC$ has the 4-step filtration \eqref{eq:4-filtration-36-dist-corr-space} with Lie bracket relations as in \eqref{eq:4step-filtration-36-dist-corr-space-each-step}. Hence, one can associate a Cartan geometry $(\ttau\colon\tcG\to\tcC,\tpsi)$ of type $(\mathfrak{so}(3,4),P_{23}).$  It is a matter of  computation, using the matrix forms  \eqref{eq:var-pair-3rd-order-ODE-Cartan-conn} and \eqref{eq:tpsi-Cartan-conn-B3-P23}, to show that  $\pr_\fk\circ\iota^*\tpsi$ and $\hat\pi^*\psi$ are related as follows:
\begin{equation}
  \label{eq:Cartan-conn-modif-36-pair-3rd-order-ODE}
  \begin{gathered}
  \tilde\omega^a=\omega^a,\quad\tilde\theta^a=\theta^a,\quad \tilde\theta^0=\theta^0-\textstyle{\frac 12}\sbb_{50}\tilde\omega^4,\quad\tilde\phi_1 = \phi_1,\\ \tilde\omega^0=\omega^0-\textstyle{\frac 12 \sbb_{60}}\tilde\omega^4,\quad \tilde\phi_0 = \phi_0-\textstyle{\frac 12 \sbb_{61}}\tilde\omega^4,\quad 
  \quad \tilde\xi_0 = \xi_0+\textstyle{\frac 12 \sbb_{62}}\tilde\omega^4,\\
  \tilde\phi_2 = \phi_2-\textstyle{\frac 12 \sbb_{50;\underline 0}}\tilde\omega^4,\quad\tilde\gamma_0 =\gamma_0+ \textstyle{(-\frac 12\sbb_{50;\underline{00}}+\frac{1}{20}\sbb_{40;11}+\frac 1{20}\sbb_{42;\underline{11}}-\frac1{10}\sbb_{41;1\underline 1})}\tilde\omega^4.
\end{gathered}
\end{equation}
wherein the pull-backs $\hat\pi^*$ and $\iota^*$ are dropped. 
\end{proof}
To state the main corollary of this section  recall that the Lie algebra $\fk$  \eqref{k_so34}  has the grading
 $\fk=\fk_{-3}\oplus\cdots\oplus\fk_{1}$ where $\fk_{-1}=\fv\oplus\fe.$
Now define the Lie algebra $\fx\subset\fk$ as 
\begin{equation}
  \label{eq:LieAlg-fx-3rd-pair-ODE}  
  \fx=\fv\oplus\fk_{0}\oplus\fk_{1}.
\end{equation}
The Lie group $X$ with Lie algebra $\fx$ is
\begin{equation}
  \label{eq:36-corr-X}
  X=(\mathrm{SL}(2,\mathbb{R})\times B)\ltimes (\mathbb{R}^2\otimes\mathrm{Sym}^2\mathbb{R}^2)^{-1},
\end{equation}
where $B\subset \mathrm{GL}(2,\mathbb{R})$ is the Borel subalgebra and $(\mathbb{R}^2\otimes\mathrm{Sym}^2\mathbb{R}^2)^{-1}$ is the smallest $\mathrm{SL}(2,\mathbb{R})\times B$-invariant filtrant of  $\mathbb{R}^2\otimes\mathrm{Sym}^2\mathbb{R}^2$ .
\begin{corollary}\label{cor:36-from-pair-3rd-ODE-various-curv-conditions}
Let $(\ttau\colon\tcG\to\tcC,\tpsi)$ be a regular and normal Cartan geometry  of type  $(B_3,P_{23})$ with an infinitesimal symmetry and $(\tau\colon\cG\to\cC,\psi)$ be a conformally quasi-symplectic Cartan geometry of type $(\fk,L)$ as in \eqref{k_so34} on the leaf space of the integral curves of the infinitesimal symmetry via the quotient map $\pi\colon\tcC\to\cC.$ Then one has
\begin{equation}
  \label{eq:B-b4-36-3rd-ODE}
  \bB=\pi^*(\bb_4),
  \end{equation}
  where $\bB$ is given in  \eqref{eq:B-harm-inv-36-corr-space} and  $\bb_4$ is given in  \eqref{eq:fund-invs-generalized-pair-3rd-order-ODEs}. Furthermore,  for the naturally induced (3,6)-distribution on the  leaf space $\tM,$ as defined in  \eqref{eq:tM-quot-map-36-corr-space}, the following holds.
  \begin{enumerate}
  \item   There is a one-to-one correspondence between (3,6)-distributions whose infinitesimal symmetry  is null with respect to its canonical conformal structure and variational pairs of third order ODEs. In this case  $\tM$ is foliated by 3-dimensional null submanifolds.
  \item A variational pair of third order ODEs defines  a Cartan geometry $(\tau_1\colon\cG\to M,\psi)$ of type $(\fk,X),$ where $M\cong J^1(\RR,\RR^2)$ if and only if  \eqref{eq:Relation-Cartan-Connections} can be written as
    \begin{equation}
      \label{eq:tpsi-psi-36-0-J2}
      \pr_\fk\circ\iota^*\tpsi-\hat\pi^*\psi= 0,
          \end{equation}
i.e. the invariant conditions  $\bb_5=\bb_6=\sbb_{42;\underline{11}}+\sbb_{40;{11}}-2\sbb_{41;1\underline{1}}=0$ are satisfied. 
  \item If a variational pair of third order ODEs satisfies the invariant conditions  $\bb_6=\sbb_{42;\uo}-\sbb_{41;1}=\sbb_{40;1}-\sbb_{41;\uo}=0,$
    then  the Cartan holonomy of the corresponding (3,6)-distribution is  a proper subgroup of the  contact parabolic subgroup  $P\subset SO(3,4).$ 
    \item Cartan geometries of type $(\fk,L)$ whose quasi-contactification is the flat (3,6)-distribution are in one-to-one correspondence with the 2-parameter family of 6-dimensional torsion-free  $\mathrm{SL}(2,\RR)\times \mathrm{GL}(2,\RR)$-structures. 
  \end{enumerate}
\end{corollary}
\begin{proof}
  As in Corollary \ref{cor:36-from-pair-3rd-ODE-various-curv-conditions}, the first claim is shown by using expression \eqref{eq:Cartan-conn-modif-36-pair-3rd-order-ODE} to form $\tpsi,$ as given in \eqref{eq:tpsi-Cartan-conn-B3-P23}, and using structure equations \eqref{eq:curv-2form-pair-3rd-ODE-variational} to compute the harmonic curvature $\bB$ for the (3,6)-distribution, as given in \eqref{eq:B-harm-inv-36-corr-space}.

  Part \emph{(1)} is shown by using \eqref{eq:Cartan-conn-modif-36-pair-3rd-order-ODE} to express the canonical  bilinear form $\iota^*\tg$,  defined in \eqref{eq:h-36-Bryant-metric},  as
  \begin{equation}
    \label{eq:h0-36-general-from-ODE}
    \iota^*\tg=\omega^4\circ\theta^0-\omega^3\circ\theta^2+\omega^2\circ\theta^3-\half\sbb_{50}\omega^4\circ\omega^4\in \Gamma(\mathrm{Sym}^2 T^*\pi^*\cG).
      \end{equation}
       As discussed in the proof of the first part of Proposition \ref{prop-relharmonics} in \ref{sec:235-4th-ODE-quasi-cont},  on $\pi^*\cG$ the infinitesimal symmetry can be taken to be $\tfrac{\partial}{\partial\omega^4}.$ Thus,  the vanishing of $\sbb_{50}$ is equivalent  to the nullity of the infinitesimal symmetry. As discussed in \ref{sec:pairs-third-order}, the integrability of $I$ in \eqref{eq:I-Backlund-integ-3-ODE} is equivalent to $\sbb_{50}=0$ and  is necessary and sufficient for the Cartan geometry $(\tau\colon\cG\to\cC,\psi)$ of type $(\fk,L)$ to arise from a pair of third order ODEs. By Theorem \ref{prop-variational}, being conformally quasi-symplectic implies that such pairs of ODE are variational.
Lastly,  using \eqref{eq:Cartan-conn-modif-36-pair-3rd-order-ODE}, the integrability of $I$ implies that the Pfaffian system $\iota^* \tI$ is integrable on $\iota^*\tilde\cG$ where 
      \begin{equation}
  \label{eq:tI-36-integrable-null-conormal}
\tI= \{\tomega^3,\ttheta^3,\ttheta^0\}.
  \end{equation}
The leaves of this Pfaffian system  are 3-dimensional submanifolds of $\tM$ with null tangent space everywhere   with respect to the conformal structure induced by $\iota^*\tg$ in  \eqref{eq:h0-36-general-from-ODE}. Recall that $\tM$ is the leaf space \eqref{eq:tM-quot-map-36-corr-space} which is   equipped with a (3,6)-distribution.
  
Part \emph{(2)} follows directly from \eqref{eq:Cartan-conn-modif-36-pair-3rd-order-ODE}. Consequently, one can check that the Cartan curvature \eqref{eq:curv-2form-pair-3rd-ODE-variational} is horizontal with respect to the fibration $\tau_1\colon\cG\to M,$ where $M$ is the leaf space  $\nu\colon\cC\to M=\cC\slash \cI_{\cV},$ wherein $\cI_{\cV}$ is the foliation of $\cC$ by the integral curves of $\cV,$ defined in \eqref{eq:filtration-pair-3rd-order-ODE-each-step} and $\tau_1=\nu\circ\tau$. Using  \eqref{eq:3-filtration-pair-3rd-ODE}  and the integrability of the Pfaffian system \eqref{eq:I-Backlund-integ-3-ODE}, it is straightforward to show   $M\cong J^1(\RR,\RR).$

Part \emph{(3)} follows by computing $\iota^*\tpsi$ in part (2) and noting that the only obstructions for it to be $\fp^{op}$-valued are the vanishing of the 1-forms
\[
  \begin{gathered}
    \txi_3=\tfrac{1}{5}(\sbb_{42;\uo\ud}-\sbb_{41;1\ud })\theta^3  + \tfrac{1}{5}(\sbb_{42;\uo 2}-\sbb_{41;1 2})\omega^3 + \tfrac{1}{5}(\sbb_{42;\uo}-\sbb_{41;1})\theta^0,\\
    \tgamma_3=\tfrac{1}{5}(\sbb_{40;1\ud}-\sbb_{41;\uo\ud})\theta^3  +\tfrac{1}{5}(\sbb_{40;12}-\sbb_{41;\uo 2}) \omega^3 + \tfrac{1}{5}(\sbb_{40;1}-\sbb_{41;\uo}) \theta^0.
  \end{gathered}
\]
 which is satisfied if and only if the coefficients of  $\theta^0$ are zero.

 Part \emph{(4)} is similar to part (4) in Corollary \ref{cor:235-from-4th-ODE-various-curv-conditions}. The flat (3,6)-distribution is characterized by $\bB=0$ and by \eqref{eq:B-b4-36-3rd-ODE} they correspond to pairs of third order ODEs for which the only non-vanishing curvature components can be $\bb_5$ or $\bb_6.$ By \cite{CS-special}, such structures define special symplectic connections and depend on 2 parameters.  The vanishing of the Wilczy\'nski invariants $\bb_2,\bb_3,\bb_4$ imply that the leaf space $\lambda\colon\cC\to S:=\cC/\cI_\cE$ of the solution curves of the pair of ODEs is equipped with a $\mathrm{SL}(2,\RR)\times\mathrm{GL}(2,\RR)$-structure, with connection $\nabla$ (\cite{CDT-ODE}). The vanishing of $\bb_1$ implies that these  structures are among torsion-free parabolic conformally symplectic structures in \cite{CS-cont1}.   Unlike  part (4) in Corollary \ref{cor:235-from-4th-ODE-various-curv-conditions}, straightforward computation shows that the Ricci curvature of such torsion-free $\nabla$ is always symmetric. 

\end{proof}

\begin{remark}
  Using the filtration \eqref{eq:3-filtration-pair-3rd-ODE} and assuming  $\bb_5=0$ which implies the integrability of \eqref{eq:I-Backlund-integ-3-ODE},   the Cartan geometry induced on $J^1(\RR,\RR^2)$ in  part (2) of Corollary \ref{cor:36-from-pair-3rd-ODE-various-curv-conditions} defines a 3-dimensional \emph{path geometry}  with additional properties which we do not discuss here. Moreover, it would be interesting to give an interpretation of the curvature conditions for variational pairs of third order ODEs that are equivalent to the descent of the ODE geometry to    a Cartan geometry on $J^0(\RR,\RR^2)\cong\RR^3.$ Such descent should manifest itself as a property of the corresponding (3,6)-distribution.

\end{remark}

\subsubsection{Parametric expressions, examples and local generality} 
\label{sec:36-local-form-invar}
Unlike the notation used in \ref{sec:spec-conf-quasi} and Theorem \ref{prop-variational} for jet coordinates, to simplify the parametric expressions in this section we use $(x,y_1,y_2,p_1,p_2,q_1,q_2,r_1,r_2)$ for jet coordinates of $J^3(\RR,\RR^2)$ with contact system 
\[\rho_0^{s}=\exd y^s-p^s\exd x,\quad \rho_1^{s}=\exd p^s-q^s\exd x,\quad \rho_2^{s}=\exd q^s-r^s\exd x,\quad \rho^0=\exd x,\quad 1\leq s\leq 2.\]
If $\cC\subset J^3(\RR,\RR^2)$ is a co-dimension two submanifold defined by a pair of third order ODEs
\begin{equation}
  \label{eq:pair-3rd-order-odes}
  r^1=f^1(x,y^1,y^2,p^1,p^2,q^1,q^2),\qquad   r^2=f^2(x,y^1,y^2,p^1,p^2,q^1,q^2)
\end{equation}
then the pull-back of the contact forms on $J^3(\RR,\RR^2)$ to $\cC$ is given by
\begin{equation}
  \label{eq:pair-3rd-order-contact-system}
  \rho^{s}_0=\exd y^s-p^s\exd x,\quad \rho^{s}_1=\exd p^s-q^s\exd x,\quad \rho^{s}_2=\exd q^s-f^s(x,y^1,y^2,p^1,p^2,q^1,q^2)\exd x,\quad \rho^0=\exd x,
\end{equation}
where  $1\leq s\leq 2.$
As was mentioned in \ref{sec:pairs-third-order}, the geometry of point equivalence class of pairs of third order ODEs can be expressed as  regular and normal Cartan geometries of type $(\mathfrak k,L).$ After imposing appropriate coframe adaptation on 1-forms \eqref{eq:pair-3rd-order-contact-system} with $1\leq r,s,t,u,v\leq 2$, the fundamental invariants of a pair of third  order ODEs, given as $\bb_i$'s in \eqref{eq:fund-invs-generalized-pair-3rd-order-ODEs}, is found as follows. Define
\[  \begin{aligned}
    E^r_{st}=&\textstyle{\frac{\partial^2 f^r}{\partial q^s\partial q^t}},\\
    F^r_s=&\textstyle{\frac{\exd}{\exd x}\frac{\partial f^r}{\partial q^s}-\frac{1}{3}\frac{\partial f^r}{\partial q^u}\frac{\partial f^u}{\partial q^s}-\frac{\partial f^r}{\partial p^s}}\\
   G^r_{s}=&\textstyle{\frac{\exd^2}{\exd x^2}\frac{\partial f^r}{\partial q^s}  - \frac 23\frac{\exd}{\exd x}\frac{\partial f^r}{\partial q^u} \frac{\partial f^u}{\partial q^s}- \frac 23\frac{\partial f^u}{\partial q^s} \frac{\exd}{\exd x}\frac{\partial f^r}{\partial q^u}} -3\frac{\exd}{\exd x}\frac{\partial f^r}{\partial p^s}  +\frac 29\frac{\partial f^r}{\partial q^s} (\frac{\partial f^t}{\partial q^u}\frac{\partial f^u}{\partial q^t}+\frac{\partial f^t}{\partial q^t}\frac{\partial f^u}{\partial q^u})  +\frac{\partial f^r}{\partial q^s}\frac{\partial f^u}{\partial p^u}  +\frac{\partial f^r}{\partial p^s}\frac{\partial f^u}{\partial q^u}   + 6 \frac{\partial f^r}{\partial y^s} \\
   H_{rs}=&\textstyle{\frac{\exd}{\exd x}\frac{\partial^3 f^u}{\partial q^u\partial q^r\partial q^s} +\frac{32}{5}\frac{\partial^3 f^u}{\partial q^u \partial q^r\partial p^s} -\frac{27}{5}\frac{\partial^3 f^u}{\partial q^u \partial q^r\partial q^s}
    + \frac{74}{15} \frac{\partial}{\partial q^r}\left(\frac{\partial f^u}{\partial q^s}\frac{\partial^2 f^v}{\partial q^v\partial q^u}\right)
    -\frac{18}{5} \frac{\partial}{\partial q^r}\left(\frac{\partial f^u}{\partial q^v}\frac{\partial^2 f^v}{\partial q^u\partial q^s}\right)}\\
    &\textstyle{+\frac{29}{15}\frac{\partial^2 f^u}{\partial q^r\partial q^s}\frac{\partial^2 f^v}{\partial q^v\partial q^u}
    +\frac{14}{45}\frac{\partial^2 f^u}{\partial q^r\partial q^u}\frac{\partial^2 f^v}{\partial q^v\partial q^s}}.
\end{aligned}
\]
Then, one obtains that
\[
  \begin{gathered}
    \sbb_{10}=\overset{o}{E}{}^2_{11},\qquad     \sbb_{11}=\overset{o}{E}{}^2_{21},\qquad    \sbb_{12}=\overset{o}{E}{}^2_{22},\qquad    \sbb_{13}=\overset{o}{E}{}^1_{22},\\
    \sbb_{20}=-\overset{}{F}{}^2_{1},\qquad     \sbb_{21}=\half(\overset{}{F}{}^1_{1}-\overset{}{F}{}^2_{2}),\qquad    \sbb_{22}=\overset{}{F}{}^1_{2},\\
    \sbb_{30}=\overset{}{G}{}^1_1+\overset{}{G}{}^2_2+\textstyle{\frac 89 \frac{\partial f^u}{\partial q^u}\frac{\partial f^{[t}}{\partial q^s}\frac{\partial f^{s]}}{\partial q^{t}}},\qquad  \sbb_{40}=-\overset{}{G}{}^2_1,\qquad \sbb_{41}=\half(\overset{}{G}{}^1_1-\overset{}{G}{}^2_2),\qquad \sbb_{42}=\overset{}{G}{}^1_2,\\
\sbb_{50}=0,\qquad    \sbb_{60}={H}_{11},\qquad     \sbb_{61}=\half({H}_{12}+H_{21}),\qquad     \sbb_{62}={H}_{22},\qquad 
  \end{gathered}
\]
where
\[    \overset{o}{E}{}^r_{st}=\textstyle{E^r_{st} -\frac 23E^u_{u(s}\delta^r_{t)}}
\]
and $\frac{\exd}{\exd x}=\frac{\partial}{\partial x}+p^s\frac{\partial}{\partial y^s}+q^s\frac{\partial}{\partial p^s}+f^s\frac{\partial}{\partial q^s}$ is the total derivative.

As an obvious  example consider the degenerate second order  Lagrangian
\[L(x,y^1,y^2,p^1,p^2,q^1,q^2)=g(x,y^1,y^2,p^1,p^2)q^1+h(x,y^1,y^2,p^1,p^2)q^2,\]
and its  Euler-Lagrange equations. For simplicity we only consider the case  $g=g(p^1,p^2)$ and $h=h(p^1,p^2).$
The corresponding Euler-Lagrange equations in this case are given by pair of 3rd order ODEs \eqref{eq:pair-3rd-order-odes} where
\[
  \begin{aligned}    
    f^1=&\textstyle{\frac{1}{g_{p^1}-h_{p^2}}((h_{p^2,p^2}-g_{p^1,p^2})q^2q^1-(g_{p^1p^1}-h_{p^1p^2})(q^1)^2)},\\ f^2=&\textstyle{\frac{1}{g_{p^1}-h_{p^2}}((h_{p^2,p^2}-g_{p^1,p^2})(q^2)^2-(g_{p^1p^1}-h_{p^1p^2})q^1q^2)}
  \end{aligned}
  \]
for which $\bb_1,\bb_2,\bb_3,\bb_5=0$ and $\bb_4$ and $\bb_6$ are generically nonzero. Therefore, by part (1) of Corollary \ref{cor:36-from-pair-3rd-ODE-various-curv-conditions} the  quasi-contactification of such pairs of ODEs gives  (3,6)-distributions with a null infinitesimal symmetry. Note that a Lagrangian for the trivial pair of third order ODEs is  $L=p^2q^1-p^1q^2.$

As another example consider
\[f^1=g_1(x,y^1,y^2)p^1+g_2(x,y^1,y^2),\qquad f^2=g_1(x,y^1,y^2)p^2+g_3(x,y^1,y^2),\]
where
\[\textstyle{\frac{\partial}{\partial y^1} g_2(x,y^1,y^2)+\frac{\partial}{\partial y^2} g_3(x,y^1,y^2) -\frac{\partial}{\partial x} g_1(x,y^1,y^2)}=0.\]
For such pairs of third order ODEs one has $\bb_1,\bb_2,\bb_3,\bb_5,\bb_6=0$ but $\bb_4$ generically nonzero. By part (3) of Corollary \ref{cor:36-from-pair-3rd-ODE-various-curv-conditions} the  quasi-contactification of such ODEs results in  (3,6)-distributions with a null infinitesimal symmetry whose holonomy is reduce to a proper subgroup  of the contact parabolic subgroup $P\subset B_3.$

\begin{remark}
  In analogy with our discussion in \ref{sec:local-form-invar}, in particular Remark \ref{rmk:235-Monge-Lagrangian}, it is reasonable to expect that the underdetermined scalar ODE
  \[z'=(y^1)''F_1(x,y^1,y^2,(y^1)',(y^2)',z)+(y^2)''F_2(x,y^1,y^2,(y^1)',(y^2)',z),\]
  for two functions $F_1$ and $F_2,$ satisfying certain non-degeneracy condition, describes a class of regular and normal Cartan geometries of type $(\mathfrak{so}(3,4),P_{23}).$ 
\end{remark}
  
Now we use Cartan-K\"ahler analysis   to find local generality of various classes of pairs of 3rd order ODEs. The generality of   (3,6)-distributions is 3 functions of 6 variables. Using the structure equations and the Cartan-K\"ahler analysis, it can be shown that the local generality of conformally quasi-symplectic Cartan geometries of type $(\fk,L)$ is 3 functions of 5 variables. It is clear from \eqref{eq:pair-3rd-order-odes} that any pair of 3rd order ODEs is defined by 2 function of 7 variables. Using Cartan-K\"ahler, one can compute that variational pairs of third order  ODE   locally depend on 2 functions of 5 variables. Furthermore, variational 3rd  order ODEs  satisfying   $\{\sbb_5=\bb_6=\sbb_{42;\underline{11}}+\sbb_{40;{11}}-2\sbb_{41;1\underline{1}}=0\}$ and $\{\sbb_5=\sbb_6=\sbb_{42;\uo}-\sbb_{41;1}=\sbb_{40;1}-\sbb_{41;\uo}=0\},$   as in parts (2) and (3) of Corollary \ref{cor:36-from-pair-3rd-ODE-various-curv-conditions}, locally depend on  1 function of 5 variables  and 2 functions of 4 variables, respectively.

\subsection{Causal structures from variational orthopath  geometries} 
\label{sec:caus-struct-from}
In this section the equivalence problem for variational orthopath geometries is solved and a Finslerian interpretation of them is given. We use this result to obtain causal structures, in particular pseudo-Riemannian conformal structures, with an infinitesimal symmetry via quasi-contactification and use the construction to give a geometric interpretation of several curvature conditions for the orthopath geometry. Lastly, we find the fundamental invariants of the variational orthopath geometry that corresponds  to an arbitrary (pseudo-)Finsler metric in terms of its Cartan torsion and flag curvature using which we will distinguish several classes of (pseudo-)Finsler structures. 

\subsubsection{Orthopath geometry}
\label{sec:orth-geom-as}

In order to solve the equivalence problem for variational orthopath geometries, we first  recall the solution of the equivalence problem for path geometries, as introduced in Definition  \ref{def-pathgeometry}.  Using the results in \cite{Fels-path,Grossman-Thesis}, path geometries  define regular and normal Cartan geometries $(\tau_0\colon\cP\to\cC,\sigma)$  of type $(\mathfrak{sl}(n+1,\RR),P_{12})$ where $P_{12}\subset\mathrm{SL}(n+1,\RR)$ is the parabolic subgroup that is the stabilizer of a flag of a line inside a 2-plane in $\RR^{n+1}$. Let us denote the Cartan connection $\sigma$ as
  \begin{equation}\label{eq:path-geom-cartan-conn-sigma}
  \sigma=
  \def\arraystretch{1.3}
\begin{pmatrix}    
 s &\xi_0& \xi_b\\
\omega^0 &\psi^0_0+s&\gamma_b\\
\omega^a&  \theta^a& \psi^a_b+s\delta^a_b\\
\end{pmatrix}    
\end{equation}
where $s=-\textstyle{\frac{1}{n+1}(\psi^0_0+\psi^a_a)}.$
The tangent bundle of $\cC$ has a 2-step filtration
\begin{equation}
  \label{eq:filtration-path-geometry-TC}
  T^{-1}\cC\subset T^{-2}\cC=T\cC
\end{equation}
where $T^{-1}\cC=\cE\oplus\cV$ and
 \begin{equation}
   \begin{gathered}
     \label{eq:filtration-path-geometry-each-step}     \cE=\tau_*\left(\Ker\{\omega^{n-1},\cdots,\omega^1,\theta^{n-1},\cdots,\theta^1\}\right)=\tau_*\langle\tfrac{\partial}{\partial\omega^0}\rangle,\\ \cV=\tau_*\left(\Ker\{\omega^{n-1},\cdots,\omega^0\}\right)=\tau_*\langle\tfrac{\partial}{\partial\theta^1},\cdots,\tfrac{\partial}{\partial\theta^{n-1}}\rangle,\quad [\cV,\cV]=\cV\\
     T^{-2}\cC=[\cE,\cV]=T\cC.
      \end{gathered}
 \end{equation}
The Cartan curvature was found in \cite[Appendix A]{Grossman-Thesis} and can be expressed as   
\begin{equation}
\label{eq:streqns-pathgeom}
     \exd\sigma+\sigma\w\sigma=   \def\arraystretch{1.35}
\begin{pmatrix}
  0 &  T_0 & T_b+C_b\\
  0 & T^0_0 & T^0_b +C^0_b\\
  0 & T^a_0 & T^a_b+ C^a_b
\end{pmatrix},
\end{equation}
wherein
\begin{equation}
  \label{eq:Components-curv-path}  
T_i=\half T_{ijk}\omega^j\w\omega^k,\quad T^i_j=\half T^i_{jkl}\omega^k\w\omega^l,\quad C_a=C_{abc}\omega^b\w\theta^b,\quad C^j_a=C^j_{abc}\omega^b\w\theta^c
\end{equation}
with the following range of indices \[0\leq i,j,k,l\leq n-1,\qquad 1\leq a,b,c\leq n-1.\] Furthermore, one has the following symmetries 
\begin{equation}
  \label{eq:RC-symmetries}
  T^i_{jik}=0,\quad T_{[ijk]}=0,\quad T^i_{[jkl]}=0,\quad S^i_{ijk}=0,\quad C_{abc}=C_{(abc)},\quad C_{abc}^i=C^i_{(abc)}.
\end{equation}
The harmonic invariants of  parabolic geometries of type $(\mathfrak{sl}(n+1,\RR),P_{12})$ that arise from path geometries are comprised of two components: the harmonic torsion, denoted by $\bT,$ and the harmonic curvature, denoted by $\bC,$ 
given by 
\begin{equation}
  \label{eq:path-harm-inv-representation}
  \bT= s^*T^a_{00b}\ \tfrac{\partial}{\partial s^*\theta^a}\otimes s^*\theta^b\otimes E^{-2},\qquad \bC= s^*C^a_{bcd}\ \tfrac{\partial}{\partial s^*\theta^a}\otimes s^*(\theta^b\circ \theta^c\circ \theta^d) \otimes  V^{\tfrac{1}{n-1}}\otimes E^{-2},
\end{equation}
where $s\colon\cC\to\cP$ and 
\[V=\tfrac{\partial}{\partial s^*\theta^1}\w\cdots\w\tfrac{\partial}{\partial s^*\theta^{n-1}},\qquad E=\tfrac{\partial}{\partial s^*\omega^0}.\]
It is known that if $\bC=0,$ then the path geometry descends to a projective structure on the local leaf space 
\begin{equation}
  \label{eq:leaf-space-M-proj-str}
  \nu\colon\cC\to M:=\cC\slash\cI_{\cV},
\end{equation}
where $\cI_{\cV}$ is the foliation of $\cC$ by the leaves of $\cV$. Conversely, given a projective structure $[\nabla]$ on $M,$ the corresponding path geometry  is defined on $\cC=\PP TM$ the paths of which are the canonical lifts of the geodesics of $[\nabla]$ to $\cC.$

If $\bT=0,$ then the path geometry descends to a Segr\'e structure on the leaf space 
\begin{equation}
  \label{eq:leaf-space-S-segre-str}
  \lambda\colon\cC\to S:=\cC\slash\cI_{\cE},
\end{equation}
where $\cI_{\cE}$ is the foliation of $\cC$ by the integral curves of $\cE$.

In Definition \ref{def-orthopath} orthopath geometries are defined as path geometries  augmented by a conformal structure, $[\bh]\subset\mathrm{Sym}^2(\cV^*),$ given by
\[\bh=\sh_{ab}s^*\theta^a\circ s^*\theta^b,\]
where $s\colon\cC\to\cP$ is a section and  $[\sh_{ab}]$ is  non-degenerate and symmetric  with signature $(p,q).$

There is a natural lift of functions  $\sh_{ab}$ to the principal bundle $\cP$ so that  $[\bh]$ is preserved  along the fibers of $\tau_0\colon\cP\to\cC.$ More precisely, using the action of the structure group of a path geometry on an adapted coframe $(\omega^0,\cdots,\omega^{n-1},\theta^1,\cdots,\theta^{n-1})$ one obtains another adapted coframe  $({}^*\omega^0,\cdots,{}^*\omega^{n-1},{}^*\theta^1,\cdots,{}^*\theta^{n-1})$ where
\[\theta^a\mapsto {}^*\theta^a:=\tfrac{1}{\ba}\bA^a_b\theta^b\quad\mod\quad\{\omega^0,\omega^a\},\]
for  $[\bA^a_b]\in\mathrm{GL}(n-1,\RR)$ and $\ba\in\RR\backslash\{0\}.$ Since the  transformation induced by the action of the structure group should preserve the conformal class $[\bh],$ the induced transformation  on the lifted functions $\sh_{ab},$ which we denote as $\sh_{ab}\mapsto{}^*\sh_{ab},$  is given by 
\begin{equation}
  \label{eq:sh-action-along-fibers}
  [\bh]=[\sh_{ab}\theta^a\circ\theta^b]=[{}^*\sh_{ab}{}^*\theta^a\circ{}^*\theta^b]\Rightarrow \bc\ {}^*\sh_{ab}=\ba^2(\bA^{-1})^c_a(\bA^{-1})^d_b\sh_{cd},
  \end{equation}
for some real-valued function $\bc$ on $\cP.$ By abuse of notation we do not distinguish between the functions  $\sh_{ab}$ and their natural lift  to $\cP.$ Infinitesimally, in terms of the connection forms in \eqref{eq:path-geom-cartan-conn-sigma} the group action \eqref{eq:sh-action-along-fibers} on the functions $\sh_{ab}$ is expressed as
\begin{equation}
  \label{eq:infintmal-action-hab}
  \exd \sh_{ab}-\psi^c_a\sh_{cb}-\psi_b^c\sh_{ca}+2\psi^0_0\sh_{ab}\equiv 2\sh_{ab}\phi\quad\mod\quad \{\omega^i,\theta^a\},
\end{equation}
for some 1-form $\phi$ on $\cP.$ Using \eqref{eq:sh-action-along-fibers}, one obtains a  reduction of the structure bundle $\cP$ to the sub-bundle $\cP_1\subset\cP$ defined by normalizing $\sh_{ab}$ to $\ve_{ab},$ i.e. there exists    
\begin{equation}
  \label{eq:cP_1-def-reduction}
  \iota_1\colon\cP_1\hookrightarrow\cP\quad\mathrm{where}\quad 
\cP_1=\{u\in\cP\ \vline\ \sh_{ab}(u)=\ve_{ab}\},
\end{equation}
in which $[\ve_{ab}]$ is the standard diagonal matrix for an inner product of signature $(p,q).$  As a result of this reduction, it is clear from \eqref{eq:sh-action-along-fibers} that the structure group  $P_{12}=\mathrm{GL}({n-1},\RR)\times\RR^*\ltimes P_+$ is reduced to $\mathrm{CO}(p,q)\times\RR^{*}\ltimes P_+$ where $P_+$ is the nilpotent part of $P_{12}.$ Moreover, from  \eqref{eq:infintmal-action-hab} one obtains that the pull-back of 1-forms in \eqref{eq:path-geom-cartan-conn-sigma} to $\cP_1$ satisfy
  \[\iota_1^*\psi_{ab}+\iota_1^*\psi_{ba}-2\iota_1^*\psi^0_0\ve_{ab}\equiv -2\iota_1^*\phi\ve_{ab},\quad \mod\quad \{\omega^i,\theta^a\},\]
  where $\psi_{ab}=\ve_{ac}\psi^c_b.$ Subsequently, after dropping  $\iota_1^*$ from the expressions above by abuse of notation, one can write
\begin{equation}
  \label{eq:reduction1-phi-sigma}
\psi_{ab}-\psi^0_0\ve_{ab}=\phi_{ab}+\sigma_{ab}-\phi\ve_{ab},
\end{equation}
where
\begin{equation}
  \label{eq:symmetries-sigma-psi-trace-free}
  \sigma_{ab}=\sigma_{ba}\quad  \ve^{ab}\sigma_{ab}=0,\quad \phi_{ab}=-\phi_{ba},\quad\mathrm{and}\quad \sigma_{ab}\equiv 0 \mod\{\omega^i,\theta^a\}.
  \end{equation}
In order for our notation to be consistent with the one we used for  (2,3,5) and (3,6)-distributions, we define the 1-forms $\phi_0$ and $\phi_1$ as
\begin{equation}
  \label{eq:psi00-replaces} 
  \phi_1=\psi^0_0-\textstyle{\frac{2}{n-1}}\psi^a_a,\qquad \phi_0=-\textstyle{\frac{1}{n-1}}\psi^a_a\Rightarrow \phi=\phi_1-\phi_0.
\end{equation}
Let us write
\begin{equation} \label{eq:reduction1-sigma-expressions}
 \sigma_{ab}=\ssA_{abi}\omega^i+\ssB_{abc}\theta^c,\qquad \ssA_{abi}=\ssA_{bai},\quad \ssB_{abc}=\ssB_{bac}.
\end{equation} 
Taking the  exterior derivative of relation  \eqref{eq:reduction1-phi-sigma} followed by taking the appropriate trace, using   structure equations \eqref{eq:streqns-pathgeom} and substitutions \eqref{eq:psi00-replaces}, it follows that  
\begin{equation}\label{eq:dsA-dsB-inf-action}
\begin{aligned}
&\exd \ssB^b_{\ ab}-\textstyle{\frac{n+1}{n-1}\gamma_a}+\ssB^b_{\ ab}(\phi_1-\phi_0)+\ssB^b_{\ cb}\phi^c_a\equiv 0\quad && \mathrm{mod}\quad\{\omega^i,\theta^a\},\\
  &\exd \ssA^b_{\ ab}-\textstyle{\frac{n+1}{n-1}\xi_a}+\ssB^b_{\ ab}\xi_0-\ssA^b_{\ cb}\phi^c_a+\ssA^b_{\ ab} \phi_0-\ssA^c_{a0}\gamma_c\equiv 0\quad &&\mathrm{mod}\quad\{\omega^i,\theta^a\},
\end{aligned}
\end{equation}
where $\ssB^b_{\ ab}=\ve^{bc}\ssB_{bac}, $  $\ssA^b_{\ ab}=\ve^{bc}\ssA_{bac}.$ Equations \eqref{eq:dsA-dsB-inf-action} give the infinitesimal change of $\sA^b_{\  ab}$ and  $\sB^a_{\ ba}$ along the fibers of $\tau_1\colon\cP_1\to\cC,$ where $\tau_1=\tau_0\circ\iota_1.$ We will not give the explicit group actions on $\sA^{b}_{\ ab}$ and $\sB_{\ ab}^b$ since the expressions are long and not illuminating.\footnote{For an account on Cartan's equivalence method, Cartan's reduction procedure and the infinitesimal form of the action of the structure group on the structure invariants we refer the reader to \cite{Gardner-book,Olver-book}.} As a result, one obtains a natural reduction of the structure bundle $\cP_1$ to the sub-bundle
\begin{equation}  
  \label{eq:cG-orthopath-last-reduction}
  \iota_2\colon\cG\hookrightarrow\cP_{1}\quad \mathrm{where}\quad \cG=\{u\in\cP_1\ \vline\ \sA^b_{\ ab}(u)=\sB^b_{\ ab}(u)=0\}.
\end{equation}
Using \eqref{eq:dsA-dsB-inf-action}, one obtains $\iota^*_2\xi_a\equiv 0$ and $\iota^*_2\gamma_a\equiv 0$ modulo $\{\omega^i,\theta^a\}.$ Thus, dropping $\iota_2^*,$ on $\cG$ one has
\begin{equation}
  \label{eq:MNPQ-reduced}
\xi_a=\sM_{ai}\omega^i+\sP_{ab}\theta^b,\qquad \gamma_a=\sN_{ai}\omega^i+\sQ_{ab}\theta^b,
\end{equation}
for some functions $\sM_{ai},\sN_{ai},\sP_{ab},\sQ_{ab}$ on $\cG.$ As a result of this reduction, the structure group is further reduced from $\mathrm{CO}(p,q)\times \RR^*\ltimes P_+$ to $\mathrm{CO}(p,q)\times \RR^*\ltimes \RR.$
Lastly, combined with \eqref{eq:reduction1-sigma-expressions}, the  algebraic properties of the functions $\sA_{abc}$ and $\sB_{abc},$ defined on $\cG,$ are 
\begin{equation}
  \label{eq:AB-tracefree}
\ssA^b_{\ ab}=\ssA^b_{\ ba}=\ssB^b_{\ ab}=\ssB^b_{\ ba}=0,\quad \text{and}\quad \ssA_{abc}=\ssA_{bac},\quad \ssB_{abc}=\ssB_{bac}.
\end{equation}
\begin{remark}
  Note that the pull-back of $\psi$ to the principal bundles $\cP_1$ or $\cG$ does not define a Cartan connection as it does not change equivariantly along the fibers. In the language of Cartan \cite{Gardner-book}, it defines an \emph{$\{e\}$-structure}. In the next section by imposing the  variationality condition  on this  $\{e\}$-structure  a Cartan geometry is obtained.
\end{remark}

\subsubsection{Variational orthopath geometry}
\label{sec:conf-quasi-sympl-1}

On the reduced structure bundle $\tau\colon\cG\to\cC,$ where $\tau=\tau_1\circ\iota_2,$ as obtained in \eqref{eq:cG-orthopath-last-reduction}, by inspection one finds that 
\begin{equation}
  \label{eq:rho-q-symp-orthopath}
\rho=\ve_{ab}\theta^a\w\omega^b  \in\Omega^2(\cG)
\end{equation}
is the unique semi-basic 2-form of maximal rank with respect to the fibration $\tau\colon\cG\to\cC.$ Taking any section $s\colon\cC\to\cG,$ the conformal class of $s^*\rho$ is well-defined on $\cC$  and gives the canonical conformally almost quasi-symplectic structure  $\ell\subset\biw^2T^*\cC.$  Moreover, the line bundle $\cE$ in \eqref{eq:filtration-path-geometry-each-step} is the characteristic direction for $\ell$ and $T^{-1}\cC=\cE\oplus\cV$ is isotropic, as was the case in \ref{sec:fourth-order-scalar} and \ref{sec:pairs-third-order}.

From now on we only consider conformally quasi-symplectic orthopath geometries. Using \eqref{eq:streqns-pathgeom} and the  reduced 1-forms in \eqref{eq:reduction1-phi-sigma}, one obtains
\[
\begin{aligned}
\exd\rho&=\ve_{ab}\exd\theta^a\w\omega^b-\ve_{ab}\theta^a\w\exd\omega^b\\
&=\phi_1\w\rho+\ve_{ab}(-\phi^a_c-\sigma^a_c)\w\theta^c\w\omega^b+\ve_{ab}(\half T^a_{0ij}\omega^i\w\omega^j)\w\omega^b-\ve_{ab}\theta^a\w(-\phi^b_c-\sigma^b_c)\w\omega^c\\
&=\phi_1\w\rho +2\sigma_{ab}\w\omega^a\w\theta^b+T_{0a0b}\omega^a\w\omega^0\w\omega^b+\half T_{0abc}\omega^a\w\omega^b\w\omega^c
\end{aligned}
\]
Since a necessary condition for being conformally quasi-symplectic is $\exd\rho=\phi_1\w\rho,$ one obtains
\begin{gather}
0=\sigma_{ab}\w\omega^a\w\theta^b=\ssA_{ab0}\omega^0\w\omega^a\w\theta^b+\ssA_{abc}\omega^c\w\omega^a\w\theta^b+\ssB_{abc}\theta^c\w\omega^a\w\theta^b,  \label{eq:sigma-A-B-totally-symm}\\
  T_{0a0b}=T_{0b0a},\quad T_{0[abc]}=0,     \label{eq:Fels-Torsion-is-Symmetric-Orthopath}
\end{gather}
where
\begin{equation}
  \label{eq:Torsion-lower-index}
  T_{0a0b}=\ve_{ad}T^d_{00b},\quad  T_{0abc}=\ve_{ad}T^d_{0bc}.\quad
\end{equation}
Using the symmetries in \eqref{eq:AB-tracefree},  it follows from \eqref{eq:sigma-A-B-totally-symm} that
\begin{equation}
  \label{eq:ABR-totallysymmetric}
  \sA_{abc}=\sA_{(abc)},\qquad \sB_{abc}=\sB_{(abc)},\qquad \sB_{ab0}=0.
\end{equation}
Subsequently, one has
\[0=\exd^2\rho=\exd\phi_1\w\rho-\phi_1\w\exd\rho=\exd\phi_1\w\rho.\]
Inserting  reduced 1-forms \eqref{eq:MNPQ-reduced} in structure equations \eqref{eq:streqns-pathgeom}, when $n\geq 2,$  being conformally quasi-symplectic is equivalent to $\exd\phi_1=0$ and $\exd\rho=\phi_1\w\rho$ which implies  
\begin{equation}
  \label{eq:extra-properties}
   \sQ_{ab}=\sQ_{(ab)},\quad \sM_{[ab]}=-\half T^0_{0ab},\quad \sN_{a0}=0,\quad \sP_{ab}=\sN_{ba},\quad \sM_{a0}=T^0_{00a}.
 \end{equation}

Now we are ready to state the solution of the equivalence problem for variational orthopath geometries  $\nu\colon\cC\to M$ where  $\mathrm{dim}\ \cC\geq 5.$
\begin{theorem}\label{thm:var-orthopath-equiv-prob}
  Variational orthopath geometries on $\cC$ of   dimension $2n-1\geq 5$ and $[\bh]\subset \mathrm{Sym^2}(\cV^*)$ with signature $(p,q),$  $p+q=n-1,$ are Cartan geometries $(\tau\colon\cG\to\cC,\psi)$ of type $(\mfk,L)$ where  $\mfk=\fp^{op}\slash\fp_{-2},$ $L=B\times \mathrm{CO}(p,q)$     
  as given in \eqref{k_sopq} and  The Cartan connection can be expressed as
 \begin{equation}\label{eq:var-orthpath-Cartan-conn}
  \psi=
  \def\arraystretch{1.3}
\begin{pmatrix}    
 \phi_0+\mu  &\xi_0&0\\
\omega^0 & \phi_1-\phi_0+\mu&0    \\
\omega^a&  \theta^a& \phi^a_b+\mu\delta^a_b\\
\end{pmatrix}    
\end{equation}
wherein $\mu=-\frac{1}{n+1}\phi_1.$ They are regular and normal with respect to the codifferential operator $\partial^*$ as defined in \eqref{Codifferential}.  The Cartan  curvature  is given by \eqref{eq:orthopath-Cartan-curv-2form} and \eqref{eq:rewrite2-streqns-pathgeom} and the fundamental invariants can be represented as 
\begin{equation}
  \label{eq:fund-inv-orthopath}
  \begin{gathered}
  \bA= s^*(\ssB_{abc} \theta^a\circ \theta^b\circ \theta^c)\otimes V^{\tfrac{2}{n-1}},\qquad \bT= s^*(\sT_{ab}\theta^a\circ\theta^b) \otimes V^{\tfrac{2}{n-1}}\otimes E^{-2},\\ 
  \bN= s^*(\sN_{ab}\theta^a\w\theta^b)\otimes  E^{-1} ,\qquad \bq= s^*\sq\  V^{\tfrac{-2}{n-1}},
\end{gathered}
\end{equation}
where $s\colon\cC\to\cG$ and 
\[V=\tfrac{\partial}{\partial s^*\theta^1}\w\cdots\w\tfrac{\partial}{\partial s^*\theta^{n-1}},\qquad E=\tfrac{\partial}{\partial s^*\omega^0},\qquad \sq=\textstyle{\frac{1}{n-1} }\ve^{ab}\sQ_{ab},\quad [\bh]=[\ve_{ab}s^*\theta^a\circ s^*\theta^b]\]
and the functions $\sA_{abc},\sN_{ab},\sQ_{ab}$ on $\cG$ are given in \eqref{eq:reduction1-sigma-expressions} and \eqref{eq:MNPQ-reduced}. The functions  $\sT_{ab}:=T_{a0b0}$ are the pull-back of the torsion of the corresponding path geometry  to $\cG,$ as given in   \eqref{eq:path-harm-inv-representation}, with its indices lowered using  $\bh,$ as in  \eqref{eq:Torsion-lower-index}.
\end{theorem}

\begin{proof}
To check that the derived structure equations define a regular and normal  Cartan geometry of the type above is a straightforward task and is skipped. To find the fundamental invariants of an orthopath geometry,  recall that the harmonic invariants of a path geometry are the torsion and curvature, as given in  \eqref{eq:path-harm-inv-representation}. Since we derived the Cartan connection for orthopath geometries by reducing a path geometry, its fundamental invariants would be either the harmonic invariants of the path geometry or the quantities that appeared during the reduction procedure, i.e. the quantities $\ssA,\ssB,\sN,\sM,\sP,\sQ$ in  \eqref{eq:reduction1-sigma-expressions} and \eqref{eq:MNPQ-reduced} satisfying properties  \eqref{eq:ABR-totallysymmetric} and \eqref{eq:extra-properties}. 

Note that by \eqref{eq:Fels-Torsion-is-Symmetric-Orthopath} the torsion of the path geometry, $\sT^a_b=T^a_{00b}$ is symmetric and trace-free which gives the bilinear form $\bT.$ 

Next, from the identity $\exd^2\omega_a=0$ one obtains that the vanishing of the 3-forms $\omega^0\w\omega^a\w\theta^b$ and $\omega^a\w\theta^b\w\theta^c,$ respectively, imply
\begin{equation}
  \label{eq:B-and-Q-from-A-orthopath}
  \ssA_{abc}=-\ssB_{abc;0},\qquad \overset{\circ}{\sQ}_{ab}:=\sQ_{ab}-\sq\ve_{ab}=-\textstyle{\frac{2}{n-1}\ve^{cd}\ssB_{abc;\underline{d}}},
  \end{equation}
where $\overset{\circ}{\sQ}_{ab}$ denotes the trace-free part of $\sQ_{ab}$ using the bilinear form $[\bh].$

Furthermore, in the identities $\exd^2\omega^0=0,$ and $\exd^2\phi_0=0$ the vanishing of the 3-forms $\omega^0\w\omega^a\w\theta^b$  give 
\begin{equation}
  \label{eq:N-symm-and-M-symm-from-Q-and-N}
  \sN_{(ab)}=-\half \sQ_{ab;0},\qquad \sM_{(ab)}=\sN_{(ab);0}+T^0_{00(a;\underline b)}, 
\end{equation}
respectively. It only remains to show that the curvature of the path geometry can be derived from the fundamental invariants of the orthopath geometry as given in \eqref{eq:fund-inv-orthopath}.

We first use the identity $\exd^2\omega_a=0$ and the vanishing of 3-forms $\omega^a\w\omega^b\w\theta^c$ to obtain
\begin{equation}
  \label{eq:C-for-orthopath-from-A-N}
  \sC_{abcd}=2\sN_{d[a}\ve_{b]c}-\sN_{[ab]}\ve_{cd}-\sN_{[ac]}\ve_{bd}+\sN_{[bc]}\ve_{ad}+2\ssB_{cd[a;b]}+2\ve^{ef}\ssB_{ce[b}\ssA_{a]df}.
  \end{equation}
Finally, relation \eqref{eq:reduction1-phi-sigma}  gives
\[\psi^a_b-\psi^0_0\delta^a_b=\phi^a_b+\sigma^a_b-(\phi_1-\phi_0)\delta^a_b\Rightarrow \psi^a_b=\phi^a_b+\sigma^a_b-\phi_0\delta^a_b.\]
Using structure equations \eqref{eq:streqns-pathgeom}, one can express the curvature of the initial path geometry by taking the exterior derivative of the expression above and comparing the 2-forms $\omega^c\w\theta^d$ to obtain 
\begin{equation}
  \label{eq:path-geom-curv-derived-from-orthopath}
  C^a_{bcd}=\sC^a_{bcd}+\sA^a_{bd;c}-\sB^a_{bc;\underline d}-\sA^a_{be}\sB^e_{cd}+\sB^a_{be}\sA^e_{cd}-\delta^a_b\sN_{cd}.
\end{equation}
It is a straightforward computation to show that the quantities $\sA_{abc},\sN_{ab},\sT_{ab}$ and $\sq$ constitute a complete set of fundamental invariants whose vanishing implies the variational orthopath geometry is locally equivalent to the flat model. 
\end{proof}
Furthermore,  from the reduction procedure carried out in \ref{sec:orth-geom-as} one obtains the following.
\begin{proposition}\label{prop:path-geom-red-vari-orth-geom}
  Given a path geometry $(\tau_0\colon\cP\to\cC,\sigma)$ with a choice of  $[\bh]\subset\mathrm{Sym}^2(\cV^*),$ the  principal $L$-bundle $\tau\colon\cG\to\cC$ for the corresponding orthopath geometry can be obtained as a sub-bundle $\iota_0\colon\cG\to\cP$  where $\iota_0:=\iota_1\circ\iota_2$ as defined in \eqref{eq:cP_1-def-reduction} and \eqref{eq:cG-orthopath-last-reduction}. Assuming variationality, such orthopath geometries  are regular and normal. 
\end{proposition}
By definition, every variational orthopath geometry arises in the way described in Proposition \ref{prop:path-geom-red-vari-orth-geom}.
\begin{remark}
    It can be shown that when $A_{abc}=0,$ there is an extension functor going from the variational orthopath geometry $(\tau\colon\cG\to\cC,\psi)$ to its corresponding path geometry $(\tau_0\colon\cP\to\cC,\sigma).$  
  \end{remark}

  Now we can state the second lemma needed in the proof of Proposition \ref{lem-canonicalconfqsymp}.
    \begin{lemma}\label{lemm:causal-variationality-orthopath}
      Given a regular and normal path geometry, any compatible conformally quasi-symplectic structure is locally induced by a 2-form $\rho\in\Omega^2(\cP)$ where
      \[\rho=\hat\sh_{ab}\theta^a\w\omega^b\]
      for some functions $\hat\sh_{ab}=\hat\sh_{ba}$ on $\cP$ with the property that $[\hat\sh_{ab}]$ has maximal rank. 
\end{lemma}
 \begin{proof}
The only non-trivial part is to show if $\ell$ has a closed representative then it defines a variational orthopath structure.  Recall from Definition \ref{def_compatible} that if $\ell$ is compatible then, locally, it can be written as $[\rho]$ where
   \[\rho=\hat\sh_{ab}\theta^a\w\omega^b+s_{ab}\omega^a\w\omega^b\in\Omega^2(\cP)\]
   and $\hat\sh_{ab}=\hat\sh_{ba},$  $s_{ab}=-s_{ba}$ with the rank of  $[\sh_{ab}]$ being $n-1.$ Similar to what was done to achieve  \eqref{eq:sh-action-along-fibers}, one can show the functions $\hat\sh_{ab}$ obey the same transformation rule under the action of the structure group of a path geometry. Normalizing $\hat\sh_{ab}$ to $\ve_{ab},$  the structure bundle is reduced to $\iota_1\colon\cP_1\to\cP$ as in \eqref{eq:cP_1-def-reduction} which gives relations \eqref{eq:reduction1-phi-sigma}. We can carry out further reductions \eqref{eq:cG-orthopath-last-reduction} to obtain an principal $L$-bundle $\cG\to\cC,$ but for this lemma we do not need to. Imposing necessary condition for the closedness of a representative $\iota_1^*\rho\in\ell,$ one obtains
      \[
    \begin{aligned}       
      \exd\rho&= \ve_{ab}\exd\theta^a\w\omega^b-\ve_{ab}\theta^a\w\exd\omega^b+\exd s_{ab}\omega^a\w\omega^b+2s_{ab}\exd\omega^a\w\omega^b\\
      &= \phi_1\w\rho- ({\phi_{ab}+\phi_{ba}}+2\sB_{ab0}\omega^0)\w\theta^a\w\omega^b-2s_{ab}\theta^a\w\omega^0\w\omega^b+\cdots
    \end{aligned}
  \]
  where in the last expression there will not be any other term involving the  3-form   $\omega^a\w\omega^0\w\theta^b.$ By the symmetries $\sB_{ab0}=\sB_{(ab)0}$ and $\phi_{ab}=-\phi_{ba},$ as given in \eqref{eq:reduction1-sigma-expressions} and \eqref{eq:symmetries-sigma-psi-trace-free}, one obtains  
  \[\sB_{ab0}=0,\quad s_{ab}=0.\] 
  The vanishing of $\sB_{ab0}$ agrees with what was obtained in \eqref{eq:ABR-totallysymmetric} and the vanishing of $s_{ab}$ proves the lemma since $\rho$ would be the canonical conformally quasi-symplectic structure associated to the reduced path geometry $(\tau_1\colon\cP_1\to\cC,\sigma_1)$ augmented by the conformal bundle metric $[\ve_{ab}\theta^a\circ\theta^b]$. 
   \end{proof}

\subsubsection{Causal geometry}
\label{sec:caus-geom-dimens}

A causal structure of signature $(p+1,q+1),$ where $p+q=n-1,$ on an $(n+1)$-dimensional manifold $\tM$ defines a regular and normal  parabolic geometry of type $(\fso(p+2,q+2),Q)$ where $Q=P_{12}\subset \rO(p+2,q+2)$ is the stabilizer of a flag of a null line inside a null plane in $\RR^{n+3}$ as defined in Diagram 2 before Proposition \ref{prop-qcontactgrading}.  The  Cartan connection is expressed as
\begin{equation} 
  \label{eq:Causal-general}
      \def\arraystretch{1.25}
\tilde\psi=\begin{pmatrix}
  \tphi_0 & \txi_0 & \txi_b & \txi_n & 0\\
  \tomega^0 & \tphi_1-\tphi_0 & \tgamma_b& 0 & -\txi_n\\
  \tomega^a & \ttheta^a & \tphi^a_b & -\ve^{ab}\tgamma_b & -\ve^{ab}\txi_a\\
  \tomega^n & 0 & -\ve_{bc}\ttheta^c & \tphi_0-\tphi_1 & -\txi_0\\
  0 & -\tomega^n & -\ve_{bc}\tomega^c & -\tomega^0 & -\tphi_0 
\end{pmatrix}
\end{equation}
which is $\mathfrak{so}(p+2,q+2)$-valued  with respect to the inner product
\[\langle v,v\rangle =2v_{\infty}v_{n+1}+2v_0v_{n}+\ve^{ab}v_av_b,\]
where $v=[v_{\infty},v_0,\cdots,v_{n+1}]^\intercal.$ The matrix 1-form 
 $[\phi^a_b]$ is $\fso(p,q)$-valued with respect to the inner product given by $[\ve_{ab}].$ In what follows we  lower and raise indices  using $\ve_{ab}$ e.g. $\txi^a:=\ve^{ab}\txi_b.$

The tangent bundle of $\tcC$ has a natural 3-step filtration
\begin{equation}
  \label{eq:filtration-causal-geometry-TC}
  T^{-1}\tcC\subset T^{-2}\tcC\subset T^{-3}\tcC=T\tcC
\end{equation}
where $T^{-1}\tcC=\tcE\oplus\tcV$ and
 \begin{equation}
   \begin{gathered}
     \label{eq:filtration-causal-geometry-each-step}     \tcE=\ttau_*\left(\Ker\{\tomega^{n-1},\cdots,\tomega^1,\ttheta^{n-1},\cdots,\ttheta^1\}\right)=\ttau_*\langle\tfrac{\partial} {\partial\tomega^0}\rangle,\\ \tcV=\ttau_*\left(\Ker\{\tomega^{n-1},\cdots,\tomega^0\}\right)=\ttau_*\langle\tfrac{\partial}{\partial\ttheta^1},\cdots,\tfrac{\partial}{\partial\ttheta^{n-1}}\rangle,\quad [\tcV,\tcV]=\tcV\\
     T^{-2}\tcC=[\tcE,\tcV]=\ttau_*\left(\Ker\{\tomega^{n}\}\right),\quad        T^{-3}\tcC=[\tcV,T^{-2}\tcC]=T\tcC.
      \end{gathered}
    \end{equation}
    Furthermore, the symmetric bilinear form $\tbh\colon\tcV\times\tcV\to T\tcC\slash T^{-2}\tcC,$ defined as $(v,u)\to [v,[u,w]]$ mod $T^{-2}\tcC,$ for some $w\in\Gamma(\tcE),$  has maximal rank on $\tcV$ with  signature $(p,q).$  The conformal structure $[\tbh]\subset\mathrm{Sym}^2(\tcV^*)$  is  independent of the choice of $w\in\Gamma(\tcE)$ and in fact uniquely determined by the quasi-contact structure on $T^{-2}\tcC.$ 
To express the harmonic invariants, define the 2-forms
 \[
   \begin{aligned}
     \tOmega^a&=\exd\tomega^a-\tgamma^a\w\tomega^n+(\tphi^a_b-\tphi_0\delta^a_b)\w\tomega^b-\tomega^0\w\ttheta^a,\\ 
     \tTheta^a&=\exd\ttheta^a-\tomega^n\w\txi^a+(\tphi^a_b+(\tphi_0-\tphi_1)\delta^a_b)\w\ttheta^b-\txi_0\w\tomega^a,\\
  \end{aligned}
\]
which are given by     
\[
  \begin{aligned}
    \tOmega^a&\equiv F^a_{bc}\tomega^b\w\ttheta^c\quad &&\mathrm{mod}\quad \{\tomega^0,\tomega^n\}\\
    \tTheta^a&\equiv W^a_{00b}\tomega^0\w\tomega^b+\half W^a_{0bc}\tomega^b\w\tomega^c\quad&&\mathrm{mod}\quad \{\tomega^n,\ttheta^1,\cdots,\ttheta^{n-1}\}.
      \end{aligned}
    \]
        The harmonic invariants of a causal structure can be presented as
\begin{equation}
  \label{eq:harmonic-inv-causal-represented}
  \begin{gathered}
  \bF= s^*(F_{abc}\ttheta^a\circ\ttheta^b\circ\ttheta^c)\otimes \tV^{\tfrac{2}{n-1}},\qquad \bW=s^*(W_{0a0b}\ttheta^a\circ\ttheta^b) \otimes \tV^{\tfrac{2}{n-1}}\otimes \tE^{-2},
\end{gathered}
\end{equation}
where $s\colon\tcC\to\tcG$ is a section,  $F_{abc}=\ve_{ad}F^d_{bc}=F_{(abc)},W_{0a0b}=\ve_{ac}W^c_{00b}=W_{0b0a}$  and 
\[
  \begin{gathered}    
    \tV=\tfrac{\partial}{\partial s^*\ttheta^1}\w\cdots\w\tfrac{\partial}{\partial s^*\ttheta^{n-1}},\qquad \tE=\tfrac{\partial}{\partial s^*\tomega^0},\\ \ve^{ab}W_{0a0b}=\ve^{ab}F_{abc}=0.
    \end{gathered}
\]
The line bundle $\tcE\subset T\cC$ is the analogue of null geodesic flow on the projectivized null cone bundle in pseudo-Riemannian conformal  structures. Lastly, $\tcC$ is equipped with a degenerate  conformal structure  $[s^*\tg]\subset\mathrm{Sym}^2T^*\tcC$ for any section $s\colon\tcC\to\tcG$ where  
\begin{equation}
  \label{eq:g-conf-str-bilinear-form-on-tC}
  \tg=2\tomega^0\circ\tomega^n+\ve_{ab}\tomega^a\circ\tomega^b\in\Gamma(\mathrm{Sym}^2T^*\tcG),
  \end{equation}
is well-defined. 
If $\bF=0,$ then the causal structure descends to the pseudo-Riemannian conformal structure defined by $[ s^*\tg]$    on the leaf space 
\begin{equation}
  \label{eq:leaf-space-tM-pseudoconf-str}
  \tnu\colon\tcC\to \tM:=\tcC\slash\cI_{\tcV},
\end{equation}
where $\cI_{\tcV}$ is the foliation of $\tcC$ by the leaves of $\tcV$. Conversely, given a pseudo-Riemannian conformal structure, $[\tg],$ on $\tM,$ then $\tcC$ is the sky bundle of $[\tg],$ i.e.
\[\tcC=\{[v]\in\PP T \tM\ \vline\ \tg(v,v)=0\}.\]
The sky bundle $\tcC$ is equipped with a causal structure and the integral curves of $\cE$ are the natural lifts of the null geodesics of $[\tg]$.  In this case  $\bW$ corresponds to an element in the irreducible component of the Weyl curvature module for this conformal structure under the action of $R_0=\mathrm{CO}(p,q)$, which generates the entire Weyl curvature module by the action of  $R_0\cap Q_{+}=\mathrm{CO}(p,q)\cap Q_+,$ where $R_0$ is the Levi factor of $R=P_1$ and $Q_+$ the unipotent radical of $Q=P_{12}$.

If $\bW=0,$ then the causal structure descends to a Lie contact structure  on the leaf space 
\begin{equation}
  \label{eq:leaf-space-tS-Lie-contact-str}
  \tlambda\colon\tcC\to \tS:=\tcC\slash\cI_{\tcE},
\end{equation}
where $\cI_{\tcE}$ is the foliation of $\tcC$ by the integral curves of $\tcE$. The only nonzero harmonic invariant of such Lie contact structures is generated by $\bF.$

Finally for the purposes of this article we will give the following realization result wherein  the parabolic subgroups $P_{12}$ and $P_{123}$ are the ones corresponding to causal structures as defined in \ref{sec-defquasi}. 
\begin{proposition}[Realizability]\label{prop:realization-causal-geometry}
Regular and normal Cartan geometries $(\tcG\to\tcC,\tpsi)$ of type   $(\mathfrak{so}(p+2,q+2), P_{12})$ for $n=p+q+2\geq 5$  and those of type  $(\mathfrak{so}(p+2,q+2),P_{123})$  when $n=p+q+2=4$   for which  $\tcV$ in \eqref{eq:filtration-causal-geometry-each-step} is integrable,  are locally realizable as causal structures  in the sense of Definition \ref{def-causal}
\end{proposition}
\begin{proof}
  Since in such Cartan geometries the line bundle $\tcE\subset T\tcC$ \eqref{eq:filtration-causal-geometry-each-step} is invariantly defined, one can define a map $\tilde\iota\colon\tcC\to \PP T\tM$ as $\tilde\iota(z)=[\tnu_*(\tcE_z)]\in \PP T_{\tnu(z)}\tM$ where $z\in\tcC,$  $[\tnu_*(\tcE_z)]$ denotes the projective class of the line $\tnu_*(\tcE_z)\subset T_{\tnu(z)}\tM$ for $\tcE_z\subset T_z\tcC,$ and $\tM$ is the leaf space \eqref{eq:leaf-space-tM-pseudoconf-str} defined by the quotient map $\tnu\colon\tcC\to\tM$. As a result, for any point $z\in\tcC,$ in a sufficiently small neighborhood $z\in U\subset\tcC,$ the fibers $U\cap\tcC_x$ where $\tcC_x:=\tnu^{-1}(x)$ are embedded  hypersurfaces in $\PP T_x\tM_1$ where $\tM_1:=\tnu(U).$ Moreover, taking any section $s\colon\tcC\to\tcG,$ the 1-forms $\langle s^*\tomega^0,\cdots,s^*\tomega^n,s^*\ttheta^1,\cdots,s^*\ttheta^{n-1}\rangle$ define a  coframe on $\tilde\iota(\tcC)$ via pull-back by $\tilde\iota^{-1}.$  Such coframes are  \emph{5-adapted}  according to  \cite[Section 2.3]{Omid-Sigma} as they satisfy the same normalization conditions however the absorption of torsion during coframe adaptations are done differently. Consequently, following the normalization conditions obtained from the Kostant codifferential operator \eqref{KostantCodif},   it is straightforward to show that the Cartan geometry defined by such causal structures coincides with the initial Cartan geometry one started with.
\end{proof}

\begin{remark}\label{rmk:generalized-causal-geometry-Finsler}
 Proposition \ref{prop:realization-causal-geometry} holds in sufficiently small open sets of $\cC.$ More general definitions, as in  Remark \ref{rmk:general-defin-causal-str}, has some benefits in certain  global considerations of causal structures which will not be discussed here. Such generalization in the definition of causal structures and their realizability has the same spirit as the definition of generalized Finsler structures and their realization in the form of certain Pfaffian systems as given in \cite[Proposition 1]{Bryant-Finsler}.
\end{remark}

\subsubsection{Quasi-contactification and various curvature conditions}
\label{sec:general-causal-quasi-cont}

As in \ref{sec:235-4th-ODE-quasi-cont} and \ref{sec:36-quasi-cont-from-pair-3rd-ODE}, we explain the one-to-one correspondence between  causal geometries with an infinitesimal symmetry and variational orthopath  geometries and relate their Cartan connections. Consequently, we  relate their fundamental invariants and geometrically interpret   several curvature conditions. 

\begin{proof}[Proof of Proposition \ref{prop-relharmonics}; last part]
As in the first and second parts of the proof,   the coframing $\hat\psi=(\tomega^n,\pi^*\psi)$ defines a Cartan connection for a Cartan geometry $(\hat\tau\colon\pi^*\cG\to\tcC,\hat\psi)$ of type $(\fp^{op},L)$ where $\fp^{op}\subset\mathfrak{so}(p+2,q+2)$ is the opposite contact parabolic subgroup and $L=B\times\mathrm{CO}(p,q)$ with $B\subset \mathrm{SL}(2,\RR)$ being the Borel subgroup. The 2-step filtration \eqref{eq:filtration-path-geometry-TC} and the quasi-contact form $\tomega^n$ imply that $T\tcC$ has a 3-step filtration as in \eqref{eq:filtration-causal-geometry-TC}. Hence, one obtained an associates regular and normal Cartan geometry $(\ttau\colon\tcG\to\tcC,\tpsi)$ of type $(\mathfrak{so}(p+2,q+2),Q)$ which corresponds to a causal structure.  In particular, in terms of the matrix forms \eqref{eq:var-orthpath-Cartan-conn} and \eqref{eq:Causal-general}, the Cartan connection $\pr_\fk\circ\iota^*\tpsi$ and $\hat\pi^*\psi$ are related as follows
\begin{equation}\label{eq:Cartan-conn-modif-orthopath-to-causal}
  \begin{gathered}
    \tilde\omega^a=\omega^a,\quad\tilde\theta^a=\theta^a,\quad \tilde\omega^0=\omega^0+\textstyle{\frac{1}{2}}\sq\tilde\omega^n, \quad   \tilde\phi_1 = \phi_1\quad \tilde\phi_0 = \phi_0-\textstyle{\frac{n-1}{4n}} \sq_{;0}\tilde\omega^n,\\
    \tilde\phi_{ab} =\phi_{ab}+ \left(\textstyle{\frac{\ve^{ce}\ve^{df}}{n+2}(\sA_{bcd}\sB_{aef}-\sA_{acd}\sB_{bef})-\sN_{[ab]}}\right)\tilde\omega^n\\
 \tilde\xi_0 = \xi_0+(\textstyle{\frac{\ve^{ac}\ve^{bd}}{2n(n+1)} \textstyle{\sT_{ab;\uc\underline d}}-\tfrac{n-1}{4n}\sq_{;00})}\tilde\omega^n,
\end{gathered}
\end{equation}
\end{proof}
To state the main corollary of this section,  recall that the Lie algebra $\fk$  has the grading
 $\fk=\fk_{-2}\oplus\cdots\oplus\fk_{1}$ where $\fk_{-1}=\fv\oplus\fe.$
Now define the Lie algebra $\fx\subset\fk$ as 
\begin{equation}
  \label{eq:LieAlg-fx-orthopath}  
  \fx=\fv\oplus\fk_{0}\oplus\fk_{1}.
\end{equation}
The Lie group $X$ whose Lie algebra is $\fx$ is given by
\begin{equation}
  \label{eq:X-corr-causal}
  X=(B\times \mathrm{CO}(p,q))\ltimes (l\otimes \mathbb{R}^{p,q}),
\end{equation}
where $B\subset \mathrm{SL}(2,\mathbb{R})$ is the Borel subalgebra stabilizing a line $l\subset\mathbb{R}^2$.
\begin{corollary} \label{cor:causal-from-orthopath-various-curv-cond}
Let $(\ttau\colon\tcG\to\tcC,\tpsi)$ be a causal structure of signature $(p+1,q+1)$ with an infinitesimal symmetry and $(\tau\colon\cG\to\cC,\psi)$ be the variational orthopath geometry of signature $(p,q)$ on the leaf space of the integral curves of the infinitesimal symmetry via the quotient map $\pi\colon\tcC\to\cC.$ Then one has
\begin{equation}
  \label{eq:FW-AT}
  \bF=\pi^*\bA,\quad \bW=\pi^*\bT
  \end{equation}
  where $\bF$ and $\bW$ are  given in \eqref{eq:harmonic-inv-causal-represented} and $\bA$ and $\bT$ are given in \eqref{eq:fund-inv-orthopath}. Furthermore, if $\bA=0$ then for the naturally induced pseudo-conformal structure on the  leaf space $\tM,$ as defined in  \eqref{eq:leaf-space-tM-pseudoconf-str}, the following holds.
  \begin{enumerate}
  \item   There is a one-to-one correspondence between pseudo-conformal structures with a null conformal Killing field  and variational orthopath geometries for which $\bA$ and $\bq$ vanish.  
      \item   There is a one-to-one correspondence between pseudo-conformal structures with a  null conformal Killing field that induces an orthogonal  hypersurface foliation   and variational orthopath geometries for which $\bA,\bN$ and $\bq$ vanish. In this case, the initial path geometry corresponds to a projective structure on $M,$ as defined in \eqref{eq:leaf-space-M-proj-str}. 
      \item A variational orthopath geometry defines a Cartan geometry $(\tau_1\colon\cG\to M,\psi)$ of type $(\fk,X)$  if and only if the invariants  $\bA,\bN,\bq,$ vanish and the first coframe derivatives $\sT_{ab;\uc}\ve^{bc}$ are zero for all $1\leq a\leq n.$ In this case the Cartan holonomy of the corresponding pseudo-conformal structure is  a proper subgroup of the  contact parabolic subgroup $P\subset \mathrm{SO}(p+2,q+2).$
      \item  A variational orthopath geometry defines a parabolic conformally symplectic structure of type $B_n$ or $D_n$  on the leaf space  $S$ given in \eqref{eq:leaf-space-S-segre-str},
    if and only if $\bT$ vanishes.  
  \item Variational orthopath geometries whose quasi-contactification is the flat pseudo-conformal structures in dimension $n+1$ are in one-to-one correspondence with the $\lfloor\half({n+1})\rfloor$-parameter family of $(2n-2)$-dimensional torsion-free  $\mathrm{SL}(2,\RR)\times \mathrm{CO}(p,q)$-structures. 
  \end{enumerate}
\end{corollary}
\begin{proof}
  As in Corollary \ref{cor:36-from-pair-3rd-ODE-various-curv-conditions} and \ref{cor:235-from-4th-ODE-various-curv-conditions}, the first claim is shown by using expression  \eqref{eq:Cartan-conn-modif-orthopath-to-causal}  to form $\tpsi,$ as given in \eqref{eq:Causal-general}, and use structure equations \eqref{eq:orthopath-Cartan-curv-2form} and \eqref{eq:rewrite2-streqns-pathgeom} to compute the harmonic invariants $\bF$ and $\bW$ for the causal structure.

  Part \emph{(1)} is shown by using \eqref{eq:Cartan-conn-modif-orthopath-to-causal}  to express the   bilinear form $\iota^*\tg$,  defined in  \eqref{eq:g-conf-str-bilinear-form-on-tC}, as
  \begin{equation}
    \label{eq:g0-pseudo-conformal-general}
    \iota^*\tg=2\omega^0\circ\omega^n+\ve_{ab}\omega^a\circ\omega^b+\sq\omega^n\circ\omega^n\in \Gamma(\mathrm{Sym}^2T^*\pi^*\cG).
      \end{equation}
       As discussed in the proof of the first part of Proposition \ref{prop-relharmonics} in \ref{sec:235-4th-ODE-quasi-cont}, on $\pi^*\cG$ the infinitesimal symmetry can be taken to be $\tfrac{\partial}{\partial\tomega^n}.$ Thus,  the vanishing of $\bq$ is equivalent  to the nullity of the infinitesimal symmetry. 
  
       Part \emph{(2)} follows from the fact that by structure equations \eqref{eq:rewrite2-streqns-pathgeom}, the Pfaffian system $I=\{\omega^0\}$ is integrable if and only if $\bA=\bN=\bq=0.$ Note that by \eqref{eq:g0-pseudo-conformal-general} $\tomega^0=\pi^*\omega^0$ is orthogonal  to the conformal Killing field $\tfrac{\partial}{\partial\omega^n}$ with respect to conformal structure induced by  $\iota^*\tg.$ Furthermore, in this case it can be checked that all 3rd coframe derivatives of $\sT_{ab}$ with respect to $\theta^a$'s vanish, i.e. $\sT_{ab;\underline{cde}}=0.$

       Part \emph{(3)} follows by first observing that among  the obstructions for the vanishing of  Cartan curvature \eqref{eq:rewrite2-streqns-pathgeom} upon insertion of $\tfrac{\partial}{\partial\theta^a}$'s are $\bA,\bN,\bq.$ Using the differential relations given in the proof of Theorem \ref{thm:var-orthopath-equiv-prob}, assuming that $\bA,\bN,\bq$ are zero, the only obstructions are the functions  $\sM_{ab}$ and $\sM_{a0}$ in \eqref{eq:rewrite2-streqns-pathgeom} which vanish if and only if $\sT_{ab;\underline{c}}\ve^{bc}=0$ for all $1\leq a\leq n.$
In this case the Cartan curvature for the orthopath geometry  is horizontal with respect to the fibration $\tau_1\colon\cG\to M,$ where $M$ is defined in \eqref{eq:leaf-space-M-proj-str} via the quotient map $\nu,$ and $\tau_1=\nu\circ\tau$. Consequently,  one obtains a Cartan geometry  $(\tau_1\colon\cG\to M,\psi)$ of type $(\fk,X).$  Computing $\iota^*\tpsi$ in this case, relations \eqref{eq:Cartan-conn-modif-orthopath-to-causal} imply $\pr_{\fk}\circ\iota^*\tpsi=\hat\pi^*\psi$ and all entries of  $\iota^*\tpsi$  that are not contained in $\pr_{\fk}\circ\iota^*\tpsi$ are zero. Furthermore, in this case all 2nd coframe derivatives of $\sT_{ab}$ with respect to $\theta^a$'s vanish, i.e. $\sT_{ab;\underline{cd}}=0.$

Part \emph{(4)} can be shown in two ways. One way is to use the Bianchi identities and show that the vanishing of $\bT$ implies the vanishing of certain parts of the Cartan curvature for $\psi$ which includes the 2-forms that involve $\omega^0,$ e.g. the quantities $\sM_{0a},\sC_{ab0d},\sT_{ab0b},\sT_{00b}.$ Subsequently, it can be checked that the orthopath geometry descends to a Cartan geometry on the leaf space of the paths, $S,$  and that such Cartan geometry corresponds to a parabolic conformally symplectic structure of type $B_n$ and $D_n,$ as defined in \cite{CS-cont1}. Alternatively, as was mentioned in Theorem \ref{thm:var-orthopath-equiv-prob}, $\bT$ is the torsion of the initial path geometry which was augmented with a bilinear form $[\bh],$ as explained in \ref{sec:orth-geom-as}. The vanishing of $\bT$ implies that the initial path geometry descends to an integrable Segr\'e structure on $S.$ Furthermore, the augmentation with the bilinear form $[\bh]$ results in a reduction of the Segre structure to a parabolic conformally symplectic (PCS) structure by following the exact same steps explained in \ref{sec:orth-geom-as} and \ref{sec:conf-quasi-sympl-1}.

Part \emph{(5)} follows from \eqref{eq:FW-AT} combined with part (4) that such structures descend to  PCS structures on $S.$ The vanishing of $\bA$ implies that $\sB_{abc}$ and $\overset{\circ}{\sQ}_{ac}$ are zero due to relations \eqref{eq:B-and-Q-from-A-orthopath}. As was shown in part (4),  $\bT=0$ implies the vanishing of $T^{a}_{bc}$ and therefore, by structure equations \eqref{eq:rewrite2-streqns-pathgeom} the resulting PCS structure on $S$ gives torsion-free $\mathrm{SL}(2,\RR)\times\mathrm{CO}(p,q)$-structures. Since such  structures define special symplectic connections $\nabla,$ by \cite{CS-special} they depend on $\lfloor\half(n+1)\rfloor$ parameters. Unlike  part (4) in Corollary \ref{cor:235-from-4th-ODE-various-curv-conditions}, straightforward computation shows that the Ricci curvature of such torsion-free $\nabla$ is always symmetric. 
\end{proof}

\begin{remark}
  In part (1) of Corollary \ref{cor:causal-from-orthopath-various-curv-cond} one could consider the class of variational orthopath geometries with $\bA\neq 0$ and $\bq=0,$  the quasi-contactification of which results in  causal structures with an infinitesimal symmetry that is  null with respect to the  bilinear form $[\tg]$  given  in  \eqref{eq:g0-pseudo-conformal-general}. Moreover, unlike Corollaries \ref{cor:235-from-4th-ODE-various-curv-conditions} and \ref{cor:36-from-pair-3rd-ODE-various-curv-conditions},  the descend of a variational orthopath geometry from $\cC$ to $M$ implies holonomy reduction of the corresponding quasi-contact cone structure.    We also point out that by the discussion in \ref{sec:orth-geom-as}, the vanishing of the curvature    $\bC$ implies that the path geometry descends to  a projective structure on  $M.$ Using expressions \eqref{eq:C-for-orthopath-from-A-N} and \eqref{eq:path-geom-curv-derived-from-orthopath}, it follows that the vanishing of $\bA$ and $\bN$ imply the vanishing of $\bC.$ We note that in part (2) of Corollary \ref{cor:causal-from-orthopath-various-curv-cond}, since the 1-form $\omega^0$ is defined up to a scale and is integrable, such variational orthopath geometries naturally define a so-called \emph{fiber equivalence class} of systems of second order ODEs. Unlike systems of ODEs, there has been an extensive study of scalar second order ODEs under the fiber equivalence relation e.g.  see \cite{HK-ODE}. Lastly, in the case of 4-dimensional conformal structures of neutral signature with a conformal Killing field the relation between our Cartan geometric approach and the results of \cite{JonesTod,DunajskiWest} remains to be examined.

 \end{remark}

\subsubsection{Geometry of (pseudo-)Finsler structures under divergence equivalence}
\label{sec:div-equiv-pseudo-Finsler-as-var-orthopath}

In this section we show how variationality of an orthotpath structure implies Finsler metrizability. Consequently, we define the divergence equivalence relation among (pseudo-)Finsler metrics which we show to be equivalent, as a geometric structure, to variational  orthopath geometry.

(Pseudo-)Finsler manifolds are geometric structures that naturally arise from  the classical problem of first order  calculus of variation for $k\geq 1$ functions of one variable. As was discussed in \ref{sec:spec-conf-quasi},  this problem is to find a function $u\colon (p,q)\to \RR^{n-1},n\geq 2,$  that  extremizes  the integral
\[I[u]=\int^p_q L(x,u^a(x),(u^a(x))')\exd x,\]
where $L(x,y^1,\cdots,y^{n-1},p^1,\cdots,p^{n-1})$ is a function of $2n-1$ variables.  Finslerian viewpoint  can be useful in the study of certain   global aspects of  this problem.

More precisely, when $L\neq 0,$ one can naturally associate a (pseudo)-Finsler structure to such a problem by restricting to an open subset $\cC\subset J^1(\R,\RR^{n-1})$ of the domain of $L$ where $L\neq 0.$ The map
\[(x,y^a,p^a)\to \tfrac{1}{L(x,y^a,p^a)}\left(\tfrac{\partial}{\partial x}+p^a\tfrac{\partial}{\partial y^a}\right)\]
defines an immersion $\cC\to T M,$ where $M$ is the projection of $\cC$ to $J^0(\RR,\RR^{n-1})\cong \RR^n.$ Consequently, if the vertical Hessian matrix $[\frac{\partial^2 L}{\partial p^a\partial p^b}]_{a,b=1}^{n-1}$ is non-degenerate then  $\cC$ (micro-)locally defines the bundle of unit  vectors with respect to a (pseudo-)Finsler metric.

The observation above links this section to Theorem \ref{prop-variational}. Unlike the notation used  in Theorem \ref{prop-variational}, we denote the jet coordinates as $(x,y^a,p^a)$ for simplicity. The base manifold $M$ is equipped with the 1-form $\lambda=L\exd x$ and the contact system $\langle\exd y^a-p^a\exd x\rangle_{a=1}^{n-1}.$ We now proceed to explain Theorem \ref{prop-variational} in the language of Finsler structures. We start by showing how a variational orthopath structure defines an abstract geometric structure which we refer to as a  generalized (pseudo-)Finsler structure.

Given a variational orthopath geometry, take a closed representative $\rho_0\in\Gamma(\ell)$. Locally,  let $\alpha^0\in \Omega^1(\cC)$ be a primitive 1-form for $\rho_0,$ i.e.
\[\rho_0=\exd\alpha^0,\]
with the property that $\alpha^0$ is semi-basic with respect to the fibration $\cC\to M.$ The existence of such a primitive  is guaranteed by \cite[Lemma 2.7]{AT-Book}. The argument invokes Poincare Lemma for the operator $\exd_V$ on $J^1(\RR,\RR^{n-1}).$ 

Let $\hat w\in\Gamma(\cE)$ be a section of $\cE.$ Since $\cE$ is the characteristic direction of $\rho_0,$ one can always find  a modification of the form $\alpha^0\to\alpha^0+\exd t$ such that
\[\hat w\im\alpha^0=f\neq 0,\]
where $\exd t$ is a total differential along the integral curves of  $\cE.$ Defining $w={\hat w}/{f}\in\Gamma(\cE),$ one has
\[w\im\alpha^0=1.\]
The resulting $\alpha^0$ is a contact form on $\cC$ with $w$ as its Reeb vector field. Choose a section  $s\colon\cC\to\cG$  such that $w\im s^*\omega^0=1.$ Then  $(\alpha^0,s^*\omega^1,\cdots,s^*\omega^{n-1},s^*\theta^1,\cdots,s^*\theta^{n-1})$ is a coframe on $\cC$ with the property that  $(\alpha^0,s^*\omega^1,\cdots,s^*\omega^{n-1})$ are semi-basic with respect to the fibration $\nu\colon\cC\to M$ defined in \eqref{eq:leaf-space-M-proj-str}. Since $v\im\alpha^0=v\im\omega^0=1$ and both are semi-basic, it follows that
\begin{equation}
  \label{eq:alpha0-omega0-relation-semibasic}
  \alpha^0=\omega^0+x_a\omega^a
\end{equation}
for some functions $x_a$'s on $\cC,$ where we have dropped $s^*$ by abuse of notation.
 
Using structure equations \eqref{eq:streqns-pathgeom} and the definition of $\ell$ in \eqref{eq:rho-q-symp-orthopath}, the following holds
\begin{equation}
  \label{eq:gen-finsler-metric-characteristic-property}
  \begin{aligned}   
  \exd\alpha^0&=\ve_{ab}\theta^a\w\omega^b,\\
  \exd\omega^a&\equiv \alpha^0\w\theta^a\quad\mathrm{mod}\quad \{\omega^1,\cdots,\omega^{n-1}\}.
    \end{aligned} 
\end{equation}
In case $\ve_{ab}=\delta_{ab}$ it is immediate from  \eqref{eq:gen-finsler-metric-characteristic-property} that the triple $(\cC,\alpha^0,M)$ is a generalized Finsler structure as defined in \cite[Section 2.2.11]{Bryant-Finsler}. In the same spirit the following definition follows naturally.
\begin{definition}\label{def:gen-pseudo-finsl-str}  
  A generalized (pseudo-)Finsler structure of signature $(p+1,q)$, denoted as a triple $(\cC,\alpha^0,M),$ is a fibration $\nu\colon\cC\to M$ such that $\cC$ and $M$ have dimensions $2n-1$ and $n,$ respectively,  $\alpha^0$ is a contact form on $\cC,$ referred to as the Hilbert form, and there is a coframing,  $(\alpha^0,\omega^a,\theta^a)_{a=1}^{n-1},$ with the property that $(\alpha^0,\omega^a)_{a=1}^{n-1}$ are semi-basic with respect to $\cC\to M,$ and \eqref{eq:gen-finsler-metric-characteristic-property} holds wherein $\ve_{ab}$ has signature $(p,q)$ and $p+q=n-1.$
\end{definition}
Following  \cite{Bryant-Finsler}, generalized (pseudo)-Finsler structures can be interpreted as a Finslerian generalization of (pseudo-)Riemannian metrics. In particular, the bundle of unit vectors in Riemannian geometry is replaced by the so-called indicatrix bundle which, locally, can be viewed as the bundle of unit vectors with respect to a Finsler norm. In the case of strictly pseudo-Finsler structure the indicatrix bundle locally arises as  the bundle of unit  vectors with respect to a pseudo-Finsler norm. The integral curves of the Reeb vector field of $\alpha^0$ project to the  \emph{geodesics} of the generalized (pseudo-)Finsler structure on $M$ along tangent vectors with positive norm.   
In this article we will not elaborate further on various interpretations and definitions of (pseudo-)Finsler structures.

\begin{remark}
Note that here we restrict to the bundle of unit tangent vectors, which depending on one's convention can be time-like or space-like in the case of strictly pseudo-Finsler metrics. For instance,   one may consider pseudo-Finsler structures of signature $(p+1,q),q\geq 1$ whose indicatrix bundle locally represent space-like vectors of unit length with respect to a pseudo-Finsler norm. Thus, unit space-like vectors for a pseudo-Finsler norm of signature $(p+1,q)$ coincide with unit time-like vectors with respect to  a pseudo-Finsler norm of signature $(q,p+1),$ when $q\geq 1.$ As a result, without having any convention on being time-like or space-like, we take the bundle of unit tangent vectors as the local model for generalized (pseudo-)Finsler structures. 
\end{remark}

For what follows in this section and the next one we need to recall Cartan's solution of the equivalence problem for generalized (pseudo-)Finsler structures as defined in Definition \ref{def:gen-pseudo-finsl-str}.

\begin{theorem}[Cartan \cite{Cartan-Finsler}]\label{thm:General-FW-solut-equiv-probl-1}
  Every (pseudo-)Finsler structure $(\cC,\alpha^0,M)$ of signature $(p+1,q)$ where $\cC$ is $(2n-1)$-dimensional and $p+q=n-1$ defines a Cartan geometry $(\sigma\colon\cF\to\cC,\varphi)$ of type $(\fg,H)$ where  $\fg=\RR^{n}\rtimes\mathfrak{so}(p+1,q)$ and $H=\mathrm{SO}(p,q)\subset \mathrm{SO}(p+1,q).$ The Cartan connection and its Cartan curvature can be expressed as
  \begin{equation}
    \label{eq:CartanConn-FW-1}  
 \def\arraystretch{1.3}
\varphi=    \begin{pmatrix}    
 0 & 0& 0\\
\alpha^0 &0&-\beta_b\\
\alpha^a&  \beta^a& \psi^a{}_b\\
\end{pmatrix}    \qquad {\boldsymbol\varphi}:=\exd\varphi+\varphi\w\varphi= \begin{pmatrix}    
 0 & 0& 0\\
0 &0&-{\bbeta}_b\\
{\balpha}^a&  {\bbeta}^b& \Psi^a{}_b\\
\end{pmatrix} 
    \end{equation} 
    where  $\beta_a=\ve_{ab}\beta^b,$ and
    \begin{subequations}
      \label{eq:general-FW-streqns-0}
      \begin{align}
        \label{eq:general-FW-streq-1}
        \balpha^a&= I^a{}_{bc}\alpha^b\w\beta^c, \\
        \label{eq:general-FW-streq-2}
        \bbeta^a&=  R^a_{\ b}\alpha^0\w\alpha^b+\half R^a_{\ bc}\alpha^b\w\alpha^c+ J^a_{\ bc}\alpha^b\w\beta^c,\\
          \label{eq:general-FW-streq-3}
        \Psi^a{}_b&= R^a_{\ b0c}\alpha^0\w\alpha^c+ \half R^a_{\ bcd}\alpha^c\w\alpha^d+K^a_{\ bcd}\alpha^c\w\beta^d+\half S^a_{\ bcd}\beta^c\w\beta^d.
      \end{align}
    \end{subequations}
Using $\ve_{ab}$ to lower and raise indices, the algebraic properties  
    \begin{equation}
      \label{eq:symmetries-FW-invariants}
      \begin{gathered}
        I_{abc}=I_{(abc)},\qquad J_{abc}=J_{(abc)}=\tfrac{\partial}{\partial\alpha^0}I_{abc},\quad R_{ab}=R_{(ab)},\quad R_{[abc]}=0,
      \end{gathered}
    \end{equation}
 are satisfied.   The fundamental invariants, whose vanishing imply that the Finsler structure is locally equivalent to the flat model $G\slash H,$ are 
    \begin{equation}
      \label{eq:FW-fundamental invariants}
      \bI=I_{abc}\beta^a\circ\beta^b\circ\beta^c,\quad\bR=R_{ab}\beta^a\circ\beta^b.  
    \end{equation}
  \end{theorem} 
The theorem above describes Cartan's original solution to the equivalence problem of Finsler structures \cite{Cartan-Finsler} in the language of Cartan geometry. A more modern exposition of  Cartan's solution can be found in \cite[Section 2]{Bryant-Finsler} which straightforwardly extends to the strictly pseudo-Finslerian case.  There are other choices of linear connections on $\cC$ that have been used in Finsler geometry, most notably that of Chern \cite{Chern-Finsler}. 
  \begin{remark}\label{rmk:Finsler-equiv-problem}    
  Using Theorem \ref{thm:General-FW-solut-equiv-probl-1},  it is straightforward to show that the degenerate bilinear form
\begin{equation}
  \label{eq:Finsler-biliner-form}
    \bh=\ve_{ab}s^*\beta^a\circ s^*\beta^b\in\mathrm{Sym}^2(\cV^*)
  \end{equation}
 is well-defined for any section $s\colon\cC\to\cF,$ where $\cV$ is the vertical tangent bundle for the fibration $\nu\colon\cC\to M.$  Locally, it can be shown \cite{Bryant-Finsler} that each fiber $\cC_x:=\nu^{-1}(x)$ can be immersed as centro-affine hypersurface in $T_xM$ whose second fundamental form has signature $(p,q).$ Such hypersurfaces are classically interpreted as the unit   tangent vectors with respect to a (pseudo)-Finsler metric. In this generalized setting such interpretation holds ``microlocally'' on $M$, i.e. in sufficiently small open neighborhoods of $\cC.$  
The invariant $\bI $ at each point $(x;y)\in\cC$ can be interpreted as the \emph{centro-affine} invariant of the immersed hypersurface $\cC_x\subset T_xM$ evaluated at $y\in \cC_x.$ The vanishing of $\bI$ implies that the (pseudo-)Finsler metric descends to the (pseudo-)Riemannian metric
\begin{equation}
  \label{eq:g-bilinear-form-Finsler-R}
  g=(\alpha^0)^2+\ve_{ab}\alpha^a\circ\alpha^b
  \end{equation}
on $M,$  in which case Cartan connection $\varphi$ coincides with the Levi-Civita connection for $g.$ The invariant $\bR$ is referred to as the \emph{flag curvature} of the (pseudo-)Finsler structure. 
  \end{remark} 
  \begin{remark}\label{rmk:Finsler-equiv-prob-flat-model}
    In  Theorem \ref{thm:General-FW-solut-equiv-probl-1} the flat model is taken to be $G\slash H$ which is equivalent to the bundle of unit  vectors  in $\RR^{p+1,q}$ with its standard metric of sectional curvature zero, i.e.  $\cC=\{v\in\RR^{p+1,q}\,\vline\,\langle v,v\rangle=1\}.$
Alternatively,     as in the case of (pseudo-)Riemannian geometry, one can take $G$ to be $\mathrm{SO}(p+2,q)$ or $\mathrm{SO}(p+1,q+1)$ in which case $G\slash H$ gives  the bundle of unit   vectors on $\SSS^{p+1,q}$ and $\HH^{p+1,q}$ with their standard metrics of $+1$ and $-1$ sectional curvature, respectively. 
  \end{remark}

Up until now we have shown that for any variational orthopath geometry the paths arise  as  integral curves of the  geodesic flow for a generalized (pseudo-)Finsler structure. The choice of the (pseudo-)Finsler structure is determined by the choice of a primitive 1-form $\alpha^0.$  The freedom in choosing $\alpha^0,$ as a semi-basic 1-form with respect to $\nu\colon\cC\to M,$ is scaling by a constant and a total derivative on $M,$ i.e. transformation of the form
\begin{equation}
  \label{eq:div-equiv-transformation}
  \alpha^0\rightarrow c\alpha^0+\exd f
  \end{equation}
for a constant $c\in\RR^*$ and  $f\in C^{\infty}(M,\RR).$

In the language of \emph{Griffiths' formalism} \cite{Griffiths-EDS}, using the notation from \cite{Bryant-VarCal}, so far we have shown that the paths of a variational orthopath geometry are extremals of the \emph{variational problem} $(\cC,\{0\},\alpha^0),$ since they are the integral curves of the characteristic direction of  $\rho_0=\exd\alpha^0.$ As defined in \cite[Section 6]{Bryant-VarCal}, the admissible transformations  \eqref{eq:div-equiv-transformation} correspond to solving the variational problem $(\cC,\{0\},\alpha^0)$ under divergence equivalence. This observation leads us to give the following definition.
\begin{definition}\label{def:geom-pseudo-finsl-div-equiv}
  Two generalized (pseudo-)Finsler structures $(\cC_1,\alpha^0,M_1)$ and $(\cC_2,\beta^0,M_2)$ are locally divergence equivalent at $x_i\in\cC_i,i\in\{1,2\}$ if there exists a diffeomorphism $g\colon U_1\to U_2$ for open subsets $U_i\subset\cC_i, i\in\{1,2\}$ such that $x_2=g(x_1)$ and 
  \[  g^*\beta^0= c\alpha^0+\exd f\]
  for $c\in\RR^*$ and $f\in C^{\infty}(\nu(U_1),\RR),$ where $x_i\in U_i,i\in\{1,2\}$ and $\nu\colon\cC_1\to M_1.$
\end{definition}
Now we can state the  following interpretation of variationality for an orthopath geometry. 
\begin{proposition}\label{prop:var-orthopath-is-div-equiv-pseudo-finsler}
    There is a  one-to-one correspondence between variational orthopath geometries and divergence equivalence classes of generalized (pseudo-)Finsler structures.
\end{proposition}
\begin{proof}
  By the discussion above it remains to show that every divergence equivalence class of generalized (pseudo-)Finsler structure defines a variational orthopath geometry. In order to do so we will make use of Cartan's solution for the equivalence problem of generalized (pseudo-)Finsler structures as presented in Theorem \ref{thm:General-FW-solut-equiv-probl-1}.   Let $\alpha^0$ be the  Hilbert form of a generalized (pseudo-)Finsler structure, then the conformal class $[\rho_0],$ where $\rho_0=\exd\alpha^0,$ defines a conformally quasi-symplectic structure on $\cC$ and is uniquely defined for the divergence equivalence class of the (pseudo-)Finsler structure.  Moreover, $\cC$ is equipped with the filtration \eqref{eq:filtration-path-geometry-TC} where $\cE$ is spanned by the Reeb vector field and $\cV$ is the vertical distribution for $\cC\to M.$  Subsequently, it is straightforward to check that the symmetric bilinear form $\ve_{ab}\theta^a\circ\theta^b\in\mathrm{Sym}^2(\cV^*)$ is well-defined up to scale for the divergence equivalence class and is proportional to the canonical  bilinear form $\bh$  of any (pseudo-)Finsler structure in the divergence equivalence class, as defined in Remark \ref{rmk:Finsler-equiv-problem}. Lastly, the conformally quasi-symplectic structure  $[\rho_0]$ is by construction compatible with $\bh,$ and by Lemma \ref{lemm:causal-variationality-orthopath} is canonically induced on $\cC.$ As a result,   $\cC$ carries a natural  variational orthopath structure.
\end{proof}

\subsubsection{Examples and local generality}
\label{sec:general-causal-local-form-invar}
In this section we use Proposition \ref{prop:var-orthopath-is-div-equiv-pseudo-finsler} to express the fundamental invariants of a variational orthopath geometry in terms of the Cartan torsion and flag curvature of a generalized (pseudo-)Finsler structure whose divergence equivalence class coincides with that orthopath structure.  

  Suppose that the  divergence equivalence class of a generalized (pseudo-)Finsler structure  $(\sigma\colon\cF\to\cC,\varphi)$ is given by the variational orthopath structure  $(\tau\colon\cG\to\cC,\psi)$ as described in Theorem \ref{thm:var-orthopath-equiv-prob}. Via the inclusion
  \begin{equation}
    \label{eq:Finsler-orthopath-F-G-inclusion}    
    \iota\colon\cF\to\cG,
  \end{equation}
 it follows that  the Cartan connection $\iota^{*}\psi$ in \eqref{eq:var-orthpath-Cartan-conn} is expressed  in terms of   $\varphi$  \eqref{eq:CartanConn-FW-1} as 
\begin{equation}
  \label{eq:var-ortho-from-Finsler}
  \begin{gathered}
     \omega^a=\alpha^a,\quad     \theta^a=\beta^a,\quad     \omega^0= \alpha^0+2I_b\alpha^b,\quad \phi_1=0,\quad      \phi_0=-J_b\alpha^b-I_b\beta^b, \\
  \phi_{ab}=\psi_{ab}+I_a\beta_b-I_b\beta_a-J_a\alpha_b+J_b\alpha_a,\quad         \xi_0=R\alpha^0+X_c\alpha^c-2J_b\beta^b,
  \end{gathered}
\end{equation}
where we have suppressed $\iota^*$ on the left hand side and
\[
  \begin{gathered}
    I_a:=\textstyle{\frac{1}{n+1}}\ve^{bc}I_{abc},\qquad  J_a:=\textstyle{\frac{1}{n+1}}\ve^{bc}J_{abc}=\textstyle{\frac{\partial}{\partial\alpha^0}}I_a,\qquad R:=\textstyle{\frac{1}{n-1}}\ve^{ab}R_{ab}\\
    X_c=\tfrac{1}{2(n-1)}R_{;\uc}+\half RI_c-\half J_{c;0}-\tfrac{n+2}{2(n-1)}\ve^{ab}I_aR_{bc},\quad J_{c;0}=\tfrac{\partial}{\partial\alpha^0}J_c,\quad R_{;\uc}=\tfrac{\partial}{\partial\beta^c}R.
    \end{gathered}
  \]
Relations above are found via straightforward and rather simple manipulations of the structure equations, although  finding the expression for  $X_c$ is significantly more tedious than others. We will not provide the details of these computations as they are fairly standard. 
As a result, one obtains that the fundamental invariants of the corresponding variational orthopath structure are given by
\begin{equation}
  \label{eq:orthopath-inv-from-Finsler}  
\sA_{abc}=\overset{\circ}I_{abc},\quad \sT_{ab}=\overset{\circ}R_{ab},\quad \sN_{ab}=-2I_{[a;b]}+4J_{[a}I_{b]},\quad \sq=(2\ve^{ab}I_aI_b-1) -\tfrac{2}{n-1}\ve^{ab}I_{a;\ub},
\end{equation}
where
\[\overset{\circ}{I}_{abc}:=I_{abc}-\left(I_a\ve_{bc}+I_b\ve_{ac}+I_c\ve_{ab}\right),\quad \overset{\circ}R_{ab}=R_{ab}-R\ve_{ab},\quad I_{a;b}=\textstyle{\frac{\partial}{\partial\alpha^b}}I_a,\quad  I_{a;\ub}=\textstyle{\frac{\partial}{\partial\beta^b}}I_a.\]
In particular, one has $\sQ_{ab}=-\ve_{ab}-2I_{a;\ub}-4I_aI_b+2\ve^{cd}I_dI_{cab}.$

From  relations \eqref{eq:orthopath-inv-from-Finsler} several conclusions can be drawn. If the orthopath geometry arises from a pseudo-Riemannian metric, i.e. $I_{abc}=0,$ then by \eqref{eq:orthopath-inv-from-Finsler} one has   $\sq=-1$ and $\sN_{ab}=0,\sA_{abc}=0.$ By Corollary \ref{cor:causal-from-orthopath-various-curv-cond}, the quasi-contactification of such orthopath geometries gives pseudo-Riemannian  conformal  structures, $[\tg],$ with a choice of non-null conformal Killing field. By our discussion in \ref{sec:caus-geom-dimens}, such conformal pseudo-Riemannian metrics are given by \eqref{eq:g-conf-str-bilinear-form-on-tC} and have signature $(p+1,q+1)$. Using expressions \eqref{eq:Cartan-conn-modif-orthopath-to-causal}, one finds  $[\tg]$ is given by
\[ \tg=-(\exd t)^2+(\alpha^0)^2+\ve_{ab}\alpha^a\circ\alpha^b.\]
In other words, in this  case the quasi-contactification recovers  the direct product construction of   a conformal structure of signature $(p+1,q+1)$ from a pseudo-Riemannian metric of signature $(p+1,q).$ Moreover, if the metric $g$ has constant sectional curvature, then it is well-known that the flag curvature is pure trace, i.e. $\overset{\circ}R_{ab}=0.$  By \eqref{eq:orthopath-inv-from-Finsler} one obtains that $\bT=0$ and consequently by \ref{cor:causal-from-orthopath-various-curv-cond} $\bW=0,$ i.e. $[\tg]$ is conformally flat, as expected.   

Note that relations  \eqref{eq:orthopath-inv-from-Finsler} imply that if an orthopath geometry arises as the reduction of a pseudo-conformal structure with a null infinitesimal symmetry, i.e. $\sq=0,$ then the corresponding (pseudo-)Finsler structure has  non-vanishing mean Cartan torsion, i.e. $I_a\neq 0,$ almost everywhere.

The vanishing conditions $\bN=0$ or $\bq=0$ can be used to define new classes of (pseudo-)Finsler structures that are invariant under divergence equivalence. For instance, if the Finsler metric is monochromatic, then one has $I_{a;b}=J_a=0,$  and from  \eqref{eq:orthopath-inv-from-Finsler} it follows that $\bN=0.$ If the indicatrices are centro-affine minimal hypersurfaces \cite{Wang-minimal}, i.e. $\ve^{ab}I_{a;\ub}=0,$ then by \eqref{eq:orthopath-inv-from-Finsler} the condition $\sq=0$ is equivalent to
\[\|\tr\bI\|^2:=\ve^{ab}I_aI_b=\half,\]
We do not know if the condition above  has been investigated among centro-affine minimal hypersurfaces.

Lastly, we recall that centro-affine hypersurfaces for which  $I_a=0$ correspond to affine spheres with center at the origin e.g. see \cite[Section 6.1.2]{AffineDiffGeom}. Such  (pseudo-)Finsler structures, when they have  Lorentzian signature, were  considered in \cite{Minguzzi} as a case of Finslerian  space-times. Moreover, using the invariant bilinear form $\bh$ in \eqref{eq:Finsler-biliner-form}, the flag curvature can be decomposed into a trace part and a trace-free part, i.e.
\[\bR=R\bh+\overset{\circ}{\bR}.\]
The condition $\overset{\circ}{\bR}=0$ defines the well-known class of  Finsler structures of \emph{scalar flag curvature.} 

Now consider  generalized (pseudo-)Finsler structures whose fibers are  affine spheres with center at the origin, i.e. $I_a=0,$ and its flag curvature is trace-free, i.e. $R=0.$ Then relations 
\eqref{eq:var-ortho-from-Finsler} imply that in this case the  Cartan connection $\varphi$ for the (pseudo-)Finsler structure and $\psi$ for the corresponding orthopath geometry satisfy 
\begin{equation}
  \label{eq:orthopath-Finsler-holonomy-red}
  \varphi=\pr_{\fg}\circ\iota^*\psi
\end{equation}
where $\pr_{\fg}$ is the natural projection of $\fk$ to $\fg=\RR^{n}\rtimes\mathfrak{so}(p+1,q)$. This is due to the fact that, as mentioned in \eqref{eq:symmetries-FW-invariants}, $J_{a}=\tfrac{\partial}{\partial\alpha^0}I_a.$ Furthermore, it follows  from  \eqref{eq:var-ortho-from-Finsler}  that $I_a=0$ and $R=0$ imply
$\iota^*\xi_0=\iota^*\phi_0=\iota^*\phi_1=0.$ Thus,  the Cartan holonomy of the variational orthopath geometry of generalized  (pseudo-)Finsler structures with $I_a=0$ and $R=0$ is reduced to $\RR^{n}\rtimes\mathrm{SO}(p+1,q).$

\begin{remark}\label{rmk:affine-centroaffine-Fubini}
 It is known that the cubic form $\bI$ encodes the centro-affine properties of the set of unit vectors in each tangent space, i.e. the immersed hypersurfaces $\cC_x\subset T_xM$ in a generalized (pseudo-)Finsler structure. Since the cubic form $\bA$ is the trace-free part of $\bI,$ up to an overall weight, it encodes the affine properties of hypersurfaces $\cC_x.$ Finally, the Fubini cubic form $\bF$ of a causal structure   encodes the projective properties of hypersurfaces $\tcC_x\subset\PP T_x\tM$ which, via quasi-contactification, can be considered as cones over the affine hypersurfaces $\cC_x$ with affine cubic form $\bA.$  See \cite[Chapter 4]{Sasaki-Book} on such relations between projective and affine geometry of hypersurfaces.  
\end{remark}

Using Cartan-K\"ahler analysis, it follows that variational orthopath structures on $n$-dimensional manifolds satisfying $\bA=0$ locally depend on $\half(n+2)(n-1)$ functions of $n$-variables. Recall that pseudo-conformal structures on $(n+1)$-dimensional manifolds locally depend on  $\half(n+2)(n-1)$ functions of $(n+1)$-variables. More generally, our discussion shows that a variational orthopath structure is determined by the divergence equivalence of a Lagrangian $L\colon TM\to \RR$ and therefore their local generality is given by  1-function of $(2n-1)$-variables. This can be compared with the fact  that a causal structure on an $(n+1)$-dimensional  manifold $\tM$  is defined by a generic  codimension one sub-bundle of the projectivized tangent bundle $\tcC\subset\PP T\tM,$ i.e. the sky bundle of the causal structure. As a result,  the   local generality of causal structures depend on 1 function of $2n$-variables. This count confirms the expected local generality of  a  geometric structure with an infinitesimal symmetry which can be stated as follows. If the local generality of a geometric structure depends on $p$ functions of $q$ variables, the existence of an infinitesimal symmetry reduces the generality to $p$ functions of $q-1$ variables.

\section{Variationality of chains in CR structures}\label{sec:vari-chains-cr-orthopath}
A non-degenerate  partially integrable almost CR-structure of hypersurface type on a manifold $M$ of dimension $2n+1$ is given by a contact structure $\mathcal{H}\subset TM$ together with an integrable almost complex structure $J:\mathcal{H}\to\mathcal{H}$ such that $\mathfrak{L}(JX,JY)=\mathfrak{L}(X,Y),$ where $\mathfrak{L}(X,Y)_x:=\mathrm{pr}_{TM/\mathcal{H}}([X,Y]_x)$ denotes the Levi bracket. Identifying $TM/\mathcal{H}$ at $x\in M$ with $\mathbb{R}$, the map $\mathfrak{L}_x$ corresponds to the imaginary part of a Hermitian form of signature $(p,q)$, where $p+q=n$.  The structure $(M,\mathcal{H},J)$ has a canonically associated    regular and normal parabolic geometry $(\mathcal{P}\to M,\omega)$ of type $(\mathrm{SU}(p+1,q+1),P)$, where the parabolic subgroup $P$ is the stabilizer of a non-zero null line in $\mathbb{C}^{n+2}$. The  Lie algebra of $P$ will be denoted by $\fp$. The partially integrable almost CR structure is integrable, i.e. a CR structure, if and only if the associated Cartan connection is torsion-free.

Explicitly, we represent $\mathfrak{su}(p+1,q+1)$ with respect to the Hermitian form
$$(z^0,z^1,\cdots,z^{n+1})\mapsto  \tfrac{\mathrm{i}}{2}(z^0\bar{z}^{n+1}-z^{n+1}\bar{z}^0)+\ve_{a\barb}z^a\bar{z}^b,  $$
where in this section $\ve_{a\barb}$ and $\ve_{ab}$  are real-valued diagonal with the first $p$ entries equal to $1$ and the next $q$ entries equal to $-1$ and $p+q=n$ and, unlike sections before, the range of indices are
\[1\leq a,b,c,d\leq n.\] This yields the following matrix form for the Cartan connection
\begin{equation}
  \label{eq:CR-Cartan-conn}
  \omega=\begin{pmatrix}
  \nu-\tfrac{1}{2} \sigma & -\mathrm{i}\eta_{b} & -\tfrac{1}{4}\eta^0 \\
    \zeta^a & \sigma_b{}^a & \tfrac{1}{2}\eta^a \\
    2\zeta^0 & 2\mathrm{i} \zeta_{b} & -\nu-\tfrac{1}{2}\sigma \\
  \end{pmatrix},
  \end{equation}
where  $\zeta^0, \eta^0$ are real,  $\zeta_{b}=\ve_{b \bara}\zeta^{\bara}$, $\eta_{b}=\ve_{b \bara}{\eta^{\bara}}$, $\zeta^{\bara}=\overline{\zeta^a},$ $\eta^{\bara}=\overline{\eta^a},$ $\sigma_b{}^a$ is contained in $\mathfrak{su}(p,q)$ and consequently $\sigma=\sigma^a_a$ is purely imaginary. The  grading of $\mathfrak{su}(p+1,q+1)$ corresponding to $P$ is a contact grading
$$\mathfrak{su}(p+1,q+1)=\p_{-2}\oplus\p_{-1}\oplus\p_0\oplus\p_1\oplus\p_2,$$ where $\zeta^0\in\p_{-2}$, $\zeta^a\in\p_{-1}$, $\nu, \sigma_b^a\in\p_{0}$, $\eta^a\in\p_{1}$ and $\eta^0\in\p_{2}$.

On a CR manifold there is a family of canonical curves called \emph{chains}. As unparametrized curves, they can be defined as projections of integral curves of  constant vector fields $\omega^{-1}(X)$ from $\mathcal{P}$ to $M$, where $X$ is a non-zero element of the grading component $\p_{-2}$. In particular, these curves are transverse to the contact distribution $\mathcal{H}$. Moreover, for each point $x\in M$ and each transverse direction $l_x$ there is precisely one unparametrized chain through $x$ in  direction $l_x$. It follows that chains define a  path geometry on the open subset $\cC\subset \mathbb{P}TM$ of transverse directions. 

Following \cite{CZ-CR}, one can describe the set of transverse directions in terms of the Cartan bundle as follows. Let $S\subset P$ be the stabilizer of the line $l\subset\mathfrak{su}(p+1,q+1)/\p$ spanned by an element in the grading component $\p_{-2}$. Then one verifies that the Lie algebra of $S$ is $\mathfrak{s}=\p_0\oplus\p_2\subset \p:=\fp_{0}\oplus\fp_1\oplus\fp_2$. Moreover, the $P$-orbit $P\cdot l\subset \mathfrak{su}(p+1,q+1)/\p$ is the set of all lines $\mathfrak{su}(p+1,q+1)/\p$ not contained in $\p^{-1}/\p$. Consequently,   $P/S$ can be identified with the set of lines in $\mathfrak{su}(p+1,q+1)/\p$ transverse to  $\p^{-1}/\p$. Since, via the Cartan connection, one has an identification $\mathcal{H}\cong \mathcal{P}\times_P(\p^{-1}/\p)$, one concludes that
 $\mathcal{P}/S=\mathcal{P}\times_{P} P/S$ can be identified with the set $\cC\subset \mathbb{P}TM$ of directions transverse to the contact sub-bundle $\mathcal{H}$.

This implies that one can view the Cartan connection $\omega\in\Omega^1(\mathcal{P},\mathfrak{su}(p+1,q+1))$ associated with the almost CR manifold as a Cartan connection on $\mathcal{P}\to\mathcal{P}/S=\cC$. In particular, it determines an isomorphism 
\begin{equation}\label{eq-isoTC}
T\cC\cong\mathcal{P}\times_S(\mathfrak{su}(p+1,q+1)/\mathfrak{s}).
\end{equation}
Via this isomorphism, the line bundle $\mathcal{E}$ defining the path geometry corresponds to the subspace induced by $\p_{-2}$ and the vertical bundle $\mathcal{V}$ corresponds to the subspace induced by $\p_{1}$.

Using this, we  now show that $\cC$ is actually equipped with a natural orthopath geometry, which  we refer to as the \emph{canonical orthopath geometry of chains}.
\begin{proposition}
The open subset $\mathcal{C}\subset\mathbb{P}TM$ of  directions transverse to the contact distribution $\mathcal{H}\subset TM$ is equipped with a canonical orthopath geometry.
\end{proposition}

 \begin{proof}
It remains to verify that the vertical bundle $\mathcal{V}\subset T\cC$ has a naturally induced conformal class $[\bh]$ of bundle metrics  $\bh\in\Gamma(\mathrm{Sym}^2\mathcal{V}^*)$ of signature $(2p,2q)$. As mentioned above, via \eqref{eq-isoTC}, the vertical bundle $\mathcal{V}$ corresponds to the subspace $\p_{1}\cong\p/\mathfrak{s}\subset\mathfrak{su}(p+1,q+1)/\mathfrak{s}$. An element in $S$ can be written as $g=g_0g_2$, where $g_2\in\mathrm{exp}(\p_2)$ acts trivially on  $\p/\mathfrak{s}$ and $g_0$ acts by an element in $\mathrm{CU}(p,q)$. 
In particular, the $S$-action  preserves a bilinear form of signature $(2p,2q)$ up to scale and the result follows. 
\end{proof}

As for the path geometry of chains studied in \cite{CZ-CR}, there is an extension functor  from Cartan geometries of type $(\mathfrak{su}(p+1,q+1),S)$ to Cartan geometries of type  $(\mfk,L)$, where $\mfk= (\mathfrak{sl}(2,\mathbb{R})\times\mathfrak{co}(p,q))\ltimes (\mathbb{R}^2\otimes \mathbb{R}^{(p,q)}),$   $L=B\times\mathrm{CO}(p,q),$ and $B\subset SL(2,\RR)$ is the Borel subgroup. Such an extension is determined by a Lie group homomorphism $i:S\to L$ and a linear map $\alpha:\mathfrak{su}(p+1,q+1)\to\mathfrak{k}$ such that (i) $\alpha\circ\mathrm{Ad}(s)=\mathrm{Ad}(s)\circ\alpha$ for all $s\in S$, (ii) $\alpha$ restricted to $\mathfrak{s}$ coincides with the derivative of $i$, and (iii) $\alpha$ induces an isomorphism $\mathfrak{su}(p+1,q+1)/\mathfrak{s}\cong \mfk/\mfl$. Explicitly, the map $\alpha$ is given by 
 \begin{equation}
 \label{eq-alpha}
 \alpha\left(
 \begin{pmatrix}
  \nu-\tfrac{1}{2} \sigma & -\mathrm{i}\eta_{{b}} & -\tfrac{1}{4}\eta^0 \\
    \zeta^a & \sigma_b{}^a & \tfrac{1}{2}\eta^a \\
    2\zeta^0 & 2 \mathrm{i}\zeta_{{b}} & -\nu-\tfrac{1}{2}\sigma \\
  \end{pmatrix} \right)
  =\begin{pmatrix}    
 \nu  &-\half \eta_0&0 & 0\\
\zeta^0 & -\nu&0 & 0   \\
\zeta^a_1&  \eta^a_1& \pi^a_b & -\varpi^a_b-\half\varpi\delta^a_b\\
\zeta^a_2&  \eta^a_2& \varpi^a_b+\half\varpi\delta^a_b& \pi^a_b
\end{pmatrix},  
\end{equation}
where
\[\zeta^a=\zeta^a_1+\mathrm{i}\zeta^a_2,\quad\eta^b=\eta^b_1+\mathrm{i}\eta^b_2,\quad \sigma^{\ a}_b=\pi^{a}_{b}+\mathrm{i} \varpi^{a}_b \ \]
with the property that
\[\ve_{ac}\pi_b^{c}=:\pi_{ab}=-\pi_{ba},\qquad \ve_{ac}\varpi^c_b=:\varpi_{ab}=\varpi_{ba},\qquad \varpi=\varpi^a_a.\]
Given a Cartan geometry $(\mathcal{P}\to \cC,\omega)$ of type  $(\mathfrak{su}(p+1,q+1),S)$, one has a bundle map
\begin{equation}
\label{eq-bundleinclusion}
\iota:\mathcal{P}\to\mathcal{G}:=\mathcal{P}\times_SL,
\end{equation}
which is equivariant with respect to $i:S\to L$. Moreover, there is a unique Cartan connection $\psi\in\Omega^1(\mathcal{G},\mfk)$ such that $\iota^*\psi=\alpha\circ\omega$.

Now we can state the main result of this section. 
\begin{theorem}\label{thm-chains}
  The canonical orthopath geometry defined by the chains of a non-degenerate partially integrable almost  CR structure is variational if and only if the almost CR structure is integrable. Such variational orthopath structures  satisfy the invariant conditions $\bA=0$ and $\bq=0$.
\end{theorem}
\begin{proof}
By the preceding discussion, it is clear that the bilinear form
\[\bh=\ve_{ab}\eta^a_1\circ\eta^b_1+\ve_{ab}\eta^a_2\circ\eta^b_2\]
is well-defined on the vertical tangent bundle $\cV.$ Moreover, the corresponding compatible conformally quasi-symplectic structure is induced by the 2-form
\[\rho=\ve_{ab}\zeta^a_1\w\eta^b_1-\ve_{ab}\zeta^a_2\w\eta^b_2=\half \ve_{\bara b}\zeta^{\bara}\w\eta^b+\half \ve_{a\barb}{\zeta^{a}\w\eta^{\barb}}\in\Omega^1(\cP),\]
where $\zeta^{\bara}=\overline{\zeta^a},$ $\eta^{\bara}=\overline{\eta^a}.$ 
Let us assume that the almost CR structure is integrable. To show variationality, one needs  to  show  $\exd\rho=\phi\w\rho$ for a 1-form $\phi\in\Omega^1(\cP).$ To do so, recall from \cite[Appendix]{CM-CR} that the Cartan curvature for the regular and normal Cartan connection \eqref{eq:CR-Cartan-conn} when $n>2$ can be written as 
\[\Omega=\exd\omega+\omega\w\omega= \def\arraystretch{1.3}
 \begin{pmatrix}
  0 & -\ri\beeta_{b} & -\tfrac{1}{4}\beeta^0 \\
    0 & \Sigma_b{}^a & \tfrac{1}{2}\beeta^a \\
    0 & 0 & 0 \\
  \end{pmatrix}\]
where 
\begin{equation}
  \label{eq:CR-curvature-2form-entries}
  \begin{gathered}    
    \Sigma^a_{\ b}=C^a_{\ bc\bard}\zeta^c\w\zeta^{\bard}+E^{\ a}_{b\ c}\zeta^{c}\w\zeta^0-E^a_{\ b\barc}\zeta^{\barc}\w\zeta^0\\
    \beeta^a=E^{a}_{\ b\barc}\zeta^b\w\zeta^{\barc}+F^{\ a}_{b}\zeta^b\w\zeta^0+G^{\ a}_{\barb}\zeta^{\barb}\w\zeta^0\\
    \beeta^0=-2\ri  F_{a\barb}\zeta^a\w\zeta^{\barb}+H_{a}\zeta^a\w\zeta^0+H_{\bara}\zeta^{\bara}\w\zeta^0.
      \end{gathered}
    \end{equation}
    Defining $E_{a\barb  c}=\overline{\ve_{\bara d}E^d_{\ b\barc}},$ $F_{\barb a }=\overline{\ve_{\bara c}F^{\ c}_b}$ $G_{ba}= \overline{\ve_{\bara c}G^{\ c}_{\barb}},$ 
one has
    \[E_{a\barb c}=E_{c\barb a},\quad F_{\bara b}=F_{b\bara},\quad G_{ab}=G_{ba}\]
    It is straightforward to show that the symmetric properties above imply $\exd\rho=0.$

    Now assume the orthopath geometry is variational. The Cartan curvature of  partially integrable almost CR structures is 
    \[\Omega= \def\arraystretch{1.3}
 \begin{pmatrix}
  \bnu-\half\Sigma & -\ri\beeta_{b} & -\tfrac{1}{4}\beeta^0 \\
    \bzeta^a & \Sigma_b{}^a & \tfrac{1}{2}\beeta^a \\
    0 & \bzeta_{a} & -\bnu-\half\Sigma \\
  \end{pmatrix}\]
where
    \[\bzeta^a\equiv X^a_{\ bc}\zeta^b\w\zeta^c+Y^a_{\ \ b\barc}\zeta^b\w\zeta^{\barc}+Z^a_{\ \barb\barc}\zeta^{\barb}\w\zeta^{\barc}\mod\{\zeta^0\}\]
    for some functions $X^a_{\ bc},Y^a_{\ b\barc},Z^a_{\ \barb\barc}$ on $\cP$ all of which are zero if and only if the CR structure is torsion-free. One computes that
\[
  \begin{aligned}   
    \exd\rho=&\half\ve_{a\barb}\bzeta^a\w\eta^{\barb}+\half\ve_{\bara b}\bzeta^{\bara}\w\eta^{b}-\half\ve_{\bara b}\zeta^{\bara}\w\beeta^b- \half\ve_{a\barb}\zeta^a\w\beeta^{\barb}.
    \end{aligned}
  \]
  Variationality of the orthopath geometry implies that in $\exd\rho$ the coefficient of  all 3-forms  $\zeta^a\w\zeta^b\w\eta^{c}$ and $\zeta^{\bara}\w\zeta^{\barb}\w\eta^{\barc}$ must be zero. This implies the vanishing of $Z^a_{\ \barb\barc}$  everywhere. 
 Finding all 3-forms in the equality  $\exd\rho=\phi\w\rho$ that involve $\eta^a$ and $\eta^{\bara}$ gives 
  \[\half(X_{\barb ac }\zeta^a\w\zeta^c\w\eta^{\barb}+Y_{\barb a\barc}\zeta^a\w\zeta^{\barc}\w\eta^{\barb}+X_{b \bara\barc }\zeta^{\bara}\w\zeta^{\barc}\w\eta^{b}+Y_{b\bara c}\zeta^{\bara}\w\zeta^{c}\w\eta^{b})+*=\phi\w\rho.\]
where no term in $*$ involves $\eta^a$ or $\eta^{\bara}.$  This implies
  \begin{equation}
    \label{eq:X-Y-W}
    X_{\barb ac}=\half(W_c\ve_{a\barb}-W_a\ve_{c\barb}),\quad Y_{\barb a\barc}=W_{\barc}\ve_{a\barb},\quad \phi\equiv -W_c\zeta^c-W_{\barc}\zeta^{\barc}\mod \{\zeta^0\}
      \end{equation}
for some functions $W_c$ on $\cP.$  To show $W_c=W_{\barc}=0,$ one computes
  \[
    \begin{aligned}
      \exd\zeta^0=&\nu\w\zeta^0+\ri\ve_{a\barb}\zeta^{a}\w\zeta^{\barb}\\
      \Rightarrow 0=\exd^2\zeta^0\equiv& \ve_{a\barb}\bzeta^a\w\zeta^{\barb}-\ve_{a\barb}\zeta^a\w\bzeta^{\barb}\equiv -2\ri W_{d}\ve_{\barb c} \zeta^{\barb}\w\zeta^{c}\w\zeta^{d}+2\ri  W_{\bard}\ve_{b\barc} \zeta^{b}\w\zeta^{\barc}\w\zeta^{\bard}\mod \{\zeta^0\}.
    \end{aligned}
  \]
 Thus, it follows that $W_c=0$, which, by \eqref{eq:X-Y-W}, gives the vanishing of $X^a_{\ bc},Y^a_{\ b\barc}$ everywhere. Therefore,  the almost CR structure is torsion-free, hence, integrable.

 Lastly, if $(\tau\colon\cG\to\cC,\psi)$ is the  variational orthopath geometry of chains in a  CR structure, by the discussion at the beginning of this section, one has the inclusion \[\iota\colon\cP\to\cG,\]
 using which one has 
    \begin{equation}\label{eq:orthopath-conn-from-CR-conn}
  \iota^*\psi=
  \def\arraystretch{1.3}
\begin{pmatrix}    
 \nu  &-\half \eta_0&0 & 0\\
\zeta^0 & -\nu&0 & 0   \\
\zeta^a_1&  \eta^a_1& \pi^a_b & -\varpi^a_b-\half\varpi\delta^a_b\\
\zeta^a_2&  \eta^a_2& \varpi^a_b+\half\varpi\delta^a_b& \pi^a_b
\end{pmatrix}    
\end{equation}
Using \eqref{eq:CR-curvature-2form-entries} and  the Cartan curvature of a variational orthopath geometry, given in \eqref{eq:orthopath-Cartan-curv-2form}, it is straightforward to find the fundamental invariants \eqref{eq:fund-inv-orthopath}. In particular, one has
\begin{equation}
  \label{eq:ANq-chains}
    \begin{gathered}
    \iota^*\bA=0,\quad \iota^*\bq=0,\quad \iota^*\bN=2\ve_{a b}\eta^a_1\w\eta^b_2\otimes E^{-1}. 
  \end{gathered}
\end{equation}
Using \eqref{eq:rewrite2-streqns-pathgeom}, one obtains the coefficients  of $\iota^*\bT$ to be 
\[\sT_{ab}=-\Re(F_{b\bara}+G_{\barb\bara}),\quad \sT_{(a+n)b}=\sT_{b(a+n)}=-\Im(G_{\barb\bara}),\quad \sT_{(a+n)(b+n)}=-\Re(F_{b\bara} -G_{\barb\bara})\]
which is trace-free because, by \eqref{eq:CR-curvature-2form-entries}, one obtains $\Re F_a^{\ a}=0.$

In the case of 3-dimensional CR structures the Cartan curvature \eqref{eq:CR-curvature-2form-entries} additionally satisfies $\Sigma^a_{\ b}=0$ which does not change anything in the rest of the proof. 
\end{proof}

\begin{remark}
  By Corollary \ref{cor:causal-from-orthopath-various-curv-cond}, $\bA=0$ and $\bq=0$ imply that the quasi-contactification of such orthopath geometries gives rise  to pseudo-conformal structures with a null infinitesimal symmetry. By the discussion in \ref{sec:general-causal-local-form-invar}, variational orthopath geometries with $\bA=0$ correspond to the divergence equivalence classes of  generalized (pseudo-)Finsler structures whose trace-free part of the Cartan torsion is zero.  The trace-free part of the Cartan torsion is referred to as the \emph{Matsumoto tensor} in Finsler geometry. The class of Finsler metrics with vanishing Matsumoto tensor is comprised of the so-called \emph{(pseudo-)Randers metrics} and the \emph{(pseudo)-Kropina metrics}, e.g. see \cite{HuangMo, Javaloyes}.

  Unlike the construction of the Kropina metrics associated to  CR chains in   \cite{CMMM-CR}, which is evidently singular on the contact hyperplane in each tangent space, in our derivation no singularity is evident. This is due to the fact that we are proving the existence of such (pseudo-)Finsler metrics in the open set of  transverse directions to the contact distribution. However, having computed $\bA=0,$ i.e. the vanishing of the Matsumoto tensor, using  \cite{HuangMo, Javaloyes}, it follows that such (pseudo-)Finsler metrics are either (pseudo)-Randers metrics or (pseudo-)Kropina metrics. Subsequently, one can rule out the case of (pseudo-)Randers metrics since the ODE system for their geodesics, under point transformations, is  defined in all directions, which is not the case for ODE systems that define chains. Hence, chains can only be realized as geodesics of (pseudo-)Kropina metrics.  
 
  Furthermore, by our discussion in \ref{sec:general-causal-local-form-invar}, the condition $\bq=0$ implies that the  Cartan torsion  for any (pseudo-)Finsler representative of the divergence equivalence class is almost everywhere nonzero. Moreover, by the expression of the curvature entries of the path geometry, $C^a_{bcd},$ in  \eqref{eq:path-geom-curv-derived-from-orthopath} and the expression for  $\bN$ and $\bA$ in \eqref{eq:ANq-chains}, it follows that the curvature $\bC$ in \eqref{eq:path-harm-inv-representation} is never zero and the path geometry of chains never defines a projective structure, as was observed in \cite{CZ-CR}.
\end{remark}

\begin{remark}
  Lastly,  recall that chains   can be viewed via the Fefferman construction discussed in Example \ref{ex-Feff}. Given a CR manifold $M$ of hypersurface type, there is a circle bundle $\tilde{M}$ over $M$ equipped with a natural pseudo-Riemannian conformal structure, referred to as the Fefferman conformal structure.  It is known that any Fefferman conformal structure has a null conformal Killing field $\xi$. Chains are  projections of null geodesics  with initial directions \emph{not} orthogonal to $\xi$ from $\tilde{M}$ onto $M$ \cite{Fefferman-CR, Koch-chains}. 
  Now, the conformal Killing field $\xi$ of a Fefferman conformal structure lifts to an infinitesimal symmetry on the sky bundle $\tilde{\cC}$, which we denote by the same name. If one restricts to the subset of the sky bundle of null lines not orthogonal to $\xi$, the lifted symmetry is transverse, as can be seen from \eqref{eq-quascond}. Due to the tight relationship between the normal CR Cartan connection and the corresponding normal conformal Cartan connection in the integrable case \cite{CG-CR}, the orthopath geometry on the local leaf space determined by $\xi$ can be identified with the orthopath geometry of chains discussed above.
  Note that this is not the case if we start with a partially integrable almost CR structure of hypersurface type. There is a generalized Fefferman construction that gives rise to a conformal structure with a null conformal Killing field also in this case,  but the (variational) orthopath geometry obtained by reduction and the orthopath geometry of chains in general differ from each other.

\end{remark}

\section*{Acknowledgments}

The authors thank Ian Anderson, Andreas \v Cap, Boris Doubrov and Igor Zelenko for enlightening discussions.
OM would like to thank   Miguel \'Angel Javaloyes, Wojciech Kry\'nski, Miguel S\'anchez Caja, Eivind Schneider, and Travis Willse   for helpful exchange of ideas and comments. KS would like to thank Pawe\l\ Nurowski for insightful conversations.  The research leading to these results has received funding from the Norwegian Financial Mechanism 2014-2021 with project registration number 2019/34/H/ST1/00636. OM gratefully acknowledges partial support by the grant  PID2020-116126GB-I00 provided via the Spanish Ministerio de Ciencia e Innovaci\'on MCIN/ AEI /10.13039/50110001103 as well as partial funding from the Norwegian Financial Mechanism 2014-2021 (project registration number 2019/34/H/ST1/00636), the Tromsø Research Foundation (project “Pure Mathematics in Norway”), and the UiT Aurora project MASCOT.

 \appendix
\setcounter{equation}{0}
\setcounter{subsection}{0}
 \setcounter{theorem}{0}

\section*{Appendix}
\renewcommand{\theequation}{A.\arabic{equation}}
\renewcommand{\thesection}{A}
 
In the appendix we provide some of the structure equations that are used in the text. The structure equations for contact equivalent classes scalar 4th order ODEs  with respect to the matrix form \eqref{eq:4-ODE-Conn2} is given by the Cartan curvature 2-form
\begin{equation}\Psi=\exd\psi+\psi\w\psi=
   \label{eq:curv-matrix-4-ODE-Conn2}
    \def\arraystretch{1.1}
   \textstyle{
\begin{pmatrix}
  \textstyle{}\Phi_0-\textstyle{\frac 15}\Phi_1& \Xi_0& 0& 0& 0 \\
\Omega^0 & \textstyle{\frac 45}\Phi_1-\Phi_0& 0 & 0&  0\\
\Omega^1& \Theta^1 & \textstyle{2\Phi_0-\frac 65}\Phi_1 & \Xi_0 & 0  \\
\Omega^2 & -2\Omega^1 & 2\Omega^0 & -\textstyle{\frac 15}\Phi_1 & -\Xi_0 \\
0 & \Omega^2 & 0 & -2\Omega^0 & \textstyle{\frac 45}\Phi_1-2\Phi_0\\
\end{pmatrix}}
 \end{equation}
where, using the notation
\[\nu^{ij}=\omega^i\w\omega^j,\quad \nu^{i\underline j}=\omega^i\w\theta^j,\]
one has
\begin{equation}
  \label{eq:curv-2-forms-4-ODE}
  \begin{aligned}
     \Omega^2=&\sx_5\nu^{23}\\
     \Omega^1 =& \sx_6\nu^{23}+\textstyle{\frac{12}{7}}\sx_5\nu^{13}+\textstyle{\frac 13}\sw_1\nu^{03},\\ 
     \Omega^0 =& \textstyle{-\frac 97\sx_5\nu^{03}+\sx_7\nu^{12}+\sx_8\nu^{13}+\sx_9\nu^{23}+\scc_1\nu^{2\uo}+(\frac 83\scc_0-\frac 13\sx_7)\nu^{3\uo}}\\
     \Theta^1 =& \sw_1\nu^{02}+\sw_0\nu^{03}+\textstyle{\frac{48}{7}\sx_5\nu^{12}+\sx_{10}\nu^{13}+\sx_{11}\nu^{23}+\frac 47\sx_5\nu^{3\uo}}\\
\Phi_0 =& \textstyle{-\frac{17}{7}\sx_5\nu^{02}+\sx_{12}\nu^{03}+\sx_{13}\nu^{12}+\sx_{14}\nu^{13}-4\scc_1\nu^{1\uo} +\sx_{15}\nu^{23}+(\frac{28}{3}\scc_0-\frac 53\sx_{8})\nu^{2\uo}+ \sx_{16}\nu^{3\uo}}\\
\Phi_1 =& \textstyle{-\frac{45}{7}\sx_5\nu^{02}+(-\frac{25}{18}\sx_{10}+\frac{16}{9}\sx_6+\frac{25}{18}\sx_{12})\nu^{03} +3(\sx_{13}-\sx_8)\nu^{12}+\sx_{17}\nu^{13} -12\scc_1\nu^{1\uo}}\\
&\textstyle{+\sx_{18}\nu^{23}+(20\scc_0-4\sx_{7})\nu^{2\uo}+(-\frac{47}{23}\sx_8+\frac{12}{23}\sx_{13}+\frac{35}{23}\sx_{16})\nu^{3\uo}}\\
     \Xi_0 =& \textstyle{\frac{16}{7}\sx_5\nu^{01}+(\frac{7}{18}\sx_{10}+\frac{20}{9}\sx_6+\frac{11}{18}\sx_{12})\nu^{02}+\sx_{19}\nu^{03}+\sx_{20}\nu^{12}}\\
     & \textstyle{+\sx_{21}\nu^{13}+(\frac 43\sx_8-\frac{32}{3}\scc_0)\nu^{1\uo}+\sx_{22}\nu^{23}+(-\frac{7}{23}\sx_{13}+\frac{14}{23}\sx_8 +\frac{16}{23}\sx_{16})\nu^{2\uo}+\sx_{23}\nu^{3\uo}}
 \end{aligned}
\end{equation}
for some functions $\scc_0,\scc_1,\sw_0,\sw_1,\sx_5,\cdots,\sx_{23}.$

For variational 4th order ODEs, discussed in \ref{sec:inverse-probl-vari}, the structure equations above simplify to the following
\begin{equation}
  \label{eq:curv-2-forms-4-ODE-variational}
  \begin{aligned}
    \Omega^2=&\Omega^1=\Phi_1=0\\
     \Omega^0  =& 5 \scc_0\nu^{12}+\scc_0\nu^{3\uo}+\sx_8\nu^{13}+\sx_9\nu^{23}\\
     \Theta^1 =& \sw_0\nu^{03}+\sx_{12}\nu^{13}+\sx_{19} \nu^{23} \\
     \Phi_0=&\sx_{12}\nu^{03}+\sx_8\nu^{12}+(\sx_{20}-4\sx_{9})\nu^{13}+\sx_{21}\nu^{23}+\scc_0\nu^{2\uo}+\sx_8\nu^{3\uo}\\
     \Xi_0=&-4\scc_0\nu^{1\uo}+\sx_{12}\nu^{02}+\sx_{19}\nu^{03}+\sx_{20}\nu^{12}+\sx_{21}\nu^{13}+\sx_{22}\nu^{23}+\sx_{8}\nu^{2\uo}+\sx_9\nu^{3\uo}
 \end{aligned}
\end{equation}

For Cartan geometries of type $(B_3,P_{23})$ discussed in \ref{sec:pairs-third-order}, the Cartan curvature for the Cartan connection $\psi$ as given in \eqref{eq:var-pair-3rd-order-ODE-Cartan-conn} is
\begin{equation}
  \label{eq:Pairs-3-ODE-Cartan-curv}
  \Psi=\begin{pmatrix}
  \textstyle{\Phi_0-\frac 15\Phi_1} & \Xi_0 & 0 \\
  \Omega^0 & \textstyle{\frac 45\Phi_1-\Phi_0} & 0  \\
  \Omega^a & \Theta^a& \Phi^a_b-\frac 15\Phi_1\delta^a_b \\
\end{pmatrix}
\quad \mathrm{where}\quad 
[\Phi^a_b]=
\begin{pmatrix}
  \textstyle{-\Phi_2} & -\Gamma_0 & 0 \\
\Theta^0 & 0 & \Gamma_0\\
0 & -\Theta^0 & \textstyle{\Phi_2}
\end{pmatrix}
\end{equation}
where using the notation
\[\nu^{ij}=\omega^i\w\omega^j,\quad \nu^{i\underline j}=\omega^i\w\theta^j,\quad \nu^{\underline i\underline j}=\theta^i\w\theta^j\]
one obtains the following \emph{first order} structure equations
 \begin{equation}
   \label{eq:curv-2form-pair-3rd-ODE}
   \begin{aligned}
    \Omega^3 =& \sx_{20}\nu^{3\ut}-\sbb_{10}\nu^{\ud\ut}-\sbb_{11}\nu^{2\ut}+\sbb_{11}\nu^{3\ud}-\sbb_{12}\nu^{23}\\
    \Theta^3 =& \sx_{19}\nu^{3\ut}+\sbb_{11}\nu^{\ud\ut}+\sbb_{12}\nu^{2\ut}-\sbb_{12}\nu^{3\ud}+\sbb_{13}\nu^{23}\\
    \Omega^2 =& \sbb_{20} \nu^{\uz\ut}+\sbb_{10} \nu^{\uo\ut}+\sbb_{12} \nu^{13}+\sbb_{11}\nu^{1\ut}+\sx_{26} \nu^{\ud\ut}  +(-\ts{\frac{19}{5}} \sx_{19}-\sx_{22}) \nu^{23}+(\ts\frac{59}{5}\sx_{20}-\sx_{24}) \nu^{2\ut}\\
    &-\sbb_{21} \nu^{3\uz}-\sbb_{11} \nu^{3\uo}+(-\ts{\frac{19}{5}} \sx_{20}+\sx_{24}) \nu^{3\ud}+\sx_{25} \nu^{3\ut}\\
    \Theta^2 =& -\sbb_{21} \nu^{\uz\ud}-\sbb_{11} \nu^{\uo\ut}-\sbb_{13} \nu^{13}-\sbb_{12} \nu^{1\ut}+\sx_{24} \nu^{\ud\ut}+\sx_{21} \nu^{23}+\sx_{22} \nu^{2\ut}+\sbb_{22} \nu^{3\uz}+\sbb_{12} \nu^{3\uo}\\
    &+(8 \sx_{19}-\sx_{22}) \nu^{3\ud}+\sx_{23} \nu^{3\ut}\\
    \Omega^1 =& -\sbb_{20} \nu^{\uz\ud}+\sbb_{40} \nu^{\uz\ut}-\sbb_{10} \nu^{\uo\ud}-\sbb_{12} \nu^{12}-\sbb_{11} \nu^{1\ud}-\sx_{19} \nu^{13}-3 \sx_{20} \nu^{1\uz}+\sbb_{21} \nu^{2\uz} +\sbb_{11} \nu^{2\uo}\\
    &-11 \sx_{20} \nu^{2\ud}+\sx_{41} \nu^{\ud\ut}+\sx_{37} \nu^{23}+\sx_{38} \nu^{2\ut}+(-\sbb_{41}+\half\sbb_{30}) \nu^{3\uz}-4 \sx_{20} \nu^{3\uo}+\sx_{39} \nu^{3\ud}+\sx_{40} \nu^{3\ut}\\
    \Theta^1 =& \sbb_{21} \nu^{\uz\ud}+(-\half \sbb_{30}-\sbb_{41}) \nu^{\uz\ut}-4 \sx_{19} \nu^{1\ut}-11 \sx_{19} \nu^{2\ud}-3 \sx_{19} \nu^{3\uo}+\sx_{20} \nu^{\uo\ut}     +\sx_{32} \nu^{23}+\sx_{33} \nu^{2\ut}\\
    &+\sx_{34} \nu^{3\ud}+\sx_{35} \nu^{3\ut}+\sx_{36} \nu^{\ud\ut}+\sbb_{11} \nu^{\uo\ud}+\sbb_{12} \nu^{1\ud}-\sbb_{12} \nu^{2\uo}+\sbb_{13} \nu^{12}-\sbb_{22} \nu^{2\uz}+\sbb_{42} \nu^{3\uz}\\
    \Theta^0 =& -\ts\frac{36}{5}\sx_{20} \nu^{\uz\ut}-\sx_{16} \nu^{\uo\ut}-\sx_{18} \nu^{13}+(-\sx_{17}-\sbb_{50}) \nu^{1\ut}-\frac 32 \sbb_{50} \nu^{2\ud} +\sx_{31} \nu^{\ud\ut}+\sx_{27} \nu^{23}+\sx_{28} \nu^{2\ut}\\
    &\ts -\frac{36}{5}\sx_{19} \nu^{3\uz}+(\sx_{17}-\sbb_{50}) \nu^{3\uo}+\sx_{29} \nu^{3\ud}+\sx_{30} \nu^{3\ut}\\
  \end{aligned}
\end{equation}
for some functions $\sbb_{10},\cdots,\sbb_{52},\sx_{16}\cdots,\sx_{96}$ not all of which appear above. In particular the curvature part of the harmonic invariants is not included due to the length of the expressions involved. 
Assuming that such geometries are PCQ, as discussed in \ref{sec:36-inverse-probl-vari}, the full structure equations is obtained from the following components of the Cartan curvature $\Psi$ defined in \eqref{eq:Pairs-3-ODE-Cartan-curv}.
 \begin{equation}
   \label{eq:curv-2form-pair-3rd-ODE-variational}
   \begin{aligned}
    \Omega^3 =&\Theta^3=\Omega^2=\Theta^2=\Phi_1=0\\
    \Omega^1 =&  \sx_{37} \nu^{23}+ \sx_{39} \nu^{3\ud}+ \sx_{52} \nu^{\ud\ut}- \sx_{65} \nu^{2\ut}- \sx_{94} \nu^{3\ut}+ \sbb_{40} \nu^{\uz\ut}- \sbb_{41} \nu^{3\uz}\\
    \Theta^1 =& - \sx_{37} \nu^{2\ut}+ \sx_{39} \nu^{\ud\ut}- \sx_{70} \nu^{3\ud}+ \sx_{83} \nu^{23}- \sx_{91} \nu^{3\ut}- \sbb_{41} \nu^{\uz\ut}+ \sbb_{42} \nu^{3\uz}\\
    \Theta^0 =& - \sbb_{50}\nu^{1\ut}+  \sbb_{62} \nu^{23}-\textstyle{\frac 32} \sbb_{50} \nu^{2\ud}+ \sx_{58} \nu^{2\ut}-\sbb_{50}\nu^{3\uo}  +( \sx_{58}-2  \sbb_{61}) \nu^{3\ud}+ \sx_{88} \nu^{3\ut}+ \sbb_{60} \nu^{\ud\ut}\\  
    \Omega^0 =& - \sbb_{60}\nu^{3\uo} +(2  \sx_{58}-2  \sbb_{61}) \nu^{13}- \sbb_{50} \nu^{12}- \sbb_{60} \nu^{1\ut}+( \half \sx_{76}76- \half\sx_{75}+ \sx_{86}) \nu^{23}- \sbb_{60} \nu^{2\ud}\\
    &+ \sx_{73} \nu^{2\ut}     +(- \sx_{65}+2  \sx_{39}) \nu^{3\uz}+(- \sx_{73}+ \sx_{78}) \nu^{3\ud}- \sx_{96} \nu^{3\ut}+ \sx_{52} \nu^{\uz\ut}+ \sx_{55} \nu^{\ud\ut}\\
    \Phi_0 =& - \sx_{65} \nu^{\uz\ut}+\half \sbb_{50}\nu^{1\ud}  - \sx_{58} \nu^{1\ut}-\half \sbb_{50}\nu^{2\uo} - \sbb_{61}\nu^{2\ud} + \sx_{73}\nu^{\ud\ut} -\sx_{84}\nu^{23}+\sx_{70}\nu^{3\uz}\\ & +( \half \sx_{75}+\half \sx_{76}+ \sx_{86}+ \sx_{88}) \nu^{2\ut}  +( \sx_{58}-2  \sbb_{61}) \nu^{3\uo}+(- \sx_{76}- \sx_{86}- \sx_{88}) \nu^{3\ud}- \sx_{90}\nu^{3\ut} \\
    \Phi_2 =& (-2  \sx_{65}+2  \sx_{39}) \nu^{\uz\ut}- \sbb_{60} \nu^{\uo\ut}+\half \sbb_{50} \nu^{1\ud}  - \sbb_{62} \nu^{13}+(-3  \sx_{58}+2  \sbb_{61}) \nu^{1\ut}+\half \sbb_{50}\nu^{2\uo} \\
    &+(- \sx_{58}+ \sbb_{61}) \nu^{2\ud}+ \sx_{78} \nu^{\ud\ut}+( \sx_{82}+ \sx_{84}8) \nu^{23}+ \sx_{75} \nu^{2\ut}+(-2  \sx_{37}-2  \sx_{70}) \nu^{3\uz}\\
    &+(-3  \sx_{58}+4  \sbb_{61}) \nu^{3\uo}+ \sx_{76} \nu^{3\ud}+ \sx_{77} \nu^{3\ut}\\
    \Xi_{0} =& (-2  \sx_{37}- \sx_{70}) \nu^{\uz\ut}+ \sbb_{50}\nu^{\uo\ud} +(-2  \sx_{58}+2  \sbb_{61}) \nu^{\uo\ut}- \sx_{81} \nu^{23}- \sx_{82} \nu^{2\ut}- \sx_{83} \nu^{3\uz}\\
    &- \sx_{84} \nu^{3\ud}- \sx_{86} \nu^{\ud\ut}+ \sx_{89} \nu^{3\ut}+ \sbb_{62}\nu^{1\ut} + \sbb_{62}\nu^{2\ud} + \sbb_{62}\nu^{3\uo} \\
    \Gamma_0 =& (- \sx_{65}+ \sx_{39}) \nu^{\uz\ud}+ \sx_{94} \nu^{\uz\ut}- \sbb_{50} \nu^{1\uo}- \sbb_{60}\nu^{\uo\ud} -  \sbb_{62}\nu^{12} - \sx_{58}\nu^{1\ud} + \sx_{88} \nu^{1\ut}\\
    &+(- \sx_{37}- \sx_{70}) \nu^{2\uz}+(- \sx_{58}+2  \sbb_{61}) \nu^{2\uo}+( \half \sx_{76}+ \sx_{88}+\half  \sx_{75}) \nu^{2\ud}+  \sx_{88}\nu^{3\uo}+ \sx_{89} \nu^{23}\\
    &+ \sx_{90} \nu^{2\ut}- \sx_{90} \nu^{3\ud}+ \sx_{91} \nu^{3\uz}+ \sx_{93} \nu^{3\ut}+ \sx_{96} \nu^{\ud\ut}
  \end{aligned}
\end{equation}

The Cartan curvature of variational  orthopath geometries, as defined in Theorem \ref{thm:var-orthopath-equiv-prob}, is given by 
\begin{equation}
  \label{eq:orthopath-Cartan-curv-2form}
  \Psi=\exd\psi+\psi\w\psi=\begin{pmatrix}
  \textstyle{\Phi_0-\tfrac{1}{n+1}\Phi_1} & \Xi_0 & 0 \\
  \Omega^0 & \textstyle{\tfrac{n}{n+1}\Phi_1-\Phi_0} & 0  \\
  \Omega^a & \Theta^a& \Phi^a_b-\tfrac{1}{n+1}\Phi_1\delta^a_b \\
\end{pmatrix}
\end{equation}
where
\begin{equation}
  \label{eq:rewrite2-streqns-pathgeom}
  \begin{aligned}
  \Omega^a&=\ssB^a_{bc}\omega^b\w\theta^c\\ 
  \Omega^0&=\sN_{ab}\omega^a\w\omega^b-\sQ_{ab}\theta^a\w\omega^b\\
  \Theta^a&=-\ssA^a_{bc}\omega^b\w\theta^c+\sT^a_{b}\omega^0\w\omega^b+\half \sT^a_{bc}\omega^b\w\omega^c\\
\Phi_0&=\sN_{ba}\omega^a\w\theta^b-\sM_{a0}\omega^0\w\omega^a+\sM_{ab}\omega^a\w\omega^b\\
\Phi_1&=0\\
\Phi_{ab}&=\half\sT_{abij}\omega^i\w\omega^j+\sC_{abid}\omega^i\w\theta^d+\half\sE_{abcd}\theta^c\w\theta^d\\
\Xi_0&=-\sM_{ab}\omega^b\w\theta^a-\sM_{a0}\omega^0\w\theta^a-\sN_{ab}\theta^a\w\theta^b+T_{00b}\omega^0\w\omega^b+\half T_{0ab}\omega^a\w\omega^b,\\ 
\end{aligned}
\end{equation}
for some functions $\sA^a_{bc},\sB^a_{bc},\sN_{ab},\sQ_{ab},\sT^a_{b},\sT^a_{bc},\sT_{abij},\sC_{abid},\sE_{abcd},\sM_{ai},T_{ijb}$ on $\cG.$ 

\section*{Data availability statement}
No data was collected for any part of this article. 

\section*{Conflict of Interest}
The authors have no conflict of interest to declare that are relevant to this article. 

%-------------BIBLIOGRAPHY--------
\bibliographystyle{alpha}   
\bibliography{Quasicontactification}
\end{document}